\documentclass[xcolor=x11names,reqno,12pt]{amsart}
 \usepackage{caption}
 \usepackage{amsmath}
 \usepackage{amsthm}
 \usepackage{amssymb}
 \usepackage{latexsym,longtable}
  \usepackage{mathrsfs}
 \usepackage[vcentermath,enableskew,stdtext]{youngtab}
  \usepackage[left=1.1in,right=1.1in,top=1in,bottom=1.09in]{geometry}
 \usepackage[all]{xy}
    \SelectTips{cm}{10}     %use the nicer arrowheads
    \everyxy={<2.5em,0em>:} % Sets the scale I like
 \usepackage{fancyhdr}      %%%%%%%%%%%% Pagestyle stuff %%%%%%%%%%%%%%%
 \usepackage{enumitem}
 \usepackage{bm}
 \usepackage{stmaryrd}
 \usepackage{etex}
 \usepackage{tikz}
 \usepackage{float}
 \usepackage{array}
 \usepackage{mathtools}
 \restylefloat{figure}
\numberwithin{equation}{section}
\usepackage{url}
 \usepackage{tabmac}
 \usepackage{ytableau}

\usepackage{amsfonts,epsfig,color}

\allowdisplaybreaks[4]
\makeatletter
\newcommand*{\centerfloat}{%
  \parindent \z@
  \leftskip \z@ \@plus 1fil \@minus \textwidth
  \rightskip\leftskip
  \parfillskip \z@skip}
\makeatother
\newcounter{ctr}
\theoremstyle{plain}
\newtheorem{theorem}{Theorem}[section]
\newtheorem{lemma}[theorem]{Lemma}
\newtheorem{corollary}[theorem]{Corollary}
\newtheorem{proposition}[theorem]{Proposition}
\newtheorem{conjecture}[theorem]{Conjecture}

\theoremstyle{definition}
\newtheorem{definition}[theorem]{Definition}

\newtheorem{remark}[theorem]{Remark}
\newtheorem{example}[theorem]{Example}

\newcommand{\ignore}[1]{}

\renewcommand{\AA}{\ensuremath{\mathbb{A}}}

\newcommand{\B}{\ensuremath{\mathcal{B}}}
\newcommand{\cL}{\ensuremath{\mathcal{L}}}

\newcommand{\cC}{\ensuremath{\mathcal{C}}}
\newcommand{\CC}{\ensuremath{\mathbb{C}}}

\newcommand{\F}{\ensuremath{\mathcal{F}}}
\newcommand{\FF}{\ensuremath{\mathbb{F}}}
\newcommand{\g}{\ensuremath{\mathfrak{g}}}
\newcommand{\gl}{\ensuremath{\mathfrak{gl}}}

\newcommand{\HH}{\ensuremath{H}}

\newcommand{\h}{\ensuremath{\mathfrak{h}}}
\renewcommand{\hom}{\text{\rm Hom}}

\newcommand{\idelm}{\ensuremath{id}}

\renewcommand{\L}{\ensuremath{\mathscr{L}}}

\DeclareMathOperator{\modd}{mod}

\renewcommand{\O}{\ensuremath{\mathcal{O}}}

\newcommand{\QQ}{\ensuremath{\mathbb{Q}}}
\newcommand{\sR}{\ensuremath{\mathsf{R}}}

\renewcommand{\sl}{\ensuremath{\mathfrak{sl}}}
\newcommand{\sgn}{\text{\rm sgn}}

\newcommand{\fs}{\ensuremath{\mathfrak{s}}}

\newcommand{\tE}{\ensuremath{{\tilde{E}}}}

\newcommand{\ZZ}{\ensuremath{\mathbb{Z}}}
\newcommand{\ZZl}{\ensuremath{\ZZ^\ell}}
\newcommand{\ZZpl}{\ensuremath{\ZZ_{\ge 0}^\ell}}
\newcommand{\ZZx}{\ensuremath{\mathbb{Z}[x_1^{\pm 1}, \ldots, x_\ell^{\pm 1}]}}
\newcommand{\ZZxq}{\ensuremath{\mathbb{Z}[x_1^{\pm 1}, \ldots, x_\ell^{\pm 1}][[q]]}}

\newcommand{\ZZqqix}{\ensuremath{\AA[x_1^{\pm 1}, \ldots, x_\ell^{\pm 1}]}}

\newcommand{\be}{\begin{equation}}
\newcommand{\ee}{\end{equation}}
\newcommand{\Res}{\text{\rm Res}}
\renewcommand{\SS}{\ensuremath{\mathcal{S}}}
\newcommand{\zH}{\ensuremath{\mathcal{H}}} %zero Hecke monoid
\newcommand{\zaH}{\ensuremath{\widetilde{\mathcal{H}}}} %zero Hecke monoid for extended affine
\newcommand{\zs}{\ensuremath{\mathsf{s}}} %simple ref in zero Hecke
\newcommand{\tsr}{\ensuremath{\otimes}}

\newcommand{\eS}{\widehat{\SS}}

\newcommand{\Rez}{\text{\rm R}} %restrict and set to 0

\newcommand{\gd}{\ensuremath{\trianglerighteq}} %dominance order

\newcommand{\hatsl}{\ensuremath{\widehat{\sl}}}

\newcommand{\sh}{\text{\rm shape}}

\DeclareMathOperator{\poly}{poly}
\DeclareMathOperator{\wt}{wt}

\DeclareMathOperator{\Gr}{Gr}

\DeclareMathOperator{\chr}{char}

\DeclareMathOperator{\nr}{\mathbf{n}}

\DeclareMathOperator{\ns}{\mathbf{s}}

\DeclareMathOperator{\length}{length}  %for Weyl group since \ell conflict...may change to \ell though
\DeclareMathOperator{\CT}{CT}

\newcommand{\eroot}[1]{\ensuremath{\varepsilon_{#1}}}

\newcommand{\mynone}{~}

\newcommand{\cl}{{\text{\rm cl}}}
\newcommand{\aff}{{\text{\rm aff}}}
\DeclareMathOperator{\kat}{kat}
\DeclareMathOperator{\bottom}{bottom}

\DeclareMathOperator{\charge}{charge}
\DeclareMathOperator{\content}{content}
\DeclareMathOperator{\RowFrank}{RowFrank}
\DeclareMathOperator{\inv}{inv}

\DeclareMathOperator{\kattype}{kattype}

\DeclareMathOperator{\sort}{sort}

\newcommand{\Oint}{\ensuremath{\O_{\text{\rm int}}}}

\def\Tiny{\fontsize{6pt}{6pt}\selectfont}
\def\TTiny{\fontsize{4.8pt}{4pt}\selectfont}

\newcommand{\ce}{\ensuremath{\tilde{e}}} %crystal e
\newcommand{\cf}{\ensuremath{\tilde{f}}}

\newcommand{\crc}[1]{#1\star}

\newcommand{\rev}{\ensuremath{{\text{\rm rev}}}}

\newcommand{\SYT}{\text{\rm SYT}}

\newcommand{\SSYT}{\text{\rm SSYT}}

\newcommand{\Tabloids}{\text{\rm Tabloids}}
\newcommand{\AGD}{\text{\rm AGD}}
\newcommand{\sfb}{\ensuremath{\mathsf{b}}}
\newcommand{\sfp}{\ensuremath{\mathsf{p}}}
\newcommand{\sfc}{\ensuremath{\mathsf{s}}}
\newcommand{\tsfp}{\ensuremath{\tilde{\mathsf{p}}}}
\newcommand{\twist}[1]{\ensuremath{\F_{#1}}}
\DeclareMathOperator{\dtau}{\twist{\tau}}
\DeclareMathOperator{\dtaui}{\twist{\tau^{-1}}}

\newcommand{\wnotb}[2]{\ensuremath{\mathsf{w}_{[#1,#2)}}}
\newcommand{\wnota}[2]{\ensuremath{\mathsf{w}_{\vec{#1}}}}

\renewcommand{\top}{\text{\rm top}} %top of an i-string (E cannot be applied)

\newcommand{\spa}{\hspace{.3mm}}

\newlength{\mycellsize}
\newcommand\mytbl[1]{
\vcenter{
\let\\=\cr
\baselineskip=-16000pt \lineskiplimit=16000pt \lineskip=0pt
\halign{&\mytblcell{##}\cr#1\crcr}}}

\newcommand{\mytblcell}[1]{{%
\def \arg{#1}\def \void{}%
\ifx \void \arg
\vbox to \mycellsize{\vfil \hrule width \mycellsize height 0pt}%
\else \unitlength=\mycellsize
\begin{picture}(1,1)
\put(0,0){\makebox(1,1){$#1\vphantom{\crc{#1}}$}}
\put(0,0){\line(1,0){1}}
\put(0,1){\line(1,0){1}}
\put(0,0){\line(0,1){1}}
\put(1,0){\line(0,1){1}}
\end{picture}%
\fi}}

\newlength{\cellsize}
\newcommand\mytableau[1]{
\vcenter{
\let\\=\cr
\baselineskip=-16000pt \lineskiplimit=16000pt \lineskip=0pt
\halign{&\mytableaucell{##}\cr#1\crcr}}}

\newcommand{\mytableaucell}[1]{{%
\def \arg{#1}\def \void{}%
\ifx \void \arg
\vbox to \cellsize{\vfil \hrule width \cellsize height 0pt}%
\else \unitlength=\cellsize
\begin{picture}(1,1)
\put(0,0){\makebox(1,1){$#1\vphantom{\crc{#1}}$}}
\put(0,0){\line(1,0){1}}
\put(0,1){\line(1,0){1}}
\put(0,0){\line(0,1){1}}
\put(1,0){\line(0,1){1}}
\end{picture}%
\fi}}

\newcommand\boldtableau[1]{
\vcenter{
\let\\=\cr
\baselineskip=-16000pt \lineskiplimit=16000pt \lineskip=0pt
\halign{&\boldtableaucell{##}\cr#1\crcr}}}

\newcommand{\boldtableaucell}[1]{{%
\def \arg{#1}\def \void{}%
\ifx \void \arg
\vbox to \cellsize{\vfil \hrule width \cellsize height 0pt}%
\else \unitlength=\cellsize
\begin{picture}(1,1)
\put(0,0){\makebox(1,1){$\mathbf{#1\vphantom{\crc{#1}}}$}}
\put(0,0){\line(1,0){1}}
\put(0,1){\line(1,0){1}}
\put(0,0){\line(0,1){1}}
\put(1,0){\line(0,1){1}}
\end{picture}%
\fi}}

\title[Schur positivity of Catalan functions]{
Demazure crystals and the Schur positivity of Catalan functions
}
\begin{document}

\author{Jonah Blasiak}
\address{Department of Mathematics, Drexel University, Philadelphia, PA 19104}
\email{jblasiak@gmail.com}

\author{Jennifer Morse}
\address{Department of Math, University of Virginia, Charlottesville, VA 22904}
\email{morsej@virginia.edu}

\author{Anna Pun}
\address{Department of Mathematics, Drexel University, Philadelphia, PA 19104}
\email{annapunying@gmail.com}

\thanks{Authors were supported by NSF Grants DMS-1855784 (J.~B.)
and DMS-1855804 (J.~M.).}

\begin{abstract}
Catalan functions, the graded Euler characteristics of certain vector bundles on the flag variety,
are a rich class of symmetric functions
which include $k$-Schur functions and parabolic Hall-Littlewood polynomials.
We prove that Catalan functions indexed by partition weight are the characters of $U_q(\widehat{\mathfrak{sl}}_\ell)$-generalized Demazure
crystals as studied by  
Lakshmibai-Littelmann-Magyar and Naoi.
We obtain
Schur positive formulas for these functions, settling conjectures of Chen-Haiman 
and Shimozono-Weyman. 
Our approach more generally gives key positive formulas for
graded Euler characteristics of certain vector bundles on Schubert varieties by matching them to characters of generalized Demazure crystals.
\end{abstract}
\maketitle

\section{Introduction}
\label{s intro}

The \emph{Kostka-Foulkes polynomials} $K_{\lambda\mu}(q)$ originated in the character theory of $GL_\ell(\FF_q)$
and their study has since flourished.
They express the modified Hall-Littlewood polynomials in the Schur basis of the ring of symmetric functions,
$H_\mu(\mathbf{x};q) = \sum_\lambda K_{\lambda\mu}(q) \,s_\lambda(\mathbf{x})$,
and are $q$-weight multiplicities defined via a $q$-analog of Kostant's partition function $\mathcal P$:
\begin{equation*}
%\label{kofo}
K_{\lambda\mu}(q)=\sum_{w \in \SS_\ell}\sgn(w)\mathcal P_q(w(\lambda+\rho)-(\mu+\rho))
\quad\text{for}\quad
%\prod_{\text{positive roots }\alpha} \!\frac{1}{(1-tx^\alpha)} = \sum_\beta \mathcal P_q(\beta) x^\beta\,,
\prod_{\alpha\in\Delta^+} \!\frac{1}{(1-q \mathbf{x}^\alpha)} = \sum_{\gamma \in \ZZ^\ell} \mathcal P_q(\gamma) \mathbf{x}^\gamma\,.
\end{equation*}
The positivity property, $K_{\lambda\mu}(q) \in \ZZ_{\ge 0}[q]$,
has deep geometric and combinatorial
significance:
$K_{\lambda\mu}(q)$
are affine Kazhdan-Lusztig polynomials~\cite{Lusztiggreenpoly,L},
give characters of cohomology rings of Springer fibers~\cite{HottaSpringer,Springerconstructionreps},
%for $GL_\ell$
record the Brylinski filtration of weight spaces~\cite{Brylinski},
and are sums over tableaux weighted by the Lascoux-Sch\"utzenberger charge statistic~\cite{LSfoulkes}.

A broader framework has emerged over the last decades.
%A expansive picture has emerged over the last decades.
Broer \cite{BroerNormality} and Shimozono-Weyman \cite{SW},
%?note: SW made explicit in terms of roots I think
in their study of nilpotent conjugacy class closures,
replaced the set of all positive roots
$\Delta^+$ by a {\it parabolic} subset---the roots $\Delta(\eta) \subset \Delta^+$ above a block diagonal matrix.
 %$\Delta^+$ with a subset of roots $\Delta(\eta)$ above a block diagonal matrix.
Panyushev~\cite{Panyushev} and Chen-Haiman \cite{ChenThesis} went further,
taking
any one of Catalan many upper order ideals $\Psi\subset\Delta^+$.
The associated symmetric \mbox{{\it Catalan functions}}, $H(\Psi; \mu)(\mathbf{x};q) = \sum_\lambda K^\Psi_{\lambda \mu}(q)\, s_\lambda(\mathbf{x})$,
indexed by $\Psi$ and partition $\mu$,
are graded Euler characteristics of vector bundles on the flag variety.

The broader scope deepened ties to Kazhdan-Lusztig theory, advanced by the discovery of
LLT polynomials~\cite{LLT,LT00,KirillovShimozono,GH}, and inspired a generalization of Jing's Hall-Littlewood
vertex operators~\cite{SZ}.  Catalan functions were connected to spaces of coinvariants
of fusion products in the WZW theory~\cite{FJKLM, FL},
$k$-Schur functions and Gromov-Witten invariants~\cite{BMPS2,BMPS,LMksplit,LMquantum},
%homology rperesentatives for affine the Grassmannian,,
and affine crystals \cite{LOS, NYenergy, SchillingWarnaar, Shimozonoaffine}.
Positivity remained a central theme;
extending earlier work of Broer, Chen-Haiman posed

\begin{conjecture}
\label{ec intro schur pos}
The Catalan functions $H(\Psi; \mu)$ are Schur positive:
$K_{\lambda\mu}^\Psi(q)\in\mathbb Z_{\geq 0}[q]$.
\end{conjecture}

The picture in the {\it dominant rectangle} case, when $\Psi = \Delta(\eta)$ {\it and} $\mu$ is
constant on parabolic blocks, is beautifully complete.
These Catalan functions were equated with characters of $U_q(\hatsl_\ell)$-Demazure
crystals \cite{Shimozonoaffine}, and Schur positive formulas were established
using Kirillov-Reshetikhin (KR) crystals~\cite{SchillingWarnaar, Scyclageposet} and
%Kirillov-Schilling-Shimozono
rigged configurations~\cite{KSS}.
The view of Catalan functions as  Euler characteristics
ties their positivity to a conjecture on higher cohomology vanishing, first posed by
Broer in the parabolic case and later extended by Chen-Haiman to arbitrary  $\Psi$ and partition $\mu$; it was
settled by Broer \cite{BroerNormality} in the dominant rectangle case.

The cohomology of vector bundles associated to Catalan functions,
%geometric picture surrounding Catalan functions,
particularly for $\Psi = \Delta^+$, has been extensively studied
\cite{Broer, BroerNormality,Brylinski, Hague,Hs1, KLT, MehtavanderKallen, Panyushev}.
%The main tool used in 
Hague \cite{Hague} extended Broer's cohomology vanishing result %for the dominant rectangle case 
to some other classes of weights in the parabolic case, using Grauert-Riemenschneider vanishing and
Frobenius splitting results of Mehta and van der Kallen \cite{MehtavanderKallen}.
Panyushev \cite{Panyushev}
established higher cohomology
vanishing
for a large subclass of Catalan functions; %using Grauert-Riemenschneider vanishing;
it includes the case
$\mu$ is strictly decreasing and $\Psi$ arbitrary.
Nonetheless, for arbitrary partitions $\mu$,
the vanishing
conjecture remains open  even for parabolic  $\Psi$.

The gold standard is to settle Conjecture~\ref{ec intro schur pos} with a manifestly positive
formula. Many attempts to extend the Lascoux-Sch\"utzenberger charge formula for Kostka-Foulkes
polynomials were made.
Shimozono-Weyman \cite{SW} conjectured such a formula
for the parabolic Catalan functions  $H(\Delta(\eta);\mu)$,
hinging on an intricate tableau procedure
called  {\it katabolism}.
Soon after, katabolism led to the origin of $k$-Schur
functions~\cite{LLM}, and
more recently, Chen-Haiman~\cite{ChenThesis}
proposed a variant of katabolism to solve Conjecture~\ref{ec intro schur pos} completely.
However, katabolism offered no traction for proofs.

We are now able to paint the picture in its entirety by moving to a larger framework of
\emph{tame nonsymmetric Catalan functions}  $H(\Psi;\mu;w)$,
depending on an additional input $w\in \SS_\ell$;
%with an additional index $w\in \SS_\ell$;
they are Euler characteristics of vector bundles on Schubert varieties
and specialize to Catalan functions when  $w = w_0$.  Our findings include
%the following results.

\newcommand{\myemph}[1]{\textsf{#1}}

\vspace{-.6mm}
\begin{itemize}[leftmargin=.8cm]
%\begin{itemize}[leftmargin=.92cm]
\item[(1)]
\myemph{Tame nonsymmetric Catalan functions are characters of $U_q(\hatsl_\ell)$-generalized Demazure crystals},
subsets of $B(\Lambda^1) \tsr \cdots \tsr B(\Lambda^p)$ of the form
\mbox{$\F_{w_p} \big( \cdots  \F_{ w_2}\big( \F_{w_1} (u_{\Lambda^1})
  \! \tsr \! u_{\Lambda^2} \big) \cdots \! \tsr \! u_{\Lambda^p} \big)$}
%  \mbox{$\F_{w_1} \big( \F_{ w_2}\big( \cdots   \F_{w_p} (u_{\Lambda^p}) \cdots
%  \! \tsr \! u_{\Lambda^2} \big) \! \tsr \! u_{\Lambda^1} \big)$}
%$\F_{w_1} \big( u_{\Lambda^1} \tsr \F_{w_{2}} \big( u_{\Lambda^{2}}  \tsr \cdots \F_{w_p} (u_{\Lambda^p}) \big) \big)$,
where $B(\Lambda^i)$ is a highest weight crystal, $w_i = s_{j_1}\cdots s_{j_k}$ lies in the affine symmetric group $\eS_\ell$,
and $\F_{w_i}(S) = \bigcup_{b \in S; a_1, \ldots, a_k \ge 0} \cf_{j_1}^{a_1} \cdots \cf_{j_k}^{a_k} b$.
Lakshmibai-Littelmann-Magyar \cite{LakLitMag} introduced these crystals in their study of Bott-Samelson varieties.
\\[-3mm]

\item[(2)]
\myemph{Tame nonsymmetric Catalan functions are key positive}, implying and generalizing Conjecture~\ref{ec intro schur pos}.
By the  powerful theory of Demazure crystals \cite{Josephdemazurecrystal, KashiwaraDemazure, LakLitMag, LittelmannCrystal},
$U_q(\hatsl_\ell)$-generalized Demazure crystals restrict to disjoint unions of $U_q(\sl_\ell)$-Demazure crystals,
implying that their characters are key positive.
\\[-3mm]

\item[(3)]
\myemph{Positive combinatorial formulas for the key coefficients of (2)}.
We draw on techniques of Naoi~\cite{NaoiVariousLevels}
to match generalized Demazure crystals with
a family of {\it \mbox{DARK crystals}},
 Demazure-like subsets of tensor products of
KR crystals.
%, where the combinatorics is simpler.
Explicit katabolism combinatorics arises  naturally by
 unraveling the $\F_{w_i}$
 %and tensor
 operators on the DARK side.
%find that miraculously katabolism arises  naturally,
%and explicit combinatorial formulas and
%We deduce
\\[-3.4mm]

\item[(4)]
%its construction %(in the spirit of \cite{LakLitMag,LakMagGL}).\\
\myemph{A katabolism tableau formula for Catalan functions}.
 In the parabolic case,
 it agrees with and settles the Shimozono-Weyman conjecture.
%this comes out of an algorithm which detects membership
\\[-3.2mm]

\item[(5)]
\myemph{A conjectural module-theoretic strengthening of (2)}, generalizing the earlier
higher cohomology vanishing conjectures of Broer and Chen-Haiman.
\\[-3.2mm]

\item[(6)]
\myemph{The  $t=0$ nonsymmetric Macdonald polynomials $E_\alpha(\mathbf{x}; q, 0)$
are tame nonsymmetric Catalan functions.}
Dating back to Sanderson~\cite{Sanderson}, the $E_\alpha(\mathbf{x}; q, 0)$ are characters
of certain  $U_q(\hatsl_\ell)$-Demazure crystals.
This topic has recently regained popularity~\cite{Alexanderssondemazurelusztig,AlexanderssonSawhneymajorindex,
AssafNSMacdonald,AssafGonzalez,AssafGonzalez2, LNSSS3, LNSSSfpsac, OrrShimozononsmac},
and in particular Assaf-Gonzalez~\cite{AssafGonzalez,AssafGonzalez2} gave a key positive formula
for $E_\alpha(\mathbf{x}; q, 0)$.
Our results yield
a different key positive formula,
which generalizes Lascoux's tableau formula for
cocharge Kostka-Foulkes polynomials~\cite{La}.
%and solves a problem of Lascoux \cite{LascouxBookpolynomials}.
\end{itemize}

\section{Main results}

The basic approach of~\cite{BMPS} is to open the door to powerful inductive techniques
by realizing $k$-Schur functions as a subclass of (symmetric) Catalan functions.
In a similar spirit,
our inductive approach here depends crucially on viewing
%identifying the suitable general
the Catalan functions as a subclass of a larger family
of nonsymmetric Catalan functions.

Nonsymmetric Catalan functions are Euler characteristics of vector bundles on Schubert varieties
%are elements of the ring $\ZZ[q][\mathbf{x}] = \ZZ[q][x_1, \dots, x_\ell]$ and can be
and can be defined by a Demazure operator formula.
Fix $\ell \in \ZZ_{\ge 0}$. % be a nonnegative integer.
The symmetric group $\SS_\ell$ acts on
$\ZZ[q][\mathbf{x}] = \ZZ[q][x_1, \dots, x_\ell]$
by
%permuting variables;
%the simple transposition $s_i \in \SS_\ell$ swaps $x_i$ and $x_{i+1}$.
permuting the $x_i$; let $s_i \in \SS_\ell$ denote the simple transposition which swaps $i$ and $i+1$.
Let $\zH_\ell$ denote the 0-Hecke monoid of $\SS_\ell$ with generators $\zs_1, \ldots, \zs_{\ell-1}$. It is obtained from  $\SS_\ell$
by replacing the relations $s_i^2 = \idelm$ with $\zs_i^2 = \zs_i$.
For $i \in [\ell-1] := \{1,2,\dots, \ell-1\}$, the \emph{Demazure operator} $\pi_i$ is the linear operator on $\ZZ[q][\mathbf{x}]$ defined by
\begin{align}
%\partial_i(f) &= \frac{f-s_i(f)}{x_i-x_{i+1}},\\
%\pi_i(f) &= \partial_i (x_if) \\
\label{e intro pi def}
\pi_i(f) &= \frac{x_if-x_{i+1}s_i(f)}{x_i-x_{i+1}}.
\end{align}
%where  $f$ denotes an arbitrary element of $\ZZ[q][\mathbf{x}]$.
More generally, for any $w \in \zH_{\ell}$, let $w=\zs_{i_1}\zs_{i_2}\cdots \zs_{i_m} $
and define the associated \emph{Demazure operator} by
$\pi_w := \pi_{i_1}\pi_{i_2}\cdots \pi_{i_m}$;
this is well defined as the $\pi_i$ satisfy the 0-Hecke relations.

%??def key here? No, too much clutter.

A \emph{root ideal} is an upper order ideal
of the poset  $\Delta^+ = \Delta^+_\ell := \{(i,j) \mid 1 \le i < j \le \ell \big\}$
with partial order given by $(a,b) \leq (c,d)$ when $a\geq c$ and $b\leq d$.
A {\it labeled root ideal of length $\ell$}
is a triple $(\Psi, \gamma, w)$ consisting of a root ideal $\Psi\subset\Delta_\ell^+$, a weight $\gamma\in\ZZ^\ell$, and $w \in \zH_\ell$.

\begin{definition}
\label{d HH gamma Psi}
The \emph{nonsymmetric Catalan function} associated to
the labeled root ideal $(\Psi, \gamma, w)$ of length  $\ell$ is
\begin{align}
\label{e d HH gamma Psi}
H(\Psi;\gamma; w)(\mathbf{x};q) := \pi_w \Big( \poly\Big(  \prod_{(i,j) \in \Psi} \big(1-q x_i/x_j\big)^{-1} \mathbf{x}^\gamma \Big) \Big)  \, \in \ZZ[q][\mathbf{x}],
\end{align}
%where  $\mathbf{x}^\gamma := x_1^{\gamma_1} \cdots x_\ell^{\gamma_\ell}$, and 
where $\poly$ denotes the polynomial truncation operator,
 defined by its action on key polynomials:
$\poly(\kappa_\alpha) = \kappa_\alpha$ for $\alpha \in \ZZ_{\ge 0}^\ell$ and $\poly(\kappa_\alpha)=0$ for $\alpha \in \ZZ^\ell\setminus \ZZpl$
(see \S\ref{s Rotation identity}).
\end{definition}

In the case $w = \mathsf{w}_0$, the longest element in $\zH_\ell$,
we recover the (symmetric) Catalan functions studied in \cite{BMPS2, BMPS,ChenThesis, Panyushev}.

\subsection{The rotation theorem}

For a root ideal  $\Psi \subset \Delta^+_\ell$, define
%following tuple which
%%recoinformally is the number of "non
%records the size of the stuff below  $\Psi$ row by row
%%of the corresponding lower order ideal in each row:
the tuple
$\nr(\Psi) = (\nr(\Psi)_1,\dots, \nr(\Psi)_{\ell-1})$ $\in [\ell]^{\ell-1}$ by
\begin{align}
\label{e intro nr def}
\nr(\Psi)_i &:= \big|\big\{j \in \{i,i+1, \dots, \ell\} : (i,j) \spa \notin \spa \Psi\big\}\big|.
%\s(\Psi)_i &:= \big|\big\{(i,j) \in \Delta^+ \setminus \Psi : j \in [\ell] \big\}\big| + 1,\\
\end{align}
%??see example  if space

\begin{definition}%[Tame nonsymmetric Catalan functions]
A labeled root ideal $(\Psi, \gamma, w)$ of length  $\ell$ is \emph{tame}
if the right descent set $\{i \in [\ell-1] \mid w \spa \zs_i = w\}$ of $w$
contains $\nr(\Psi)_1+1, \nr(\Psi)_1+2,\dots, \ell-1$;
informally, this means that $\pi_w$ symmetrizes the columns which intersect $\Psi$.
%A \emph{tame nonsymmetric Catalan function} is a nonsymmetric Catalan function associated to
%a tame labeled root ideal.
%In this case, we also say that the associated $H(\Psi;\gamma; w)$ is a \emph{tame nonsymmetric Catalan function}.
\emph{We also say that the associated nonsymmetric Catalan function is tame}.
\end{definition}

%We have a good understanding of the following class of nonsymmetric Catalan functions.

%The operator  $\Phi$ on $\ZZ[q][\mathbf{x}]$ is defined by
%\begin{align}
%\label{e intro Phi operator}
%\Phi(f) := f(x_2, \dots, x_\ell, qx_1).
%\end{align}
Define the  $\ZZ[q]$-algebra homomorphism $\Phi$ of  $\ZZ[q][\mathbf{x}]$ by
\begin{align}
\label{e intro Phi operator}
\Phi(x_i) =  x_{i+1} \text{ for  $i \in [\ell-1]$}, \quad  \Phi(x_\ell) = qx_1.
\end{align}

A crucial finding of this paper is the following operator formula for tame nonsymmetric Catalan functions.

%The operator  $\Phi$ arises in a recurrence for nonsymmetric Macdonald polynomials.
%A crucial  finding of this paper is that it can be used to give an operator formula for tame nonsymmetric Catalan functions.
%(We will see in Section \ref{s consequences for nonsymmetric macdonald polynomials}
%that the connection to
%nonsymmetric Macdonald polynomials is no coincidence.)

%
%We show that any tame nonsymmetric Catalan function $H(\Psi; \gamma; w)$ with  $\gamma_1 \ge 0$ is equal to
%$x_1^{\gamma_1} \Phi$ applied to a smaller nonsymmetric Catalan function obtained by
%removing the first part of  $\gamma$ and the first row of $\Psi$ (Theorem \ref{t ns Catalan rotate}).
%Iterating this yields

\begin{theorem}
\label{t character rotation formula}
For any  tame labeled root ideal $(\Psi, \gamma,w)$
with  $\gamma \in \ZZ_{\ge 0}^\ell$,
\begin{align}
\label{ec character rotation formula}
H(\Psi; \gamma; w) = \pi_w \spa x_1^{\gamma_1}\spa \Phi \spa\pi_{\sfc(n_1)} x_1^{\gamma_2}\spa \Phi \spa \pi_{\sfc(n_2)} x_1^{\gamma_3} \cdots \spa \Phi \spa \pi_{\sfc(n_{\ell-1})}  x_1^{\gamma_\ell} \spa ,
\end{align}
where $(n_1, \dots, n_{\ell-1}) = \nr(\Psi)$ and
 $\sfc(d) := \zs_{\ell-1}\zs_{\ell-2} \cdots \zs_d \in \zH_\ell$ for $d \in [\ell]$.
%\begin{align}
%\label{ec character rotation formula}
%H(\Psi; \gamma; w) = \pi_w  \spa x_1^{\gamma_1}\Phi \spa \pi_{\sfc(n_1)} \spa x_1^{\gamma_2}\Phi \spa \pi_{\sfc(n_2)} \spa x_1^{\gamma_3} \cdots \Phi \spa \pi_{\sfc(n_b)} \spa x_1^{\gamma_b} x_2^{\gamma_{b+1}} \cdots x_{\ell-b+1}^{\gamma_{\ell}},
%\end{align}
%where $b$ is the number of nonempty rows of $\Psi$.
\end{theorem}

Its proof requires an in-depth understanding of polynomial truncation and is given in Section \ref{s Rotation identity}.
The operator  $\Phi$ arises in a recurrence for nonsymmetric Macdonald polynomials, and we will see in
Section \ref{s consequences for nonsymmetric macdonald polynomials} that its appearance here is no coincidence.
% as nonsymmetric Catalan functions contain the  $t=0$ nonsymmetric Macdonald polynomials as a subclass.)

%the following theorem giving an elegant expression for the tame nonsymetric Catalan functions in terms
%of  $\Phi$ and the $\pi_i$'s.
%Note that it is crucial to leave the class of symmetric Catalan functions %(even if that is our main case of interest)
%for this induction to work.
%The nonsymmetric generality allows us to give an inductive formula for a large class of nonsymmetric Catalans
%Also, somewhat surprisingly, the right side of \eqref{ec character rotation formula} is automatically polynomially truncated,
%whereas the definition of  $H(\Psi; \gamma; w)$ includes an explicit polynomial truncation.

\subsection{Affine generalized Demazure crystals and key positivity}
\label{ss Generalized affine Demazure crystals and key positivity}

%Further properties of tame nonsymmetric Catalan functions are obtained by connecting the right side of \eqref{ec character rotation formula} to affine Demazure crystals.  We give a brief version of the crystal background needed, deferring
%a thorough treatment to \S\ref{s background on crystals}.

Theorem~\ref{t character rotation formula}  allows us to connect
tame nonsymmetric Catalan functions with affine Demazure crystals.
We describe this connection here,
but defer a
%some technical details and background to
thorough treatment of crystals to Section \ref{s background on crystals}.

Let $U_q(\g)$ be the quantized enveloping algebra of a symmetrizable Kac-Moody Lie algebra $\g$
(as in \cite{KashiwaraSurvey}).
Among the data specifying a  $U_q(\g)$-crystal $B$ are maps
$\cf_i \colon  B \sqcup \{0\} \to B \sqcup \{0\}$ for  $i$
ranging over the Dynkin node set  $I$.
For a subset $S$ of %a $U_q(\hatsl_{\ell})$ or $U'_q(\hatsl_{\ell})$-crystal
$B$ and $i \in I$, define
\[\F_i \spa S := \{\cf_i^m b \mid b\in S, m \ge 0 \} \setminus \{0\} \subset B.\]
For a dominant integral weight $\Lambda \in P^+$,
let $B(\Lambda)$ denote the highest weight $U_q(\g)$-crystal
of highest weight $\Lambda$ and $u_\Lambda$ its
highest weight element.

\begin{definition}
\label{d demazure crystal}
A $U_q(\g)$-\emph{Demazure crystal} is a subset of a  highest weight $U_q(\g)$-crystal  $B(\Lambda)$ of the form
$\F_{i_1} \cdots \F_{i_k} \{u_{\Lambda}\}$.
\end{definition}

Now specialize to  $\g = \hatsl_\ell$, our focus here.
The associated data includes Dynkin nodes $I = \ZZ/\ell\ZZ = \{0,1,\dots, \ell-1\}$, fundamental weights  $\{\Lambda_i \mid i \in I\}$, weight lattice $P = \sum_{i\in I} \ZZ \Lambda_i \oplus \ZZ\frac{\delta}{2\ell}$,
and dominant weights
$P^+ = \sum_{i\in I} \ZZ_{\ge 0}\Lambda_i + \ZZ\frac{\delta}{2\ell} \subset P$.
Let $\tau$ denote the Dynkin diagram
automorphism $I \to I, \, i \mapsto i+1$.
Let $\widetilde{\SS}_\ell$ denote the extended affine symmetric group
and $\zaH_\ell$ its 0-Hecke monoid.
%??note \widetilde{\SS}_\ell is not used until quite a bit later, so no need to say anything but the bare minimum in the intro
The generators of  $\zaH_\ell$ are denoted $\tau$ and $\zs_i$ $(i \in I)$, and relations include
$\tau \zs_i \tau^{-1}= \zs_{\tau(i)} = \zs_{i+1}$, braid relations,
and $\zs_i^2 =\zs_i$.

\begin{definition}
\label{d mathcal D}
%A subset  $S$ of a $U_q(\g)$-crystal is $U_q(\g)$-\emph{Demazure-like} if it is isomorphic to
%a disjoint union of $U_q(\g)$-Demazure crystals.
Let  $\mathcal{D}(\hatsl_\ell)$ be the set of all subsets $S \subset B$
such that  %$B = B(\Lambda^1) \tsr \cdots \tsr B(\Lambda^p)$,
  $B$ is a tensor product of highest weight $U_q(\hatsl_\ell)$-crystals
and the image of $S$ under $B \xrightarrow{\cong} \bigsqcup_{\Lambda \in M} B(\Lambda)$
%??really the pair S,B
is a  disjoint union of $U_q(\hatsl_\ell)$-Demazure crystals.  Here, $M$ is a multiset of elements of  $P^+$.
%, i.e.,
% $\theta(S)$ is a disjoint union of  $U_q(\g)$-Demazure crystals  $\theta \colon B \xrightarrow{\cong} \bigsqcup_i B(\Lambda^i)$.
\end{definition}

For  $\Lambda \in P^+$, define the bijection of sets $\twist{\tau} \colon B(\Lambda) \to B(\tau(\Lambda))$ by
$\cf_{j_{1}}^{d_{1}} \cdots \cf_{j_{k}}^{d_{k}}(u_\Lambda) \mapsto \cf_{\tau(j_{1})}^{d_{1}} \cdots \cf_{\tau(j_{k})}^{d_{k}}(u_{\tau(\Lambda)})$, for any  $j_1,\dots, j_k \in I$ and $d_i \in \ZZ_{\ge 0}$;
%??that this is well-defined is equivalent to B(tau \Lambda) being iso to its crystal twist, which follows from Kashiwara's uniqueness for local crystal bases
we also denote by  $\twist{\tau}$ the bijection $B(\Lambda^1) \tsr \cdots \tsr B(\Lambda^p) \xrightarrow{\twist{\tau} \tsr \cdots \tsr \twist{\tau}}
B(\tau(\Lambda^1)) \tsr \cdots \tsr B(\tau(\Lambda^p))$, for  $\Lambda^1,\dots, \Lambda^p \in P^+$.
We can regard $\F_i$  $(i \in I)$ and $\twist{\tau}$ as operators on  $\mathcal{D}(\hatsl_\ell)$
and as such they satisfy the relations of $\zaH_\ell$ (by \cite{KashiwaraDemazure, NaoiVariousLevels}---see \S\ref{ss hatsl demazure crystals}).
This yields a well-defined operator
$\F_w \colon  \mathcal{D}(\hatsl_\ell) \to \mathcal{D}(\hatsl_\ell)$ for any
$w \in \zaH_\ell$.
%??cite kashiwara here or nearby? Yes done
For  $\Lambda \in P^+$ and $w \in \zaH_\ell$,
denote by $B_w(\Lambda) = \F_{w} \{u_\Lambda\}$ the associated $U_q(\hatsl_\ell)$-Demazure crystal.

\begin{theorem}[{Combinatorial Excellent Filtration~\cite{Josephdemazurecrystal,LakLitMag}}]
\label{t monomial times Demazure}
For any $\Lambda^1, \Lambda^2 \in P^+$ and  $w \in \zaH_\ell$,
$ B_w(\Lambda^2) \tsr u_{\Lambda^1}$ is isomorphic to a disjoint union of  $U_q(\hatsl_\ell)$-Demazure crystals.
\end{theorem}

A  $U_q(\hatsl_\ell)$-\emph{generalized Demazure crystal} is a subset of
a tensor product of highest weight crystals of the form
%??because of twists, hard to name which crystal it lives in
$\F_{w_1} \big( \F_{ w_2}\big( \cdots  \F_{ w_{p-1}} \big( \F_{w_p} \{u_{\Lambda^p }\} \tsr u_{\Lambda^{p-1}} \big) \cdots
  \tsr  u_{\Lambda^2} \big) \tsr  u_{\Lambda^1}  \big)$
%$\F_{w_1} \big( u_{\Lambda^1} \tsr \F_{w_{2}} \big( u_{\Lambda^{2}}  \tsr \cdots \tsr \F_{w_p} (u_{\Lambda^p}) \big) \big)$
for some $\Lambda^1, \dots, \Lambda^p \in P^+$ and  $w_1, \dots, w_p \in \zaH_\ell$.
Theorem \ref{t monomial times Demazure} and the well-definedness of $\F_w$ on  $\mathcal{D}(\hatsl_\ell)$ show that
these are well-defined and yield the following corollary
(this argument is essentially due to \cite{LakLitMag},
with the extended affine setup treated carefully in \cite{NaoiVariousLevels}).

\begin{corollary}
\label{c monomial times Demazure}
Any $U_q(\hatsl_\ell)$-generalized Demazure crystal is isomorphic to a disjoint union of $U_q(\hatsl_\ell)$-Demazure crystals.
\end{corollary}

%The theorem first statement is known as the combinatorial excellent filtration theorem, and the consequence
%for generalized Demazure crystal is proven carefully in the extended affine generality by Naoi (see \cite[Proposition 4.4]{NaoiVariousLevels}).

Our focus is on the following subclass of
$U_q(\hatsl_\ell)$-generalized Demazure crystals:
for $\mathbf{w} = (w_1, w_2, \dots, w_p) \in (\zH_\ell)^p$ and a partition $\mu = (\mu_1 \ge \cdots \ge \mu_p \ge 0)$,
define the associated \emph{affine generalized Demazure (AGD) crystal} by
\begin{multline}
\label{e AGD def}
\!\!\! \AGD(\mu;\mathbf{w}) :=
\F_{w_1} \big( \F_{\tau w_2}\big( \cdots  \F_{\tau w_{p-1}} \big( \F_{\tau w_p} \{u_{\mu^p\Lambda_1} \} \tsr u_{\mu^{p-1}\Lambda_{1}} \big) \cdots
  \tsr  u_{\mu^2\Lambda_{1}} \big) \tsr  u_{\mu^1\Lambda_{1}}  \big) \\
\subset
B(\mu^p\Lambda_{p}) \tsr \cdots \tsr B(\mu^1 \Lambda_1)\,, \end{multline}
where  $\mu^i = \mu_i-\mu_{i+1}$, with  $\mu_{p+1}:=0$.

Let $\ZZ[P]$ denote the group ring of $P$ with  $\ZZ$-basis $\{ e^\lambda \}_{\lambda \in P}$.
The \emph{character} of a $U_q(\hatsl_\ell)$-crystal $G$ is
$\chr(G) := \sum_{g \in G} e^{\wt(g)} \in \ZZ[P]$.
Define the ring homomorphism  $\zeta$ by
\begin{align}
\label{e intro P0 iso}
\zeta \colon  \ZZ[q][\mathbf{x}] \to \ZZ[P], \ \ \ x_i \mapsto e^{\Lambda_i - \Lambda_{i-1} + \frac{\ell + 1 - 2i}{2\ell}\delta} \, , \ q \mapsto e^{-\delta}.
%\cong \ZZqqix / (x_1x_2 \cdots x_\ell-1),
%e^{\varpi_i} \mapsto x_1 x_2 \cdots x_i, \text{ and } e^{-\delta/2\ell} \mapsto q^{1/2\ell}\\
\end{align}
Let $\nr(\Psi)$ be as in \eqref{e intro nr def}
and $\sfc(d) = \zs_{\ell-1}\zs_{\ell-2} \cdots \zs_d$.
For a root ideal  $\Psi$, set
\begin{align}
\label{e ns def}
\ns(\Psi) := (\sfc(\nr(\Psi)_1), \dots, \sfc(\nr(\Psi)_{\ell-1})) \in (\zH_\ell)^{\ell-1}.
\end{align}

%Using Corollary \ref{c monomial times Demazure} and Kashiwara's theory of Demazure crystals \cite{KashiwaraDemazure},
%one readily obtains a Demazure operator formula for $\chr(\AGD(\mu;\mathbf{w}))$.
%This is then not difficult to connect to \eqref{ec character rotation formula}, giving one of our main results:

%This yields the following result (see \S\ref{} for details):

\begin{theorem}
\label{t intro kat conjecture resolution 0}
Tame nonsymmetric Catalan functions of partition weight are characters of AGD crystals:
for any tame labeled root ideal $(\Psi, \mu, w)$ of length $\ell$ with partition $\mu$,
%Set $\mathbf{w} = (w, \ns(\Psi)) \in (\zH_\ell)^\ell$.
\begin{align}
\label{ec intro kat conjecture resolution 0}
\zeta(H(\Psi; \mu; w)) =  e^{-\mu_1\Lambda_0+n_\ell(\mu)\delta}\chr \! \big( \spa \AGD(\mu;(w,\ns(\Psi))) \spa \big),
\end{align}
where $n_\ell(\mu) =  \frac{|\mu|\spa (\ell-1)}{2\ell}- \frac{1}{\ell} \sum_{i=1}^\ell (i-1)\mu_i$.
\end{theorem}
\begin{proof}[Proof sketch]
From Corollary \ref{c monomial times Demazure} and Kashiwara's results on Demazure crystals \cite{KashiwaraDemazure},
one readily obtains a Demazure operator formula for the character of  $\AGD(\mu;(w,\ns(\Psi)))$,
which is not difficult to connect to the rotation Theorem~\ref{t character rotation formula}.
\end{proof}

It can further be shown  that the $U_q(\sl_\ell)$-restriction of a  $U_q(\hatsl_\ell)$-Demazure crystal is isomorphic to a disjoint union of  $U_q(\sl_\ell)$-Demazure crystals
(Theorem \ref{t restrict Demazure}).  Combining this with Corollary~\ref{c monomial times Demazure} and
Theorem~\ref{t intro kat conjecture resolution 0} proves that

\begin{corollary}
\label{c intro tame ns catalan key pos}
The tame nonsymmetric Catalan functions are key positive.
\end{corollary}

%Full proofs of Theorem \ref{t intro kat conjecture resolution 0} and Corollary \ref{c intro tame ns catalan key pos} are given in
%Section~\ref{s Schur and key positivity} as part of their more detailed versions Theorem \ref{t character AGD} and Corollary \ref{c kat conjecture resolution 0}, which in particular include explicit positive formulas for the key expansions.

More detailed versions of Theorem~\ref{t intro kat conjecture resolution 0} and Corollary~\ref{c intro tame ns catalan key pos}---Theorem~\ref{t character AGD} and Corollary~\ref{c kat conjecture resolution 0}---are
stated and proved in
Section~\ref{s Schur and key positivity}.
They include explicit positive formulas for the key expansions.

\subsection{DARK crystals}
%\subsection{Kirillov-Reshetikhin affine Demazure crystals}
\label{ss intro DARK}
%To go beyond positivity and reach positive formulas for key expansions,
%we use a technique of Naoi~\cite{NaoiVariousLevels}
%to match generalized Demazure crystals
%with subsets of tensor products of Kirillov-Reshetikhin (KR) crystals where combinatorial
%ideas are better developed.

To extract key positive formulas from Theorem \ref{t intro kat conjecture resolution 0},
we use a technique of Naoi \cite{NaoiVariousLevels} to  match generalized Demazure crystals with subsets of tensor products of
%???check that this abbreviation introduced exactly once...checked 6-29-20
KR crystals, termed DARK crystals;
the latter appears to have simpler combinatorics and, remarkably, exactly matches the katabolism combinatorics conjectured in \cite{SW}.
%by Shimozono-Weyman \cite{SW}.
%While generalized Demazure crystals have nice theoretical properties, elegant combinatorics seems to come more easily by connecting them to Kirillov-Reshetikhin crystals
%are not well-developed enough for our needs.

Let  $B^{1,s}$ denote the single row KR crystal;
it is a seminormal crystal for the subalgebra $U'_q(\hatsl_{\ell}) \subset U_q(\hatsl_{\ell})$ (see \S\ref{ss type A crystals}).
%it is a seminormal crystal for $U'_q(\hatsl_{\ell})$, the subalgebra of $U_q(\hatsl_{\ell})$ without the degree operator (see ).
Its elements are labeled by weakly increasing words of length  $s$ in the alphabet  $[\ell]$.
For $\mu = (\mu_1 \ge \cdots \ge \mu_p \ge 0)$, set $\B^\mu = B^{1,\mu_p} \otimes \cdots \otimes B^{1,\mu_1}$.

%It was already known from \cite{Shimozonoaffine} that
%Catalan functions in the dominant rectangle case are related to Demazure crystals, but one of our key insights is to
%instead match Catalan function and katabolizability to the
\begin{definition}
The \emph{Kirillov-Reshetikhin affine Demazure (DARK) crystal}
associated to $\mu = (\mu_1 \ge \cdots \ge \mu_p \ge 0)$ and
$\mathbf{w} = (w_1, \dots, w_p) \in (\zH_\ell)^p$,
is the following subset of
%the seminormal  $U'_q(\hatsl_\ell)$
$\B^\mu$\,:
\begin{align}
\label{ed DARK}
%\B^{\mu; \mathbf{w}} := \F_{w_1} \big( \sfb_{\mu_1} \tsr \tau \F_{w_2} \big( \sfb_{\mu_2} \tsr \dots \tsr \tau \F_{w_p} (\sfb_{\mu_\ell}) \big) \big) \subset \B^{\mu},\\
\B^{\mu; \mathbf{w}} := \F_{w_1} \big(\twist{\tau} \F_{w_2} \big( \cdots \twist{\tau} \F_{w_{p-1}} \big(\twist{\tau} \F_{w_p} \{\sfb_{\mu_p}\} \tsr  \sfb_{\mu_{p-1}} \big) \cdots \tsr \sfb_{\mu_2} \big) \tsr \sfb_{\mu_1} \big)\spa,
\end{align}
%See Figure \ref{ex DARK crystals} (\S\ref{ss intro examples}).
where $\sfb_s \in B^{1,s}$ is the element labeled by the word
$1^s$,
$\twist{\tau} \colon B^{1,\mu_p} \otimes \cdots \otimes B^{1,\mu_j} \to B^{1,\mu_p} \otimes \cdots \otimes B^{1,\mu_j}$ is given by adding 1 (mod  $\ell$) to each letter and then sorting
each tensor factor to be weakly increasing
(see Proposition~\ref{p tau crystal}),
 and $\F_{w_i} = \F_{j_1}\cdots \F_{j_k}$ for any chosen expression $w_i =\zs_{j_1} \cdots \zs_{j_k}$  $(i\in [p])$;
the right side of \eqref{ed DARK}
%$\B^{\mu; \mathbf{w}}$
does not depend on these choices by \cite[Theorem 3.7]{generalizedNaoi}.
See \S\ref{ss intro examples} for examples.
\end{definition}

The following modification of \cite[Proposition 5.16]{NaoiVariousLevels}
allows us to port results in crystal theory from AGD to DARK crystals.
%It is a generalization of \cite[Proposition 5.16]{NaoiVariousLevels}, obtained easily from the techniques therein.

%We are able to port results in crystal theory from AGDs to DARKs with a modifiction
%of~\cite[Proposition 5.16]{NaoiVariousLevels}.

%and also shows that  $\B^{\mu; \mathbf{w}}$ is well defined (the choice of
%These are related to AGD crystals as follows (and this also shows they are well-defined):

%The well-definedness of these subsets will follow from Theorem \ref{}also the above is not clearly well defined until we know combinatorial excellent filt

\begin{theorem}[{\cite[Corollary 3.11]{generalizedNaoi}}]
\label{t intro KR to affine iso subsets}
Let  $\mathbf{w}$, $\mu$,  $\mu^i$ be as in \eqref{e AGD def}.
There is a strict embedding of $U'_q(\hatsl_\ell)$-seminormal crystals
(see \S\ref{ss seminormal crystals})
\[\Theta_\mu \colon
\B^\mu \tsr B(\mu_1\Lambda_0)  \hookrightarrow B(\mu^p\Lambda_p) \tsr \cdots \tsr B(\mu^1 \Lambda_1);\]
it is an isomorphism from the domain onto a disjoint union of connected components of the
codomain.
And under this map,  $\Theta_\mu (\B^{\mu; \mathbf{w}} \tsr u_{\mu_1 \Lambda_0}) = \AGD(\mu;\mathbf{w}).$
Here, the  $B(s\Lambda_i)$ are regarded as $U'_q(\hatsl_\ell)$-seminormal crystals by
restriction---see \S\ref{ss type A crystals}.
\end{theorem}

\begin{remark}
%%For the objects of interest (Catalan functions, katabolism, etc.)
%%Our study of Catalan functions %and their connections to crystals
%%makes important use of
%The $U_q'(\hatsl_\ell)$-crystal structure of
%$\B^\mu \tsr B(\mu_1\Lambda_0)$ and  $B(\mu^p\Lambda_p) \tsr \cdots \tsr B(\mu^1 \Lambda_1)$,
%but not of $\B^\mu$---it does not seem to be the right object for the combinatorics of interest in this paper.
%%has too many edges to be of use
%However, the $U_q(\sl_\ell)$-restrictions of  $\B^\mu$ and $\B^\mu \tsr B(\mu_1\Lambda_0)$ are isomorphic, and
%these play an important role.
%
This article makes important use of the $U_q'(\hatsl_\ell)$-crystal structures of
$\B^\mu \tsr B(\mu_1\Lambda_0)$ and  $B(\mu^p\Lambda_p) \tsr \cdots \tsr B(\mu^1 \Lambda_1)$,
but not of $\B^\mu$---it does not seem to be the right object for the combinatorics of interest here.
However, the $U_q(\sl_\ell)$-restriction of  $\B^\mu$, being isomorphic to
%the $U_q(\sl_\ell)$-restriction
that of $\B^\mu \tsr B(\mu_1\Lambda_0)$,
\emph{is} of interest and will be frequently used.
\end{remark}

%\begin{remark}
%%For the objects of interest (Catalan functions, katabolism, etc.)
%Our study of Catalan functions and their connections to crystals makes important use of the
%$U_q'(\hatsl_\ell)$-crystal structure of
%$B(\mu_1\Lambda_0) \tsr \B^\mu$ and  $B(\mu^1\Lambda_{1}) \tsr \cdots \tsr B(\mu^p \Lambda_p)$,
%but not of $\B^\mu$; it does not seem to be the right object for the combinatorics of interest in this paper.
%%has too many edges to be of use
%However, the $U_q(\sl_\ell)$-restrictions of  $\B^\mu$ and $B(\mu_1\Lambda_0) \tsr \B^\mu$ are isomorphic, and we
%will work with these considerably.
%%$u_{\mu_1 \Lambda_0} \tsr \B^{\mu; \mathbf{w}})$ or equivalently $\AGD(\mu;\mathbf{w})$.
%\end{remark}

\subsection{Katabolism and Schur positive formulas}
\label{ss intro Katabolism and the Shimozono-Weyman conjecture}

%The class of (symmetric) Catalan functions includes several well-studied subfamilies
%for which proofs of Schur positivity have long been sought; the gold standard
%is to establish a manifestly positive formula for the coefficients.

%stablishing Schur positivity of the (symmetric) Catalan functions,
We establish the Schur positivity of Catalan functions in the strongest possible terms
with a streamlined tableau formula.  It  arises naturally from
DARK crystals
by unraveling the $\F_{w_i}$,  $\F_\tau$, and tensor operations in their construction
(in the spirit of \cite{LakLitMag,LakMagGL}).

Given a weak composition  $\alpha = (\alpha_1, \dots, \alpha_\ell) \in \ZZpl$, the \emph{diagram} of $\alpha$ consists of a left justified array of boxes with  $\alpha_i$ boxes in row  $i$ (rows are allowed to be empty).
A \emph{tabloid}  $T$ of shape $\alpha$ is a filling of the diagram of  $\alpha$ with weakly increasing rows, drawn in English notation with rows labeled  $1, 2,\dots, \ell$ from the top down.
Set $\sh(T) = \alpha$.
%(see Example \ref{ex tabloid example}).
The \emph{content} of $T$ is the vector $(c_1, \dots, c_p)$,
where $c_i$ is the number of times letter $i$ appears in $T$.

Let $\Tabloids_\ell$ denote the set of tabloids of any shape $\alpha \in \ZZpl$, and
$\Tabloids_\ell(\mu) \subset \Tabloids_\ell$ the subset with fixed content
$\mu$.
 Let $\SSYT_\ell(\mu)$ denote the subset of $\Tabloids_\ell(\mu)$
 which are {\it tableaux}, tabloids with partition shape and where entries
strictly increase down columns.
%in columns from top to bottom.
Given a tabloid  $T$, let  $T^i$ denote the  $i$-th row of  $T$ and $T^{[i,j]}$
the subtabloid of  $T$ consisting of the rows
in the interval $[i,j] := \{i, i+1, \dots, j\}$;
set  $T^{[j]} = T^{[1,j]}$.
% for the set $\{1, 2,\dots, j\}$.

\begin{definition}[Partial insertion]
\label{d intro partial insertion}
For $T \in \Tabloids_\ell$
such that $T^{[i,\ell-1]}$ is a tableau,
define $P_{i, \ell}(T) \in \Tabloids_\ell$ to be the tabloid obtained by column inserting the  $\ell$-th row of $T$ into $T^{[i,\ell-1]}$
and leaving rows  $1$ through $i-1$ of $T$ fixed.
%(this may create some empty rows at the bottom).
(There is a way to extend this definition to any tabloid  $T$ but this simpler version
is all we need for the results of this section---see Definition \ref{d partial insertion} and Remark \ref{r kat intro vs not}.)
%column inserting the $\ell$-th row of  $T$ into the subtabloid of  $T$ consisting of its  $i$-th through
%$\ell-1$-st rows.
%
%For a tabloid $T$, let $T_i$ denote the $i$-th row of $T$, thought of as a word in the alphabet of positive integers.
%Given a tabloid $T$ with $\ell$ rows, $P_{i,\ell}(T)$ is the tabloid obtained from $T$
%by replacing rows $\ell-1, \ell$ with the insertion tableau $P(T_\ell T_{\ell-1})$, then replacing rows $\ell-2, \ell-1$
%of the resulting tabloid by the insertion tableau of these two rows, and so on until rows  $i, i+1$ are replaced
%%by replacing rows $i, i+1, \dots, \ell$ of $T$ with the insertion tableau $P(T_\ell T_{\ell-1} \cdots  T_i)$
\end{definition}
%In our main case of interest below, we will always apply $P_{i,\ell}(T)$ to a tabloid such that its $i$th through $\ell-1$st rows form a tableau $T'$; in this case, $P_{i, \ell}(T)$ is obtained more efficiently by column inserting $T_\ell$ into $T'$ and leaving the bottom $i-1$ rows of $T$ fixed.

\begin{example}
\label{ex tabloid example}
 Let $\ell = 5$.  We compute $P_{2,\ell}(T)$ for the $T \in \Tabloids_\ell(4,3,3,3,2)$ below:
 %Both tabloids have content  $(4,3,3,3,1)$.
\[\begin{array}{cc}
T = {\fontsize{7pt}{5pt}\selectfont\tableau{1 &1 &1 &1 & 4 &5\\ 2 & 3   \\ 3 & 4 \\ \bl \fr[l] \\ 2 & 2 &3 & 4 & 5 }} \quad & \quad
P_{2,\ell}(T) = {\fontsize{7pt}{5pt}\selectfont\tableau{1 &1 &1 &1 & 4 &5\\ 2 &2 & 2 &3 & 3    \\ 3 & 4 \\ 4 & 5 \\ \bl \fr[l]}}
\vspace{2.3mm}
\\
\sh(T) = (6,2,2,0, 5) \quad & \quad \sh(P_{2,\ell}(T)) = (6,5,2,2,0)
%\content(T) = (4,3,3,3,1) \quad & \quad \content(P_{2,\ell}(T)) = (4,3,3,3,1)
\end{array}\]
\end{example}

\begin{definition}[Katabolism]
\label{d intro kat}
For $T \in \Tabloids_\ell$,
define $\kat(T)\in \Tabloids_\ell$ as follows:
remove all 1's from  $T$ and left justify rows,
then remove the first (top) row and add it as the new $\ell$-th row, and finally subtract 1 from all letters.

%older (essentially equivalnt) version not quite as nicely compatible with
%For a tabloid  $T$,
%let $\kat(T)$ denote the result of removing its top row $T^1$,
%then adding a new $\ell$-th row at the bottom which is  $T^1$ with all letters $1$ removed,
%and then subtracting 1 from all letters.

Let  $\mathbf{n} = (n_1, \dots, n_{p-1}) \in [\ell]^{p-1}$ and  $\mu\in \ZZ_{\ge 0}^p$ for some $p \in \ZZ_{\ge 1}$.
A tableau $T \in \SSYT_\ell(\mu)$
%??useful: pretty sure we dont need partition content
is \emph{$\mathbf{n}$-katabolizable}
if, for all $i \in [p-1]$, the tabloid $P_{n_i, \ell} \circ \kat \circ \cdots \circ P_{n_2,\ell} \circ \kat \circ P_{n_1,\ell} \circ \kat(T)$ has all its 1's on the first row.
%
%A tabloid  $T \in \Tabloids_\ell$
%is \emph{$\mathbf{n}$-katabolizable}
%if all the 1's in  $P_{n_{1},\ell}(T)$ lie in its first row and  $\kat(P_{n_{1},\ell}(T))$ is
%$(n_2, \dots, n_{p})$-katabolizable;
%for $\mathbf{n}$ the empty sequence, the only $\mathbf{n}$-katabolizable tabloid is the empty one.
\end{definition}

%\begin{definition}[Katabolism]
%\label{d intro kat}
%For $T \in \Tabloids_\ell$, define $\kat(T)\in \Tabloids_\ell$ as follows:
%remove all 1's from  $T$ and left justify rows, then delete the first (top) row and
%add it as the new $\ell$-th row. Finally, subtract 1 from all letters.
%
%
%Let  $\mathbf{n} = (n_1, \dots, n_p) \in [\ell]^p$ for some $p \in \ZZ_{\ge 0}$.
%A tableau $T \in \Tabloids_\ell$
%is \emph{$\mathbf{n}$-katabolizable}
%if for all $i \in [p]$, the tabloid $P_{n_i, \ell} \circ \kat \circ \cdots \circ P_{n_2,\ell} \circ \kat \circ P_{n_1,\ell} \circ \kat(T)$ has all its 1's in the first row.
%\end{definition}

\begin{example}
For  $\ell=5$ and $\mathbf{n} =(3,2,2,1)$, the tableau
below (left) is $\mathbf{n}$-katabolizable:
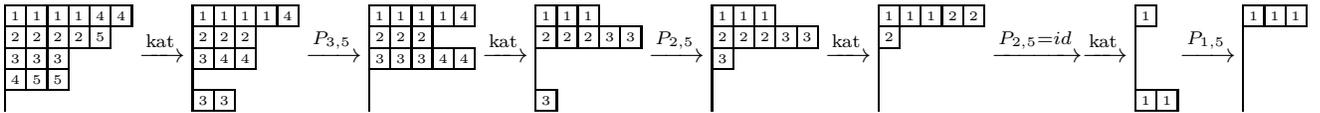
\begin{align*}
\!\!\!
{\fontsize{5.5pt}{4pt}\selectfont \tableau{
1&1&1&1&4&4\\2&2&2&2&5\\3&3&3\\ 4 & 5 &5\\  \bl \fr[l] } }
\, \spa  \footnotesize{\text{$\xrightarrow{\kat}$}} \, \spa
{\fontsize{5.5pt}{4pt}\selectfont \tableau{
1&1&1&1&4\\2&2&2\\ 3 & 4 &4\\  \bl \fr[l]\\ 3&3 } }
\, \spa  \footnotesize{\text{$\xrightarrow{P_{3,5}}$}} \, \spa
{\fontsize{5.5pt}{4pt}\selectfont \tableau{
1&1&1&1&4\\2&2&2\\3&  3 &3& 4 &4\\  \bl \fr[l]\\  \bl \fr[l] } }
\, \spa  \footnotesize{\text{$\xrightarrow{\kat}$}} \, \spa
{\fontsize{5.5pt}{4pt}\selectfont \tableau{
1&1&1\\2& 2&  2 & 3 &3\\  \bl \fr[l]\\  \bl \fr[l] \\3 } }
\, \spa  \footnotesize{\text{$\xrightarrow{P_{2,5}}$}} \, \spa
{\fontsize{5.5pt}{4pt}\selectfont \tableau{
1&1&1\\2& 2& 2 & 3 &3\\ 3\\  \bl \fr[l]\\  \bl \fr[l]  } }
\, \spa  \footnotesize{\text{$\xrightarrow{\kat}$}} \, \spa
{\fontsize{5.5pt}{4pt}\selectfont \tableau{
1& 1& 1 & 2 &2\\ 2\\  \bl \fr[l]\\  \bl \fr[l]  \\  \bl \fr[l]  } }
\, \spa  \footnotesize{\text{$\xrightarrow{P_{2,5}=\idelm}$}}
\footnotesize{\text{$\xrightarrow{\kat}$}} \, \spa
{\fontsize{5.5pt}{4pt}\selectfont \tableau{
 1\\  \bl \fr[l]\\  \bl \fr[l]  \\  \bl \fr[l]  \\ 1 &1} }
 \spa  \footnotesize{\text{$\xrightarrow{P_{1,5}}$}} \, \spa
{\fontsize{5.5pt}{4pt}\selectfont \tableau{
1& 1&1\\  \bl \fr[l]\\  \bl \fr[l]  \\  \bl \fr[l]  \\\bl \fr[l]  } }
\end{align*}
In contrast, the following tableau is not $\mathbf{n}$-katabolizable:
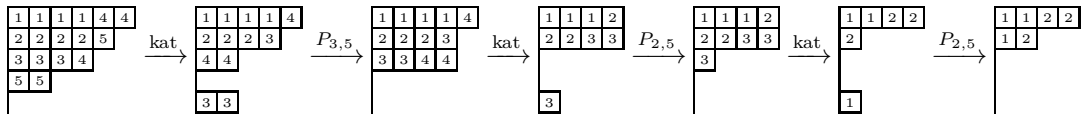
\begin{align*}
\!\!
{\fontsize{5.5pt}{4pt}\selectfont \tableau{
1&1&1&1&4&4\\2&2&2&2&5\\3&3&3&4\\ 5 &5\\  \bl \fr[l] } }
\, \spa  \footnotesize{\text{$\xrightarrow{\kat}$}} \, \spa
{\fontsize{5.5pt}{4pt}\selectfont \tableau{
1&1&1&1&4\\2&2&2&3\\ 4 &4\\  \bl \fr[l] \\3&3 } }
\, \spa  \footnotesize{\text{$\xrightarrow{P_{3,5}}$}} \, \spa
{\fontsize{5.5pt}{4pt}\selectfont \tableau{
1&1&1&1&4\\2&2&2&3\\3&3& 4 &4\\  \bl \fr[l] \\ \bl \fr[l] } }
\, \spa  \footnotesize{\text{$\xrightarrow{\kat}$}} \, \spa
{\fontsize{5.5pt}{4pt}\selectfont \tableau{
1&1&1&2\\2&2& 3 &3\\  \bl \fr[l] \\ \bl \fr[l] \\ 3 } }
\, \spa  \footnotesize{\text{$\xrightarrow{P_{2,5}}$}} \, \spa
{\fontsize{5.5pt}{4pt}\selectfont \tableau{
1&1&1&2\\2& 2& 3 &3\\ 3 \\ \bl \fr[l] \\ \bl \fr[l]} }
\, \spa  \footnotesize{\text{$\xrightarrow{\kat}$}} \, \spa
{\fontsize{5.5pt}{4pt}\selectfont \tableau{
1& 1&2& 2\\ 2  \\ \bl \fr[l]\\ \bl \fr[l] \\1 } }
\, \spa  \footnotesize{\text{$\xrightarrow{P_{2,5}}$}} \, \spa
{\fontsize{5.5pt}{4pt}\selectfont \tableau{
1& 1& 2 &2\\ 1&2\\ \bl \fr[l] \\ \bl \fr[l]\\ \bl \fr[l]} }
\end{align*}
See also Example \ref{ex big kat}.
\end{example}

%Let $\Tabloids_\ell$ denote the disjoint union, over $\alpha \in \ZZpl$, of the set of tabloids of shape $\alpha$, and let $\Tabloids_\ell(\mu) \subset \Tabloids_\ell$ denote the subset of content $\mu$.

The elements of  $\B^\mu$ are naturally labeled by biwords whose
top word is $p^{\mu_p}\cdots 2^{\mu_2} 1^{\mu_1}$ and whose bottom word
is weakly increasing on the intervals with constant top word.
Define the bijection $\inv  \colon  \B^\mu \xrightarrow{\cong} \Tabloids_\ell(\mu)$ as follows:
for all  $i$, the  $i$-th row of  $\inv(b)$ is obtained by sorting the letters above the  $i$'s in the bottom word of  $b \in \B^\mu$.
(This is essentially the well-known inverse map on biwords generalizing the inverse of a permutation---see \S\ref{ss rsk and inv}.)
Katabolism exactly characterizes the image of DARK crystals under $\inv$.

\begin{theorem}
\label{t intro katable inv match crystal}
%??danger: want to generalize to have no Psi but then easy def of katability does not work
For a partition $\mu$ and root ideal $\Psi$,
the map  $\inv$ gives a bijection
\begin{align*}
\B^{\mu;(\mathsf{w}_0, \ns(\Psi))}
\xrightarrow{\inv}
\big\{T \in \Tabloids_\ell(\mu) \mid P(T) \text{ is $\nr(\Psi)$-katabolizable}\big\}
\end{align*}
which takes content to shape.  Here,
%??useful: worth defining P(T) here since it is only done implicityly later on for tabloids
 $P(T)$ denotes the insertion tableau of the row reading word  $T^\ell \cdots T^1$
of  $T$.
%??useful not used enough to be worth it: For a tabloid  $T$ with  $\ell$ rows, define the \emph{row reading word} of  $T$ by $\reading(T) = T^\ell \cdots T^1$.
\end{theorem}

We settle Conjecture~\ref{ec intro schur pos} with a manifestly positive formula.

\begin{theorem}
\label{t kat conjecture resolution}
For any root ideal $\Psi \subset \Delta^+_\ell$ and partition
$\mu = (\mu_1 \ge \cdots \ge \mu_\ell \ge 0)$,
the associated
%(symmetric)
Catalan function has the following Schur positive expression:
%The (symmetric) Catalan function  $H(\Psi; \mu; \mathsf{w}_0)$ is Schur postiive with Schur expansion given by
\begin{align}
\label{e schur positive}
H(\Psi; \mu; \mathsf{w}_0)(\mathbf{x};q) \ = \! \sum_{\substack{U \in \SSYT_\ell(\mu)\\ \text{$U$
is $\nr(\Psi)$-katabolizable} }} \!\!\! q^{\charge(U)}s_{\sh(U)}(\mathbf{x})\,.
\end{align}
\end{theorem}
\begin{proof}[Proof sketch (details in \S\ref{ss Schur positive formula for Catalan functions})]
Combine Theorems \ref{t intro katable inv match crystal}, \ref{t intro kat conjecture resolution 0},
and \ref{t intro KR to affine iso subsets}
and select the tabloids which are  $\inv$ of the  $U_q(\sl_\ell)$-highest weight elements of  $\B^{\mu;(\mathsf{w}_0, \ns(\Psi))}$.
\end{proof}

%\begin{conjecture}[\cite{ChenThesis}]
%\label{cj Chen katab}
%For any indexed root ideal  $(\Psi, \lambda)$ with partition  $\lambda$, define the \emph{Catalan atom}
%\begin{align}
%\label{ecj catalan atom}
%\atom(\Psi;\lambda) = \big\{U \in \SSYT_\lambda \mid \, \text{U is  $\eta$-katabolizable whenever $\Psi \subset \Delta(\eta)$} \big\},
%\end{align}
%where  $\Delta(\eta)$ is as in
%\S\ref{ss Schur pos parabolic}.
%The associated Catalan function is given by
%\begin{align}
%H(\Psi ; \lambda) = \sum_{U \in \atom(\Psi; \lambda)} q^{\charge(U)} s_{\sh(U)}.
%\end{align}
%\end{conjecture}

See Example \ref{ex key positive}.
%An instructive special case is
When $\Psi = \Delta^+$, every  $U \in \SSYT_\ell(\mu)$ is  $(\nr(\Psi)=)$ $\mathbf{1}$-katabolizable and this is
the Lascoux-Sch\"utzenberger \cite{LSfoulkes} charge formula for the modified Hall-Littlewood polynomial  $H_\mu(\mathbf{x};q) = %H(\Delta^+;\mu;\mathsf{w}_0)=
\sum_{\lambda} K_{\lambda \mu}(q) s_\lambda(\mathbf{x})$.

Theorem~\ref{t kat conjecture resolution} resolves
the Shimozono-Weyman conjecture~\cite[Conjecture 27]{SW}
for the generalized Kostka polynomials  $K_{\lambda\mu}^{\Delta(\eta)}(q)$.
%The conjectured formula for the generalized Kostka polynomials of Shimozono-Weyman~\cite[Conjecture 27]{SW}
Indeed, Proposition~\ref{p recover SW def} confirms that
Shimozono-Weyman katabolizability agrees with $\nr(\Psi)$-katabolizability
for the \emph{parabolic root ideal} $\Psi=\Delta(\eta)$, defined for $\eta\in \ZZ_{\geq 0}^r$ by
%
%Theorem \ref{t kat conjecture resolution} resolves the Shimozono-Weyman katabolism conjecture \cite[Conjecture 27]{SW}
%since $\nr(\Psi)$-katabolizability for $\Psi = \Delta(\eta)$
%%(see \eqref{e d parabolic root ideal})
%is easily seen to agree with katabolizability in the sense of \cite{SW}
%(Proposition \ref{p recover SW def}).
%Here, for $\eta\in \ZZ_{\geq 0}^r$, the \emph{parabolic root ideal} $\Delta(\eta)$ is
\begin{align}
\label{e d parabolic root ideal}
& \Delta(\eta):= \big\{ \text{$\alpha \in \Delta^+_{|\eta|}$ above the block diagonal with block sizes $\eta_1,\ldots,\eta_r$}\big\}. \\
\notag
& \text{For example, \quad \quad \ \ \quad \quad
$\Delta(1,3,2) =
\ytableausetup{mathmode, boxsize=.8em,centertableaux}
% ~&*(blue!20)&*(red)&*(red)\\
{\Tiny
\begin{ytableau}
*(white)     &*(red)  &*(red)   &*(red)  &*(red)  &*(red) \\
\mynone & *(white) & *(white) & *(white) & *(red)  &*(red)  \\
\mynone &*(white)  & *(white) & *(white) & *(red)  &*(red)  \\
\mynone &*(white)  & *(white)  & *(white) & *(red) &*(red) \\
\mynone &\mynone  &\mynone  &\mynone  & *(white) & *(white) \\
\mynone &\mynone  &\mynone  &\mynone  &*(white)  & *(white)
\end{ytableau}}\,. $}\end{align}
This gives the first proof of positivity for the Catalan functions and generalized Kostka polynomials in the parabolic case.

\begin{remark}
\label{r two formulas for rectangle}
Shimozono \cite{Scyclageposet} and Schilling-Warnaar \cite{SchillingWarnaar} give a positive formula for the
%generalized Kostka coefficients  $K^{\Delta(\eta)}_{\nu \lambda}(q)$ in the
dominant rectangle Catalan functions  $H(\Delta(\eta);\mu;\mathsf{w}_0)$
(i.e., $\mu= a_1^{\eta_1}\cdots a_r^{\eta_r}$,  $a_1 \ge \cdots \ge a_r$)
% and $H(\Delta(\eta);\mu;\mathsf{w}_0) =\sum_\nu K^{\Delta(\eta)}_{\nu \mu}(q)\, s_\nu$)
using tensor products
of arbitrary KR crystals in type A.
%Theorem \ref{t kat conjecture resolution} also addresses this case,
%but with \emph{subsets} %($\B^{\mu;\mathbf{w}} \subset \B^\mu$)
%of tensor products of single row KR crystals.
%Reconciling these two different formulas is addressed in
%\cite[Conjecture 10]{KirillovShimozono} which proposes a crystal switching map to go between the two.
Included in Theorem \ref{t kat conjecture resolution} is a different formula addressing this case,
using \emph{subsets} %($\B^{\mu;\mathbf{w}} \subset \B^\mu$)
of tensor products of single row KR crystals.
Conjecture 10 of \cite{KirillovShimozono} proposes a  map to reconcile these two different formulas.
\end{remark}

We further obtain a positive combinatorial formula for the key expansion of any tame nonsymmetric Catalan function of partition weight
by similar methods (Corollary~\ref{c kat conjecture resolution 0}).

\subsection{Consequences for $t=0$ nonsymmetric Macdonald polynomials}
\label{ss intro nsmac}

%The nonsymmetric Macdonald polynomials is a basis of orthogonal polynomials for the character
%ring over  $\QQ(q,t)$.  Their theory has been developed over the last 30 years is rich with connections
%to representation theory and combinatorics.
%
A deep theory of nonsymmetric Macdonald polynomials
has developed over the last 30 years,
beginning with the work of Opdam-Heckman \cite{OpdamHeckman},
Macdonald \cite{MacdonaldaffineHeckeandorthogonal}, and Cherednik \cite{Cheredniknsmac}.
%through work of Opdam-Heckman, Macdonald, Cherednik, Knop, and Sahi.
Our results apply to the type  A nonsymmetric Macdonald polynomials at $t=0$,
$E_\alpha(\mathbf{x}; q, 0)$, a nonsymmetric generalization
of the modified Hall-Littlewood polynomials.
%?just said above
%$H_\mu = \sum_\nu K_{\nu \mu}(q)s_\nu$.
They were connected to affine
Demazure characters by Sanderson~\cite{Sanderson}
and the subject of recent results and conjectures on key positivity
\cite{Alexanderssondemazurelusztig,AlexanderssonSawhneymajorindex, AssafNSMacdonald,AssafGonzalez,AssafGonzalez2}.
%Related work includes
%Combinatorial studies of
The $t=0$ nonsymmetric Macdonald polynomials in other types have also
received considerable attention
%crystals and combinatorial models Generalizations of Sanderson's result to other types and combintorial models is also an active area
\cite{IonNSMacdonald, LNSSS3, LNSSSfpsac, OrrShimozononsmac}.
%$\pi_{\mathsf{w}_0} E = \omega H_{\alpha^+}$.
Our results yield the following.

\begin{theorem}
\label{t intro t0 nsMac key pos}
For any  $\alpha \in \ZZpl$,
%the  $t=0$ nonsymmetric Macdonald polynomial
$E_\alpha(\mathbf{x}; q, 0)$
is (1) the character of a $U_q(\hatsl_\ell)$-Demazure crystal, and
(2) key positive with key expansion
\begin{equation}
\label{EinK}
E_\alpha(\mathbf{x}; q, 0)
%\tE_\alpha(\mathbf{x};q) =
= \sum q^{\sum_i {\alpha_i \choose 2}-\charge(P(T))}\kappa_{\sh(T)}(\mathbf{x}),
\end{equation}
where  the sum is over tabloids  $T$ satisfying the katabolizability  conditions in
Corollary~\ref{c t0 nsMac key pos}/ Definition \ref{d extreme kat}.
%and $\kappa_\beta$ is the key polynomial.
Further,
 up to a specialization $x_{\ell+1}= \cdots = x_m =0$ when  $m = |\alpha| > \ell$,
$E_\alpha(\mathbf{x}; q, 0)$
is (3) a tame nonsymmetric Catalan function, and
(4) the Euler characteristic of a vector bundle on a Schubert variety.
%and  $\charge(T)$ denotes the charge of the word
%row reading word  $T^1 T^2 \cdots T^\ell$ of  $T$.
\end{theorem}
\begin{proof}
Statement (1) is due to Sanderson \cite{Sanderson}, and we also recover it as a special case
of our character formula \eqref{ec character AGD 1} for AGD crystals (see Theorem \ref{t Sanderson}).
%result on characters Theorem \ref{t character AGD}
Statement (2) is proved in Corollary \ref{c t0 nsMac key pos}, (3) in Theorem \ref{t nonsymmetric Macdonald},
and (4) follows from (3) and Theorem \ref{t cohomology}.
\end{proof}

The formula~\eqref{EinK} generalizes Lascoux's formula for
cocharge Kostka-Foulkes polynomials~\cite{La},
answering a call put out in~\cite[Conjecture 15]{AlexanderssonSawhneymajorindex},
\cite[p. 267-268]{LascouxBookpolynomials} for
a description of the key coefficients of $E_\alpha(\mathbf{x}; q, 0)$ in this style.
%generalizing the Lascoux-Sch\"utzenberger charge formula for Kostka-Foulkes
%coefficients.
Assaf-Gonzalez~\cite{AssafGonzalez, AssafGonzalez2} studied the problem from a different point of view
and realized the coefficients in terms of crystals on
nonattacking fillings with no coinversion triples
(objects defined in \cite{HHLnonsymmetric}).
See also Remark~\ref{r compare AssafGonzalez}.

\subsection{Consequences for $k$-Schur functions}
\label{ss intro kschur}

%The {$k$-Schur functions} are a family of symmetric functions constructed defined by an algorithm involving katabolism~\cite{LLM}.
%They arose in the study of Macdonald polynomials and were
%subsequently connected

%The {$k$-Schur functions} are a family of symmetric functions which arose in the study of Macdonald polynomials
%and originally constructed by an algorithm involving katabolism~\cite{LLM}.
%They were subsequently connected
%to Gromov-Witten invariants and affine Schubert calculus
%\cite{LMktableau, LMquantum, LamSchubert}.

The {$k$-Schur functions} are  a family of symmetric
functions which arose in the study
of Macdonald polynomials
\cite{LLM} and were subsequently connected
to Gromov-Witten invariants and affine Schubert calculus
\cite{LMktableau, LMquantum, LamSchubert}.
For $\mu =  (k \ge \mu_1 \ge \cdots \ge \mu_\ell \ge 0)$, define the {\it $k$-Schur Catalan function} by $\fs_\mu^{(k)}(\mathbf{x};q)$ $:=H(\Delta^k(\mu); \mu; \mathsf{w}_0)(\mathbf{x};q)$, where
the root ideal $\Delta^k(\mu)$ is determined by $\nr(\Delta^k(\mu))_i= $ $\min\{k-\mu_i+1, \ell-i+1\}$ for all  $i \in [\ell]$.
Building off of \cite{LamSchubert,LLMSMemoirs1,LMktableau},
it was established \cite{BMPS} that
the  $q=1$
specializations $\{\fs_\mu^{(k)}(\mathbf{x};1)\}$
represent Schubert classes in
the homology of the
affine Grassmannian $\Gr_{SL_{k+1}}$.
%It was shown in \cite{BMPS} that the $k$-Schur Catalan functions agree with a definition involving strong marked tab and that  $\fs_\mu^{(k)}(\mathbf{x};1)$ are homology Schubert classes of $\Gr_{SL_{k+1}}$.
% $$these agree with several other versions of  $k$-Schur functions including cohomomlogy at  $q=1$.

%Another by-product of our results concerns {$k$-Schur functions}, a family of
%functions arising from a study of Macdonald polynomials
%and constructed from a set of tableaux defined by katabolism~\cite{LLM}.
%%Inspired by this,

%It is quite different from
%the recent combinatorial formula given in \cite{BMPS},
%which uses chains in Bruhat order on the affine symmetric group $\eS_{k+1}$ and the spin statistic.

A combinatorial formula for the Schur expansion of $\fs_\mu^{(k)}$ was given in \cite{BMPS}
in terms of chains in Bruhat order on the affine symmetric group $\eS_{k+1}$ and the spin statistic.
Theorem~\ref{t kat conjecture resolution} yields a  very different formula:

%The following katabolism formula we obtain from
%%and is without a hint of katabolism.
%%We can now derive a katabolism formula from
%%Theorem~\ref{t kat conjecture resolution}.

\begin{corollary}
\label{c k Schur}
The  $k$-Schur function $\fs_{\mu}^{(k)}$ has
the following Schur positive expansion:
%can be expressed as a sum over certain tableau:
%The graded $k$-Schur function hasve the following Schur positive expansion:
\begin{align}
\label{ec k Schur}
\fs_{\mu}^{(k)} = \sum_{\substack{T \in \SSYT_\ell(\mu)\\
\text{$T$ is $\nr(\Delta^k(\mu))$-katabolizable} }}
%\text{$T$ is $(k-\mu_1+1, \dots, k-\mu_\ell+1)$-katabolizable} }}
\!\!\! q^{\charge(T)}s_{\sh(T)}.
\end{align}
Namely, $T$ occurs in the sum as follows:
remove the $\mu_1$ 1's from the first row of $T$ and column insert the remainder of row 1 into
rows larger than $\min\{k-\mu_1, \ell-1\}$; remove $\mu_2$ 2's from first row and column insert its remainder
into rows larger than $\min\{k-\mu_2, \ell-2\}$; continue until reaching an $i$ such that there are
not $\mu_i$  $i$'s in the first row; $T$ survives if no such $i$ occurs.
%older version
%remove the $\mu_1$ 1's from the first row of $T$ and column insert the remainder of row 1 into
%rows larger than $k-\mu_1$; remove $\mu_2$ 2's from first row and column insert its remainder
%into rows larger than $k-\mu_2$; continue until reaching an $i$ such that there are
%not $\mu_i$  $i$'s in the first row; $T$ survives if no such $i$ occurs.
%(This is not an exact translation of Definition \ref{d intro kat}, but is equivalent to it by Remark \ref{r no roots superstandard}.)
\end{corollary}

%The original definition of  $k$-Schur functions involved katabolism, but as proving anything about the construction appeared difficult, other later definitions received more attention.
%Theorem~\ref{t kat conjecture resolution} yields the following katabolism formula for the
%Schur expansion of $k$-Schur functions
%which is similar in spirit and conjecturally equal to the original.

Example \ref{ex key positive} illustrates \eqref{ec k Schur} for $\fs_{22211}^{(3)}$.
This formula %\eqref{ec k Schur}
has the same spirit as
the original definition of $k$-Schur functions \cite{LLM}, which expressed them
in terms of sets of tableaux called super atoms $\AA_\mu^{(k)}$, constructed using
Shimozono-Weyman
katabolism and crystal reflection operators.

%The $k$-Schur functions  were originally \cite{LLM} defined by a katabolism construction
%in terms of sets of tableaux $\AA_\mu^{(k)}$ called super atoms:
%$A^{(k)}_\mu = \sum_{T \in \AA_\mu^{(k)}} q^{\charge(T)}s_{\sh(T)}$.
%%Our version of katabolism is similar in spirit to
%It was previously conjectured that  $A^{(k)}_\mu = \fs_\mu^{(k)}$.  This remains open, but a natural route is now by proving the following stronger claim:
%
%which also expressed these functions as a charge weighted sum over
%In fact, we conjecture the following:

\begin{conjecture}
The set of tableaux appearing in \eqref{ec k Schur}
%$\nr(\Delta^k(\mu))$-katabolizable tableaux
is equal to the super atom $\AA_\mu^{(k)}$.
\end{conjecture}

%See Example \ref{ex key positive} an illustration of this formula for $\fs_{22211}^{(3)}$.
%(in particular the left side of Figure \ref{f key positive}) for an illustration of % Corollary \ref{c k Schur}  for the $\fs_{22211}^{(3)}$.

\begin{figure}[t]
\vspace{0mm}
\centerfloat
\begin{tikzpicture}[xscale = 1.48, yscale = 1]
\tikzstyle{vertex}=[inner sep=0pt, outer sep=4pt]
\tikzstyle{framedvertex}=[inner sep=3pt, outer sep=4pt, draw=gray]
\tikzstyle{aedge} = [draw, thin, ->,black]
\tikzstyle{dotaedge} = [draw, thin, ->,gray, dashed]
\tikzstyle{edge} = [draw, thick, -,black]
\tikzstyle{LabelStyleH} = [text=black, anchor=south]
\tikzstyle{LabelStyleHn} = [text=black, anchor=north]
\tikzstyle{LabelStyleV} = [text=black, anchor=east]

%\begin{scope}[xshift= -100]
%\node (a) at (0,0) {\tiny $1$};
%\end{scope}
%
%\begin{scope}[xshift= -80]
%\node (a) at (0,0) {\tiny $1$};
%\node (b) at (0,-1) {\tiny $2$};
%
%\draw[aedge] (a) to node[LabelStyleV]{\Tiny $\cf_{1}$} (b);
%\end{scope}

\begin{scope}[xshift= -84]
\node (z) at (0,1.36) {\small $\B^{(1); (\zs_2\zs_1)}$};
\node (z2) at (0,.76) {\scriptsize $\F_{\zs_2 \zs_1}(\sfb_1)$};

\node (a) at (0,0) {\tiny $1$};
\node (b) at (0,-1) {\tiny $2$};
\node (d) at (0, -2) {\tiny $3$};
\node (key) at (0, -4.65) {\scriptsize  $\kappa_{001}$};
\node (cat) at (0, -5.54) {\scriptsize $H(\Delta^+; 100; \mathsf{w}_0)$};
\node (z3) at (0, -6.25) {\footnotesize $\B^{(1,0,0); (\mathsf{w}_0,\zs_2\zs_1, \zs_2\zs_1)}$};

\draw[aedge] (a) to node[LabelStyleV]{\Tiny $\cf_{1}$} (b);
\draw[aedge] (b) to node[LabelStyleV]{\Tiny $\cf_{2}$} (d);
\end{scope}

\begin{scope}[xshift = -25]
\node (z) at (0,1.36) {\small $\B^{(1,1); (\idelm, \zs_2\zs_1)}$};
\node (z2) at (0,.76) {\scriptsize $\twist{\tau} \F_{\zs_2 \zs_1}(\sfb_1)  \! \tsr \! \sfb_1  $};

\node (a) at (0,0) {\tiny $21$};
\node (b) at (0,-1) {\tiny $31$};
\node (d) at (0, -2) {\tiny $11$};
\node (key) at (0, -4.65) {\scriptsize  $\kappa_{101} + q \spa \kappa_{200}$};
\node (cat) at (0, -5.54) {\scriptsize $H(\Delta^+; 110; \zs_2)$};
\node (z3) at (0, -6.25) {\footnotesize $\B^{(1,1,0); (\zs_2,\zs_2\zs_1, \zs_2\zs_1)}$};

\draw[aedge] (a) to node[LabelStyleV]{\Tiny $\cf_{2}$} (b);
%\draw[aedge] (b) to node[LabelStyleV]{\Tiny $\cf_{0,1}$} (d);
\end{scope}

\begin{scope}[xshift = 12]
\node (z) at (1,1.36) {\small $\B^{(1,1); (\zs_1, \zs_2\zs_1)}$};
\node (z2) at (1,.76) {\scriptsize $\F_{\zs_1} ( \twist{\tau} \F_{\zs_2 \zs_1}(\sfb_1) \! \tsr \! \sfb_1 )$};

\node (a) at (0,0) {\tiny $21$};
\node (b) at (0,-1) {\tiny $31$};
\node (c) at (1, -1) {\tiny $32$};
\node (d) at (0, -2) {\tiny $11$};
\node (e) at (1, -2) {\tiny $12$};
\node (f) at (2, -2) {\tiny $22$};
\node (key) at (1, -4.65) {\scriptsize  $\kappa_{011} + q \spa \kappa_{020}$};
\node (cat) at (1, -5.54) {\scriptsize $H(\Delta^+; 110; \zs_1\zs_2)$};
\node (z3) at (1, -6.25) {\footnotesize $\B^{(1,1,0); (\zs_1\zs_2,\zs_2\zs_1, \zs_2\zs_1)}$};

\draw[aedge] (a) to node[LabelStyleV]{\Tiny $\cf_{2}$} (b);
\draw[aedge] (b) to node[LabelStyleH]{\Tiny $\cf_{1}$} (c);
%\draw[aedge] (b) to node[LabelStyleV]{\Tiny $\cf_{0,1}$} (d);
\draw[aedge] (d) to node[LabelStyleH]{\Tiny $\cf_{1}$} (e);
\draw[aedge] (e) to node[LabelStyleH]{\Tiny $\cf_{1}$} (f);
\end{scope}

\begin{scope}[xshift = 94]
\node (z) at (1,1.36) {\small $\B^{(1,1); (\zs_2\zs_1, \zs_2\zs_1)}$};
\node (z2) at (1,.76) {\scriptsize $\F_{\zs_2 \zs_1} (\twist{\tau} \F_{\zs_2 \zs_1}(\sfb_1) \! \tsr \!  \sfb_1 ) $};

\node (a) at (0,0) {\tiny $21$};
\node (b) at (0,-1) {\tiny $31$};
\node (c) at (1, -1) {\tiny $32$};
\node (d) at (0, -2) {\tiny $11$};
\node (e) at (1, -2) {\tiny $12$};
\node (f) at (2, -2) {\tiny $22$};
\node (g) at (1, -3) {\tiny $13$};
\node (h) at (2, -3) {\tiny $23$};
\node (i) at (2, -4) {\tiny $33$};
\node (key) at (1, -4.65) {\scriptsize  $\kappa_{011} + q \spa \kappa_{002}$};
\node (cat) at (1, -5.54) {\scriptsize $H(\Delta^+; 110; \mathsf{w}_0)$};
\node (z3) at (1, -6.25) {\footnotesize $\B^{(1,1,0); (\mathsf{w}_0,\zs_2\zs_1, \zs_2\zs_1)}$};

\draw[aedge] (a) to node[LabelStyleV]{\Tiny $\cf_{2}$} (b);
\draw[aedge] (b) to node[LabelStyleH]{\Tiny $\cf_{1}$} (c);
%\draw[aedge] (b) to node[LabelStyleV]{\Tiny $\cf_{0,1}$} (d);
\draw[aedge] (d) to node[LabelStyleH]{\Tiny $\cf_{1}$} (e);
\draw[aedge] (e) to node[LabelStyleH]{\Tiny $\cf_{1}$} (f);
\draw[aedge] (e) to node[LabelStyleV]{\Tiny $\cf_{2}$} (g);
\draw[aedge] (f) to node[LabelStyleV]{\Tiny $\cf_{2}$} (h);
\draw[aedge] (h) to node[LabelStyleV]{\Tiny $\cf_{2}$} (i);
\draw[aedge] (g) to node[LabelStyleH]{\Tiny $\cf_{1}$} (h);
\end{scope}

\begin{scope}[xshift = 180]
\node (z) at (1,1.36) {\small $\B^{(2,1,1); (\idelm, \zs_2\zs_1, \zs_2\zs_1)}$};
\node (z) at (1,.76) {\scriptsize $\twist{\tau} \F_{\zs_2 \zs_1} (\twist{\tau} \F_{\zs_2 \zs_1}(\sfb_1) \! \tsr \! \sfb_1) \! \tsr \!  \sfb_2 $};

\node (a) at (0,0) {\tiny $3211$};
\node (b) at (0,-1) {\tiny $1211$};
\node (c) at (1, -1) {\tiny $1311$};
\node (d) at (0, -2) {\tiny $2211$};
\node (e) at (1, -2) {\tiny $2311$};
\node (f) at (2, -2) {\tiny $3311$};
\node (g) at (1, -3) {\tiny $2111$};
\node (h) at (2, -3) {\tiny $3111$};
\node (i) at (2, -4) {\tiny $1111$};
\node (key) at (1, -4.7) {$\substack{\kappa_{211} + \spa q \spa \kappa_{301} + \spa q \spa \kappa_{202} \\[1mm] + \spa q^2\kappa_{301} + \spa q^3\kappa_{400}}$};
\node (cat) at (1, -5.54) {\scriptsize $H(\Delta^+; 211; \zs_2)$};
\node (z3) at (1, -6.25) {\footnotesize $\B^{(2,1,1); (\zs_2,\zs_2\zs_1, \zs_2\zs_1)}$};

\draw[dotaedge] (a) to node[LabelStyleV]{\Tiny $\cf_{0}$} (b);
\draw[aedge] (b) to node[LabelStyleH]{\Tiny $\cf_{2}$} (c);
%\draw[aedge] (b) to node[LabelStyleV]{\Tiny $\cf_{1,1}$} (d);
\draw[aedge] (d) to node[LabelStyleH]{\Tiny $\cf_{2}$} (e);
\draw[aedge] (e) to node[LabelStyleH]{\Tiny $\cf_{2}$} (f);
\draw[dotaedge] (e) to node[LabelStyleV]{\Tiny $\cf_{0}$} (g);
\draw[dotaedge] (f) to node[LabelStyleV]{\Tiny $\cf_{0}$} (h);
\draw[dotaedge] (h) to node[LabelStyleV]{\Tiny $\cf_{0}$} (i);
\draw[aedge] (g) to node[LabelStyleH]{\Tiny $\cf_{2}$} (h);
\end{scope}
\end{tikzpicture}
\vspace{-1.1mm}
%\captionsetup{width=\linewidth}
\caption{\label{ex DARK crystals}
%\!\!\!
Five DARK crystals and associated data.
%Constructing a DARK crystal using the $\F_i$,  $\twist{\tau}$, and  $\tsr$ operations.
}
\end{figure}

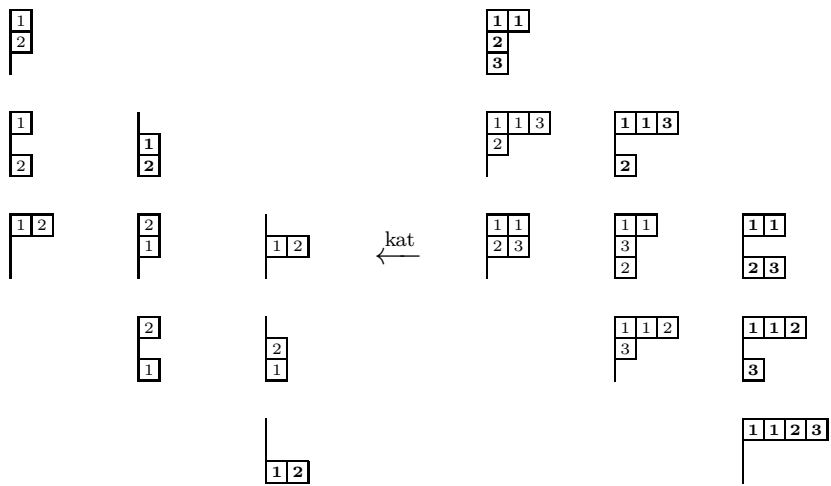
\begin{figure}
\centerfloat
\vspace{-1mm}
\begin{tikzpicture}[xscale = 1.7, yscale = 1.36]
\tikzstyle{vertex}=[inner sep=0pt, outer sep=4pt]
\tikzstyle{framedvertex}=[inner sep=3pt, outer sep=4pt, draw=gray]
\tikzstyle{aedge} = [draw, thin, ->,black]
\tikzstyle{dotaedge} = [draw, thin, ->,gray, dashed]
\tikzstyle{edge} = [draw, thick, -,black]
\tikzstyle{LabelStyleH} = [text=black, anchor=south]
\tikzstyle{LabelStyleHn} = [text=black, anchor=north]
\tikzstyle{LabelStyleV} = [text=black, anchor=east]

\begin{scope}[xshift = 66][xscale=.5, yscale=.65][yshift=0]
%\node (a) at (-6,-2) {
%\parbox{8.2cm}{
%}
%};

%\node (z) at (1,1.36) {\small $\big\{T \in \Tabloids_{\ell}(2,1,1) \mid \inv(T) \in \B^{(2,1,1); (\idelm, \zs_2\zs_1, \zs_2\zs_1)}\big\}$};
%\node (z) at (1,.76) {\footnotesize $\sfb_2 \! \tsr \! \twist{\tau} \F_{21} (\sfb_1 \! \tsr \! \twist{\tau} \F_{21}(\sfb_1)) $};
\node[anchor=west] (a0) at (-.88,-2) {$\xleftarrow{\kat}$};
\node[anchor=west] (a0) at (0,0) {$\fontsize{5.7pt}{4pt}\selectfont\tableau{}$};
\node[anchor=west] (a) at (0,0) {$\fontsize{5.7pt}{4pt}\selectfont\tableau{\mathbf{1}&\mathbf{1}\\\mathbf{2}\\\mathbf{3}}$};
\node[anchor=west] (b) at (0,-1) {$\fontsize{5.7pt}{4pt}\selectfont\tableau{1&1&3\\2 \\ \bl \fr[l]}$};
\node[anchor=west] (c) at (1, -1) {$\fontsize{5.7pt}{4pt}\selectfont\tableau{\mathbf{1}&\mathbf{1}&\mathbf{3}\\ \bl \fr[l]  \\\mathbf{2}}$};
\node[anchor=west] (d) at (0, -2) {$\fontsize{5.7pt}{4pt}\selectfont\tableau{1&1\\2&3 \\ \bl \fr[l]}$};
\node[anchor=west] (e) at (1, -2) {$\fontsize{5.7pt}{4pt}\selectfont\tableau{1&1\\3\\2}$};
\node[anchor=west] (f) at (2, -2) {$\fontsize{5.7pt}{4pt}\selectfont\tableau{\mathbf{1}&\mathbf{1}\\ \bl \fr[l] \\\mathbf{2}&\mathbf{3}}$};
\node[anchor=west] (g) at (1, -3) {$\fontsize{5.7pt}{4pt}\selectfont\tableau{1&1&2\\ 3 \\ \bl \fr[l]}$};
\node[anchor=west] (h) at (2, -3) {$\fontsize{5.7pt}{4pt}\selectfont\tableau{\mathbf{1}&\mathbf{1}&\mathbf{2}\\ \bl \fr[l] \\\mathbf{3}}$};
\node[anchor=west] (i) at (2, -4) {$\fontsize{5.7pt}{4pt}\selectfont\tableau{\mathbf{1}&\mathbf{1}&\mathbf{2}&\mathbf{3}\\ \bl \fr[l]\\ \bl \fr[l]}$};
\end{scope}

\begin{scope}[xshift = -40][xscale=.4, yscale=.65][yshift=0]
\node[anchor=west] (a0) at (0,0) {$\fontsize{5.7pt}{4pt}\selectfont\tableau{}$};
\node[anchor=west] (a) at (0,0) {$\fontsize{5.7pt}{4pt}\selectfont\tableau{1\\2\\ \bl \fr[l] }$};
\node[anchor=west] (b) at (0,-1) {$\fontsize{5.7pt}{4pt}\selectfont\tableau{1\\ \bl \fr[l] \\ 2} $};
\node[anchor=west] (c) at (1, -1) {$\fontsize{5.7pt}{4pt}\selectfont\tableau{\bl \fr[l] \\ \mathbf{1} \\ \mathbf{2}}$};
\node[anchor=west] (d) at (0, -2) {$\fontsize{5.7pt}{4pt}\selectfont\tableau{1&2\\ \bl \fr[l]\\ \bl \fr[l]}$};
\node[anchor=west] (e) at (1, -2) {$\fontsize{5.7pt}{4pt}\selectfont\tableau{2\\1\\ \bl \fr[l]}$};
\node[anchor=west] (f) at (2, -2) {$\fontsize{5.7pt}{4pt}\selectfont\tableau{\bl \fr[l] \\1&2 \\ \bl \fr[l]}$};
\node[anchor=west] (g) at (1, -3) {$\fontsize{5.7pt}{4pt}\selectfont\tableau{2\\ \bl \fr[l] \\ 1}$};
\node[anchor=west] (h) at (2, -3) {$\fontsize{5.7pt}{4pt}\selectfont\tableau{\bl \fr[l] \\2 \\1}$};
\node[anchor=west] (i) at (2, -4) {$\fontsize{5.7pt}{4pt}\selectfont\tableau{\bl \fr[l]\\ \bl \fr[l] \\ \mathbf{1}&\mathbf{2}}$};
\end{scope}
\end{tikzpicture}
\vspace{-.6mm}
\caption{\label{ex DARK crystals2}
\!\!\! The image under $\inv$ of the rightmost two crystals in Figure \ref{ex DARK crystals}.}
%\captionsetup{width=\linewidth}
\end{figure}

\subsection{Examples}
\label{ss intro examples}

Here, we provide running examples for reference throughout the article.  Similar examples
 are also given in Figure \ref{f last fig} on the last page.

Figure~\ref{ex DARK crystals} (right) depicts the DARK crystal  $\B^{\mu;\mathbf{w}}$ for $\ell = 3$, $\mu = (2,1,1)$, $\mathbf{w} = (\idelm,\zs_2\zs_1, \zs_2\zs_1)$;
it can be constructed step by step using the $\F_i$,  $\twist{\tau}$, and tensor operations  as illustrated.

The first two lines give two different names for each DARK crystal.
%(discussed further in \S\ref{s Schur and key positivity}).
%Each key polynomial in these expansions corresponds to a connected component (of solid edges) give their decompositions
The connected components of solid edges decompose them
into $U_q(\sl_{\ell})$-Demazure crystals, each of which has
character equal to a key polynomial;
the key expansions of their charge weighted characters (see \S\ref{s Schur and key positivity})
are given in the third to last line, written so that reading left to right gives the components top
to bottom, e.g.,
%$\{2211,2311, 3311\}$ is  $\kappa_{202}$. ??explain charge where? okay to defer to later section
$\{3211\}$ has character  $\kappa_{211} = x_1^2x_2x_3$.
By Corollary \ref{c kat conjecture resolution 0},
%\eqref{ec kat conjecture resolution 0}.
these characters are
%of these DARK crystals are
%is of the form  $\B^{\lambda; (w, \ns(\Psi))}$
tame nonsymmetric Catalan functions (second to last line),
though this requires rewriting the DARK crystals appropriately (last line),
e.g.,
%we must first
%%(this is essentially the combination of Theorems \ref{t intro kat conjecture resolution 0} and \ref{t intro KR to affine iso subsets}).
%trailing 0's must be interpreted appropriately, e.g.,
$\F_{\zs_1} ( \twist{\tau} \F_{\zs_2 \zs_1}(\sfb_1) \! \tsr \! \sfb_1 )
= \F_{\zs_1\zs_2} (\twist{\tau} \F_{\zs_2 \zs_1} (\twist{\tau} \F_{\zs_2\zs_1}(\sfb_0) \! \tsr \! \sfb_1) \! \tsr \!  \sfb_1)
= \B^{(1,1,0); (\zs_1\zs_2,\zs_2\zs_1, \zs_2\zs_1)}$.
Here, $\sfb_s$ denotes the element of $B^{1,s}$ labeled by $1^s$, with  $\sfb_0$ the empty word (see \S\ref{ss single row KR}).

%For instance, checking that ${\fontsize{5.4pt}{4pt}\selectfont\tableau{1&1\\\bl \fr[l] \\2&3}}$ is $\mathbf{w}$-katabolizable
% $\big\{T \in \Tabloids_{\ell}(2,1,1) \mid \inv(T) \in \B^{(2,1,1); (\idelm, \zs_2\zs_1, \zs_2\zs_1)}\big\}$

The dashed arrows are the $\cf_0$-edges of
%With the dotted edges, the connected components  in addition
$\B^{\mu;\mathbf{w}} \tsr u_{2\Lambda_0}$
(technically this is just a subset of the $U'_q(\hatsl_\ell)$-seminormal crystal $\B^{\mu} \tsr B(2\Lambda_0)$
but we often think of it as coming with the edges  $\cf_i$,  $\ce_i$ ($i \in I$)
which have both ends in the subset).
By Theorem \ref{t intro KR to affine iso subsets},
$\AGD(\mu;\mathbf{w}) = \Theta_\mu (B^{\mu;\mathbf{w}} \tsr u_{2\Lambda_0})$,
which
is isomorphic to a disjoint union of $U_q(\hatsl_{\ell})$-Demazure crystals;
%by Corollary \ref{c monomial times Demazure};
the corresponding decomposition of  $B^{\mu;\mathbf{w}}$
is given  by the components of dashed and solid edges (in the rightmost crystal).
Here there are two such components, so $\AGD(\mu;\mathbf{w})$ is
not a single $U_q(\hatsl_{\ell})$-Demazure
crystal; this  demonstrates a fundamental difference between this work and
earlier work \cite{KMOTU, Sanderson, Shimozonoaffine}
%add??\cite{NYenergy}  no. doesnt directly discuss Demazure
relating generalizations of
Kostka-Foulkes polynomials to Demazure crystals, where only single $U_q(\hatsl_{\ell})$-Demazure
crystals were used.

Figure \ref{ex DARK crystals2} depicts the tabloids obtained by applying $\inv$ to the rightmost two DARK crystals in Figure \ref{ex DARK crystals}.
 % $\B^{(2,1,1); (\idelm, \zs_2\zs_1, \zs_2\zs_1)}$,
By Theorem \ref{t kat and inv crystal} (the full version of Theorem \ref{t intro katable inv match crystal}),
the tabloids on the right are also the
 $T \in \Tabloids_\ell(211)$ which are $\mathbf{w}$-katabolizable in the sense of
Definition~\ref{d kat};
%---see Example \ref{ex gen def katab};
the ones on the left are the
$T \in \Tabloids_\ell(11)$ which are $(\zs_2\zs_1, \zs_2\zs_1)$-katabolizable.
The bold tabloids, by
reading off their shapes and charges,
give the rightmost two key expansions in Figure \ref{ex DARK crystals};
this will be explained in Corollary \ref{c kat conjecture resolution 0}.

%\section{Cohomology of vector bundles on Schubert varieties}
\section{Higher cohomology vanishing and nonsymmetric Catalan functions}
\label{s cohomology}
%We first give a geometric description of the nonsymmetric Catalan functions
%which extends that of the symmetric Catalan functions given in \cite{ChenThesis, Panyushev, BMPS}.

This section uses notation in \S\ref{s intro}, \eqref{e intro pi def}--\eqref{e d HH gamma Psi}, and Definition \ref{d poly}, but is otherwise notationally independent from the remainder of the paper.
%This serves to provide context for our results but is not necessary for the remainder of the paper.

Let  $G = GL_\ell(\CC)$ and  $B \subset G$ the standard upper triangular Borel subgroup.
For $w \in \SS_\ell$, let $X_w = \overline{B \cdot wB} \subset G/B$ denote the Schubert variety.
Given a $B$-module $N$, let $G\times_B N$ denote the homogeneous $G$-vector bundle on
$G/B$ with fiber $N$ above $B \in G/B$,
%, and let $\L_{G\times_B N}(\alpha))$ denote the line bundle on $G\times_B N$.
and let $\L(N)$ denote the locally free $\O_{G/B}$-module of its sections.
We also denote by $\L(N) = \L(N)|_{X_w}$ the restriction of $\L(N)$ to $X_w$.
%The restriction  $\L(N)|_{X_w}$ of $\L(N)$ to $X_w$ is again denoted $\L(N)$.

%Then $\pi$ makes $U$ the total space of a line bundle on $G \times_B N$, and
Consider the adjoint action of $B$ on  the Lie algebra $\mathfrak{u}$ of strictly upper triangular matrices.
The $B$-stable (or ``ad-nilpotent'')
ideals of  $\mathfrak{u}$ are in bijection with root ideals via
the map sending the root ideal $\Psi$ to the $B$-submodule, call it $\mathfrak{u}_\Psi$, of $\mathfrak{u}$ with weights  $\{\epsilon_i-\epsilon_j \mid (i,j)\in \Psi\}$.

The \emph{character} of a $B$-module $N$ is
%\begin{align}
%\label{ed character}
$\chr(N) =
\sum_{\alpha \in \ZZ^\ell}
 \dim (N_\alpha) \spa \mathbf{x}^{\alpha}$,
% \end{align}
where $N_\alpha = \{ v \in N \mid$ $\text{diag}(x_1, \dots, x_\ell) v =
\mathbf{x}^\alpha v \}$
is the \emph{$\alpha$-weight space} of  $N$ and $\mathbf{x}^\alpha := x_1^{\alpha_1} \cdots x_\ell^{\alpha_\ell}$.
Let  $\mathsf{d}$ be the  $\ZZ$-linear operator on $\ZZx$  satisfying  $\mathsf{d}(\mathbf{x}^{\alpha}) = \mathbf{x}^{-\alpha}$,
so that $\chr(N^*) = \mathsf{d}(\chr(N))$;
extend it to an operator on  $\ZZxq$ by
$\mathsf{d}(\sum_{d\ge 0} f_d q^d)$ $= \sum_{d \ge 0} \mathsf{d}(f_d) q^d$.
% for any  $f_d \in \ZZx$.

For $\gamma \in \ZZ^\ell$, let $\CC_\gamma$ denote the one-dimensional $B$-module of weight $\gamma$.

%The following is a generalization of
%Weyl character formula Borel-Weil-Bott for Schubert varieties.
We need the following result of Demazure
\cite[\S5.5]{DemazureSchubert}
(this assumes $G$ is semisimple; see also \cite[II.14.18 (a)]{Jantzenalgebraicgroups} where reductive  $G$ are allowed).
\begin{theorem}
\label{t Schubert Demazure}
For any weight $\gamma \in \ZZ^\ell$ and  $w \in \SS_\ell$,
\begin{align*}
\mathsf{d} \circ \pi_w  \circ \mathsf{d} (\mathbf{x}^\gamma)  = \sum_{i \ge 0} (-1)^i \chr  H^i\big(X_w, \L(\CC_\gamma) \big) .
\end{align*}
\end{theorem}

Nonsymmetric Catalan functions appear naturally
as graded Euler characteristics, extending a description of the Catalan functions in
\cite{Panyushev, ChenThesis}:

\begin{theorem}
\label{t cohomology}
%Nonsymmetric Catalan functions are graded Euler characteristics
%of vector bundles on Schubert varieties:
For any labeled root ideal  $(\Psi, \gamma, w)$,
%For any root ideal $\Psi \subset \Delta^+_\ell$, weight $\gamma \in \ZZ^\ell$, and $w \in \SS_\ell$,
\begin{align}
\label{et cohomology}
H(\Psi;\gamma;w) =  \poly  \circ \, \mathsf{d} \Big(\sum_{i, j \ge 0} (-1)^i q^j \chr   H^i\big(X_w, \L(S^j \mathfrak{u}_\Psi^*  \tsr \CC_\gamma^*) \big) \Big),
\end{align}
where %$\mathfrak{u}_\Psi$ denote the $B$-module with weights $\Psi$
$S^j \mathfrak{u}_\Psi^*$ denotes the $j$-th symmetric power of the  $B$-module  $\mathfrak{u}_\Psi^*$.
\end{theorem}
\begin{proof}
The series $\mathsf{d} \big( \prod_{(i,j) \in \Psi} \big(1-q x_i/x_j\big)^{-1} \mathbf{x}^\gamma \big)$ gives the
character of
 $\bigoplus_j S^j \mathfrak{u}_\Psi^* \tsr \CC_\gamma^*$
where $q$ keeps track of the grading.
%The $B$-module $\bigoplus_j (S^j \mathfrak{u}_\Psi \tsr \CC_\gamma)^*$
Each homogeneous component $S^j \mathfrak{u}_\Psi^* \tsr \CC_\gamma^*$
has a $B$-module filtration into one-dimensional weight spaces.
Then by the additivity of the Euler characteristic and Theorem \ref{t Schubert Demazure},
%Let $\widetilde{M}_j$ denote its associated graded
\begin{align*}
 \sum_{i, j \ge 0} (-1)^i q^j \chr H^i\big(X_w, \L(S^j \mathfrak{u}_\Psi^*  \tsr \CC_\gamma^*) \big)
%&= \sum_{i, j \ge 0} (-1)^i q^j \chr \! \Big( H^i\big(X_w, \L(\widetilde{M}_j) \big) \Big),\\
= \mathsf{d} \circ \pi_w  \circ \mathsf{d} \circ  \mathsf{d} \Big( \prod_{(i,j) \in \Psi} \big(1-q x_i/x_j\big)^{-1} \mathbf{x}^\gamma \Big).
\end{align*}
Applying $\poly \circ \, \mathsf{d}$ to both sides,  the right side
becomes the nonsymmetric Catalan function $H(\Psi;\gamma;w)$ from Definition \ref{d HH gamma Psi}
after using $\poly\circ \, \pi_w = \pi_w \circ \poly$ (Proposition \ref{p poly part} (i)).
%
%
%Applying  $\poly \circ \, \mathsf{d}$ to both sides, and then using  $\poly\circ \, \pi_w = \pi_w \circ \poly$ (Proposition \ref{p poly part} (i)), we obtain the nonsymmetric Catalan function $H(\Psi;\gamma;w)$ from Definition \ref{d HH gamma Psi}.
\end{proof}

\begin{remark}
A version of \eqref{et cohomology} holds for any  $B$-module  $N$, with the product over  $\Psi$ in the definition of  $H(\Psi; \gamma;w)$ replaced by a product over the multiset of weights of $N$.
However, restricting to the $\mathfrak{u}_\Psi$ is natural from the
geometric
perspective of
\cite{Panyushev}, e.g.,
for $w = w_0$ and  $\Psi = \Delta^+$,
$H^i(G/B, \L(\bigoplus_j S^j \mathfrak{u}^*  \tsr \CC_\gamma^*)) \cong H^i(T^*(G/B),\theta^* \L(\CC_\gamma^*))$, where
  $T^*(G/B)$ is the cotangent bundle of the flag variety  and $\theta \colon T^*(G/B) \to G/B$ the projection.
%,  which allows more general (but not arbitrary) $B$-modules  $N$ than the $\mathfrak{u}_\Psi$ considered here.
\end{remark}

For $\nu = (\nu_1 \ge  \cdots \ge \nu_\ell) \in \ZZ^\ell$,
let $V(\nu)$ be the irreducible  $G$-module of highest weight  $\nu$.
Let $\alpha \in \ZZ^\ell$ and $\alpha^+$ be the weakly decreasing rearrangement of  $\alpha$.
The \emph{Demazure module} $D(\alpha) \subset V(\alpha^+)$
is the  $B$-module  $B \spa u_\alpha$, where  $u_\alpha$ is an element of the (one-dimensional) $\alpha$-weight space of $V(\alpha^+)$.
The \emph{Demazure atom module} $\hat{D}(\alpha)$
%??better name for this?
is the quotient of $D(\alpha)$ by the sum of all Demazure modules properly contained in  $D(\alpha)$.
The characters $\kappa_\alpha(\mathbf{x}) = \chr(D(\alpha))$ and $\hat{\kappa}_\alpha(\mathbf{x}) = \chr(\hat{D}(\alpha))$ are the key polynomial
and Demazure atom, respectively which will be discussed further in \S\ref{ss gl Demazure crystal} and \S\ref{ss polynomial truncation}.

As in \cite[\S2.3]{vanderKallenBModule}, say a  $B$-module  $N$ has an  \emph{excellent filtration} (resp. \emph{relative Schubert filtration}) if its dual  $N^*$ has a  $B$-module filtration whose subquotients are isomorphic to
%$0 = \mathfrak{u}_0 \subset \mathfrak{u}_1 \subset \cdots \subset \mathfrak{u}_k = N$ with $\mathfrak{u}_{i+1}/\mathfrak{u}_{i}$ of the form $D(\alpha)^*$ (resp. $\hat{D}(\alpha)$) for some  $\alpha \in \ZZ^\ell$
Demazure modules (resp. Demazure atom modules).
%Let $S \mathfrak{u}_\Psi = \bigoplus_j S^j \mathfrak{u}_\Psi$ be the symmetric algebra of  $\mathfrak{u}_\Psi$.

\begin{conjecture}
\label{cj ns Catalan positivity}
Let $(\Psi, \mu, w)$ be a labeled root ideal  with partition $\mu$ and  $j \ge 0$.
\begin{list}{\emph{(\roman{ctr})}} {\usecounter{ctr} \setlength{\itemsep}{1pt} \setlength{\topsep}{2pt}}
\item The nonsymmetric Catalan function $H(\Psi; \mu; w)$ is a positive sum of Demazure atoms, i.e.,
%\begin{align*}
$H(\Psi;\mu;w)(\mathbf x;q)=\sum_{\alpha} K^{\Psi,w}_{\alpha, \mu}(q)\,\hat{\kappa}_\alpha(\mathbf{x})
\ \spa \text{with} \ \, K^{\Psi,w}_{\alpha,\mu}(q)\in \ZZ_{\ge 0}[q]\,.$
%\end{align*}
\item  $H^i(X_w, \L(S^j \mathfrak{u}_\Psi^*  \tsr \CC_\mu^*) ) = 0$ for  $i >0$.
\item  $H^0(X_w, \L(S^j \mathfrak{u}_\Psi^*  \tsr \CC_\mu^*) )$  has a relative  Schubert  filtration.
\item  $H^0(X_w, \L(S^j \mathfrak{u}_\Psi^*  \tsr \CC_\mu^*) )$  has an excellent filtration when $(\Psi, \mu, w)$ is tame.
\end{list}
\end{conjecture}

For tame $(\Psi, \mu, w)$, Corollary \ref{c intro tame ns catalan key pos} implies (i), while
conjectures (ii) and (iv) constitute a module-theoretic strengthening of this corollary.
Similarly, (ii) and (iii) give a module-theoretic strengthening of (i).

%See \cite[\S3]{BMPS} for more details on what is known in the  $w=w_0$ case.
%Here is brief summary of work on the case
In this paragraph we discuss the  $w = w_0$ ($X_w = G/B$) case of Conjecture
\ref{cj ns Catalan positivity}.
First note that the cohomology groups are $G$-modules, so  (iii)--(iv) hold and (ii) implies (i).
Conjecture (ii) was posed
%is a strengthening of
by Chen-Haiman \cite[Conjecture~5.4.3]{ChenThesis}; this generalized a conjecture of Broer
%(see, e.g., \cite[Conjecture~5]{SW})
for parabolic $\Psi$, which he settled in the dominant rectangle case \cite[Theorem 2.2]{BroerNormality}.
Hague \cite[Theorems 4.15 and 4.23]{Hague} extended this result to some other classes of weights
(still parabolic  $\Psi$).
Panyushev proved that (ii) holds when
%the cohomology in \eqref{et cohomology} vanishes for  $i > 0$ when
the weight $\mu - \rho + \sum_{(i,j) \in \Delta^+ \setminus \Psi} \epsilon_i- \epsilon_j$ is weakly decreasing, where $\rho = (\ell-1, \ell-2, \dots, 0)$.
Frobenius splitting methods \cite{KLT} give another proof of a subcase of Broer's result; this method has the advantage of applying to $G$ over algebraically closed fields of prime characteristic.

When  $\Psi= \varnothing$,
$H^0(X_w, \L(\CC_\mu^*))^* = D(w \mu)$ (implying (iii)-(iv))
and the cohomology vanishing (ii) are results of
Demazure \cite{DemazureSchubert,DemazureBulletin};
a gap in \cite{DemazureSchubert} is bypassed by another proof method of Andersen \cite[\S4.3]{AndersenSchubertvariety};
see also II.14.18 (b) and II.14.15 (e) of \cite{Jantzenalgebraicgroups}.

\section{Background on crystals}
\label{s background on crystals}

We begin by reviewing crystals for any  symmetrizable Kac-Moody Lie algebra  $\g$ and prove that restrictions of Demazure crystals are disjoint unions of Demazure crystals.
We then fix notation and conventions for  $\g = \hatsl_\ell$ following Naoi \cite{NaoiVariousLevels} and Kac \cite{KacBook}; note that the notation $I, P,P^+, \alpha_i, \alpha_i^\vee$ is for general  $\g$ in
\S\ref{ss seminormal crystals}--\ref{ss restrict Demazure} and for
 $\hatsl_\ell$ from \S\ref{ss lie algebra hatsl} through the remainder of the paper.
%While much of this background is
%standard we follow Naoi which has some useful conventions.

%symmetrizable Kac-Moody Lie algebra
\subsection{$U_q(\g)$-(seminormal) crystals}
\label{ss seminormal crystals}
%Let $U_q(\g)$ be a quantized enveloping algebra of a symmetrizable Kac-Moody Lie algebra
%(as in \cite{KashiwaraSurvey});
The quantized enveloping algebra  $U_q(\g)$
%of a symmetrizable Kac-Moody Lie algebra $\g$
is specified by
%the following data:
a Dynkin node set  $I$, coweight lattice  $P^*$, weight lattice $P = \hom_\ZZ(P^*,\ZZ)$, coroots  $\{\alpha_i^\vee\}_{i\in I} \subset P^*$, roots $\{\alpha_i \}_{i \in I} \subset P$, and a symmetric bilinear form  $(\cdot, \cdot) \colon  P \times P \to \QQ$ \spa
 subject to several conditions (see \cite[\S2.1]{KashiwaraSurvey}).
This data given,
a \emph{$U_q(\g)$-seminormal crystal} is
a set $B$ equipped with a \emph{weight function} $\wt\colon B \to P$ and \emph{crystal operators}
$\ce_i, \cf_i \colon B \sqcup \{0\} \to B \sqcup \{0\}$ ($i \in I$) such that for all $i \in I$ and $b \in B$, there holds $\ce_i(0) = \cf_i(0) = 0$ and
\begin{align*}
%\label{e crystal def 2}
&\wt(\ce_ib)= \wt(b) + \alpha_i \text{ whenever $\ce_ib \neq 0$}, \ \text{ and } \,
%\label{e crystal def 2b}
 \wt(\cf_ib)= \wt(b) - \alpha_i \text{ whenever $\cf_ib \neq 0$}; \\
\notag
& \varepsilon_i(b) := \max\{k \ge 0 \mid \ce_i^kb \neq 0\} < \infty, \quad \phi_i(b) := \max\{k \ge 0 \mid \cf_i^kb \neq 0\} < \infty;  \\
\notag
& \langle \alpha_i^\vee, \wt(b)  \rangle = \phi_i(b) - \varepsilon_i(b); \\
\notag
& \cf_i(\ce_ib) = b\text{ whenever $\ce_ib \neq 0$, \ \ \ and \ } \ce_i(\cf_ib) = b\text{ whenever $\cf_ib \neq 0$}.
\end{align*}
This agrees with the notion of a seminormal crystal in \cite[\S7]{KashiwaraSurvey},
the notion of a crystal in \cite{NaoiVariousLevels}, and the notion of a $P$-weighted  $I$-crystal in \cite{Shimozonoaffine}.
%though the latter two are not defined in full kac moody case

A \emph{strict embedding} of $U_q(\g)$-seminormal crystals  $B, B'$ is an injective map
$\Psi \colon B \sqcup \{0\} \to B' \sqcup \{0\}$ such that $\Psi(0) = 0$  and
 $\Psi$ commutes with  $\wt$,   $\varepsilon_i$,  $\phi_i$,
 $\ce_i$,  and $\cf_i$ for all  $i\in I$.
It is necessarily an isomorphism from  $B$ onto
a disjoint union of connected components of  $B'$.

For $U_q(\g)$-seminormal crystals $B_1$ and $B_2$,
their tensor product $B_1 \tsr B_2 = \{b_1 \tsr b_2 \mid b_1 \in B_1, b_2 \in B_2\}$ is
the $U_q(\g)$-seminormal crystal with weight function
$\wt(b_1 \otimes b_2) = \wt(b_1) + \wt(b_2)$ and crystal operators
(we use the convention opposite Kashiwara's)
%This is opposite Kashiwara convention now
\begin{align}
\label{e crystal tensor 1}
\ce_i(b_1 \otimes b_2) &= \begin{cases}
\ce_ib_1 \tsr b_2  & \text{ if $ \varepsilon_i(b_1) > \phi_i(b_2)$}, \\
b_1 \tsr \ce_ib_2   & \text{ if $\varepsilon_i(b_1) \le \phi_i(b_2)$}.
\end{cases}\\
\label{e crystal tensor 2}
\cf_i(b_1 \otimes b_2) &= \begin{cases}
\cf_ib_1 \tsr b_2 & \text{ if $\varepsilon_i(b_1) \ge \phi_i(b_2)$}, \\
 b_1 \tsr \cf_ib_2  & \text{ if $\varepsilon_i(b_1)< \phi_i(b_2)$}.
\end{cases}
\end{align}

%??decided to use seminormal crytsals for U'. this may have some bad interactions with intro notation.  updated intro accordingly I believe
Assume for this paragraph that the roots and coroots are linearly independent.
Let  $\Oint(\g)$ denote the category whose objects are the $U_q(\g)$-modules isomorphic to a direct sum of integrable highest weight  $U_q(\g)$-modules
(see, e.g., \cite[\S2.4]{KashiwaraSurvey}).
% Let $P$ and $I$ be the weight lattice and Dynkin nodes for $U_q(\g)$.
Any  $M$ in  $\Oint(\g)$ has a unique local crystal basis  $(\cL,\B)$ up to isomorphism \cite{Kas1},
and extracting the associated combinatorial data yields a $U_q(\g)$-seminormal crystal
(see \cite[\S4.2, \S7.5]{KashiwaraSurvey}).
We define a \emph{$U_q(\g)$-crystal} to be a $U_q(\g)$-seminormal crystal arising in this way.
For $\Lambda \in P^+ = \{\lambda \in P \mid \langle \alpha_i^\vee, \lambda\rangle \ge 0\}$, the \emph{highest weight $U_q(\g)$-crystal}  $B(\Lambda)$
is the  $U_q(\g)$-crystal arising from the local crystal basis of the irreducible highest weight module $V(\Lambda)$ in $\Oint(\g)$.
So with this notation,
any $U_q(\g)$-crystal is a disjoint union of highest weight $U_q(\g)$-crystals by \cite{Kas1}.

\subsection{Restricting Demazure crystals}
\label{ss restrict Demazure}

Let $U_q(\g)$, $P^*$, $P$, $\{\alpha_i^\vee\}_{i\in I}$,
$\{\alpha_i \}_{i \in I}$ be as in \S\ref{ss seminormal crystals}.
Let $J \subset I$ and $\hat{P}^* \subset P^*$ be such that $\{\alpha_i^\vee\}_{i \in J} \subset \hat{P}^*$.
As $P = \hom_\ZZ(P^*,\ZZ)$,  restricting maps from  $P^*$ to  $\hat{P}^*$ yields
a projection $z \colon  P \to \hat{P} := \hom_\ZZ(\hat{P}^*, \ZZ)$.
Assume that the sets $\{\alpha_i^\vee\}_{i \in I}, \{\alpha_i\}_{i \in I},\{\alpha_i^\vee\}_{i \in J}, \{z(\alpha_i) \}_{i \in J}$ are linearly independent.
The algebra  $U_q(\g)$ has generators  $e_i, f_i$,  $i \in I$, and  $q^h$,  $h \in P^*$.
Let $U_q(\g_J) \subset U_q(\g)$ be the subalgebra generated by  $e_i, f_i$,  $i \in J$, and  $q^h$,  $h \in \hat{P}^*$;
it is a quantized enveloping algebra and its defining data includes
$J, \{\alpha_i^\vee\}_{i \in J} \subset \hat{P}^*, \{z(\alpha_i) \}_{i \in J} \subset \hat{P}$.
%??are the new roots lin ind? no.    do we need to assume this?  yes for some things

It is straightforward to verify that for any  $M$ in  $\Oint(\g)$, the local crystal basis $(\cL,\B)$ of $M$ is also a local crystal basis of the
$U_q(\g_J)$-restriction of $M$ and so is isomorphic to the direct sum of local crystal bases of
 highest weight $U_q(\g_J)$-modules by \cite{Kas1}
(see \cite[\S4.6]{KashiwaraSurvey} for a similar result).
%(this is similar to the discussion in \cite[\S4.6]{KashiwaraSurvey}).
Moreover, the associated $U_q(\g)$-crystal  $B$ of $(\cL,\B)$ and $U_q(\g_J)$-crystal  $\hat{B}$ of $\Res_{U_q(\g_J)}(\cL,\B)$
are related as follows:
 $\hat{B}$ is obtained from $B$ by replacing its weight function with $z \circ \wt \colon B \to \hat{P}$ and taking only the crystal operators  $\ce_i,\cf_i$ for $i \in J$.
We say $\hat{B}$ is the \emph{$U_q(\g_J)$-restriction}
of  $B$ and denote it $\Res_J B$
or similar---see \S\ref{ss type A crystals}.
%notation as discussed in \S\ref{ss type A crystals}).

The following crystal restriction theorem will be important for obtaining key positivity results.
Its proof %of the following restriction theorem
was communicated to us by Peter Littelmann,
and we are also grateful to Wilberd van der Kallen  who
pointed us to his module-theoretic version \cite[Theorem 6.3.1]{vanderKallenBModule}.
A more general module-theoretic version was recently given in \cite[Appendix A]{AssafGonzalez2}.
%A module-theoretic result generalizing van der Kallen's to the affine setting was recently given in \cite[Appendix A]{AssafGonzalez2}.

\begin{theorem}
\label{t restrict Demazure}
Let  $U_q(\g_J) \subset U_q(\g)$ be as above.
For any  $U_q(\g)$-Demazure crystal $S$,
$\Res_{J} S$ is isomorphic to a disjoint union
of $U_q(\g_J)$-Demazure crystals.
\end{theorem}

Here, $S =  \F_{i_1} \cdots \F_{i_k} \{u_{\Lambda}\} \subset B(\Lambda)$ for some highest weight  $U_q(\g)$-crystal  $B(\Lambda)$
(as in Definition \ref{d demazure crystal})
and  $\Res_J S$ denotes the set  $S$ regarded as a subset of  $\Res_J B(\Lambda)$, which
is isomorphic to a disjoint union of  highest weight $U_q(\g_J)$-crystals
by the discussion above.

\begin{proof}
As $\{\alpha_i^\vee\}_{i \in I}$ is linearly independent, we can
choose $\{\Lambda_j\}_{j \in I} \subset P$ such that $\langle \alpha_i^\vee, \Lambda_j\rangle = m \delta_{ij}$
for $i,j \in I$ and  $m\in \ZZ_{\ge 1}$.
Set  $\bar{J} = I \setminus J$ and
$\rho_{\bar{J}}= \sum_{i \in \bar{J}} \Lambda_i$.
Put  $c = 1 + \max\{\varepsilon_i(b) \mid b \in S\}$.
Consider the $U_q(\g)$-crystal $B({c\rho_{\bar{J}}})$
with highest weight element $u_{c\rho_{\bar{J}}}$. Since $\phi_i(u_{c\rho_{\bar{J}}}) = \langle \alpha_i^\vee, c\rho_{\bar{J}} \rangle = c \spa m$ for  $i \in \bar{J}$,
the tensor product
rule \eqref{e crystal tensor 2} implies
\begin{align}
\label{e restrict thm}
\text{$\cf_i(b \tsr u_{c\rho_{\bar{J}}}) \notin S \tsr u_{c\rho_{\bar{J}}}$ \ \  for all  $i \in \bar{J}$ and $b \in  S$.}
\end{align}
By \cite[\S2.11]{Josephdemazurecrystal},
$S \tsr u_{c\rho_{\bar{J}}} $ is a disjoint union of $U_q(\g)$-Demazure crystals,
each of the form  $\F_{i_1} \cdots \F_{i_m} \{b \tsr  u_{c\rho_{\bar{J}}} \}$ for some  $b \in S$ and, by \eqref{e restrict thm}, we must have $i_j \in J$;
moreover, $\F_{i_1} \cdots \F_{i_m} \{b \tsr u_{c\rho_{\bar{J}}} \} =  (\F_{i_1} \cdots \F_{i_m} \{b\}) \tsr u_{c\rho_{\bar{J}}} $,
which follows from \eqref{e crystal tensor 2} and  $\phi_i(u_{c\rho_{\bar{J}}}) = \langle \alpha_i^\vee, c\rho_{\bar{J}} \rangle  = 0$ for $i \in J$.
Thus  $S$ is a disjoint union of sets of the form  $\F_{i_1} \cdots \F_{i_m} \{b\}$, and these are  $U_q(\g_J)$-Demazure crystals as
$\ce_i(b \tsr u_{c\rho_{\bar{J}}} ) = 0$ implies $\ce_i(b) = 0$ for $i \in J$.
\end{proof}

\begin{remark}
\label{r sl to gl crystal}
Let  $U_q(\g_J) \subset U_q(\g)$
be as above and assume $J = I$. Then a subset  $S$ of a $U_q(\g)$-crystal  $B$
%Then a subset $S \subset B$ belongs to  $\mathcal{D}(\g)$ (see Definition \ref{d demazure crystal}) if and only if
%$\Res_{J} S$ belongs to  $\mathcal{D}(\g_J)$.
is isomorphic to a disjoint union of
$U_q(\g)$-Demazure crystals if and only if $\Res_J S$ is isomorphic to a disjoint union of
$U_q(\g_J)$-Demazure crystals.
%($\Res_{\hat{\g}} S$ is just  $S$ regarded as a subset of $\Res_{\hat{\g}} B$).
This is immediate from the definitions since
$B$ and  $\Res_J B$ have the same $\cf_i$-edges for all $i \in J = I$.
\end{remark}

\subsection{The affine Lie algebra $\hatsl_\ell$}
\label{ss lie algebra hatsl}

%We mostly follow the conventions of Naoi \cite{NaoiVariousLevels} and Kac \cite{KacBook}.
Let $\hatsl_\ell$ be the complex affine Kac-Moody Lie algebra of type $A_{\ell-1}^{(1)}$,
with associated
%Denote Assoicated data as follows:
Dynkin nodes $I = \ZZ/\ell\ZZ = \{0, 1, \ldots , \ell-1\}$
and Cartan matrix $A = (a_{ij})_{i,j\in I}$.
Let  $\h \subset \hatsl_\ell$ be the Cartan subalgebra, which has a basis consisting of
the simple coroots $\{\alpha_i^\vee \mid  i \in  I\} \subset \h$ together with the scaling element $d \in \h$.
We have the simple roots $\{\alpha_i \mid  i \in I\} \subset \h^*$, with pairings
$\langle \alpha_i^\vee, \alpha_j\rangle = a_{ij}$
and $\langle d, \alpha_i \rangle = \delta_{i 0}$ $(i,j \in I)$.
The fundamental weights $\{\Lambda_i \mid i \in I\} \subset \h^*$
are defined by
$\langle \alpha_i^\vee, \Lambda_j\rangle = \delta_{ij}$
for $i,j \in I$,
$\langle d, \Lambda_{0}\rangle = 0$, and
\begin{align}
\label{e tau computation}
%??useful really want [\ell] here as opposed to something like I, since for instance RHS cannot be taken mod ell
\langle d, \Lambda_{i} - \Lambda_{i-1} \rangle = \frac{2i-1-\ell}{2\ell}  \ \ \ \text{ for } i \in [\ell].
%\langle \Lambda_{i} - \Lambda_{i-1}, d \rangle = \frac{2(i-1)-(\ell-1)}{2\ell},  \text{ for all } i \in [\ell].
%\langle \Lambda_{0} - \Lambda_{\ell-1}, d \rangle = \frac{(\ell-1)}{2\ell}
%\langle \Lambda_i - \Lambda_{i+1}, d \rangle = \frac{\ell-1-2i}{2\ell}$, for all $i \in I$,
\end{align}
%The latter convention is not always standard but is adopted in \cite{NaoiVariousLevels} to ensure that  $\tau$ is nice---see.
%We have adopted this convention
The convention \eqref{e tau computation} is
implicit in \cite{NaoiVariousLevels} and ensures
the extended affine Weyl group acts nicely on  $\alpha_i$ and  $\Lambda_i$, which will be important
in \S\ref{ss formal twist}.
%see \S\ref{ss the affine symmetric group and 0Hecke monoid}.
The $\{\Lambda_i \mid i \in I\}$ together with
the null root $\delta = \sum_{i\in I} \alpha_i$ form
a basis for  $\h^*$; note that $\langle \alpha_i^\vee, \delta\rangle = 0$ for  $i \in I$
and $\langle d, \delta \rangle =1$.
%and $K = \sum_{i\in I} a_i^\vee \alpha_i^\vee$ the canonical central element.

Let $P =
\bigoplus_{i\in I} \ZZ \Lambda_i \oplus \ZZ \frac{\delta}{2\ell}  \subset \h^*$
 be the weight lattice and
$P^+ = \sum_{i\in I} \ZZ_{\ge 0}\Lambda_i + \ZZ \frac{\delta}{2\ell}$
the dominant weights.
Let  $\cl \colon  \h^* \to \h^*/\CC \delta$ be the canonical projection, and
set  $P_{\cl} = \cl(P) = \bigoplus_{i\in I} \ZZ \, \cl(\Lambda_i)$.
Let $\aff \colon  \h^*/\CC\delta \to \h^*$ be the section of $\cl$ satisfying $\langle d, \aff(\lambda) \rangle = 0$ for all $\lambda \in h^* / \CC \delta$.
Set  $\varpi_i = \aff(\cl(\Lambda_i- \Lambda_0))$ for  $i \in I$ (hence  $\varpi_0 = 0$).

Let $\sl_\ell \subset \hatsl_\ell$ be the simple Lie subalgebra with Dynkin nodes $I \setminus \{0\} = [\ell-1]$,
Cartan subalgebra  $\mathring{\h} = \bigoplus_{i \in [\ell-1]}\CC \alpha_i^\vee \subset \h$,
and fundamental weights $\{\mathring{\varpi}_i \mid i \in [\ell-1]\} \subset (\mathring{\h})^*$.
The associated weight lattice $\mathring{P} = \bigoplus_{i \in [\ell-1]} \ZZ \mathring{\varpi}_i$ is
naturally viewed as the image of  $P$ under the projection  $\h^* \to \h^*/(\CC\delta \oplus \CC \Lambda_0) = (\mathring{\h})^*$;
moreover,  $\varpi_i$ maps to  $\mathring{\varpi}_i$ and $\bigoplus_{i \in [\ell-1]} \ZZ \varpi_i \subset \h^*$ maps isomorphically onto
$\mathring{P} \subset (\mathring{\h})^*$.

\subsection{Type A crystals}
\label{ss type A crystals}
Let  $U_q(\hatsl_\ell)$
be the quantized enveloping algebra
specified by the data  $I, P^* = \hom_\ZZ(P,\ZZ), P, \{\alpha_i^\vee\}_{i\in I}, \{\alpha_i\}_{i\in I}$
above and the symmetric bilinear form  $(\cdot, \cdot) \colon  P \times P \to \QQ$ defined by $(\alpha_i, \alpha_j) = a_{ij}$,  $(\alpha_i, \Lambda_0) = \delta_{i0}$,
$(\Lambda_0, \Lambda_0) = 0$.
The subalgebra $U_q(\sl_\ell) \subset U_q(\hatsl_\ell)$ fits the form in \S\ref{ss restrict Demazure},
with Dynkin node subset $[\ell-1] \subset I$,
coweight lattice $\bigoplus_{i\in [\ell-1]} \ZZ \alpha_i^\vee$, and weight lattice  $\mathring{P}$.
%, coroots  $\{\alpha_i^\vee\}_{i\in [\ell-1]}$, and roots $\{?(\alpha_i)\}_{i\in I}$.
Let  $U_q(\gl_\ell)$ be as in
%\cite[Chapter 7]{HK} or
\cite[\S5]{KashiwaraSurvey};
data includes Dynkin nodes  $[\ell-1]$, weight lattice $\ZZ^\ell$, and roots $\{\epsilon_i-\epsilon_{i+1}\}_{i \in [\ell-1]}$.

Let $U'_q(\hatsl_\ell) \subset U_q(\hatsl_\ell)$ be the subalgebra generated by
$e_i, f_i$,  $i \in I$, and  $q^h$,  $h \in P_{\cl}^*=  \bigoplus_{i \in I}\ZZ \alpha_i^\vee$;
it can be considered a quantized enveloping algebra with data $I, \{\alpha_i^\vee\}_{i\in I} \subset P_{\cl}^*,
\{\cl(\alpha_i)\}_{i\in I} \subset P_{\cl}$ (it fits the form in \cite[Definition 2.1]{KashiwaraSurvey}),
but note that the roots are not linearly independent.
\emph{For $U'_q(\hatsl_\ell)$, we work with $U'_q(\hatsl_\ell)$-seminormal crystals so that
we can work with both KR crystals and restrictions
of  $U_q(\hatsl_\ell)$-crystals and treat them uniformly, while for  $\g = \sl_\ell$,  $\gl_\ell$, or  $\hatsl_\ell$ we only need  $U_q(\g)$-crystals.}

We fix some notation for restricting crystals and specify the projection  $z$ of weight lattices (as in \eqref{ss restrict Demazure}) for each case.
%the maps $z$ from \S\ref{ss restrict Demazure}.
% between the above algebras.
%Our main cases of interest are the following:
For a $U_q(\gl_\ell)$-crystal (resp. $U_q(\hatsl_\ell)$-crystal) $B$,
its  $U_q(\sl_\ell)$-restriction $\Res_{\sl_\ell} B$ has edges  $\ce_i,\cf_i$,  $i\in [\ell-1]$, and
$z$ is the canonical projection $\ZZ^\ell \to \ZZ^\ell/\ZZ(1,\dots, 1) \cong \mathring{P}$, $\epsilon_i \mapsto \mathring{\varpi}_i - \mathring{\varpi}_{i-1}$ (resp.  $P \to \mathring{P}$).
%The $U_q(\sl_\ell)$-restriction of a $U_q(\gl_\ell)$-crystal  $B$, denoted  $\Res_{\sl_\ell} B$, has the same edges as  $B$ and
%$z$ is the canonical projection  $\ZZ^\ell \to \ZZ^\ell/\ZZ(1,\dots, 1) \cong \mathring{P}$,  $\epsilon_i \mapsto \mathring{\varpi}_i - \mathring{\varpi}_{i-1}$.
For a $U'_q(\hatsl_\ell)$-seminormal crystal $B$,
its  $U_q(\sl_\ell)$-restriction $\Res_{\sl_\ell} B$ has edges  $\ce_i,\cf_i$,  $i\in [\ell-1]$, and
$z$ is the canonical projection $P_{\cl} \to \mathring{P}$
(this does not fit the form in \S\ref{ss restrict Demazure} and it need not yield a $U_q(\sl_\ell)$-crystal, but it does so for all $U'_q(\hatsl_\ell)$-seminormal crystals considered in this paper).
For a $U_q(\hatsl_\ell)$-crystal $B$, its  $U'_q(\hatsl_\ell)$-restriction has the same edges as  $B$ and
$z$ is  $\cl \colon P \to P_{\cl}$ (it is easily verified that this always yields a $U'_q(\hatsl_\ell)$-seminormal crystal).

\subsection{The  affine symmetric group and 0-Hecke monoid}
\label{ss the affine symmetric group and 0Hecke monoid}

The \emph{extended affine symmetric group $\widetilde{\SS}_\ell$} %of type $\hat{A}_{\ell-1}$
is the group generated by $\tau$ and $s_i$ $(i \in I)$ with relations
\begin{align}
\label{e extended affine Weyl pi i squared}
s_i^2 &= \idelm \\
\label{e extended affine Weyl far commutation}
s_is_j&=s_js_i  \text{ \ \qquad if  $a_{ij} = 0$ \, (equivalently, $i \notin \{j-1, j+1\}$) }\\
\label{e extended affine Weyl braid}
s_is_{i+1}s_i &= s_{i+1}s_is_{i+1} \\
\label{e extended affine Weyl tau}
\tau s_i &= s_{i+1} \tau \\
\label{e extended affine Weyl tau power}
\tau^\ell &= \idelm
\end{align}
Here, $i,j$ denote arbitrary elements of $I = \ZZ/ \ell \ZZ$.
The \emph{affine symmetric group $\eS_\ell$} is the subgroup of  $\widetilde{\SS}_\ell$
generated by the $s_i$ for  $i \in I$, and
\emph{the symmetric group $\SS_\ell$} is the subgroup
generated by $s_i$ for $i \in [\ell-1]$.
We have $\widetilde{\SS}_\ell = \Sigma \ltimes \eS_\ell$, where  $\Sigma = \{\tau^i \mid i \in [\ell]\} \cong \ZZ/\ell\ZZ$;
as in \S\ref{ss Generalized affine Demazure crystals and key positivity}, we also denote by
$\tau$ the Dynkin diagram automorphism $I \to I, \, i \mapsto i+1$, so that
 $\tau s_i \tau^{-1} = s_{\tau(i)}$.

Following the conventions of \cite{NaoiVariousLevels}, $\widetilde{\SS}_\ell$ is also naturally realized as a subgroup of $GL(\h^*)$:
%also the extended affine Weyl group associated to $\hatsl_\ell$, and
for $i \in I$, $s_i$ acts by  $s_i(\lambda) = \lambda - \langle \alpha_i^\vee, \lambda \rangle \alpha_i$
for  $\lambda \in \h^*$, and
$\tau \in GL(\h^*)$ is determined by  $\tau (\Lambda_i) = \Lambda_{i+1}$ for  $i\in I$
and  $\tau(\delta) = \delta$.
Another useful description of  $\Sigma \subset \widetilde{\SS}_\ell$ is as
the subgroup of   $\widetilde{\SS}_\ell  \subset GL(\h^*)$ which takes the set $\{\alpha_i \mid i \in I\}$ to itself;
moreover,
%$\Sigma$ acts by Dynkin diagram automorphisms on this set
 $\sigma(\alpha_i) = \alpha_{\sigma(i)}$ for all $\sigma \in \Sigma, i \in I$.
%where the seclatter  $\sigma$ is the Dynkin diagram automorphism.
%(this can be deduced from \S\ref{ss lie algebra hatsl} or see \cite[\S2.2]{NaoiVariousLevels}),
%This will be revisited in \S\ref{ss Sanderson theorem}.
%We also have $\widetilde{\SS}_\ell = \SS_\ell \ltimes T$ where $T \xrightarrow{\cong} \ZZ^{\ell}/{\ZZ(1,\dots, 1)}$

The \emph{0-Hecke monoid $\zaH_\ell$ of  $\widetilde{\SS}_\ell$} is the monoid generated by $\tau$ and $\zs_i$ ($i \in I$) with relations
\eqref{e extended affine Weyl far commutation}--\eqref{e extended affine Weyl tau power}
(with $\zs_i$'s in place of  $s_i$'s) together with
%$\zs_i^2 = \zs_i$ for $i \in I$.
\begin{align}
\label{e 0Hecke pi i squared pi i}
\zs_i^2 &= \zs_i
\end{align}
for  $i \in I$.
%\label{e 0Hecke far commutation}
%s_is_j&=s_js_i  \text{ \qquad \ \ \qquad for $|i-j|>1$ and $\{i,j\} \neq \{0,\ell-1\}$ }\\
%\label{e 0Hecke braid}
%s_is_{i+1}s_i &= s_{i+1}s_is_{i+1} \\
%% \text{\quad \ \ \quad  for all $i \in I$}.\\
%\label{e 0Hecke tau}
%\tau s_i &= s_{i+1} \tau
%\end{align}
%Here, $i$ denotes an arbitrary element of $I$ and indices are taken mod $\ell$.
The \emph{0-Hecke monoid $\zH_\ell$ of $\SS_\ell$} is
the submonoid of  $\zaH_\ell$ generated by $\zs_i$ for $i \in [\ell-1]$.

The \emph{length} of $w \in \eS_\ell$, denoted $\length(w)$,  is the minimum $m$ such that $w=s_{i_1}s_{i_2}\cdots s_{i_m}$ for some $i_j\in I$.
For $w \in \widetilde{\SS}_\ell$, we can write $w = \tau^i v$, $v \in \eS_\ell$;
define $\length(w) = \length(v)$.
An expression for $w \in  \widetilde{\SS}_\ell$ as a product of  $\tau$'s and  $s_i$'s is \emph{reduced} if it uses $\length(w)$  $s_i$'s.
Length and reduced expressions for elements of $\zaH_\ell$ are defined similarly.

\subsection{Dynkin diagram automorphisms and crystals}
\label{ss formal twist}
Any  $\sigma \in \Sigma$, viewed as an element of $GL(\h^*)$,
satisfies $\sigma(P) = P$ and since  $\sigma(\delta) = \delta$, it also yields an element of $GL(\h^*/\CC \delta)$ which satisfies
$\sigma(P_{\cl}) = P_{\cl}$; hence $\sigma$ yields automorphisms  of $P$ and  $P_{\cl}$.

%For a  $P$-weighted (resp. $P_{\cl}$-weighted) $I$-crystal $B$,
For $\sigma \in \Sigma$ and
$U_q(\hatsl_\ell)$-crystals (resp. $U'_q(\hatsl_\ell)$-seminormal crystals)  $B, B'$,
a bijection of sets $\theta \colon B \to B' $ is a  \emph{$\sigma$-twist} if
\begin{align*}
\sigma(\wt(b)) &= \wt(\theta(b)), \ \, \text{ and } \\
  \theta(\ce_{i}b) &=\ce_{\sigma(i)}\theta(b), \ \   \theta(\cf_{i}b)= \cf_{\sigma(i)}\theta(b) \ \text{ for all  $i \in I$, \ where  $\theta(0) := 0.$}
\end{align*}

For any $\Lambda \in P^+$, there is a unique  $\sigma$-twist
$\F_\sigma^\Lambda \colon B(\Lambda) \to B(\sigma(\Lambda))$,
which follows from  $\sigma(\alpha_i) = \alpha_{\sigma(i)}$
% $(i \in I)$
%the description of  $\Sigma \subset \widetilde{\SS}_\ell \subset GL(\h^*)$ as the subgroup
%which takes the set $\{\alpha_i \mid i \in I\}$ to itself (see, e.g.,
and the uniqueness of  local crystal bases of highest weight modules \cite{Kas1}.
%the local crystal basis of  $V(\sigma(\Lambda))$ \cite{Kas1}.
%and extracting the associated combinatorial data yields a $U_q(\g)$-seminormal crystal
%twist the module by \sigma and then note that the local base is still a local base, then use uniqueness

It is easily verified that if  $\theta_1 \colon B_1 \to B_1'$ and  $\theta_2 \colon B_2 \to B_2'$ are  $\sigma$-twists,
then  so is $\theta_1 \tsr \theta_2 \colon B_1 \tsr B_2 \to B_1' \tsr B_2'$.
Thus  the tensor product of maps
\[ \F_\sigma^{\Lambda^1} \tsr \cdots \tsr \F_\sigma^{\Lambda^p}
\colon B(\Lambda^1) \tsr \cdots \tsr B(\Lambda^p)  \to B(\sigma(\Lambda^1)) \tsr \cdots \tsr B(\sigma(\Lambda^p))
%\text{ given by  $\F_\sigma^{(\Lambda^1,\dots, \Lambda^p)} =
%\F_\sigma^{\Lambda^1} \tsr \cdots \tsr \F_\sigma^{\Lambda^p} $}
\]
%$\F_\sigma^{\Lambda^1} \tsr \cdots \tsr \F_\sigma^{\Lambda^p}$
is the natural choice of  $\sigma$-twist
from any tensor product  $B(\Lambda^1) \tsr \cdots \tsr B(\Lambda^p)$ of  highest weight $U_q(\hatsl_\ell)$-crystals, $\Lambda^1,\dots, \Lambda^p \in P^+$.
We let  $\twist{\tau}$ denote the operator on $\mathcal{D}(\hatsl_\ell)$
(see Definition \ref{d mathcal D})
which takes  $S \subset B(\Lambda^1) \tsr \cdots \tsr B(\Lambda^p)$ to  $\F_\tau^{\Lambda^1} \tsr \cdots \tsr \F_\tau^{\Lambda^p}(S)$.
This agrees with and explains the definition
of $\twist{\tau}$ in \S\ref{ss Generalized affine Demazure crystals and key positivity}.
Similarly, there is a unique $\tau$-twist
of $U'_q(\hatsl_\ell)$-seminormal crystals
 $\twist{\tau} \colon \B^\mu \to \B^\mu$, explained in~\S\ref{ss kat}.
%$U'_q(\hatsl_\ell)$-seminormal crystal for the   $\B^\mu$ is

\subsection{$U_q(\hatsl_\ell)$-Demazure crystals}
\label{ss hatsl demazure crystals}

Recall %from \S\ref{ss Generalized affine Demazure crystals and key positivity}
that for a subset $S$ of a seminormal crystal $B$ and $i \in I$, $\F_i S = \{\cf_i^kb \mid b\in S, k \ge 0 \} \setminus \{0\} \subset B$.
%and $\mathcal{D}(\g)$ is set of subsets of  $U_q(\g)$-crystals isomorphic to a disjoint union of $U_q(\g)$-Demazure crystals.

\begin{proposition}
\label{p lemma 4.3 Naoi}
The operators $\F_i$  $(i \in I)$ and $\twist{\tau}$ take $U_q(\hatsl_\ell)$-Demazure crystals to $U_q(\hatsl_\ell)$-Demazure crystals.
Hence they can be regarded as operators on  $\mathcal{D}(\hatsl_\ell)$
%(rather than operators on arbitrary subsets of crystals)
and as such they satisfy the 0-Hecke relations
\eqref{e extended affine Weyl far commutation}--\eqref{e 0Hecke pi i squared pi i} of  $\zaH_\ell$.
\end{proposition}
\begin{proof}
This follows from \cite[Lemma 4.3]{NaoiVariousLevels} and its proof
(which is largely based on \cite{KashiwaraDemazure}).
\end{proof}

Thus for any  $w \in \zaH_\ell$, we can define  $\F_w \colon  \mathcal{D}(\hatsl_\ell) \to \mathcal{D}(\hatsl_\ell)$ by
\[\F_w = \F_{c_1}\F_{c_{2}} \cdots \F_{c_m}, \]
where $w = c_1 \cdots c_m$ with each $c_j \in \{\zs_i \mid i \in I\} \sqcup  \{\tau\}$  and  $\F_{\zs_i}:= \F_i$,
and this is independent of the chosen expression for  $w$.
Recall that for  $\Lambda \in P^+$ and $w \in \zaH_\ell$,
$B_w(\Lambda) := \F_{w} \{u_\Lambda\}$.
%which is a $U_q(\hatsl_\ell)$-Demazure crystal contained in $B(\tau^i(\Lambda))$ where  $w = \tau^i v$,  $v \in unextended 0-Hecke?$.
% a slight variant/generalization of earlier def.
We thus have $\F_{w'}B_w(\Lambda) = \F_{w'}\F_w \{u_\Lambda\}
= \F_{w'w} \{u_\Lambda\}= B_{w'w}(\Lambda)$
for any $\Lambda \in P^+$ and $w, w' \in \zaH_\ell$.

\subsection{$U_q(\gl_\ell)$-Demazure crystals and key polynomials}
\label{ss gl Demazure crystal}

The symmetric group $\SS_\ell$ acts on  $\ZZ^\ell$ by permuting coordinates.
It is also convenient to define an action of $\zH_\ell$ on $\ZZ^\ell$ by
\begin{align}
\label{e Hl action}
\zs_i \spa \alpha =
\begin{cases}
s_i \spa \alpha &\text{if  $\alpha_i \ge \alpha_{i+1}$},\\
\alpha &     \text{if $\alpha_i \le \alpha_{i+1}$}.
\end{cases}
\end{align}

%The highest weight $U_q(\gl_\ell)$-crystals, denoted $B^\gl(\nu)$,
%%(crystals of the irreducible modules in  $\Oint(\gl_\ell)$)
%%integrable highest weight $U_q(\gl_\ell)$-modules
%are parameterized by weakly decreasing $\nu \in \ZZ^\ell$;
%denote by $u_\nu$ the highest weight element of $B^\gl(\nu)$.
%For weakly decreasing $\nu \in \ZZ^\ell$, let $B^\gl(\nu)$ denote the $U_q(\gl_\ell)$-crystal of highest weight  $\nu$
%and $u_\nu$ its highest weight element; such $\nu$ parameterize the highest weight $U_q(\gl_\ell)$-crystals.

Let $B^\gl(\nu)$ denote the highest weight $U_q(\gl_\ell)$-crystal
and $u_\nu$ its highest weight element, parameterized by $\nu \in \{\lambda \in \ZZ^\ell \mid \lambda_1 \ge \cdots \ge \lambda_\ell\}$, the dominant integral weights for  $U_q(\gl_\ell)$.
Definition \ref{d demazure crystal} defines $U_q(\gl_\ell)$-Demazure crystals but let us make this more explicit.
They are indexed by elements of $\ZZ^\ell$.
Let  $\alpha \in\ZZ^\ell$.  Denote by $\alpha^+$ the weakly decreasing rearrangement of  $\alpha$
and $\sfp(\alpha) \in \zH_\ell$ the shortest element such that $\sfp(\alpha) \alpha^+ = \alpha$.
%??either or both may be defined earlier
Define the \emph{$U_q(\gl_\ell)$-Demazure crystal}
indexed by  $\alpha$ to be $BD(\alpha) = \F_{\sfp(\alpha)} \{u_{\alpha^+}\} \subset B^\gl(\alpha^+)$.
%A \emph{$U_q(\gl_\ell)$-Demazure crystal} is defined to be a subset of some $B^\gl(\nu)$ of
%the form $B_{w}^\gl(\lambda) := \F_w u_\nu \subset B^\gl(\lambda)$ for $w \in \zH_\ell$.
%By a $U_q(\gl_\ell)$-Demazure crystal, we mean a subset $BD(\beta)$ for some  $\beta \in \ZZ^\ell$.

\begin{remark}
\label{r gl demazure crystals}
Analogous results to \S\ref{ss hatsl demazure crystals} hold for $U_q(\gl_\ell)$-Demazure crystals.
In particular, the  $\F_i$  $(i \in [\ell-1])$ can be regarded as operators on the set of $U_q(\gl_\ell)$-Demazure crystals  and as
such satisfy the 0-Hecke relations
\eqref{e extended affine Weyl far commutation}, \eqref{e extended affine Weyl braid},
\eqref{e 0Hecke pi i squared pi i} of $\zH_\ell$.
\end{remark}

%We establish some basic facts about the $\gl_\ell$-\emph{Demazure operators} or \emph{isobaric divided difference operators}  $\pi_i$ defined in
%\eqref{e intro pi def}
Consider the group ring of the  $\gl_\ell$-weight lattice  $\ZZ[\ZZ^\ell] = \ZZx$.
It has the monomial basis $\mathbf{x}^\alpha := x_1^{\alpha_1} x_2^{\alpha_2} \cdots x_\ell^{\alpha_\ell}$,
as $\alpha$ ranges over $\ZZl$.
Recall from \eqref{e intro pi def} that the Demazure operators $\pi_i$ are given by
$\pi_i = \frac{x_i-x_{i+1}s_i}{x_i-x_{i+1}}$ for  $i \in [\ell-1]$.
They were defined
as operators on  $\ZZ[q][\mathbf{x}]$, but we will also regard them as operators on
 $(\ZZx)[[q]]$  or $\ZZqqix$ for a ground ring  $\AA$, given by the same formula.
%They were defined as operators on  $\ZZ[q][\mathbf{x}]$ but we can also regard them as operators on
% the larger ring $\ZZxq$.
%(earlier they were defined as operators on a smaller ring but from now on we ...).
%?? turn to remark?  okay as is
%Here we regard these as operators  $\pi:\ZZx \to \ZZx$ or sometimes on a larger ring.
%We $\pi_i$ defined operators the $\gl_\ell$-\emph{Demazure operators}  (to distingush them from later ones, call them
%The symmetric group $\SS_\ell$ acts by permuting indices.
They  satisfy the 0-Hecke relations \eqref{e extended affine Weyl far commutation}, \eqref{e extended affine Weyl braid},
\eqref{e 0Hecke pi i squared pi i} of  $\zH_\ell$ (see e.g. \cite{RS}).
Thus, just as we discussed for $\F_w$ in \S\ref{ss hatsl demazure crystals},
 $\pi_w$ makes sense for any  $w \in \zH_\ell$ and  $\pi_w \pi_{w'} = \pi_{w w'}$ for all  $w, w' \in \zH_\ell$.

%Though in \eqref{e intro pi def} the  $\pi_i$ were defined as operators on  $\ZZ[q][x_1,\dots, x_\ell]$

\begin{definition}
\label{d def key}
For $\alpha \in \ZZ^\ell$, define the \emph{key polynomial} or Demazure character by
\begin{align}
\label{ed def key}
\kappa_\alpha &= \pi_{\sfp(\alpha)} \mathbf{x}^{\alpha^+}.
\end{align}
\end{definition}
If $\alpha \in \ZZl$
% (\alpha_1,\ldots,\alpha_\ell)$
is weakly decreasing, then $\kappa_\alpha$ is simply the monomial $\mathbf{x}^\alpha$,
while if $\alpha$ is weakly increasing, then $\kappa_\alpha$ is the Schur function $s_{\alpha^+}(\mathbf{x}) = s_{\alpha^+}(x_1,x_2,\ldots,x_\ell)$.

We record several facts about key polynomials for later use.
First,
 it follows from $\pi_{\zs_i} \pi_{w'} = \pi_{\zs_i w'}$ for all  $w' \in \zH_\ell$, that
%??useful: also need $\kappa_\alpha = \pi_{w} \mathbf{x}^{\alpha^+}$ for any  $w \in \zH_\ell$ such that
% $w \alpha^+ = \alpha$ and  $\zs_i \sfp(\alpha) \alpha^+= \sfp(\zs_i \alpha) \alpha^+$
%There holds (see \cite{RS} or use Remark \ref{r gl demazure crystals})
%??boost to 0-Hecke on subscripts? yes, this is a cleaner result (both versions are true)
\begin{align}
\label{e pi i on key}
\pi_i \spa \kappa_{\alpha}
= \kappa_{\zs_i \alpha},
%= \begin{cases}
%\kappa_{s_i \alpha} & \text{ if $\alpha_i \ge \alpha_{i+1}$, } \\
%\kappa_\alpha       & \text{ if $\alpha_i < \alpha_{i+1}$.}
%\end{cases}
\end{align}
where $\zs_i \alpha$ is as in \eqref{e Hl action}.

Next, note that for  $f \in \ZZx$,  $s_i(f) = f$ if and only if
$\pi_i(f) = f$ if and only if $f$ is symmetric in  $x_i, x_{i+1}$.
Further, for $f, g \in \ZZx$ with $s_i(f) = f$,
\begin{align}
\label{e pi i commutes with symmetric}
 \pi_i(fg) = f\pi_i(g).
\end{align}
%Note that  $\pi_i(f) = \pi_i(f)$ if and only if $f$ is symmetric in  $x_i, x_{i+1}$;
%the bottom line is says that  $\kappa_\alpha$ is symmetric in  $x_i, x_{i+1}$ if $\alpha_i < \alpha_{i+1}$.
It is immediate from Definition \ref{d def key} and \eqref{e pi i commutes with symmetric} that
\begin{align}
\label{e gl keys}
(x_1\cdots x_\ell)^d \kappa_{\alpha} = \kappa_{\alpha + (d, \dots, d)} \quad \text{ for all  $d\in \ZZ$ and  $\alpha \in \ZZl$}.
\end{align}

\begin{proposition}
\label{p key basis}
The key polynomials  $\{\kappa_\alpha \mid \alpha \in \ZZpl \}$ form a basis for  $\ZZ[x_1, \dots, x_\ell]$
and $\{\kappa_\alpha \mid \alpha \in \ZZl \}$ form a basis for  $\ZZx$.
\end{proposition}
\begin{proof}
The first holds by \cite[Corollary 7]{RS}, and the second then follows from \eqref{e gl keys}.
\end{proof}

\begin{remark}
\label{r key not stable}
We caution that though the key polynomials  $\kappa_\alpha$, $\alpha \in \ZZpl$, have the  $\gl_\infty$-stability property $\kappa_\alpha = \kappa_{(\alpha,0)} = \kappa_{(\alpha,0,0)}  = \cdots$, this is not so for the $\kappa_\beta$ with  $\beta \in \ZZl \setminus \ZZpl$.
\end{remark}

The \emph{character} of a subset  $S$ of a  $U_q(\gl_\ell)$-crystal
 is $\chr_\gl (S) = \sum_{b \in S} \mathbf{x}^{\wt(b)} \in \ZZx$.

\begin{proposition}
\label{p key character}
%Let  $C$ be a  $U_q(\gl_\ell)$-component of  $\B^{\mu;\mathbf{w}}$ with highest weight  $u_C$
%and suppose $C \cong BD^\gl(\alpha)$ for  $\alpha \in P_0$.
%Regarding  $C$ as a subset of the  $U_q(\gl_\ell)$-crystal of  $\B^{\mu}$,
%we have  $C \cong BD^\gl(\gamma)$, where  $\gamma = \sfp(\alpha) \gamma^+$ and  $\gamma^+ = \content(u_C)$.
The characters of $U_q(\gl_\ell)$-Demazure crystals are key polynomials: for any $\alpha \in \ZZ^\ell$,
\begin{align}
\chr_\gl (BD(\alpha)) = \sum_{b \in BD(\alpha)} \mathbf{x}^{\wt(b)} = \kappa_\alpha(x_1, \dots, x_\ell).
\end{align}
\end{proposition}
\begin{proof}
This is a consequence of \cite{KashiwaraDemazure}.  Note that the setup of \cite{KashiwaraDemazure} encompasses
the $\gl_\ell$ case with weight lattice  $\ZZ^\ell$ (see \cite[\S5]{KashiwaraSurvey}), and the
Demazure operators defined therein match the $\pi_i$ in the definition of key polynomials.
\end{proof}

\section{The rotation theorem for tame nonsymmetric Catalan functions}
%Key polynomials and nonsymmetric Catalan functions}
\label{s Rotation identity}

We give the proof of the rotation Theorem \ref{t character rotation formula},
which requires Demazure operator identities and
an in-depth study of polynomial truncation. %and its interaction with the multiplication by  $x_i$ operators.
%It is somewhat surprising that polynomial truncation is so important here
%since in our previous studies of Catalan functions \cite{BMPS,BMPS2}, it was a
%minor issue---after all, the series %version of  $H(\Psi;\gamma;w)$ (without the truncation)
%$\pi_w \big(  \prod_{(i,j) \in \Psi} (1-q x_i/x_j)^{-1} \mathbf{x}^\gamma \big)$
%can be recovered from the polynomials  $H(\Psi;\gamma+(d,\dots, d);w)$ for arbitrarily  large  $d$.
%However, it %(and its interaction with multiplication by %plays a critical role in the rotation theorem.
Interestingly, the expression it gives %(the right side of \eqref{ec character rotation formula})
for tame nonsymmetric Catalan functions is automatically polynomially truncated,
%whereas we had to explicitly add the truncation operator in their definition.
whereas we had to explicitly add the truncation in our definition of these functions.
%Interestingly, while we had to
%explicitly add polynomial truncation in the
%definition of nonsymmetric Catalan functions,
%the  expression in the rotation theorem
%is automatically polynomially truncated.

\begin{definition}
\label{d poly}
The \emph{polynomial truncation operator}, denoted  $\poly$,  is the linear operator on
 $\ZZx$ determined by its action on the basis $\{\kappa_\alpha \mid \alpha \in \ZZ^\ell\}$\,:
\[\poly(\kappa_\alpha) = \begin{cases}
\kappa_\alpha & \text{if $\alpha \in \ZZ_{\ge 0}^\ell$}\spa, \\
0 &\text{otherwise}.
\end{cases}\]
We extend this in the natural way to a linear operator on $\ZZxq$ by
\break $\poly(\sum_{d\ge 0} f_d q^d)$ $= \sum_{d \ge 0} \poly(f_d) q^d$ for any  $f_d \in \ZZx$.
\end{definition}

\subsection{Root expansion}
A straightforward yet surprisingly powerful recursion played an important role for the Catalan functions in \cite{BMPS}.
This is easily generalized to the nonsymmetric setting.
For a root ideal $\Psi$, we say  $\alpha\in \Psi$ is a \emph{removable root of  $\Psi$} if  $\Psi \setminus \alpha$ is a root ideal.
For $\alpha = (i,j) \in \Delta^+_\ell$, write
$\eroot{\alpha} = \epsilon_i - \epsilon_j \in \ZZ^\ell$.
% for the corresponding positive root.

\begin{proposition}
\label{p inductive computation atom}
%\label{p Rlambda recurrence nonroot HH}
%??adding braces on singletons here makes it look much worse
%The Catalan functions satisfy the following recurrences.
Let $(\Psi, \gamma, w)$ be a labeled root ideal.
%For any root $\beta$ addable to $\Psi$,
%\begin{align}\label{e Rlambda recurrence nonroot HH}
%\HH(\Psi; \gamma; w) = \HH(\Psi \cup \beta; \gamma; w) - q\;\! \HH(\Psi\cup \beta; \gamma+ \eroot{\beta}; w).
%\end{align}
For any removable root $\alpha$ of $\Psi$,
\begin{align}
\label{e Rlambda recurrence HH}
\HH(\Psi;\gamma; w) = \HH(\Psi \setminus \alpha; \gamma; w) + q\;\! \HH(\Psi; \gamma+ \eroot{\alpha}; w). \qquad
%\label{e Rlambda recurrence}
%H_\gamma^{\Psi} = H_\gamma^{\Psi \setminus \alpha} + tH_{\gamma+ \eroot{\alpha}}^{\Psi}.
\end{align}
\end{proposition}
\begin{proof}
Apply the linear operator $\pi_w \circ \poly$ to the following identity of series:
\begin{align*}
\prod_{(i,j) \in \Psi} \! \big(1-q x_i/x_j\big)^{-1} \mathbf{x}^\gamma
&= \big(1-q\mathbf{x}^{\eroot{\alpha}}\big)^{-1} \! \prod_{(i,j) \in \Psi \setminus \alpha} \! \big(1-q x_i/x_j\big)^{-1} \mathbf{x}^\gamma \\
&= \big(1+q\mathbf{x}^{\eroot{\alpha}}\big(1-q\mathbf{x}^{\eroot{\alpha}}\big)^{-1}\big)\prod_{(i,j) \in \Psi \setminus \alpha}\! \big(1-q x_i/x_j\big)^{-1} \mathbf{x}^\gamma  \\
&= \! \prod_{(i,j) \in \Psi \setminus \alpha} \!\! \big(1-q x_i/x_j\big)^{-1} \mathbf{x}^\gamma
\spa + \spa q \!\! \prod_{(i,j) \in \Psi} \! \big(1-q x_i/x_j\big)^{-1} \mathbf{x}^{\gamma+\eroot{\alpha}}.  \qedhere
\end{align*}
\end{proof}

\subsection{Polynomial truncation}
\label{ss polynomial truncation}

Polynomial truncation is better understood using
%In practice, computing the polynomial part is facilitated using
the following symmetric bilinear form
which comes from Macdonald theory and was given a self-contained treatment by Fu and Lascoux
\cite{FLkernel}.
%There is a symmetric bilinear form  $(\cdot, \cdot)$ on $\ZZx$ as follows:
%$\ZZx \times \ZZx \to \ZZ$ given by  % by .
For $f, g \in \ZZx$, define
\begin{align*}
(f,g) = \CT \! \bigg(f(x_1, \dots, x_\ell)g(x_\ell^{-1},\dots, x_1^{-1}) \prod_{(i,j) \in \Delta^+_\ell} \! (1-x_i/x_j)\bigg),
\end{align*}
where $\CT$ denotes taking the constant term.

%For $i \in [\ell-1]$, define the linear operators  $\pi_i$ and  $\hat{\pi}_i$ on  $\ZZx$ by
%\begin{align}
%\hat{\pi}_i(f) &= (\pi_i - 1)(f).
%\end{align}
For  $\alpha \in \ZZ^\ell$, define the \emph{Demazure atom} by
$\hat{\kappa}_\alpha = \hat{\pi}_{\sfp(\alpha)} \mathbf{x}^{\alpha^+}$.
Here, $\hat{\pi}_i:= \pi_i-1$ and $\hat{\pi}_w := \hat{\pi}_{i_1} \cdots \hat{\pi}_{i_m}$, where $w = \zs_{i_1} \cdots \zs_{i_m}$ is a reduced expression;
this is well defined since the $\hat{\pi}_i$ satisfy the braid relations.

%\begin{theorem}[{\cite[Theorem 2.8.1]{LascouxBookpolynomials}}]
\begin{theorem}[{\cite[Theorem 15]{FLkernel}}]
%The statement in FLkernel is only for nonnegative integer vectors but seems to easily generalize to \ZZ^\ell
\label{t inner product key}
The key polynomials and Demazure atoms are dual bases with respect to $(\cdot,\cdot )$: for $\alpha, \beta \in \ZZ^\ell$, $(\kappa_{\alpha}, \hat{\kappa}_{w_0 \beta}) = \delta_{\alpha,\beta}$,
where  $\delta$ is the Kronecker delta and $w_0$ is the longest permutation in  $\SS_\ell$.
%??warning is very important this is S_l and not H_l
%??move Kronecker delta def to early notation section? no
%$(\hat{\kappa}_{\alpha}, \kappa_{w_0 \beta}) = \delta_{\alpha,\beta}$. ??probably dont need
\end{theorem}
\begin{proof}
The statement in \cite[Theorem 15]{FLkernel} is for $\alpha, \beta \in \ZZ_{\ge 0}^\ell$, and this yields the statement for $\alpha, \beta \in \ZZ^\ell$ too
since  it implies that for $d$ sufficiently large,
%the result for $\alpha, \beta \in \ZZ_{\ge 0}^\ell$ implies
$(\kappa_{\alpha}, \hat{\kappa}_{w_0 \beta}) =$ \break
$((x_1\cdots x_\ell)^d \kappa_{\alpha}, (x_1\cdots x_\ell)^d \hat{\kappa}_{w_0 \beta} ) \! = \! (\kappa_{\alpha+ (d,\dots,d)}, \hat{\kappa}_{w_0 \beta+ (d,\dots, d)}) \! = \! \delta_{\alpha+ (d,\dots,d) ,\beta+ (d,\dots, d)}
\! = \!\delta_{\alpha, \beta}$. \qedhere
%, where the second to last equality holds for $d$ sufficiently large
\end{proof}

%Hence we can compute the coefficient  of $\kappa_\lambda$ in the key expansion of any $f \in \ZZxq$ by
Hence, letting $c_{\alpha,\beta} \in \ZZ_{\ge 0}$ denote the atom to monomial expansion coefficients, i.e., $\hat{\kappa}_{\alpha} = \sum_{\beta \in \ZZ^\ell}c_{\alpha,\beta}\mathbf{x}^\beta$, the coefficient of $\kappa_\alpha$ in the key expansion of any $f \in \ZZx$ is given by
\begin{align*}
& \CT\! \bigg(f \!\! \prod_{(i, j) \in \Delta^+} \!\!\! (1-x_i/x_j) \hat{\kappa}_{w_0 \alpha}(x_\ell^{-1}, \dots, x_1^{-1}) \bigg)
= \sum_{\beta \in \ZZ^\ell} c_{w_0 \alpha,\beta} \CT\! \bigg(f \!\! \prod_{(i, j) \in \Delta^+}\!\!\! (1-x_i/x_j) \mathbf{x}^{-\rev(\beta)} \bigg) \\
&= \sum_{\beta \in \ZZ^\ell} c_{w_0 \alpha,\beta} \bigg( \text{coefficient of } \mathbf{x}^{\rev(\beta)}\text{ in the monomial expansion of } \, f \!\! \prod_{(i, j) \in \Delta^+} \!\!\! (1-x_i/x_j) \bigg),
\end{align*}
% and $\tilde{f} := f \prod_{(i, j) \in \Delta^+}(1-x_ix_j^{-1})$.
where $\rev(\beta) := (\beta_\ell, \dots, \beta_1)$ denotes the reverse of any  $\beta = (\beta_1, \dots, \beta_\ell)\in \ZZl$.
%In particular, if $\lambda \in \NN^\ell$, then $\alpha$ ranges over $\NN^\ell$ of size $|\lambda|$.
We package this into the following corollary:

\begin{corollary}
\label{c coef of key}
For  $f \in \ZZx$,
%Set $\tilde{f} := f \prod_{(i, j) \in \Delta^+}(1-x_ix_j^{-1})$.
%the coefficient of $\kappa_\lambda$ in the key expansion of $f$ is
\begin{align}
\label{ec coef of key}
f = \sum_{\alpha, \beta \in \ZZl} c_{w_0 \alpha,\beta} \Big(\text{coefficient of } \mathbf{x}^{\rev(\beta)}\text{ in } \, f \!\! \prod_{(i, j) \in \Delta^+} \!\!\! (1-x_i/x_j)\Big) \kappa_\alpha,
\end{align}
\begin{align}
\label{ec coef of key poly}
\poly(f) =
\sum_{\alpha, \beta \in \ZZpl}
 c_{w_0 \alpha,\beta} \Big(\text{coefficient of } \mathbf{x}^{\rev(\beta)}\text{ in } \, f \!\! \prod_{(i, j) \in \Delta^+} \!\!\! (1-x_i/x_j) \Big) \kappa_\alpha.
\end{align}
\end{corollary}

\begin{proposition}
\label{p poly part}
Let $\gamma \in \ZZ^\ell$ and $w \in \zH_\ell$ be arbitrary.
\begin{itemize}[align=left, itemindent = -12pt]
\item[\emph{(i)}] For any $f \in \ZZxq$, $\poly(\pi_i (f)) = \pi_i(\poly(f))$.
\item[\emph{(ii)}] For any $\alpha \in \ZZ_{\ge 0}^\ell$, $\poly(\mathbf{x}^\alpha) = \mathbf{x}^\alpha$.
\item[\emph{(iii)}] If  $\sum_{a = k}^\ell \gamma_a < 0$ for some  $k \in [\ell]$,
then $\poly(\mathbf{x}^{\gamma}) = 0$.
\item[\emph{(iv)}] If  $\sum_{a = k}^\ell \gamma_a < 0$ for some  $k \in [\ell]$, then $H(\Psi; \gamma; w) = 0$ for any root ideal~$\Psi \subset \Delta^+_\ell$.
\item[\emph{(v)}] If  $\gamma_{m+1} = \cdots = \gamma_\ell = 0$, then
 $H(\Psi;\gamma;w) = H(\Psi';\gamma;w )$ for any $\Psi, \Psi' \subset \Delta^+_\ell$ such that
 $\Psi \cap \Delta^+_m = \Psi' \cap \Delta^+_m$.
\end{itemize}
\end{proposition}
\begin{proof}
Statement (i) is immediate from the definition of polynomial truncation and \eqref{e pi i on key}.
Both $\{\mathbf{x}^\alpha \mid \alpha \in \ZZ_{\ge 0}^\ell\}$
and $\{\kappa_\alpha \mid \alpha \in \ZZ_{\ge 0}^\ell\}$ are $\ZZ$-bases for $\ZZ[x_1,\dots, x_\ell]$
(Proposition~\ref{p key basis}).
Since $\poly$ acts as the identity on the latter basis by definition, (ii) follows.

To prove (iii), by \eqref{ec coef of key poly}, it suffices to show that for any term $c \spa \mathbf{x}^\zeta$ in the
monomial expansion of $\mathbf{x}^\gamma \prod_{(i, j) \in \Delta^+}(1-x_i/x_j)$, we have
$\zeta \notin \ZZpl$.
%We prove $\poly(\mathbf{x}^{\gamma})=0$ by showing that for any  $\zeta \in \ZZ_{\ge 0}^\ell$,
%the coefficient of $\kappa_{\zeta}$ in the key expansion of $\mathbf{x}^\gamma$ is 0.
%By the above computation, this coefficient is
%\[\sum_{\alpha \in \ZZ_{\ge 0}^\ell} c_{\zetaw_0,\alpha} \Big( \text{coefficient of } \mathbf{x}^{\rev(\alpha)}\text{ in the monomial expansion of } \mathbf{x}^\gamma \!\prod_{(i, j) \in \Delta^+}(1-x_ix_j^{-1}) \Big).\]
Indeed, for such a term we must have
%of $\mathbf{x}^\gamma \! \prod_{(i, j) \in \Delta^+}(1-x_ix_j^{-1})$,
$\sum_{a = k}^\ell \zeta_a \le \sum_{a = k}^\ell \gamma_a < 0$ as needed.

To prove (iv), recall
$H(\Psi;\gamma; w)
%pi_w \big( \poly \! ))$, where
% $\sum_{m \ge 0}$
%and  $\bbH(\Psi;\gamma; w) = \big(\mathbf{x}^\gamma \! \spa \prod_{(i,j) \in \Psi}  \big(1+ q x_i/x_j + q^2 (x_i/x_j)^2 + \cdots \big)$
=  \pi_w \big( \poly \! \spa \big(\mathbf{x}^\gamma \! \spa \prod_{(i,j) \in \Psi}  \big(1+ q x_i/x_j + q^2 (x_i/x_j)^2 + \cdots \big) \big) \big)$
from Definition \ref{d HH gamma Psi}.
Any term $c \spa \mathbf{x}^\zeta = \prod_{(i,j)\in \Psi} q^{d_{ij}}(x_i/x_j)^{d_{ij}}$ arising in the expansion of the product over $\Psi$
satisfies
 %$f_d \in \ZZx$ be the coefficient of  $q^d$ in the product over  $\Psi$.
%Any term  $c \spa \mathbf{x}^\zeta$ in the monomial expansion of  $f_d$ satisfies
$\sum_{a = k}^\ell \zeta_a \le 0$, so $\sum_{a = k}^\ell (\gamma+ \zeta)_a < 0$  and
 $\pi_w ( \poly (c \, \mathbf{x}^{\gamma+\zeta})) = 0$ by (iii).
Thus $H(\Psi; \gamma; w) = 0$.
Statement (v) follows similarly
 from the observation that
any term $c \spa \mathbf{x}^\zeta = \prod_{(i,j)\in \Psi} q^{d_{ij}}(x_i/x_j)^{d_{ij}}$
with $d_{ij} > 0$ for some root $(i,j)$ with  $j > m$,
%arising in the product over  $\Psi$
satisfies $\sum_{a = j}^\ell (\gamma+\zeta)_a < 0$.
 %$\poly (c \, \mathbf{x}^{\gamma+\zeta}) = 0$ for any
%which arose by using a term  $q^i $ in the monomial expansion of  $f_d$
% any
\end{proof}

\begin{corollary}
The nonsymmetric Catalan functions lie in  $(\ZZ[q])[x_1,\dots, x_\ell]$
rather than the larger  $(\ZZ[x_1,\dots,x_\ell])[[q]]$, i.e., they are finite sums of key polynomials  $\kappa_\alpha$,  $\alpha\in \ZZ_{\ge 0}^\ell$, with coefficients which are polynomials in  $q$ with integer coefficients.
%$which are polynomials in  $q$.
\end{corollary}
\begin{proof}
Similar to the proof of (iv) above,
one checks that in computing  $H(\Psi;\gamma;w)$,~any~term $q^d \spa \mathbf{x}^\zeta\! =\! $ $ \prod_{(i,j)\in \Psi} q^{d_{ij}}(x_i/x_j)^{d_{ij}}$
with $d > \ell|\gamma|$ satisfies $\sum_{a = k}^\ell (\gamma+\zeta)_a < 0$ for some  $k \in [\ell]$.
%Apply Proposition \ref{p poly part} (iii). okay to omit? I think so
%
%coefficient of  $q^d$ in the series $\bbH(\Psi;\gamma; w)$, for $d > \ell|\gamma|$,
%is a sum of monomial terms  $c \spa \mathbf{x}^\zeta$ with  $\sum_{a = k}^\ell \zeta_a < 0$ for some  $k \in [\ell]$.
%Apply Proposition \ref{p poly part} (iii).
\end{proof}

\subsection{Identities for Demazure operators and polynomial truncation}

%Let $\Phi$ be the operator on $\ZZqx$ defined by $\Phi(f) = f(x_2, \dots, x_\ell, qx_1)$ (just as in
%\eqref{e intro Phi operator} but for simplicity we used the domain $\ZZ[q][x_1,\dots, x_\ell]$).

%We assemble identities for the proof of the rotation theorem.
Recall from \eqref{e intro Phi operator} that  $\Phi$ is the operator on $\ZZ[q][\mathbf{x}]$ given
by $\Phi(f) = f(x_2, \dots, x_\ell, qx_1)$; here we will regard it as an
operator on  $\ZZ[q,q^{-1}][x_1^{\pm 1}, \dots, x_\ell^{\pm 1}]$.

%here we regard  $\Phi$ as an operator on  $\ZZqx$ rather than  ).
%$f \in \ZZ[x_1, \dots, x_\ell][[q]]$, define
%Since $\Phi$ is almost the rotation automorphism of polynomials, $\Phi$ twists most of the $\pi_i$ by this automorphism:

\begin{proposition}
\label{p Phi sort of commutes}
For any $f \in \ZZ[q,q^{-1}][x_1^{\pm 1}, \dots, x_\ell^{\pm 1}]$,
\[\pi_{i+1} \Phi(f) = \Phi \pi_{i}(f) \text{ \quad for $i = 1, \dots, \ell-2$}.\]
Thus, recalling that  $\tau \zs_i \tau^{-1} = \zs_{i+1}$, we have
\begin{align*}
\pi_{\tau v \tau^{-1}} \Phi(f) = \Phi \pi_{v}(f) \text{ \quad for any  $v \in \zH_{\ell-1} \times \zH_1 \subset \zH_\ell$}.
\end{align*}
\end{proposition}
\begin{proof}
This is a direct computation from the definition of the Demazure operator $\pi_i$:
\begin{multline*}
\Phi(\pi_i(f)) %&
= \Phi\Big(\frac{x_if-x_{i+1}s_i(f)}{x_i-x_{i+1}}\Big)
%\frac{x_{i+1}\Phi(f)-x_{i+2}\Phi(s_i(f))}{x_{i+1}-x_{i+2}} \\
%&
= \frac{x_{i+1}\Phi(f)-x_{i+2}s_{i+1}(\Phi(f))}{x_{i+1}-x_{i+2}} =
\pi_{i+1}(\Phi(f)). \qedhere
\end{multline*}
\end{proof}

\begin{lemma}
\label{l poly x1}
For any $f \in \ZZ[x_1^{\pm 1},\dots, x_{\ell-1}^{\pm 1}]$ and  $a \ge 0$,
 $\poly(x_1^a \Phi(f)) = x_1^a \Phi(\poly( f))$.
%??note : this is Phi for ell
\end{lemma}
\begin{proof}
Since $\poly$ and $\Phi$ are linear operators, it is enough to prove this for $f$ ranging over the  $\ZZ$-basis $\{\kappa_{\zeta} \mid \zeta \in \ZZ^{\ell-1} \}$ of  $\ZZ[x_1^{\pm 1},\dots, x_{\ell-1}^{\pm 1}]$.
In light of Remark \ref{r key not stable}, computing $\poly(\kappa_\zeta)$ is nontrivial as we have defined polynomial truncation with respect to the basis $\{\kappa_{\alpha} \mid \alpha \in \ZZ^{\ell} \}$ of  $\ZZx$.
However, we can use Demazure operators:
write $\kappa_\zeta = \pi_v \mathbf{x}^{(\mu,0)}$
with $\mu = \zeta^+ \in \ZZ^{\ell-1}$
%the weakly decreasing rearrangement of $\zeta$
and $v = \sfp(\zeta) \in \zH_{\ell-1}$
as in Definition \ref{d def key} but for  $\ell-1$ in place of  $\ell$.
%is and $v \in \SS_{\ell-1}$ is such that $v(\mu) = \zeta$.
%Let $v^+ = \tau v \in \SS_1 \times \SS_{\ell-1} \subset \SS_\ell$.
%denote the permutation result of obtained from rotating $v$, i.e., $v = %s_{i_1}s_{i_2}\cdots s_{i_m}$, $v^+ = s_{i_1+1}s_{i_2+1}\cdots s_{i_m+1}$.
 Then
\begin{align*}
x_1^a \Phi(\poly(\pi_v \mathbf{x}^{(\mu, 0)} )) =
%x_1^a \pi_{\tau v \tau^{-1}}\Phi(\poly( \mathbf{x}^{(\mu, 0)} ) ) =
 \pi_{\tau v \tau^{-1}}x_1^a\Phi(\poly( \mathbf{x}^{(\mu, 0)} ) ) =
\begin{cases}
\pi_{\tau v \tau^{-1}}\mathbf{x}^{(a,\mu)} & \text{if $\mu \in \ZZ_{\ge 0}^{\ell-1}$}\\
0 & \text{otherwise},
\end{cases}
\end{align*}
where the first equality is by Propositions \ref{p poly part} (i) and \ref{p Phi sort of commutes} and then \eqref{e pi i commutes with symmetric};
the second equality uses
Proposition \ref{p poly part} (ii) for the top line and Proposition \ref{p poly part} (iii) for the bottom line
($\mu$ weakly decreasing implies $\mu_{\ell-1} < 0$ if $\mu \notin \ZZ_{\ge 0}^{\ell-1}$).

On the other hand, there holds
\begin{align*}
\poly(x_1^a \Phi(\pi_v \mathbf{x}^\mu)) =
\pi_{\tau v \tau^{-1}}\poly(x_1^a \Phi(\mathbf{x}^\mu)) =
\pi_{\tau v \tau^{-1}}\poly(\mathbf{x}^{(a,\mu)}) =
\begin{cases}
\pi_{\tau v \tau^{-1}}\mathbf{x}^{(a,\mu)} & \text{if $\mu \in \ZZ_{\ge 0}^{\ell-1}$}\\
0 & \text{otherwise}.
\end{cases}
\end{align*}
%where the first equality is by \eqref{e pi i commutes with symmetric} and Propositions \ref{p poly part} (i) and \ref{p Phi sort of commutes}; the last equality uses $a \ge 0$ and Proposition \ref{p poly part} (ii) for the top line and Proposition \ref{p poly part} (iii) for the bottom line.
The justification is just as in the previous paragraph (the last equality uses  $a \ge 0$).
%
 %($\mu$ decreasing implies $\mu_{\ell-1} < 0$ if $\mu \notin \ZZ_{\ge 0}^{\ell-1}$).
%This agrees with the right side of \eqref{e pi v plus} since $\Phi(\kappa_\zeta) = \Phi(\pi_v \mathbf{x}^\mu) = \pi_{v^+} \mathbf{x}^{(0,\mu)} $.
%By Proposition \ref{p Phi sort of commutes}
\end{proof}

Lascoux  \cite[\S4.1]{LascouxBookpolynomials} gives a partial description of a Monk's rule for key polynomials,
i.e.  $x_i \kappa_\alpha$ expanded in key polynomials.  The computations therein are similar
to the next three lemmas, which we need
% though less explicit as we only need enough
for polynomial part computations.
%However, since proofs are not given in full, we include the following result for completeness.
Recall that $\hat{\pi}_i = \pi_{i}-1$.

\begin{lemma}
For any  $f \in \ZZx$,
\begin{align}
\label{el pi and xi 1}
 x_{i+1}\pi_i(f) &= \hat{\pi}_i(x_if) \!\! && \text{ for $i\in  [\ell-1]$,}\\
%x_{i+1}^{-1}\hat{\pi}_i(g) = \pi_i(g x_i^{-1}) \quad \text{ and } \quad
\label{el pi and xi 2}
x_{i}^{-1}\pi_i(f) &= \hat{\pi}_i(x_{i+1}^{-1} f) \!\! && \text{ for $i\in  [\ell-1]$,} \\
\label{el pi and xi 3}
%x_{i}^{-1}\hat{\pi}_{i-1}(f) &= \pi_{i-1}( x_{i-1}^{-1} f) \quad \text{ for $2 \le i \le \ell$.}
\hphantom{aaaaaaaaaaa} x_{i}^{-1}\pi_{i-1}(f) &= \pi_{i-1}( x_{i-1}^{-1} f) +  x_{i}^{-1} f  && \text{ for $2 \le i \le \ell$.} \hphantom{aaaaaaaaaaa}
\end{align}
\end{lemma}
\begin{proof}
%These can all be proved by direct computation using \eqref{e pi def}.  We verify \eqref{el pi and xi 1}:
The identity \eqref{el pi and xi 1} is proved by direct computation:
%The first for example,
\begin{align*}
\hat{\pi}_i(x_if)
=
\frac{x_i (x_if)-x_{i+1}s_i(x_if)}{x_i-x_{i+1}}  - \frac{(x_i-x_{i+1})x_i f}{x_i-x_{i+1}}
=\frac{x_{i+1}x_i f-x_{i+1}^2s_i(f)}{x_i-x_{i+1}}
=x_{i+1}\pi_i(f).
\end{align*}
Multiplying both sides by $x_i^{-1}x_{i+1}^{-1}$ (which commutes with  $\pi_i$)
 yields \eqref{el pi and xi 2}, and \eqref{el pi and xi 3} is a
rearrangement of  \eqref{el pi and xi 1}.
%??useful alternative proof
%It may be checked directly that the divided difference operator $\partial_i$ satisfies the $i$-derivation property:
%\begin{align}
%\partial_i(fg) = \partial_i(f)s_i(g) + f\partial_i(g).
%\end{align}
\end{proof}

\begin{lemma}
\label{l hat pi mult xi}
Let $f \in \ZZx$ such that
$s_i(f) = f$ \spa for \spa $a < i \le \ell-1$.  Then
\begin{align}
\label{el hat pi mult xi 2}
x_\ell\spa \pi_{\ell-1} \pi_{\ell-2} \cdots \pi_a (f)= \hat{\pi}_{\ell-1} \hat{\pi}_{\ell-2} \cdots \hat{\pi}_a (x_af) = \hat{\pi}_{\ell-1} \pi_{\ell-2} \cdots \pi_a (x_af).
\end{align}
\end{lemma}
\begin{proof}
Applying \eqref{el pi and xi 1} repeatedly yields the first equality of \eqref{el hat pi mult xi 2}.
%\begin{multline*}
%x_\ell\pi_{\ell-1}\pi_{\ell-2}  \cdots \pi_a (f) = \hat{\pi}_{\ell-1}x_{\ell-1}\pi_{\ell-2} \cdots \pi_a(f) = \cdots =
%\hat{\pi}_{\ell-1}\hat{\pi}_{\ell-2} \cdots \hat{\pi}_a (x_af).
%\end{multline*}
For the second equality, we use that $\hat{\pi}_i (f) = 0$ for  $i > a$ and  $\hat{\pi}_i\pi_j = \pi_j\hat{\pi}_i$ for  $i > j+1$, and compute as follows:
%%$\hat{\pi}_i(x_a f)= x_a \hat{\pi}_i( f)=0$) for $i = a+1, \dots, \ell-1$;
%$\hat{\pi}_i (\pi_j \pi_{j-1}\cdots \pi_a (x_a f)) = 0$ as well,
% % = \pi_j \pi_{j-1}\cdots \pi_a x_a \hat{\pi}_i (f) = 0$ for
%for $i > a $ and  $j < i-1$.
%Thus
\begin{align*}
& \hat{\pi}_{\ell-1} \cdots \hat{\pi}_a (x_af) =
\hat{\pi}_{\ell-1} \cdots \hat{\pi}_{a+1} \pi_a(x_a f) -
\hat{\pi}_{\ell-1} \cdots \hat{\pi}_{a+1} (x_a f)\\
& =  \hat{\pi}_{\ell-1} \cdots \hat{\pi}_{a+1} \pi_{a} (x_a f)=
  \hat{\pi}_{\ell-1} \cdots \hat{\pi}_{a+2} \pi_{a+1}\pi_{a} (x_a f)
    -\hat{\pi}_{\ell-1} \cdots \hat{\pi}_{a+2} \pi_{a} (x_a f)
%=\hat{\pi}_{\ell-1} \cdots \hat{\pi}_{a+2} \pi_{a+1}\pi_{a} (x_a f)
%    -\hat{\pi}_{\ell-1} \cdots \hat{\pi}_{a+3} \pi_{a} (x_a \hat{\pi}_{a+2}(f))
    \\
& = \hat{\pi}_{\ell-1} \cdots \hat{\pi}_{a+2} \pi_{a+1}\pi_{a} (x_a f)
= \cdots
= \hat{\pi}_{\ell-1} \pi_{\ell-2} \cdots \pi_{a}(x_a f).   \qedhere
\end{align*}
\end{proof}

\begin{lemma}
\label{l x i on key}
For $i \in [\ell]$ and  $\alpha \in \ZZl$,
$x_i^{-1}\kappa_\alpha \in \ZZ\big\{\kappa_\beta \mid \beta^+ = \alpha^+ - \epsilon_j \text{ for some $j \in [\ell]$} \big\}$.
%is a sum of key polynomials $\kappa_{\beta}$ with $\beta^+ = \alpha^+ - \epsilon_j$ for some $j$.
\end{lemma}
\begin{proof}
Write $\kappa_\alpha=\pi_v\mathbf{x}^\mu$ with  $\mu = \alpha^+$ and  $v = \sfp(\alpha)$ as in
Definition \ref{d def key}.
The proof is by induction on  $\length(v)$.
For the base case $v = \idelm$, let  $z$ be the index
such that $\mu_i = \mu_{i+1} = \cdots = \mu_z > \mu_{z+1}$
(interpret  $\mu_{\ell+1} = -\infty$ so that $z = \ell$ if $\mu_i = \cdots = \mu_\ell$).
%$z = \min\big(\{j \in [i,\ell-1] \mid \mu_j > \mu_{j+1}\} \cup \{\ell \}\big)$.
%holds since $x_i\mathbf{x}^\mu$ is either z key polynomial itself (if $\mu_{i-1} > \mu_i $) or if $\mu_{i-1} = \mu_i$,
Then $x_i^{-1} \mathbf{x}^\mu = \hat{\pi}_i \hat{\pi}_{i+1} \cdots \hat{\pi}_{z-1} x_z^{-1} \mathbf{x}^\mu$,
which belongs to $\ZZ\{\kappa_\beta \mid \beta^+  = \mu - \epsilon_z \}$ by \eqref{e pi i on key}.

Now suppose $v \neq \idelm$.  Choose a length additive factorization $v= \zs_j u$.
%where this is a length additive factorization reduced expansion.
Using \eqref{el pi and xi 2} and \eqref{el pi and xi 3} we obtain
\begin{align*}
x_i^{-1}\pi_v\mathbf{x}^\mu
&= x_i^{-1}\pi_{j}\pi_u\mathbf{x}^\mu
= \begin{cases}
%x_i^{-1}(\hat{\pi}_{j}+1)\pi_u\mathbf{x}^\mu =
%x_i^{-1}\hat{\pi}_{j}\pi_u\mathbf{x}^\mu + x_i^{-1}\pi_u\mathbf{x}^\mu
\pi_{j}x_{i-1}^{-1}\pi_u\mathbf{x}^\mu + x_i^{-1}\pi_u\mathbf{x}^\mu & \text{if $j = i-1$}, \\
\hat{\pi}_{j}x_{i+1}^{-1}\pi_u\mathbf{x}^\mu  & \text{if $j = i$}, \\
\pi_{j}x_i^{-1}\pi_u\mathbf{x}^\mu & \text{otherwise}.
\end{cases}
\end{align*}
%We prove only the second case as the others are similar and easier.
By the inductive hypothesis,
$x_{i-1}^{-1}\pi_u\mathbf{x}^\mu$, $ x_i^{-1}\pi_u\mathbf{x}^\mu$, and
$x_{i+1}^{-1}\pi_u\mathbf{x}^\mu$
belong to $\ZZ\big\{\kappa_\beta \mid \beta^+ = \mu - \epsilon_j \text{ for some $j \in [\ell]$} \big\}$.
Hence the result follows from \eqref{e pi i on key}.
\end{proof}

\begin{lemma}
\label{l poly xell}
For any $f \in \ZZx$,
 $x_\ell \poly(x_\ell^{-1} \hat{\pi}_{\ell-1}(f)) = \poly( \hat{\pi}_{\ell-1} (f))$.
\end{lemma}
\begin{proof}
%Since $\poly$ is a linear operator,
It is enough to prove this identity for $f$ ranging over a  $\ZZ$-basis of $\ZZx$.
We choose the basis $\{\mathbf{x}^\gamma \mid \gamma \in \ZZ_{\ge 0}^{\ell} \} \sqcup \{\kappa_{\alpha} \mid \alpha \in \ZZ^\ell \setminus \ZZ_{\ge 0}^\ell\}$.
First consider  $f = \mathbf{x}^\gamma$ with  $\gamma \in \ZZpl$;
set $c = \gamma_{\ell-1}, d = \gamma_\ell$.
Then
\begin{align*}
\hat{\pi}_{\ell-1} \mathbf{x}^{\gamma} &=
\begin{cases}
x_1^{\gamma_1}\cdots x_{\ell-2}^{\gamma_{\ell-2}}(x_{\ell-1}^{c-1} x_{\ell}^{d+1} + x_{\ell-1}^{c-2} x_{\ell}^{d+2} + \dots + x_{\ell-1}^{d}x_{\ell}^{c} ) & \text{if $c > d$},\\
0 & \text{if $c=d$}, \\
-x_1^{\gamma_1}\cdots x_{\ell-2}^{\gamma_{\ell-2}}(x_{\ell-1}^{c} x_{\ell}^{d} + x_{\ell-1}^{c+1} x_{\ell}^{d-1} + \dots + x_{\ell-1}^{d-1}x_{\ell}^{c+1} ) & \text{if $c < d$}.
\end{cases}
\end{align*}
Since $x_\ell$ appears with a positive power in each summand, we have
 $x_\ell \poly(x_\ell^{-1} \hat{\pi}_{\ell-1}\mathbf{x}^\gamma )=
  \hat{\pi}_{\ell-1} \mathbf{x}^\gamma
  = \poly( \hat{\pi}_{\ell-1} \mathbf{x}^\gamma)$
by Proposition \ref{p poly part} (ii).

Now consider $\alpha \in \ZZ^\ell \setminus \ZZ_{\ge 0}^\ell$.
%We have $\hat{\pi}_{\ell-1}\kappa_\alpha = $
Since $\hat{\pi}_{\ell-1} \kappa_\alpha$  is a sum of key polynomials indexed by rearrangements of $\alpha$
(by \eqref{e pi i on key}), $\poly( \hat{\pi}_{\ell-1} \kappa_\alpha) = 0$.
By Lemma \ref{l x i on key},
$x_\ell^{-1} \hat{\pi}_{\ell-1}\kappa_\alpha$ lies in
$\ZZ\big\{\kappa_\beta \mid \beta^+ = \alpha^+ - \epsilon_j \text{ for some $j \in [\ell]$} \big\}$,
so $\poly(x_\ell^{-1} \hat{\pi}_{\ell-1}\kappa_\alpha) = 0$ as well.
\end{proof}

Let $\wnotb{i}{j} \in \zH_\ell$ be the longest element of the  submonoid generated by $\zs_i, \dots, \zs_{j-1}$, i.e.,
%???not best notation but okay lengthen vector symbol
the element of  $\zH_\ell$ corresponding to the permutation which reverses the interval $[i,j]$;
we will also use the shorthand $\wnota{a}{\ell} := \wnotb{a}{\ell}$.
%with one-line notation  $1 \,2 \, \cdots i-1\, j\, j-1 \, \cdots \, i \, j+1 \, \cdots \, \ell$.

\begin{corollary}
\label{c hat pi xa mult}
For any $g \in \ZZx$ and  $a \in [\ell-1]$,
\begin{align}
\label{ec hat pi xa mult}
\poly(\pi_{\ell-2}\pi_{\ell-3}\cdots \pi_a \pi_{\wnota{a+1}{\ell} } (x_a g))
+ x_\ell \poly(\pi_{\wnota{a}{\ell}} (g))
&= \poly(\pi_{\wnota{a}{\ell}}(x_a g)).
\end{align}
\end{corollary}
\begin{proof}
%Subtracting  $\poly(\pi_{\wnota{a}{\ell}s_a } (x_a g))$ from both sides of \eqref{ec hat pi xa mult} and using $\wnota{a}{\ell} s_a = s_{\ell -1}\wnota{a}{\ell}$ and $\pi_{\wnota{a}{\ell}} = \pi_{s_{\ell-1}}\pi_{s_{\ell-1}\wnota{a}{\ell}}$, we see that
%Noting that $\pi_{\wnota{a}{\ell}} = \pi_{\ell-1}\cdots \pi_{a}\pi_{\wnota{a+1}{\ell}}$, we see that  \eqref{ec hat pi xa mult} is equivalent to
Rewriting the right side of \eqref{ec hat pi xa mult} using $\pi_{\wnota{a}{\ell}} = \pi_{\ell-1}\cdots \pi_{a}\pi_{\wnota{a+1}{\ell}}$,
 %shows \eqref{el hat pi mult xi} is equivalent to
we obtain the equivalent statement
\begin{align*}
x_\ell \poly( \pi_{\wnota{a}{\ell}} (g))
&= \poly( \hat{\pi}_{\ell-1} \pi_{\ell-2} \cdots \pi_a \pi_{\wnota{a+1}{\ell}} (x_a g)).
\end{align*}
To prove this, we compute
\begin{align*}
x_\ell \poly(\pi_{\wnota{a}{\ell}} (g))
&= x_\ell \poly\big( x_\ell^{-1} x_\ell \pi_{\ell-1} \pi_{\ell-2} \cdots \pi_a (\pi_{\wnota{a+1}{\ell}} g) \big) &&\text{}\\
&= x_\ell \poly\big( x_\ell^{-1} \hat{\pi}_{\ell-1} \pi_{\ell-2}\cdots \pi_a (x_a \pi_{\wnota{a+1}{\ell}} g) \big) &&\text{by Lemma \ref{l hat pi mult xi} }\\
&= \poly\big( \hat{\pi}_{\ell-1} \pi_{\ell-2}\cdots \pi_a (x_a \pi_{\wnota{a+1}{\ell}} g) \big) &&\text{by Lemma \ref{l poly xell}}\\
&= \poly\big( \hat{\pi}_{\ell-1} \pi_{\ell-2}\cdots \pi_a \pi_{\wnota{a+1}{\ell}} (x_a g) \big) && \text{by \eqref{e pi i commutes with symmetric}}. \qedhere
\end{align*}
%fact that $\pi_i$ commutes with multiplication by $x_j$ for any $j \notin \{i,i+1\}$ for the last equality.
\end{proof}

\subsection{Proof of Theorem \ref{t character rotation formula}}

%The next theorem shows how to express a tame nonsymmetric Catalan function
% %$H(\Psi ; \gamma ; \wnota{a+1}{\ell})$
% in terms of a smaller one
%%$ H( \sR(\Psi); \sR(\gamma) ; \wnota{a}{\ell})$
%by peeling off its first row,
%which we can then iterate to unravel
%any tame nonsymmetric Catalan function
%%$H(\Psi ; \gamma ; \wnota{a+1}{\ell})$
%one row at a time and obtain the desired expression involving  $\pi_i$'s and
% $\Phi$'s.

The next theorem shows how to express a tame nonsymmetric Catalan function
$H(\Psi ; \gamma ; \wnota{a+1}{\ell})$
 %$H(\Psi ; \gamma ; \wnota{a+1}{\ell})$
in terms of a smaller one
$ H( \sR(\Psi); \sR(\gamma) ; \wnota{a}{\ell})$
by peeling off its first row,
which we can then iterate to unravel
%$H(\Psi ; \gamma ; \wnota{a+1}{\ell})$
%any tame nonsymmetric Catalan function
it one row at a time and obtain
 the desired expression involving  $\pi_i$'s and
 $\Phi$'s.

%The next theorem shows how to peel off the first row of
%$H(\Psi; \gamma; w)$. This allows us to unravel it one row at a time and obtain the desired expression involving  $\pi_i$'s and
% $\Phi$'s.

%The crux of the proof is the following theorem which allows us to unravel
%$H(\Psi; \gamma; w)$ one row at a time to an expression involving  $\pi_i$'s and
% $\Phi$'s.

%The proof goes by unraveling using the following the
%We show that any tame nonsymetric Catalan function $H(\Psi; \gamma; w)$ with  $\gamma_1 \ge 0$ is equal to
%$x_1^{\gamma_1} \Phi$ applied to a smaller nonsymmetric Catalan function obtained by
%removing the first part of  $\gamma$ and the first row of $\Psi$ (Theorem \ref{t ns Catalan rotate}).

\begin{theorem}
\label{t ns Catalan rotate}
Let $\gamma \in \ZZ^\ell$ and $\Psi$ be a root ideal of length $\ell$.
Let $\sR(\gamma) = (\gamma_2, \dots, \gamma_\ell,0)$ and $\sR(\Psi) \subset \Delta^+_{\ell}$
be $\sR(\Psi) := \{(i-1,j-1) \mid (i,j) \in \Psi, i>1\} \sqcup \{(i,\ell) \mid i \in [\ell-1]\}$;
this is the root ideal obtained from  $\Psi$ by removing its first row, shifting what remains up 1 and left 1, and adding a full column of roots on the right.
Set $a = \nr(\Psi)_1$.
If  $\gamma_1 \ge 0$, then
\begin{align}
\label{et ns Catalan rotate}
H(\Psi ; \gamma ; \wnota{a+1}{\ell}) = x_1^{\gamma_1}\Phi \big( H( \sR(\Psi); \sR(\gamma) ; \wnota{a}{\ell})\big).
\end{align}
\end{theorem}

\begin{remark}
\label{r trailing 0}
The last column of roots in $\sR(\Psi)$
%is only there as a convenient way to
%does not affect the right side and
%is irrelevant
%can be regarded as
is just a place holder
%so that right side is a nonsymmetric Catalan function of length  $\ell$
to make the right side % $H( \sR(\Psi); \sR(\gamma) ; \wnota{a}{\ell})$
a length  $\ell$ nonsymmetric Catalan function:
%(we want to view both sides as elements of $\ZZ[q][x_1,\dots, x_\ell]$).
since  $\sR(\gamma)_\ell =0$, by Proposition \ref{p poly part} (v),
$H( \sR(\Psi); \sR(\gamma) ; \wnota{a}{\ell}) = H( \Psi'; \sR(\gamma) ; \wnota{a}{\ell})$
for any $\Psi' \subset \Delta^+_\ell$ with  $\Psi' \cap \Delta^+_{\ell-1} = \sR(\Psi) \cap \Delta^+_{\ell-1}$.
%stay in the world of  $\gl_\ell$-characters on both sides).
\end{remark}

\begin{example}
Let us verify Theorem \ref{t ns Catalan rotate} for $\ell = 2$, $\gamma = (3,2)$, $\Psi = \Delta^+$.
Then $a = 1$, $\wnota{a+1}{\ell} = \idelm$, $\wnota{a}{\ell} = \zs_1$, $\sR(\gamma)= (2,0)$, $\sR(\Delta^+) = \Delta^+$.
We compute both sides of \eqref{et ns Catalan rotate}:
\begin{multline*}
H(\Delta^+ ; \gamma ; \wnota{a+1}{\ell})
= \poly \!\big( x_1^3x_2^2(1-q {x_1}/{x_2})^{-1}  \big)\\
= \poly \!\big( x_1^3x_2^2 +  qx_1^4x_2 + q^2x_1^5 + q^3x_1^6x_2^{-1} + \cdots \big)
= x_1^3x_2^2 + qx_1^4x_2 + q^2x_1^5.
\end{multline*}
\vspace{-6.6mm}
\begin{multline*}
x_1^{\gamma_1}\Phi\big(H(\sR(\Delta^+) ; \sR(\gamma) ; \wnota{a}{\ell})\big)
=x_1^{\gamma_1}\Phi \pi_{1} \poly \!\big(x_1^2 (1-q{x_1}/{x_2})^{-1} \big)
=x_1^{\gamma_1}\Phi \pi_{1}(x_1^2)\\
=x_1^{\gamma_1}\Phi \big(x_1^2 + x_1x_2 + x_2^2 \big)
= x_1^3(x_2^2 + qx_1x_2 + q^2x_1^2 )
= x_1^3x_2^2 + qx_1^4x_2 + q^2x_1^5.
\end{multline*}
\end{example}

\begin{example}
Let $\ell =3, \gamma = 211$, and  $\Psi = \Delta^+$.  Then  $a=1$ and Theorem \ref{t ns Catalan rotate} yields
$H(\Delta^+; 211; \zs_2) = x_1^2 \Phi (H(\Delta^+; 110; \mathsf{w}_0))$.
%be seen Through Corollary \ref{c kat conjecture resolution 0},
This can be viewed (via Corollary \ref{c kat conjecture resolution 0}) as the identity of characters
corresponding to going from the fourth to fifth crystal in Figure \ref{ex DARK crystals}.
%Going from the fourth to fifth crystal in Figure \ref{ex DARK crystals} corresponds to, on the level of characters,
%the following application of Theorem \ref{t ns Catalan rotate},
%For  $\ell =3, \gamma = 211$, and  $\Psi = \Delta^+$,
%$H(\Delta^+; 211; \zs_2) = x_1^2 \Phi (H(\Delta^+; 110; \mathsf{w}_0))$.
\end{example}

%Recall $\hat{\mu} = (\mu_2,\mu_3,\dots, \mu_\ell, 0)$.
\begin{proof}[Proof of Theorem \ref{t ns Catalan rotate}]
%The proof is by induction of We now prove \eqref{et ns Catalan rotate} for all $\gamma$ in this set and  root ideals  $\Psi$ by induction on
%$|\Psi|$ and dominance order? on the weight $\gamma$.
%can assume  $\gamma$ belongs to the finite set of weights
%$\{\alpha \in \ZZ^\ell \mid |\alpha| = |\gamma|, \sum_{i=j}^\ell \alpha_i \ge 0 \text{ for all } j \in [\ell] \}$.
%%??change to dominance order statement  $d0^{\ell-1} gd' \alpha$ okay as is
The proof is by induction on  $\sum_{j=2}^\ell \sum_{i = j}^\ell \gamma_i$ and $|\Psi|$.
The former quantity is not bounded below, so to make this induction valid we first handle the
following ``base case'':
%To make this a valid inductive argument
suppose $\sum_{i = j}^\ell \gamma_i < 0$ for some  $2 \le j \le \ell$. Thus
$\sum_{i = j-1}^\ell \sR(\gamma)_i < 0$,
and so by Proposition \ref{p poly part} (iv),
$H(\Psi ; \gamma ; \wnota{a+1}{\ell}) = 0 = x_1^{\gamma_1}\Phi \big(H( \sR(\Psi); \sR(\gamma) ; \wnota{a}{\ell})\big)$.
%Note that since we are inducting on $\sum_{j=1}^\ell \sum_{i = j}^\ell \gamma_i$, we may assume
%\eqref{et ns Catalan rotate} holds for any weight $\gamma + \varepsilon{\alpha}$ .

Next, the base case $|\Psi| = 0$ holds by Lemma \ref{l poly x1} (it is here we need $\gamma_1 \ge 0$):
\begin{align*}
H(\varnothing ; \gamma; \idelm) = \poly(\mathbf{x}^\gamma)= \poly(x_1^{\gamma_1}\Phi(\mathbf{x}^{\sR(\gamma)})) = x_1^{\gamma_1} \Phi(\poly(\mathbf{x}^{\sR(\gamma)}))=
x_1^{\gamma_1}\Phi\big( H(\sR(\varnothing); \sR(\gamma); \idelm)  \big);
\end{align*}
we have also used Remark \ref{r trailing 0} for the last equality.
%??useful there is a sort of hidden base case when \sum_{j=2}^\ell \sum_{i = j}^\ell \gamma_i = 0, but we can just allow this to go to -1, where we know it's true
%also note that the last equality holds by Remark \ref{r trailing 0}.
%we have also used Corollary \ref{c trailing 0} for the last equality.

We may assume from now on that $|\Psi| > 0$ and $\sum_{i = j}^\ell \gamma_i \ge 0$ for $j \ge 2$.
If there is a removable root $\alpha$ of $\Psi$ not in the first row, then
\begin{align*}
H(\Psi; \gamma; \wnota{a+1}{\ell})
&= H(\Psi \setminus \alpha; \gamma; \wnota{a+1}{\ell}) + q H(\Psi; \gamma + \varepsilon_{\alpha}; \wnota{a+1}{\ell})\\
&= x_1^{\gamma_1} \Phi\big(H(\sR(\Psi \setminus \alpha) ; \sR(\gamma) ; \wnota{a}{\ell})\big)
+ q x_1^{\gamma_1}\Phi\big(H(\sR(\Psi); \sR(\gamma + \varepsilon_{\alpha}); \wnota{a}{\ell})\big)\\
&=x_1^{\gamma_1}\Phi\big(H(\sR(\Psi); \sR(\gamma); \wnota{a}{\ell}\big),
\end{align*}
where the first and third equalities are by Proposition \ref{p inductive computation atom} and the second is by the inductive hypothesis.

Now we may assume $\Psi$ consists of a single nonempty row.
%Since we are assuming  $|\Psi| >0$, If $\Psi \neq \varnothing$,
%We now compute, beginning by expanding
Hence we can expand on the only removable root $(1,a+1)$ (Proposition \ref{p inductive computation atom}) to obtain the first equality below:
\begin{align}
H(\Psi; \gamma; \wnota{a+1}{\ell}) \notag
&= H(\Psi \setminus (1,a+1) ; \gamma ; \wnota{a+1}{\ell}) + qH(\Psi; \gamma+\epsilon_1-\epsilon_{a+1}; \wnota{a+1}{\ell}) \\
\notag
&= \pi_{\ell-1}\pi_{\ell-2}\cdots \pi_{a+1}H(\Psi \setminus (1,a+1) ; \gamma ; \wnota{a+2}{\ell}) + qH(\Psi; \gamma+\epsilon_1-\epsilon_{a+1}; \wnota{a+1}{\ell}) \\
\notag
&= \pi_{\ell-1}\pi_{\ell-2}\cdots \pi_{a+1}x_1^{\gamma_1}\Phi\big(H(\varnothing; \sR(\gamma); \wnota{a+1}{\ell})\big)
+ qx_1^{\gamma_1+1} \Phi\big(H(\varnothing; \sR(\gamma) - \epsilon_a; \wnota{a}{\ell})\big)\\
\notag
&= \pi_{\ell-1}\pi_{\ell-2}\cdots \pi_{a+1}x_1^{\gamma_1}\Phi \pi_{\wnota{a+1}{\ell}} \poly(\mathbf{x}^{\sR(\gamma)})
+ qx_1^{\gamma_1+1} \Phi \pi_{\wnota{a}{\ell}} \poly( \mathbf{x}^{\sR(\gamma) - \epsilon_a} )\\
%\notag
%&= x_1^{\gamma_1}\Phi \pi_{\ell-2} \cdots \pi_{a}\pi_{\wnota{a+1}{\ell}} \poly(\mathbf{x}^{\sR(\gamma)})
%+ qx_1^{\gamma_1+1} \Phi \pi_{\wnota{a}{\ell}} \poly( \mathbf{x}^{\sR(\gamma) - \epsilon_a} )\\
\notag
&= x_1^{\gamma_1}\Phi \Big( \pi_{\ell-2} \pi_{\ell-3}\cdots \pi_{a}\pi_{\wnota{a+1}{\ell}} \poly(\mathbf{x}^{\sR(\gamma)})
+  x_\ell \pi_{\wnota{a}{\ell}} \poly( \mathbf{x}^{\sR(\gamma) - \epsilon_a} ) \Big)\\
%\notag
%&= x_1^{\gamma_1}\Phi\big(H(\varnothing; \sR(\gamma); s_{\ell-2} \cdots s_a \wnota{a+1}{\ell} )\big)
%+ qx_1^{\gamma_1+1} \Phi\big(H(\varnothing; \sR(\gamma) - \epsilon_a; \wnota{a}{\ell})\big)\\
%%\label{e almost done}
%\notag
%&= x_1^{\gamma_1}\Phi\Big( H(\varnothing; \sR(\gamma); s_{\ell-2} \cdots s_a \wnota{a+1}{\ell} )
%+ x_\ell H(\varnothing; \sR(\gamma) - \epsilon_a; \wnota{a}{\ell})  \Big)\\
\notag
&= x_1^{\gamma_1} \Phi \pi_{\wnota{a}{\ell}} \poly(\mathbf{x}^{\sR(\gamma)}) \\
\notag
&= x_1^{\gamma_1} \Phi\big(H(\sR(\varnothing); \sR(\gamma); \wnota{a}{\ell})\big).
\end{align}
The second equality is by $\pi_{\wnota{a+1}{\ell}} = \pi_{\ell-1}\pi_{\ell-2}\dots \pi_{a+1} \pi_{\wnota{a+2}{\ell}}$ and Definition \ref{d HH gamma Psi},
%the definition of the nonsymmetric Catalan functions,
the third
is by the inductive hypothesis and Remark \ref{r trailing 0} (note that we have the first part of $\gamma + \epsilon_1 - \epsilon_{a+1}$ is still $\ge 0$),
the fifth is by Proposition \ref{p Phi sort of commutes}, \eqref{e pi i commutes with symmetric}, and
%$x_1^{\gamma_1}\pi_{\wnota{a+1}{\ell}} = \pi_{\ell-1}\pi_{\ell-2}\dots \pi_{a+1} x_1^{\gamma_1} \pi_{\wnota{a+2}{\ell}}$
$\Phi(x_\ell) = qx_1$,
and the sixth is by Corollary \ref{c hat pi xa mult} with  $g = \mathbf{x}^{\sR(\gamma) - \epsilon_a}$ and Proposition \ref{p poly part} (i).
%To complete the proof we must show that \eqref{e almost done} is equal to the desired  $x_1^{\gamma_1} \Phi\big(H(\varnothing; \sR(\gamma); \wnota{a}{\ell})\big)$.
%It suffices to show
\end{proof}

\begin{proof}[Proof of Theorem \ref{t character rotation formula}]
Our goal is to prove \eqref{ec character rotation formula}, reproduced here for convenience:
%Recall that $\sfc(d) := s_{\ell-1}s_{\ell-2} \cdots s_d \in \SS_\ell$ for $d \in [\ell]$.
\begin{align*}
H(\Psi; \gamma; w) = \pi_w x_1^{\gamma_1}\Phi\pi_{\sfc(n_1)} x_1^{\gamma_2}\Phi \pi_{\sfc(n_2)} x_1^{\gamma_3} \cdots \Phi \pi_{\sfc(n_{\ell-1})}  x_1^{\gamma_\ell}.
\end{align*}
We proceed by induction on
%$m = \min \{i \in [\ell] \mid \gamma_i = \cdots = \gamma_\ell = 0 \}$.
$m$, the minimum index such that  $\gamma_m = \gamma_{m+1} = \cdots = \gamma_\ell = 0$  (set $m=\ell+1$ if  $\gamma_\ell \ne 0$).
The base case $m = 1$, $\gamma = \mathbf{0}$ holds since  $H(\Psi; \gamma; w) = 1$ by
Proposition \ref{p poly part} (v).
%, while the right side of  \eqref{ec character rotation formula} is clearly $1$ as well.
Now assume  $m > 1$.  By the tameness assumption,  $w$ has a length additive factorization
$w = v \, \wnota{n_1+1}{\ell}$.
%$v \in \zH_\ell$ the minimal coset representative.  ??not so great to say for 0-Hecke
Thus Theorem~\ref{t ns Catalan rotate} gives
\begin{align*}
H(\Psi ; \gamma ; w) = \pi_{v} \spa H(\Psi ; \gamma ; \wnota{n_1+1}{\ell}) =
 \pi_{v} \spa x_1^{\gamma_1} \Phi \big( H(\sR(\Psi); \sR(\gamma) ; \wnota{n_1}{\ell} ) \big).
\end{align*}
Applying the inductive hypothesis to  $H(\sR(\Psi); \sR(\gamma) ; \wnota{n_1}{\ell} )$, we obtain
\begin{align*}
\pi_{v} \spa x_1^{\gamma_1} \Phi \big( H(\sR(\Psi); \sR(\gamma) ; \wnota{n_1}{\ell} ) \big)
%H(\Psi ; \gamma ; w)
=\  & \pi_{v} \spa x_1^{\gamma_1} \Phi \pi_{\wnota{n_1}{\ell}} \spa x_1^{\gamma_2}\Phi\pi_{\sfc(n_2)} x_1^{\gamma_3}\cdots \Phi \pi_{\sfc(n_{\ell-1})}  x_1^{\gamma_{\ell}} \Phi \pi_{\sfc(1)}  x_1^{0}
\\[.4mm]
=\ &
\pi_{v} \spa \pi_{\wnota{n_1+1}{\ell}} \spa x_1^{\gamma_1} \Phi \pi_{\sfc(n_1)} \spa x_1^{\gamma_2}\Phi\pi_{\sfc(n_2)} x_1^{\gamma_3} \cdots \Phi \pi_{\sfc(n_{\ell-1})}  x_1^{\gamma_{\ell}},
\end{align*}
giving the desired \eqref{ec character rotation formula};
for the second equality, we have used the operator identity
\begin{align*}
x_1^{\gamma_1}\Phi \pi_{\wnota{n_1}{\ell}}
=
x_1^{\gamma_1}\Phi \pi_{\wnotb{n_1}{\ell-1}} \pi_{\sfc(n_1)}
= \pi_{\wnota{n_1+1}{\ell}}  \spa x_1^{\gamma_1}\Phi\pi_{\sfc(n_1)},
\end{align*}
where the last equality is by Proposition \ref{p Phi sort of commutes} and \eqref{e pi i commutes with symmetric}.
\end{proof}

\section{DARK crystals and katabolism}

\label{s DARK and katabolism}
We show that for any DARK crystal  $\B^{\mu;\mathbf{w}}$, katabolism is exactly the condition on  $\Tabloids_\ell(\mu)$ which detects membership in
 $\inv(\B^{\mu;\mathbf{w}})$.
A connection between KR crystals
and Catalan functions
in the dominant rectangle case
has been well established (see Remark \ref{r two formulas for rectangle}).
One of our key insights is that to go beyond this case,
%handle the general flag variety case
DARK crystals are needed rather than full tensor products of KR crystals.

%lines up exactly with is none other than the condition which By translating the above description of generalized Demazure modules to type A and taking the inverse word, we obtain the following streamlined version of the katabolism algorithm.
%When the affine Demazure modules above are also $U_q(\sl_n)$-modules, we can obtain characters by looking at $\QQ$-tableaux alone.
%We first need a model for these partial affine Demazure modules.

\subsection{Single row Kirillov-Reshetikhin crystals}
\label{ss single row KR}

%It is enough to work in the product of Kirillov-Reshetikhin rows even though these aren't the exact crystals we want, they agree with the desired ones for  $\sl_n$ action.

We will only need an explicit description of the KR crystals $B^{1,s}$ in type A.
For any positive integer $s$, the $U'_q(\hatsl_\ell)$-seminormal crystal  $B^{1,s}$ consists of all weakly increasing
words of length $s$ in the alphabet  $[\ell]$, with weight function  $\wt \colon B^{1,s} \to P_{\cl}$ given by
%\begin{align}
%\varepsilon(b) &= \sum_i \varepsilon_i(b) \Lambda_i \qquad \phi(b) = \sum \phi_i(b) \Lambda_i\\
%\wt(b) &= \phi(b) - \varepsilon(b)
%\end{align}
%The level $\ell$ KR-crystal is constructed as follows:
\begin{align}
%\mu &= a_0\Lambda_0 + a_1\Lambda_1 + a_2\Lambda_2 \\
%b^\mu &= (a_0, a_1, a_2)\\
%c_\mu &= (a_1,a_2,a_0)\\
%\varepsilon(b^\mu) &= \mu\\
%\varepsilon(c_\mu) &= a_1\Lambda_0 + a_2\Lambda_1 + a_0\Lambda_2\\
%\wt(b) &= (c_\ell-c_1)\Lambda_0 + (c_1-c_2)\Lambda_1 + \cdots + (c_{\ell-1}-c_\ell) \Lambda_{\ell-1}
\label{e KR weight}
\wt(b) &= \cl \big( c_1(\Lambda_1-\Lambda_0) + c_2(\Lambda_2-\Lambda_1) + \cdots + c_\ell(\Lambda_0 - \Lambda_{\ell-1}) \big),  \ \text{ for $b = 1^{c_1}2^{c_2}\cdots \ell^{c_\ell}$}
%\text{for } b &= \underbrace{11\cdots 1}_{c_1}\underbrace{22 \cdots 2}_{c_2} \underbrace{33\cdots 3}_{c_3} \cdots \underbrace{\ell\ell \cdots \ell}_{c_\ell},
\end{align}
(i.e.,  $b$ is the weakly increasing word with content $(c_1, \dots, c_\ell)$),
%??where def content? nowhwere for words but okay since just clarifies
and crystal operators defined as follows:
for $i \in [\ell-1]$ and $b \in B^{1, s}$,
$\ce_i(b)$ is obtained from $b$ by changing its leftmost $i+1$ to an $i$,
and $\cf_i(b)$ by changing its rightmost $i$ to an $i+1$;
if there are no $i+1$'s, $\ce_i(b) = 0$, and if there are no $i$'s, $\cf_i(b)= 0$.
The element $\ce_0(b)$ is obtained from $b$ by removing a letter 1 from the beginning and adding a letter $\ell$ to the end,
and $\cf_0(b)$ is obtained by removing a letter $\ell$ from the end and adding a letter 1 to the beginning;
if there are no 1's, $\ce_0(b) = 0$, and  if there are no $\ell$'s,  $\cf_0(b) = 0$.
%So $\phi_0(b)$ is the number of letters $n$ in $b$ and $\varepsilon_0(b$) is the number of letters 1 in b. In the crystal graph $B^{1,s}$ all
See Figure \ref{f KR crystals}.

%\begin{remark}
%\label{r zeros KRs}
We also define $B^{1,0} = \{\sfb_0 \}$ to be the trivial
$U'_q(\hatsl_\ell)$-seminormal crystal, i.e., $\wt(\sfb_0)=0$ and $\ce_i(\sfb_0)= \cf_i(\sfb_0) = 0$ for all $i\in I$, and view
$\sfb_0$ as the empty word.
%Note that  $B(0\Lambda_i)= B(0)= \{u_0\}$ is the trivial $U_q(\g)$-crystal.
%\end{remark}

\subsection{Products of KR crystals}
\label{ss products of KR rows}

%We now describe the crystal $\B^\mu$ in detail, in \S\ref{ss intro DARK}.  We now describe this object in more detail.
We now describe in detail the crystals $\B^\mu$ which were briefly introduced in \S\ref{ss intro DARK}.
%review and elaborate on the definition of , and then describe this crystal

\begin{definition}
A \emph{biword} is a pair of words
$b = \begin{pmatrix}v_1 \!\! & v_2 & \cdots & v_m \\[-.5mm] w_1 \!\! & w_2 & \cdots  & w_m\end{pmatrix}$
with $v_i, w_i \in \ZZ_{\ge 1}$,
such that  for  $i< j$, $v_i > v_j$ or ($v_i = v_j$ and  $w_i \le w_j$).
Define $\top(b) := v_1 \cdots v_m$, the \emph{top word of  $b$}, and
$\bottom(b) := w_1 \cdots w_m$, the \emph{bottom word of  $b$}.
The \emph{$i$-th block} of  $b$, denoted $b^i$, is the (contiguous) subword of  $w_1 \cdots w_m$ below the letters $i$ in $v_1 \cdots v_m$.
%so that $b = \begin{pmatrix}\ell \cdots \ell (\ell-1)\cdots (\ell-1)  \cdots 1 \cdots 1 \\[-.5mm] b^\ell b^{\ell-1} \cdots b^1 \end{pmatrix}$;
Thus a pair of words  $b$ is a biword if and only if
$\top(b)$ is weakly decreasing  and its blocks are
weakly increasing.
The \emph{content} of  $b$, denoted $\content(b)$,
is the vector $(c_1, c_2, \dots, c_\ell)$,
where $c_i$ is the number of occurrences of the letter $i$ in $\bottom(b)$.
%??content_i not currently used
%The \emph{content} of  $b$, denoted $\content(b)$, is the sequence whose  $i$-th part  $\content_i(b)$ is the number of
%letters $i$ in $\bottom(b)$.
\end{definition}

Recall that for a partition  $\mu = (\mu_1 \ge \cdots \ge \mu_p \ge 0)$,
we let $\B^\mu = B^{1,\mu_p} \otimes \cdots \otimes B^{1,\mu_1}$, a $U'_q(\hatsl_\ell)$-seminormal crystal.
We identify its elements with the biwords
whose bottom word has letters in  $[\ell]$ and
whose top word is $p^{\mu_p}\cdots 2^{\mu_2} 1^{\mu_1}$ (see Example \ref{ex inv});
we use a biword  $b$ interchangeably with its bottom word when the
crystal  $\B^\mu$ it belongs to is clear.

%It is easily seen that can also be
\begin{remark}
We can also regard $\B^\mu$  as a $U_q(\gl_\ell)$-crystal (temporarily denote it $\B^\mu_\gl$)
with weight function $\B^\mu_\gl \to \ZZ^\ell$,  $b \mapsto \content(b)$
and the same edges %$\ce_i, \cf_i$-edges ($i \in [\ell-1]$)
as $\Res_{\sl_\ell} \B^\mu$  (the restriction from  $U'_q(\hatsl_\ell)$ to  $U_q(\sl_\ell)$);
moreover, $\Res_{\sl_\ell} \B^\mu_\gl = \Res_{\sl_\ell} \B^\mu$ by \eqref{e KR weight}.
From now on we write  $\B^\mu$ for both the $U'_q(\hatsl_\ell)$-seminormal and $U_q(\gl_\ell)$-crystal, and will clarify when necessary.
\end{remark}

The crystal operators $\ce_i$ and  $\cf_i$ on $\B^\mu$ are determined by the above description of  $\ce_i$ and  $\cf_i$ on $B^{1,s}$ and
the tensor product rule \eqref{e crystal tensor 1}--\eqref{e crystal tensor 2}.
For $i \in [\ell-1]$, they have the following streamlined description.
Let $b \in \B^\mu$.
%The word  $\ce_i(b)$ is obtained by changing a single $i+1$ in  $b$ to an  $i$.
Place a left parenthesis ``('' below each letter $i+1$ in $b$ and a right parenthesis ``)'' below each letter $i$.
Match parentheses in the usual way.
The unmatched parentheses correspond to a subword consisting of $i$'s followed by $i+1$'s.
Then $\ce_i(b)$ is obtained from  $b$ by changing the leftmost unmatched $i+1$ to an $i$, and
$\cf_i(b)$ by changing the rightmost unmatched $i$ to an $i+1$;
if there are no unmatched $i+1$'s, $\ce_i(b) = 0$, and if there are no unmatched $i$'s, $\cf_i(b) = 0$.

\begin{example}
\label{ex paren match}
%Let $\ell =4$.
%Let $b = 2234 \, 13334 \, 11122 \in \B^{554}$.
We illustrate the parentheses matching rule for computing  $\ce_2$ and  $\cf_2$ of the element
 $b \in \B^{554}$ below,  with the unmatched letters in bold.
\[\begin{array}{r@{\hspace{1.3mm}}r@{\hspace{1.3mm}}c@{\hspace{.3mm}}c@{\hspace{.3mm}}c@{\hspace{.3mm}}c@{\hspace{1.4mm}}c@{\hspace{.3mm}}c@{\hspace{.3mm}}c@{\hspace{.3mm}}c@{\hspace{.3mm}}c@{\hspace{1.4mm}}c@{\hspace{.3mm}}c@{\hspace{.3mm}}c@{\hspace{.3mm}}c@{\hspace{.3mm}}c@{\hspace{.3mm}}c@{\hspace{.3mm}}c@{\hspace{.3mm}}c@{\hspace{.3mm}}c@{\hspace{.3mm}}c@{\hspace{.3mm}}c@{\hspace{.3mm}}c@{\hspace{.3mm}}c@{\hspace{.3mm}}c@{\hspace{.3mm}}c@{\hspace{.3mm}}cccccccccccccccccccc}
b &= &\mathbf{2}&\mathbf{2}&3&4 &2&\mathbf{2}&\mathbf{3}&3&3
&1&1&1&2&2\\
&  & \bm{)}&\bm{)}&(&
&)&\bm{)}&\bm{(}&(& (
& \, & \, & \, &)&)\\
\ce_2(b) &=& \mathbf{2}&\mathbf{2}&3&4 &2&\mathbf{2}&\mathbf{2}&3&3
&1&1&1&2&2\\
\cf_2(b) &=& \mathbf{2}&\mathbf{2}&3&4
&2&\mathbf{3}&\mathbf{3}&3&3
&1&1&1&2&2\\
\end{array}\]
%The 0-string of $b$ is
%\begin{align*}
%\ce_0(b) &= 12244 \, 134 \, 1111\\
%b &= 12244 \, 134 \, 1112\\
%\cf_0(b) &= 12244 \, 134 \, 1122\\
%\cf_0^2(b) &= 12244 \, 134 \, 1222\\
%\cf_0^3(b) &= 22244 \, 134 \, 1222
%\end{align*}
\end{example}

\begin{figure}[t]
\centerfloat
\begin{tikzpicture}[xscale = 1.49,yscale = 1.66]
\tikzstyle{vertex}=[inner sep=0pt, outer sep=2pt, fill = white]
\tikzstyle{edge} = [draw, ->,black]
\tikzstyle{LabelStyleH} = [text=black, anchor=south, yshift = -0.2ex]
\tikzstyle{LabelStyleV} = [text=black, anchor=east, xshift = 0.44ex]
\tikzstyle{labelStyleV} = [text=black, anchor=east, xshift = 0.44ex, yshift = -0.44ex]
\tikzstyle{labelStyleH} = [text=black, anchor=south, xshift = .14ex, yshift = -0.4ex]

\begin{scope}
\node[vertex] (t00) at (-6.36,1.1) {\scriptsize $33$};
\node[vertex] (t10) at (-5,2) { \scriptsize $13$};
\node[vertex] (t11) at (-6.36,2) { \scriptsize $23$};
\node[vertex] (t20) at (-3.64,2.9) {\scriptsize $11$};
\node[vertex] (t21) at (-5,2.9) {\scriptsize $12$};
\node[vertex] (t22) at (-6.36,2.9) {\scriptsize $22$};

\draw[edge] (t00) to node[labelStyleH]{{\Tiny$\cf_0$ \ \ \ }} (t10);
\draw[edge] (t11) to node[LabelStyleV]{{\Tiny$\cf_2$}} (t00);
\draw[edge] (t10) to node[LabelStyleH]{{\Tiny$\cf_1$}} (t11);
\draw[edge] (t10) to node[labelStyleH]{{\Tiny$\cf_0$ \ \ \ }} (t20);
\draw[edge] (t21) to node[LabelStyleV]{{\Tiny$\cf_2$}} (t10);
\draw[edge] (t11) to node[labelStyleH]{{\Tiny$\cf_0$ \ \ \ }} (t21);
\draw[edge] (t22) to node[LabelStyleV]{{\Tiny$\cf_2$}} (t11);
\draw[edge] (t20) to node[LabelStyleH]{{\Tiny$\cf_1$}} (t21);
\draw[edge] (t21) to node[LabelStyleH]{{\Tiny$\cf_1$}} (t22);

\node[vertex] (s00) at (-1.4,1) {\scriptsize $333$};
\node[vertex] (s10) at (-0.4,1.67) {\scriptsize $133$};
\node[vertex] (s11) at (-1.4,1.67) {\scriptsize $233$};
\node[vertex] (s20) at (0.6,2.33) {\scriptsize $113$};
\node[vertex] (s21) at (-0.4, 2.33) {\scriptsize $123$};
\node[vertex] (s22) at (-1.4,2.33) {\scriptsize $223$};
\node[vertex] (s30) at (1.6,3) {\scriptsize $111$};
\node[vertex] (s31) at (0.6,3) {\scriptsize $112$};
\node[vertex] (s32) at (-0.4,3) {\scriptsize $122$};
\node[vertex] (s33) at (-1.4,3) {\scriptsize $222$};

\draw[edge] (s11) to node[LabelStyleV]{{\Tiny$\cf_2$}} (s00);
\draw[edge] (s10) to node[LabelStyleH]{{\Tiny$\cf_1$}} (s11);
\draw[edge] (s21) to node[LabelStyleV]{{\Tiny$\cf_2$}} (s10);
\draw[edge] (s22) to node[LabelStyleV]{{\Tiny$\cf_2$}} (s11);
\draw[edge] (s20) to node[LabelStyleH]{{\Tiny$\cf_1$}} (s21);
\draw[edge] (s21) to node[LabelStyleH]{{\Tiny$\cf_1$}} (s22);
\draw[edge] (s31) to node[LabelStyleV]{{\Tiny$\cf_2$}} (s20);
\draw[edge] (s32) to node[LabelStyleV]{{\Tiny$\cf_2$}} (s21);
\draw[edge] (s33) to node[LabelStyleV]{{\Tiny$\cf_2$}} (s22);
\draw[edge] (s30) to node[LabelStyleH]{{\Tiny$\cf_1$}} (s31);
\draw[edge] (s31) to node[LabelStyleH]{{\Tiny$\cf_1$}} (s32);
\draw[edge] (s32) to node[LabelStyleH]{{\Tiny$\cf_1$}} (s33);
\end{scope}

\begin{scope}[xscale=.92][yscale=1][xshift=0]
\node[vertex] (p00) at (2.6,1) {\scriptsize $313$};
\node[vertex] (p01) at (1.1,1) {\scriptsize $323$};
\node[vertex] (p10) at (3.8,1.95) {\scriptsize $311$};
\node[vertex] (p11) at (2.16,1.95) {\scriptsize $312$};
\node[vertex] (p12) at (1.1,1.95) {\scriptsize $322$};
\node[vertex] (p13) at (2.6,2.2) {\scriptsize $213$};
\node[vertex] (p20) at (3.8,2.9) {\scriptsize $211$};
\node[vertex] (p21) at (2.6,2.9) {\scriptsize $212$};

\draw[edge] (p00) to node[LabelStyleH]{{\Tiny$\cf_1$}} (p01);
\draw[edge] (p11) to node[LabelStyleH]{{\Tiny$\cf_1$}} (p12);
\draw[edge] (p10) to node[LabelStyleH]{\quad{\Tiny$\cf_1$}} (p11);
\draw[edge] (p20) to node[LabelStyleH]{{\Tiny$\cf_1$}} (p21);
\draw[edge] (p20) to node[LabelStyleV]{{\Tiny$\cf_2$}} (p10);
\draw[edge] (p21) to node[LabelStyleV]{{\Tiny$\cf_2$}} (p13);
\draw[edge] (p13) to node[labelStyleV]{{\Tiny$\cf_2$}} (p00);
\draw[edge] (p12) to node[LabelStyleV]{{\Tiny$\cf_2$}} (p01);
\end{scope}
\end{tikzpicture}
\captionsetup{width=.95\linewidth}
%\captionsetup{labelsep=space}
\caption{\label{f KR crystals} \!\!\!\! \small \text{For $\ell =3$, the KR crystal $B^{1,2}$ (left) and $\Res_{\sl_\ell}\spa \B^{(2,1)} = \Res_{\sl_\ell}\spa B^{1,1} \! \tsr \! B^{1,2}$ (right).}}
\end{figure}

%\subsection{Insertion and its compatibility with crystal operators}
\subsection{RSK and crystals}

We review the beautiful connection between  $U_q(\gl_\ell)$-crystals and classical tableau combinatorics,
which may be attributed to Kashiwara-Nakashima \cite{KasNak}, and Lascoux-Sch\"utzenberger \cite{LS} who anticipated much of the
combinatorics before the development of crystals.
Other good references include
\cite{Shimozono2005CrystalsFD} and \cite[Chapter 7]{HK}.

The crystals  $\B^\mu$ are compatible with the following variant of the Robinson-Schensted-Knuth correspondence described in
\cite[A.4.1, Proposition 2]{Fultontableaux}.
Let  $b$ be a biword.
The \emph{insertion tableau $P(b)$} of  $b$
%We define the insertion and recording tableau of a biword $b$ as follows:
is the ordinary insertion tableau of the word $\bottom(b)$.  It can be obtained by applying
the Schensted row insertion algorithm to the letters of  $\bottom(b)$ from left to right
or by column inserting each letter from right to left.
%factors  $\bottom(b) = $
The \emph{recording tableau $Q(b)$} of  $b$ is obtained by column inserting the bottom word
%$b^\ell \cdots b^1$
of $b$ from right to left and recording each newly added box with the corresponding top letter.
%as each letter in the block $b^i$ is inserted, record this in a separate tableau of shape  $\sh(P(b))$ by placing a letter  $i$ in a new box the corresponding letter above it.
More precisely, $Q(b)$ is the tableau with the same shape as $P(b)$ such that the skew shape \spa $\sh(P(b^i b^{i-1} \cdots b^1))/\sh(P(b^{i-1} \cdots b^1))$
is filled with  $i$'s for all $i$.

Recall from \S\ref{ss intro Katabolism and the Shimozono-Weyman conjecture} that $\SSYT_\ell(\mu)$ denotes the subset of
$\Tabloids_\ell(\mu)$ consisting of tabloids with partition shape whose columns strictly increase from top to bottom.
(This is the set of semistandard Young tableaux of content $\mu$ with at most $\ell$ rows, but with the fine print
that we  regard them as having  $\ell$ rows some of which may be empty.)

\begin{theorem}[{see \cite[Theorem 3.6]{Shimozono2005CrystalsFD}}]
\label{t Blambda comps etc}
The decomposition of the $U_q(\gl_\ell)$-crystal $\B^{\mu}$
%a disjoint union of $U_q(\gl_\ell)$-highest weight crystals, with decomposition given by
into highest weight $U_q(\gl_\ell)$-crystals is given by
\begin{align}
\quad \quad \B^{\mu} = \!\! \bigsqcup_{U \in \SSYT_\ell(\mu)} \! \cC_U,
\quad \text{ where
$\cC_U := \{b \in \B^{\mu} \mid Q(b) = U\} \cong B^{\gl}(\sh(U))$}.
\end{align}
Here, $B^{\gl}(\nu)$ denotes the highest weight $U_q(\gl_{\ell})$-crystal of highest weight $\nu$.
\end{theorem}

\subsection{The $\inv$ bijection and RSK}
\label{ss rsk and inv}
A biword can be thought of as a sequence of biletters
 $\genfrac(){0pt}{2}{v_1}{w_1}  \genfrac(){0pt}{2}{v_2}{w_2} \cdots \genfrac(){0pt}{2}{v_m}{w_m}$
 which is weakly decreasing for the order
$\genfrac(){0pt}{2}{v}{w} \ge \genfrac(){0pt}{2}{v'}{w'}$ if and only if  $v > v'$ or ($v=v'$ and  $w \le w'$).
%The \emph{inverse map} $\inv$ on biwords can now be defined as follows:
Then, for a biword  $b$,  define $\inv(b)$ to be the result
of exchanging the top and bottom words of $b$ and then sorting biletters to be weakly decreasing.

It is natural to regard $\inv$ as an involution on the set of biwords.
However, as discussed in Remark \ref{r no two sided} below,
we prefer to think of $\inv$ as a bijection between biwords and tabloids,
which we can do since biwords and tabloids
%Tabloids and biwords
may be naturally identified by equating blocks with rows (see the right side of \eqref{eex inv}).
%equivalently, the rows  $T^\ell,\dots, T^1$ of  $T$ are the blocks of  $b^\ell, \dots, b^1$ of $b$.
%Potentially we could identify either or both sides of the  $\inv$ bijection with tabloids, but we adopt the following convention:
%elements of  $\B^\mu$ will be written as biwords and never tabloids;
%their inverses will typically be written as tabloids, though occasionally thought of as biwords for the purposes of computing  $\inv$.
%
Since the contents of the top and bottom words are exchanged by  $\inv$,
%$\content(\top(\inv(b))) = \content(\bottom(b))$ and $\content(\bottom(\inv(b))) = \content(\top(b))$,
it restricts to a bijection  $\inv \colon \B^\mu \xleftrightarrow{\cong} \Tabloids_\ell(\mu)$, which takes content to shape
(we gave a direct description of the map $\B^\mu \to \Tabloids_\ell(\mu)$ in \S\ref{ss intro DARK}).
%??where def content?, must be defined for biword = content of its bottom word.  Yes, it is.

%Though as we have just said biwords and tabloids can be used interchangeably, we adopt the following convention:
%crystal basis elements will be written as biwords and never tabloids;
%their inverses will typically be written as tabloids and occasionally thought of as biwords as well.

\begin{proposition}[{\cite[A.4.1, Symmetry Theorem B]{Fultontableaux}}]
\label{p P Q symmetry}
The insertion ($P$) and recording ($Q$) tableaux are exchanged by  $\inv$.
In particular, for a biword $b \in \B^\mu$,
$Q(b) = P(\inv(b))$ and for a tabloid  $T \in \Tabloids_\ell(\mu)$,
$P(T) = Q(\inv(T)).$
%$P(w,w') = Q(\inv(w,w'))$ and $Q(w,w') = P(\inv(w,w'))$.
\end{proposition}

%Biwords and tabloids are in bijection via the correspondence taking a biword
%$b = \begin{pmatrix}\ell^{\alpha_\ell} \cdots 1^{\alpha_1} \\[-.5mm] w \end{pmatrix}$ to the tabloid  $T$ of shape $\alpha$ with row reading word  $w$;
%%the row reading word of the tabloid is equal to the bottom word.
%%We therefore often identify a biword with its corresponding tabloid;
%%%In this way tabloids and biwords are the same.
%%we use both notations to connect to and generalize standard operations from tableau theory.
%equivalently, the rows  $T^\ell,\dots, T^1$ of  $T$ are the blocks of  $b^\ell, \dots, b^1$ of $b$.

%For example, biword
%$b = \begin{pmatrix}33333 \, 222 \, 11111 \\[-.5mm] 23344 \, 234 \, 11112\end{pmatrix}$
%corresponds to the tabloid
%$T = {\fontsize{7pt}{5pt}\selectfont\tableau{1 &1 &1 &1 & 2 \\ 2 & 3 & 4  \\ 2 & 3 &3 & 4 & 4 }}$.
%The blocks of $b$ and rows of $T$ are
%$b^3 = T^3 = 23344, \ \ b^2 = T^2 = 234, \ \ b^1 = T^1 = 11112.$

\begin{example}
\label{ex inv}
For the following biword $b \in \B^{554}$, we compute
 $\inv(b)$ and  $Q(b)$\spa:
\begin{align}
\label{eex inv}
b = \begin{pmatrix}3333 \, 22222 \, 11111 \\ 2234 \, 13334 \, 11222 \end{pmatrix}
%b = \begin{pmatrix}3333 \, 22222 \, 11111 \\ 2234 \, 22333 \, 11122 \end{pmatrix}
\xrightarrow{\inv} \begin{pmatrix}44 \, 3333 \, 22222 \, 111 \\ 23 \, 2223 \, 11133 \, 112\end{pmatrix}
= {\fontsize{6pt}{4.5pt}\selectfont\tableau{1&1&2\\1&1&1&3&3\\2&2&2&3\\2&3}} = T
\end{align}
\[Q(b) =P(T)= {\fontsize{7pt}{5pt}\selectfont\tableau{1&1&1&1&1&2&3\\2&2&2&2&3&3\\3}}.\]
\end{example}

\begin{remark}
\label{r no two sided}
Though it is possible to define a two-sided crystal structure on biwords in which crystal operators act on both a biword and its inverse,
this is not the perspective we take here.  Instead, we break the symmetry between the two sides by adopting the following conventions:
%a biword and its inverse as well as the  $P$ and  $Q$ tableaux:
crystal operators act only on the  $\B^\mu$ side and not the  $\Tabloids_\ell(\mu)$ side;
we are mainly interested in  $Q(b)$, not  $P(b)$, for $b\in \B^\mu$, and  $P(T)$, not  $Q(T)$,  for  $T \in \Tabloids_\ell(\mu)$ as these are the ones which identify  $\inv$ of the highest weight element of a  $U_q(\gl_\ell)$-component.
%Relatedly, to distinguish the two sides, we adopt the following convention:
Further, \emph{elements of  $\B^\mu$ will be written as biwords and never tabloids;
their inverses will be written as tabloids, though occasionally thought of as biwords for the purposes of computing  $\inv$}.
%
%Though it is possible to define a two-sided crystal structure on biwords in which crystal operators act on both a biword and its inverse,
%this is not the perspective we take here.  Instead, we break the symmetry between a biword and its inverse by adopting the following
%convention:
%\emph{elements of  $\B^\mu$ will be written as biwords and never tabloids;
%their inverses will be written as tabloids, though occasionally thought of as biwords for the purposes of computing  $\inv$}.
%
%This also breaks the symmetric between the $P$ and  $Q$ tableaux:
%crystal operators only act on the  $\B^\mu$ side and not the  $\Tabloids_\ell(\mu)$ side,
%and we are mainly interested in  $Q(b)$, not  $P(b)$, for $b\in \B^\mu$, and  $P(T)$, not  $Q(T)$,  for  $T \in \Tabloids_\ell(\mu)$.
%%Accordingly, for $b \in \B^\mu$, we are mainly interested in  $Q(b) = P(\inv(b))$ and not
%%$P(b) = Q(\inv(b))$.
%%% as the latter forgets about  $\mu$, which we dont want.
%%
%%There is only one crystal structure on  $\B^\mu$ coming from the action on the bottom word and we will not make use of the fact
%%that this crystal is isomorphic to  $\B^\alpha$ for rearrangements  $\alpha$ of  $\mu$.
\end{remark}

%\begin{lemma}
%Let $v^i$ be a weakly increasing word for $1\le i \le n$ and $(P,Q)$ the pair of tableaux with $P = P(v^n \dots v^1$ and $Q = Q(v^n,v^{n-1},\dots, v^1)$.
%Then the overlap of the pair of words $(v^{r+1},v^r)$ is equal to the number of paired parentheses in the computation of the $r$-string of $Q$.
%
%Let $\ell^{\alpha_\ell}  \dots 1^{\alpha_1}$, $v^\ell \dots v^1$ be a biword and $P, Q$ its insertion and recording tableau.
%Then the overlap of the pair of words $(v^{r+1},v^r)$ is equal to the number of paired parentheses in the computation of the $r$-string of $Q$.
%\end{lemma}

\subsection{Partial insertion and $\ce_i^{\spa \max}$}
\label{ss partial insertion}

In the remainder of Section \ref{s DARK and katabolism}, we match operations on the tabloids side with ones on the crystal side.
The material in this subsection is similar to
\cite[\S3.5]{SW}, \cite[\S2]{LascouxCrystal}	and perhaps can be considered folklore.

For an element $b$ of a  $U_q(\gl_\ell)$-crystal, define
\begin{align}
\ce_i^{\spa \max} (b) = \ce_i^{\spa \varepsilon(b)} (b),
\end{align}
 i.e., the last  element in the list $b, \ce_i(b), \ce_i^2(b), \dots$ which is not 0.
For example, in the crystal $\B^{432}$,  $\ce_1^{\spa \max}(12 \, 122 \, 1222) = 12 \, 112 \, 1111$.
%and for  $b$ as in Example \ref{ex paren match}, $\ce_2^{\spa \max} (b) = \ce_2(b)$.
More generally, for  $w \in \zH_\ell$, let $w=\zs_{i_1}\cdots \zs_{i_m}$ be any expression for  $w$ as a product of  $\zs_j$'s; define
% is any expression for  $w$ as a product of simple reflections,
$\ce_w^{\spa \max} = \ce_{i_1}^{\spa \max} \cdots \ce_{i_m}^{\spa \max}$; by Proposition \ref{p Pi braid} (ii) below, this is
independent of the chosen expression for  $w$.

Recall that $T^i$ denotes the  $i$-th row of a tabloid  $T$.

\begin{definition}[Partial insertion]
\label{d partial insertion}
Given a tabloid $T$, $P_{i}(T)$ is the tabloid obtained from $T$ by replacing rows $i$ and $i+1$ of $T$ by the tableau $P(T^{i+1}T^i)$
%(we regard $P(T^{i+1}T^i)$ as a two row tableau whose second row may be empty).
(if $P(T^{i+1}T^i)$ has only one row, then the $i+1$-st row of $P_i(T)$ is empty).
%(this may result in an empty $i+1$-st row in $P_i(T)$ even if $T^{i+1} \ne \varnothing$).
%(we allow the second row to be empty).
%For $i< j$, $P_{i,j}(T)$ is the tabloid obtained from $T$ by replacing rows $i$ through $j$ of $T$ by the $j-i+1$-row tableau $P(T_j\dots T_i)$.
More generally, for  $w=\zs_{i_1}\cdots \zs_{i_m} \in \zH_\ell$, define
% is any expression for  $w$ as a product of simple reflections,
$P_w = P_{i_1} \cdots P_{i_m}$; by Proposition \ref{p Pi braid} (iii) below, this is
independent of the chosen expression for  $w$.
For how this is related to Definition \ref{d intro partial insertion}, see Remark \ref{r kat intro vs not}.
\end{definition}

For example,  \
$P_2 \Bigg( \,\, {\fontsize{6pt}{4.5pt}\selectfont\tableau{1&1&2&2\\1&2&3&3&4\\1&1&2&2&2&3\\2&3&4}} \, \Bigg)  =
{\fontsize{6pt}{4.5pt}\selectfont\tableau{1&1&2&2\\1&1&1&2&2&2&3&3&4\\2&3\\2&3&4}}$\,.

The following commutative diagrams give a summary of
\S\ref{ss rsk and inv}--\ref{ss partial insertion}
(the left holds by Proposition \ref{p Ei max equal partial P} and the right by
Propositions \ref{p P Q symmetry}, \ref{p Ei max equal partial P}, and \ref{p Pi braid} (iv)).
%the left diagram depicts Proposition \ref{p Ei max equal partial P} and the right depicts Propositions \ref{p P Q symmetry}, \ref{p Ei max equal partial P}, and \ref{p parital insertion vs ordinary}.

\[\xymatrix@C=40pt{ \B^\mu\ar[d]_{\ce^{\spa \max}_i}  \ar[r]^(.35){\inv} & \Tabloids_{\ell}(\mu)  \ar[l] \ar[d]_{ P_i} \\
\B^\mu  \ar[r]_(.35){\inv} &  \ar[l] \Tabloids_{\ell}(\mu)}
\qquad \qquad
\xymatrix@C=48pt{ \B^\mu\ar[d]_{\ce^{\spa \max}_{\mathsf{w}_0}} \ar[r]^(.35){\inv} \ar[rd]_(.4){Q} & \Tabloids_{\ell}(\mu) \ar[l] \ar[d]^{ P \spa =\spa P_{\mathsf{w}_0}} \\
\B^\mu  \ar[r]_(.35){\inv} &  \ar[l] \Tabloids_{\ell}(\mu)}
\]

%The partial insertion operator $P_i$ is none other than the conjugate of $\ce_i^{\spa \max}$ by $\inv$\spa :

\begin{proposition}
\label{p Ei max equal partial P}
%Partial insertion agrees with $\ce_i^{\spa \max}$ under $\inv$:
%If $\genfrac(){0pt}{2}{w}{w'}$ is a biword with $\inv\genfrac(){0pt}{2}{w}{w'} = \genfrac(){0pt}{2}{v}{v'}$, then $\inv(P_i\genfrac(){0pt}{2}{w}{w'}) = \genfrac(){0pt}{2}{v}{\ce_i^{\spa \max}(v')}$.
For any biword  $b \in \B^\mu$ and  $i \in [\ell-1]$,  $\ce_i^{\spa \max}(b) = \inv (P_i( \inv(b)))$.
\end{proposition}
%The row insertion algorithm which computes $P(T)$, read the letters of  $T^1$ from left to right, a letter $x$ of  $T^1$ bumps the smallest valued entry of  $T^2$
%greater than $x$ that has not already been bumped (unless  $x$ is larger than the entries of  $T^2$ ) no letter there is no such letter it .
%Pairing up the letters from  $T^1$ and  $T^2$ which are involved in bumps, we see that  $P(T)^2$ contains exactly the unpaired entries from  $T^2$.

\begin{proof}
Set  $T = \inv(b)$.
%To ease notation, assume $i=1$ and that $T$ has two rows; our argument below works for the general case with no essential change.
Recall from \S\ref{ss products of KR rows} that $\ce_i^{\spa \max}(b)$ is obtained by viewing $i+1$'s and $i$'s as left and right parentheses and then changing all unmatched $i+1$'s to $i$'s.
We claim that $\inv (P_i (\inv(b)))$, computed using the row bumping algorithm, is obtained by the same rule except with the following \emph{greedy parentheses matching} in place of the ordinary one: read ``)''s from right to left and match each with the rightmost unmatched ``(''.
To see this, first note
that the letters in  $\top(b)$ above the $i$'s (resp. $i+1$'s) in $\bottom(b)$ are~the values of $T^i$ (resp.  $T^{i+1}$).
%we view the entries in  $T^{i+1}$ and $T^i$ as left and right parens, respectively.
%The tabloid $P_i(T)$ is computed by %working from left to right,
% inserting the letters of  $T^i$ from left to right;
The row bumping algorithm computes $P_i(T)$ by processing the letters of $T^i$ from left to right;
each letter  $x$ of  $T^i$ bumps the smallest entry of  $T^{i+1}$
greater than $x$ not already bumped (if it exists). %after converting values to positions,
Each bump corresponds to a greedy-matched pair in  $b$ and the unmatched $i+1$'s of  $b$ correspond to the entries of  $T^{i+1}$ not bumped,
which are exactly the ones that move from  $T^{i+1}$ to $(P_i(T))^i$ in computing  $P_i(T)$.
 %of a 1 and 2 and this amounts to pairing letters of  $T^1$
% (unless  $x$ is larger than the entries of  $T^2$ ) no letter there is no such letter it .
%using the row insertion algorithm to compute  $P_i$  and then exchanging values with positions to get
%one checks that this  an
 %Decoding what the bumping algorithm does after converting values to positions,

It remains to show that, given a string  $w_1\cdots w_m$ in the letters ``('' and ``)'',
the ordinary and greedy matching rules produce the same unmatched ``(''s. We proceed by induction on  $m$. Consider the subword  $w_i\cdots w_m$ where  $w_i$ is the rightmost matched ``(''; it must look like  $())\cdots )(\cdots($.
Let  $(w_i, w_{j})$ (resp.  $(w_i, w_{i+1})$) be the greedy (resp. ordinary) matched pair in this subword.
Though these pairs typically differ,
deleting the greedy-matched pair yields the same string as deleting the ordinary matched pair.  Since the position of the ``('' in
both pairs is the same, the result follows by the inductive hypothesis.
\end{proof}

\begin{proposition}
\label{p highest weight descriptions}
Let $b \in \B^\mu$ and set $T = \inv(b) \in \Tabloids_\ell(\mu)$.
Then  $b$ is a $U_q(\gl_\ell)$-highest weight element if and only if any of the following equivalent conditions holds:
%(the first is the definition):
\begin{list}{\emph{(\alph{ctr})}} {\usecounter{ctr} \setlength{\itemsep}{1pt} \setlength{\topsep}{2pt}}
\item  $\ce_i(b) = 0$ for all  $i \in [\ell-1],$
\item  $P_i(T) = T$ for  all  $i \in [\ell-1],$
\item  $T$ is a tableau, i.e.,  $T \in \SSYT_\ell(\mu)$.
%this means, by definition, that
%it has partition shape and columns strictly increase from top to bottom,
\end{list}
\end{proposition}
\begin{proof}
Condition (a) is the definition of  $b$ being a $U_q(\gl_\ell)$-highest weight element.
The equivalence (a) $\iff$ (b) is by Proposition \ref{p Ei max equal partial P}, and (b) $\iff$ (c) is clear from
computing $P(T^{i+1}T^i)$ by column insertion.
\end{proof}

\begin{proposition}
\label{p Pi braid}
Let $B^\gl(\nu)$,  $\nu = (\nu_1 \ge \cdots \ge \nu_\ell)$, be a highest weight  $U_q(\gl_\ell)$-crystal  and $u_\nu$ its highest weight element.  Then
\begin{list}{\emph{(\roman{ctr})}} {\usecounter{ctr} \setlength{\itemsep}{1pt} \setlength{\topsep}{2pt}}
\item
%BD(\nu_\ell, \dots, \nu_1) =
$\F_{\mathsf{w}_0} \{u_\nu\} = B^\gl(\nu)$.
\item The operators $\ce_1^{\spa \max}, \dots, \ce_{\ell-1}^{\spa \max}$ on  $B^\gl(\nu)$ satisfy the 0-Hecke relations of $\zH_\ell$
 $($\eqref{e extended affine Weyl far commutation}, \eqref{e extended affine Weyl braid}, and
\eqref{e 0Hecke pi i squared pi i}$)$.
%for any $b \in B$ and $i \in [\ell-2]$,
%$\ce_i^{\spa \max} \ce_{i+1}^{\spa \max} \ce_i^{\spa \max}(b)  = \ce_{i+1}^{\spa \max} \ce_{i}^{\spa \max} \ce_{i+1}^{\spa \max}(b)$.
\item The operators $P_1, \dots, P_{\ell-1}$ on $\Tabloids_\ell$ satisfy the 0-Hecke relations of  $\zH_\ell$.
\item  $\ce_{\mathsf{w}_0}^{\spa \max} (b) = u_\nu$ for any  $b \in B^\gl(\nu)$ and $P_{\mathsf{w}_0}(T) = P(T)$ for any $T \in \Tabloids_\ell$.
%??def \ce_w max. yes only used once? not often but okay
\end{list}
\end{proposition}
\begin{proof}
Statement (i) is well known; it can be deduced, for instance, from Remark \ref{r gl demazure crystals} using that $B^\gl(\nu)$ is finite.
Statement (iii) holds by (ii) and Proposition \ref{p Ei max equal partial P}.
For (ii), the
$\ce_i^{\spa \max}$ clearly satisfy the relations \eqref{e 0Hecke pi i squared pi i}; we now verify that they satisfy the braid relations \eqref{e extended affine Weyl braid} and omit the similar argument for
\eqref{e extended affine Weyl far commutation}. Let  $b \in B^\gl(\nu)$.
Let  $U_q(\g_J) \subset U_q(\gl_\ell)$ be the subalgebra isomorphic to $U_q(\gl_3)$ associated
to Dynkin node subset $J = \{i,i+1\} \subset [\ell-1]$.
The component  $B'$ of $\Res_{J}B^\gl(\nu)$
containing $b$ is isomorphic to a
highest weight $U_q(\gl_3)$-crystal
%??check why exactly Kashiwara result gives this. done. now discussion about restricting gives it exactly
(see \S\ref{ss restrict Demazure}); let $u \in B'$ be its highest weight element.
By (i) and Remark \ref{r gl demazure crystals}, $B' = \F_1 \F_{2} \F_1 \{u\} = \F_{2} \F_1 \F_{2} \{u\}$; hence
$\ce_i^{\spa \max} \ce_{i+1}^{\spa \max} \ce_i^{\spa \max}(b) = u = \ce_{i+1}^{\spa \max} \ce_{i}^{\spa \max} \ce_{i+1}^{\spa \max}(b)$.
%A similar argument gives the relations \eqref{e extended affine Weyl far commutation} and \eqref{e 0Hecke pi i squared pi i} is clear.

For (iv),  $\ce_{\mathsf{w}_0}^{\spa \max} (b) = u_\nu$ holds by (i).
Next, let  $U = P(T)$ and $\cC_U$ be the  $U_q(\gl_\ell)$-crystal component containing  $\inv(T)$.
By Proposition \ref{p highest weight descriptions}, the set $\inv(\cC_U) = \{S \in \Tabloids_\ell(\mu) \mid P(S) = U \}$ has
a unique element fixed by  $P_i$ for all  $i \in [\ell-1]$.
%using Proposition  \ref{p Ei max equal partial P} and Proposition \ref{p highest weight descriptions} here
Both  $U$ and  $P_{\mathsf{w}_0}(T)$ satisfy this property by Proposition \ref{p highest weight descriptions} and (iii), hence $P(T) = P_{\mathsf{w}_0}(T)$.
%
%alternative proof but too much kE stuff not discussed:
%There is a unique element of the Knuth equivalence class of the row reading word of  a tableau, namely  $P(T)$, hence
%by hw  $P(T)$ is the unique element in its Knuth equivalence class fixed by $P_i$
% for all  $i\in [\ell-1]$ satisfy are the unique elements in their is also
%$\ce_{\mathsf{w}_0}^{\spa \max} (b) = u_\nu$ and $P_{\mathsf{w}_0}(T)= \inv(\ce_{\mathsf{w}_0}(\inv(T))) = \inv(u) = P(T)$ for any $T \in \Tabloids_\ell$.
\end{proof}

\subsection{The  $\kat$ and  $\kat'$ operators and the automorphism $\tau$}
\label{ss kat}

%Recall from \S\ref{ss formal twist} the automorphism  $\tau$ and the
%formal twist $\tsigma(B)$ of a $U'_q(\hatsl_{\ell})$-crystal by $\sigma \in \Sigma = \{\tau^j \mid j \in [\ell]\}$.

Recall from \S\ref{ss formal twist}
that for $\sigma \in \Sigma = \{\tau^j \mid j \in [\ell]\}$
and  $U'_q(\hatsl_\ell)$-seminormal crystals $B,B'$, a  $\sigma$-twist is a bijection  $B \to B'$
taking  $i$-edges to  $\sigma(i)$-edges.
%?? would be nice to merge with twist discussion but nice to have this next to where used
For a word  $w$, let $\sort(w)$ denote its weakly increasing rearrangement.
For $m \in \ZZ$, let $\modd_{1}^\ell(m)$ be the unique  $i \in [\ell]$ such that $i \equiv m$ mod $\ell$.
%let $\modd_{0}^{\ell-1}(m)$ be the unique integer in $\{0, \dots, \ell-1\}$ congruent to $m$ mod $\ell$;

\begin{proposition}[{\cite[Proposition 5.5]{NaoiVariousLevels}}]
\label{p tau crystal}
There is a unique  $\sigma$-twist %of $U'_q(\hatsl_\ell)$-seminormal crystals
$\twist{\sigma} \colon \B^{\mu} \to \B^{\mu}$
for any $\sigma \in \Sigma$.
The  $\tau^{-1}$-twist $\dtaui$ has the following explicit description:
%in terms of words:
%for $\tau^j \in \Sigma$.
first, for $v = v_1\cdots v_s \in B^{1,s}$, $\dtaui(v_1 \cdots v_s) =
\sort\!\big( \modd_1^\ell(v_1 - 1) \cdots \modd_1^\ell(v_s-1) \big)$.
Then for a biword $b \in \B^\mu$ with blocks  $b^p, \dots, b^1$,
$\bottom(\dtaui(b)) = \dtaui(b^p) \cdots \dtaui(b^1)$ and  $\top(\dtaui(b)) =\top(b)$.
\end{proposition}

For example, with  $\ell= 4$ and $b = 233 \, 1124 \, 12223 \in \B^{543}$,
$\dtaui(b) = 122 \, 1344 \, 11124$.

% The inverse version of  $\kat$ we will show to be the following:
For  $b \in \B^\mu$, with blocks denoted $b^p, \dots,  b^1$ as usual,
define
\begin{align}\label{d katprime def}
\kat'(b) = \dtaui(b^p \cdots b^2) \in \B^{(\mu_2, \dots, \mu_p)}.
\end{align}
In other words,  $\kat'(b)$ is obtained from the biword $b$ as follows: remove the rightmost block of  $b$, subtract 1 from all bottom letters, turn
any  $0$'s into  $\ell$'s, sort each block, and finally subtract 1 from all top letters to obtain a biword in $\B^{(\mu_2, \dots, \mu_p)}$.
%by removing its Define $\kat'$ to denote removing from the top word the smallest letter and the corresponding bottom letters and then subtracting 1 (mod  $\ell$) to all letters and sorting blocks to be weakly increasing.
For example, with $\ell= 4$ and $b = \begin{pmatrix}444 \, 3333 \, 22222 \, 111111 \\[-.5mm] 233 \, 1124 \, 12223 \, 111111 \end{pmatrix}
\in \B^{6543}$,
$\kat'(b) = \begin{pmatrix}333\, 2222 \, 11111 \\[-.5mm] 122 \, 1344 \, 11124\end{pmatrix}\in \B^{543}$.

Recall from Definition \ref{d intro kat} that
for $T \in \Tabloids_\ell$,
$\kat(T)$ is defined as follows:
remove all 1's from  $T$ and left justify rows, then shift rows up by one cycling the first row to become the  $\ell$-th row, and finally subtract 1 from all letters.
The operators $\kat$ and  $\kat'$ are conjugate under $\inv$\spa:

\begin{proposition}
\label{p kat and kat'}
%??useful assumption next line no longer needed after slight change to kat def
%Let  $T \in \Tabloids_\ell$ with all its 1's lying in the first row. Then
%$\kat(T) = \inv(\kat'(\inv(T)))$.
For any $T \in \Tabloids_\ell$, $\inv(\kat(T)) = \kat'(\inv(T))$.
\end{proposition}
\begin{proof}
%Since  $\inv$ is bijective it is equivalent to prove $\inv(\kat(T)) = \kat'(\inv(T))$.
For  $j > 1$, let  $a_1 \cdots a_m$ be the row indices of the letters $j$ in  $T$, in weakly increasing order, which is also the block $\inv(T)^j$.
%the segment of the bottom word of $\inv(T)$ below the block of $j$'s on the top.
Since $\kat$ rotates rows, these $j$'s (which become $j-1$'s in $\kat(T)$) appear in rows $\modd_1^\ell(a_1-1) \cdots \modd_1^\ell(a_m-1)$ of $\kat(T)$.
Thus
\begin{align*}
(\inv(\kat(T)))^{j-1} = \sort\! \big( \modd_1^\ell(a_1-1) \cdots \modd_1^\ell(a_m-1)\big) = (\kat'(\inv(T)))^{j-1},
\end{align*}
%??note: \dtau does not subtracts 1 from top word but this is a bit confusing
where the second equality is by Proposition \ref{p tau crystal}
%($j-1$'s appear here and not  $j$'s since 1 is subtracted from all letters in the final step defining $\kat$, and
%1 is subtracted from all top letters in the final step computing $\kat'$).
($j-1$ appears on the right and not  $j$ because of the final step in the computation of $\kat'$).
The result follows.
\end{proof}

\subsection{Katabolism}

\begin{definition}
\label{d kat}
Let $\mathbf{w} = (w_1, \dots, w_p) \in (\zH_\ell)^p$.
We say $T \in \Tabloids_\ell$ is \emph{$\mathbf{w}$-katabolizable} if all the  $1$'s of $P_{w_1^{-1}}(T)$ lie in its first row and $\kat(P_{w_1^{-1}}(T))$ is
$(w_2, \dots, w_p)$-katabolizable.
For  $\mathbf{w}$ the empty sequence, the only
 $\mathbf{w}$-katabolizable tabloid is the empty one.
\end{definition}

The streamlined version of katabolism from Definition \ref{d intro kat} agrees with this one
%is a streamlined and easier to compute version of this one in the setting of
in the setting of Theorem \ref{t kat conjecture resolution}, as we now verify.

\begin{proposition}
\label{p easy kat vs gen kat}
Let  $\mu \in \ZZ_{\ge 0}^p$ and $\mathbf{n}  = (n_1,\dots, n_{p-1})\in [\ell]^{p-1}$ satisfy  $n_{i+1} \ge n_i - 1$ for all  $i \in [p-2]$.  A tableau  $U \in \SSYT_\ell(\mu)$ is
$\mathbf{n}$-katabolizable in the sense of Definition \ref{d intro kat} if and only if it is
 $(\idelm, \sfc(n_1), \dots, \sfc(n_{p-1}))$-katabolizable.
\end{proposition}
\begin{proof}
%Recall the notation $P_{i,\ell}$ from Definition \ref{d intro partial insertion}.
We first verify the following claim:  for any tabloid $T$ such that its subtabloid $T^{[i, \ell-1]}$ is a tableau,
$P_{i}\cdots P_{\ell-1}(T)$ can be obtained by column inserting $T^\ell$ into  $T^{[i, \ell-1]}$,
i.e., $P_{i}\cdots P_{\ell-1}(T) = P_{i,\ell}(T)$ in the notation of Definition \ref{d intro partial insertion}.
%To ease notation, we prove this in the case $i=1$, which easily implies the general case.
To ease notation, assume $i=1$, as this easily implies the general case.
We have $P_{1,\ell}(T) = P(T)$, the unique tableau with reading word Knuth equivalent to that
of  $T$. %insertion tableau column insertion results in a tableau,
Then by Proposition \ref{p Pi braid} (iv),
$P_{1,\ell}(T) = P(T)= P_{\mathsf{w}_0}(T) = P_{1} \cdots P_{\ell-1}P_{\wnotb{1}{\ell-1}}(T)
 = P_{1} \cdots P_{\ell-1}(T)$, where the last equality uses that $T^{[\ell-1]}$ is a tableau.
%??where \omega_{1,\ell-1} is as in.  okay to omit as notation is suggestive

Let us now see that the tabloids produced in computing the two versions of katabolism are the same:
%computing both see that it remains to see that
set  $\dot{U} = \kat(U)$.
Since $\dot{U}^{[\ell-1]}$ is a tableau,  $P_{n_1}\cdots P_{\ell-1}(\dot{U}) = P_{n_1,\ell}(\dot{U})$ by the claim.
Since $(P_{n_1,\ell}(\dot{U}))^{[n_1,\ell]}$ is a tableau, so is
%??corner case n1 = 1 should interpret [0,l-1] = [1,l-1] ...not worth discussing
$\ddot{U}^{[n_1-1, \ell-1]}$,  for $\ddot{U} := \kat(P_{n_1,\ell}(\dot{U}))$. As $n_2 \ge n_1 -1$,
$\ddot{U}^{[n_2, \ell-1]}$ is also a tableau.  Hence  $P_{n_2}\cdots P_{\ell-1}(\ddot{U}) = P_{n_2,\ell}(\ddot{U})$ again by the claim, and so on.
\end{proof}

\begin{remark}
\label{r kat intro vs not}
%Note that if  $\mathbf{n}$ does not satisfy the condition of Proposition \ref{p easy kat vs gen kat}, then
%the right notion of
%By the proof above,
With the assumption of Proposition \ref{p easy kat vs gen kat},
$P_{n_i, \ell}(T)  = P_{{\sfc(n_i)}^{-1}}(T) = P_{n_i} \cdots P_{\ell-1}(T)$ at every step of the katabolism algorithm,
so in this sense the partial insertion of Definition \ref{d partial insertion} generalizes that of Definition \ref{d intro partial insertion}.
We caution however, that without this assumption, %$n_{i+1} \ge n_i - 1$, the  main case of
%$(\ell, \mathbf{n})$-katabolizable in the sense of Definition \ref{d intro kat} if and only if it is
only Definition \ref{d kat} should be used and not Definition \ref{d intro kat}.
%$(\idelm, \sfc(n_1), \dots, \sfc(n_{p-1}))$-katabolizability is the correct notion and not that of
\end{remark}

\begin{example}
\label{ex gen def katab}
The following tabloid from Figure \ref{ex DARK crystals2} (\S\ref{ss intro examples})
is  $(\idelm, \zs_2\zs_1, \zs_2\zs_1)$-katabolizable:
% ($\ell=  3$):
\begin{align*}
&{\fontsize{5.8pt}{4.3pt}\selectfont
\text{\tableau{1&1&2\\ \bl \fr[l] \\3}}
\xrightarrow{\kat}
\text{\tableau{\bl \fr[l] \\ 2\\1}}
\xrightarrow{P_{2}}
\text{\tableau{\bl \fr[l] \\ 1& 2\\\bl \fr[l]}}
\xrightarrow{P_{1}}
\text{\tableau{1 &2 \\ \bl \fr[l] \\ \bl \fr[l] }}
\xrightarrow{\kat}
\text{\tableau{\bl \fr[l] \\ \bl \fr[l] \\ 1}}
\xrightarrow{P_{2}}
\text{\tableau{\bl \fr[l] \\  1 \\ \bl \fr[l]}}
\xrightarrow{P_{1}}
\text{\tableau{1\\ \bl \fr[l] \\ \bl \fr[l] }}
%\xrightarrow{P_{\zs_1 \zs_2}}
%\text{\tableau{1 &2 \\ \bl \fr[l] \\ \bl \fr[l] }}
%\xrightarrow{\kat}
%\text{\tableau{\bl \fr[l] \\ \bl \fr[l] \\ 1}}
%\xrightarrow{P_{\zs_1 \zs_2}}
%\text{\tableau{1\\ \bl \fr[l] \\ \bl \fr[l] }}
\xrightarrow{\kat} \varnothing
}
\end{align*}
\end{example}

\begin{example}
\label{ex big kat}
%We illustrate Theorem \ref{t kat conjecture resolution} and
%Definitions \ref{d intro kat}/\ref{d kat}.
Let  $\ell = 7$, $\mu = 4333332$, and $\Psi$ be the root ideal in red; $\Delta^+ \setminus \Psi$ is shown in blue.
We have  $\nr(\Psi)= (2,2,3,3,2,1)$.
We can visualize  $\ns(\Psi)=
(\sfc(\nr(\Psi)_1), \dots, \sfc(\nr(\Psi)_{\ell-1}))$
%(see \eqref{e ns def})
partially overlaid on the root ideal, so that row  $i$ read right to left is  $\ns(\Psi)_i$.
\ytableausetup{mathmode, boxsize=1.1em,centertableaux}
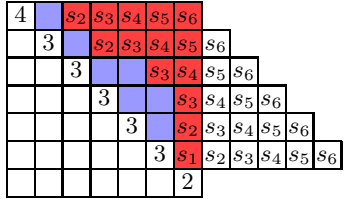
\[\scriptsize
\begin{ytableau}
   4    &*(blue!40)& *(red!75) s_2 &*(red!75) s_3 & *(red!75) s_4  &*(red!75)s_5 &*(red!75)s_6\\
        &    3   &*(blue!40)& *(red!75) s_2 &*(red!75) s_3 & *(red!75) s_4  &*(red!75)s_5 &*(red!0)s_6\\
        &        &    3    &*(blue!40)& *(blue!40) &*(red!75) s_3 & *(red!75) s_4  &*(red!0)s_5 &*(red!0)s_6\\
        &        &         &    3   &*(blue!40)&*(blue!40)& *(red!75) s_3 & *(red!0) s_4  &*(red!0)s_5 &*(red!0)s_6\\
        &        &         &        &    3     &*(blue!40) &*(red!75) s_2&*(red!0) s_3 & *(red!0) s_4  &*(red!0)s_5 &*(red!0)s_6\\
        &        &         &        &          &  3 &*(red!75) s_1& *(red!0) s_2 &*(red!0) s_3 & *(red!0) s_4  &*(red!0)s_5 &*(red!0)s_6\\
        &&&&&&2
\end{ytableau}\]
The following computation shows the tableau $U$ below to be $\nr(\Psi)$-katabolizable
or equivalently $(\idelm, \ns(\Psi))$-katabolizable (see Proposition \ref{p easy kat vs gen kat}).
So this gives one term  $q^{\charge(U)}s_{\sh(U)}$ $= q^{14} s_{876}$ of the Schur expansion of  $H(\Psi;\mu;\mathsf{w}_0 )$ from Theorem \ref{t kat conjecture resolution}.
%(and $(\idelm, \ns(\Psi))$-katabolizable---see Proposition \ref{p easy kat vs gen kat}):
%?both examples double checked
\begin{align*}
&{\fontsize{5.8pt}{4.3pt}\selectfont
U =
%\xrightarrow{P_{7, 7}} U =
\text{\tableau{1&1&1&1&4&4&4&6\\2&2&2&5&5&5&7\\3&3&3&6&6&7 \\ \bl \fr[l] \\ \bl \fr[l] \\ \bl \fr[l] \\ \bl \fr[l]}}
%\text{\tableau{1&1&1&1&4&4&4&6\\2&2&2&5&5&5&7\\3&3&3&6&6&7}}
\xrightarrow{\kat}
\text{\tableau{1&1&1&4&4&4&6\\2&2&2&5&5&6\\ \bl \fr[l] \\ \bl \fr[l] \\ \bl \fr[l] \\ \bl \fr[l] \\ 3&3&3&5}}
\xrightarrow{P_{2, 7}}
\text{\tableau{1&1&1&4&4&4&6\\2&2&2&5&5&5&6\\ 3&3&3 \\ \bl \fr[l] \\ \bl \fr[l] \\ \bl \fr[l] \\ \bl \fr[l]}}
\xrightarrow{\kat}
\text{\tableau{1&1&1&4&4&4&5\\ 2&2&2 \\ \bl\fr[l] \\ \bl\fr[l] \\ \bl\fr[l] \\ \bl\fr[l] \\ 3&3&3&5}}
\xrightarrow{P_{2, 7}}
\text{\tableau{1&1&1&4&4&4&5\\ 2&2&2&5\\ 3&3&3 \\ \bl \fr[l] \\ \bl \fr[l] \\ \bl \fr[l] \\ \bl \fr[l]}} }
\\[1.5mm]
&{\fontsize{5.8pt}{4.3pt}\selectfont
\xrightarrow{\kat}
\text{\tableau{1&1&1&4\\ 2&2&2 \\ \bl\fr[l] \\ \bl\fr[l]\\ \bl\fr[l]\\ \bl\fr[l] \\ 3&3&3&4 }}
\xrightarrow{P_{3, 7}}
\text{\tableau{ 1&1&1&4\\ 2&2&2 \\ 3&3&3&4 \\ \bl \fr[l] \\ \bl \fr[l] \\ \bl \fr[l] \\ \bl \fr[l]}}
\xrightarrow{\kat}
\text{\tableau{1&1&1 \\ 2&2&2&3 \\ \bl \fr[l] \\  \bl \fr[l] \\\bl \fr[l] \\\bl \fr[l] \\ 3 }}
\xrightarrow{P_{3, 7}}
\text{\tableau{ 1&1&1 \\ 2&2&2&3 \\ 3 \\ \bl \fr[l] \\ \bl \fr[l] \\ \bl \fr[l] \\ \bl \fr[l]}}
\xrightarrow{\kat}
\text{\tableau{ 1&1&1&2 \\ 2  \\ \bl \fr[l] \\\bl \fr[l] \\\bl \fr[l] \\\bl \fr[l] \\ \bl \fr[l]}}
\xrightarrow{P_{2, 7}}
\text{\tableau{ 1&1&1&2 \\ 2 \\ \bl \fr[l] \\ \bl \fr[l] \\ \bl \fr[l] \\ \bl \fr[l] \\ \bl \fr[l]}}
\xrightarrow{\kat}
\text{\tableau{1 \\ \bl \fr[l] \\ \bl \fr[l] \\ \bl \fr[l] \\ \bl \fr[l] \\
\bl \fr[l] \\  1}}
\xrightarrow{P_{1, 7}}
\text{\tableau{ 1 & 1\\ \bl \fr[l] \\ \bl \fr[l] \\ \bl \fr[l] \\ \bl \fr[l] \\ \bl \fr[l] \\ \bl \fr[l]}}
\xrightarrow{\kat} \varnothing
}
\end{align*}
%\begin{align*}
%&{\fontsize{5.8pt}{4.3pt}\selectfont
%\text{\tableau{1&1&1&1&4&4&4&6\\2&2&2&5&5&5&7\\3&3&3&6&6&7\\ \bl\\ \bl\\ \bl\\ \bl}}
%\xrightarrow{\kat}
%\text{\tableau{2&2&2&5&5&5&7\\3&3&3&6&6&7\\ \bl \\ \bl \\ \bl \\ \bl\\  4&4&4&6 }}
%\xrightarrow{P_{s_2s_3s_4s_5s_6}}
%\text{\tableau{2&2&2&5&5&5&7\\3&3&3&6&6&6&7\\ 4&4&4\\ \bl \\ \bl \\ \bl\\ \bl }}
%\xrightarrow{\kat}
%\text{\tableau{3&3&3&6&6&6&7\\ 4&4&4\\ \bl \\ \bl \\ \bl\\ \bl \\ 5&5&5&7}}}\\[2mm]
%&{\fontsize{5.8pt}{4.3pt}\selectfont
%\xrightarrow{P_{s_2s_3s_4s_5s_6}}
%\text{\tableau{3&3&3&6&6&6&7\\ 4&4&4&7\\ 5&5&5 \\ \bl \\ \bl\\ \bl \\ \bl}}
%\xrightarrow{\kat}
%\text{\tableau{ 4&4&4&7\\ 5&5&5 \\ \bl \\ \bl \\ \bl \\ \bl\\ 6&6&6&7}}
%\xrightarrow{P_{s_3s_4s_5s_6}}
%\text{\tableau{ 4&4&4&7\\ 5&5&5 \\ 6&6&6&7 \\ \bl \\ \bl \\ \bl\\ \bl }}
%\xrightarrow{\kat}
%\text{\tableau{ 5&5&5 \\ 6&6&6&7 \\ \bl \\ \bl \\ \bl\\ \bl \\ 7}}}\\[2mm]
%&{\fontsize{5.8pt}{4.3pt}\selectfont
%\xrightarrow{P_{s_3s_4s_5s_6}}
%\text{\tableau{ 5&5&5 \\ 6&6&6&7 \\ 7 \\ \bl \\ \bl \\ \bl\\ \bl }}
%\xrightarrow{\kat}
%\text{\tableau{ 6&6&6&7 \\ 7 \\ \bl \\ \bl \\ \bl\\ \bl\\ \bl }}
%\xrightarrow{P_{s_2s_3s_4s_5s_6}}
%\text{\tableau{ 6&6&6&7 \\ 7 \\ \bl \\ \bl \\ \bl \\ \bl\\ \bl}}
%\xrightarrow{\kat}
%\text{\tableau{ 7 \\ \bl \\ \bl \\ \bl \\ \bl\\ \bl \\ 7}}
%\xrightarrow{P_{s_1s_2s_3s_4s_5s_6}}
%\text{\tableau{ 7 & 7\\ \bl \\ \bl \\ \bl \\ \bl\\ \bl \\ \bl}}
%}\end{align*}
In contrast, the tableau $U'$ below is not $\nr(\Psi)$-katabolizable
since the katabolism algorithm produces a tabloid with
a  $1$ outside its first row just after an application of $P_{n(\Psi)_i,\ell}$.
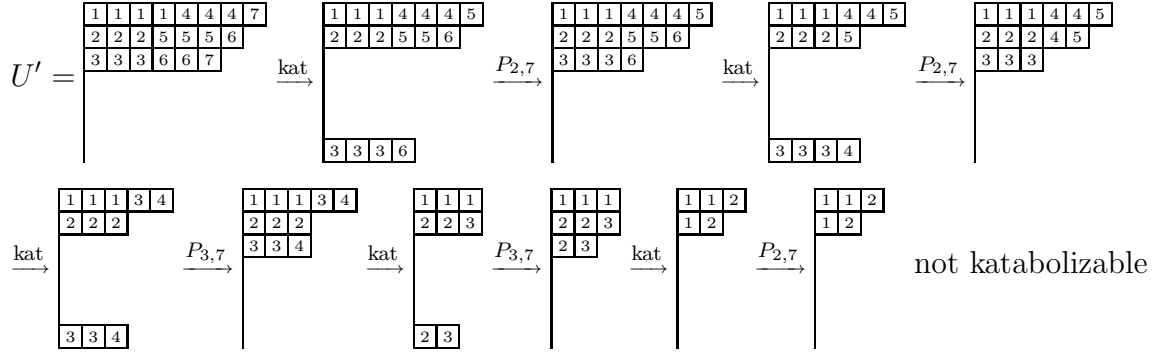
\begin{align*}
&{\fontsize{5.8pt}{4.3pt}\selectfont
U'=\text{\tableau{1&1&1&1&4&4&4&7\\2&2&2&5&5&5&6\\3&3&3&6&6&7 \\ \bl \fr[l] \\ \bl \fr[l] \\ \bl \fr[l] \\ \bl \fr[l]}}
\xrightarrow{\kat}
\text{\tableau{1&1&1&4&4&4&5\\2&2&2&5&5&6 \\ \bl \fr[l] \\ \bl \fr[l] \\ \bl \fr[l] \\ \bl \fr[l] \\ 3&3&3&6}}
\xrightarrow{P_{2, 7}}
\text{\tableau{1&1&1&4&4&4&5\\2&2&2&5&5&6\\ 3&3&3&6\\ \bl \fr[l] \\ \bl \fr[l] \\ \bl \fr[l] \\ \bl \fr[l] }}
\xrightarrow{\kat}
\text{\tableau{1&1&1&4&4&5\\ 2&2&2&5\\ \bl \fr[l] \\ \bl \fr[l] \\ \bl \fr[l] \\ \bl \fr[l] \\ 3&3&3&4}}
\xrightarrow{P_{2, 7}}
\text{\tableau{1&1&1&4&4&5\\ 2&2&2&4&5\\ 3&3&3 \\ \bl \fr[l] \\ \bl \fr[l] \\ \bl \fr[l] \\ \bl \fr[l] }}}\\[2mm]
&{\fontsize{5.8pt}{4.3pt}\selectfont
\xrightarrow{\kat}
\text{\tableau{1&1&1&3&4\\ 2&2&2 \\ \bl \fr[l] \\ \bl \fr[l] \\ \bl \fr[l] \\ \bl \fr[l] \\ 3&3&4 }}
\xrightarrow{P_{3, 7}}
\text{\tableau{ 1&1&1&3&4\\ 2&2&2 \\ 3&3&4 \\ \bl \fr[l]\\ \bl \fr[l]\\ \bl \fr[l]\\ \bl \fr[l]}}
\xrightarrow{\kat}
\text{\tableau{1&1&1 \\ 2&2&3 \\ \bl \fr[l]\\ \bl \fr[l]\\ \bl \fr[l]\\ \bl \fr[l] \\ 2&3}}
\xrightarrow{P_{3, 7}}
\text{\tableau{ 1&1&1 \\ 2&2&3 \\ 2& 3 \\ \bl \fr[l]\\ \bl \fr[l]\\ \bl \fr[l]\\ \bl \fr[l]}}
\xrightarrow{\kat}
\text{\tableau{ 1&1&2 \\ 1& 2 \\ \bl \fr[l]\\ \bl \fr[l]\\ \bl \fr[l]\\ \bl \fr[l]\\ \bl \fr[l]}}
\xrightarrow{P_{2, 7}}
\text{\tableau{ 1&1&2 \\ 1&2 \\ \bl \fr[l]\\ \bl \fr[l]\\ \bl \fr[l]\\ \bl \fr[l]\\ \bl \fr[l] }}} \quad \text{not katabolizable}
\end{align*}
%Here is what's happening on the corresponding biwords:
%\begin{align*}
%\mat{ 333333 \, 2222222 \, 11111111 \\ 333667\,2225557 \, 11114446 }
%\end{align*}
%Take $\inv$ to go to the version where the bottom words are the crystal words.
%\begin{align*}
%\mat{77 \, 666 \, 555 \, 444 \, 333 \, 222 \, 1111 \\ 23 \, 133 \, 222 \, 111 \, 333 \, 222 \, 1111}
%\end{align*}
\end{example}

%??move? this is sort of needed for the k-Schur variant.  okay to omit...decided to do it another way, but still may be good to include
\begin{remark}
\label{r no roots superstandard}
%Let  $p = \ell$,
Let $\mu  = (\mu_1 \ge \cdots \ge \mu_\ell \ge 0)$ and $\mathbf{w}=  (w_1,\dots, w_\ell) = (w_1, \ns(\Psi)) \in (\zH_\ell)^\ell$ for
a root ideal $\Psi \subset \Delta^+_\ell$ which is empty in rows  $\ge r$.
Then
%$T$ is in Definition \ref{d intro kat} with $\mathbf{n}= \nr(\Psi)$ and  $p = \ell$,
$T \in \Tabloids_\ell(\mu)$ is $\mathbf{w}$-katabolizable
 $\iff$   $T$ is  $(w_1,\dots, w_r, \idelm,\dots, \idelm)$-katabolizable  $\iff$ for $i \in [r-1]$, the tabloid  $U_i$ has all its 1's on the first row, and
$U_r$ is the superstandard tableau of shape and content $(\mu_r, \dots, \mu_\ell)$, where
$U_i := P_{w_i^{-1}} \circ \kat \circ \cdots \circ \kat \circ P_{w_2^{-1}} \circ \kat \circ P_{w_1^{-1}}(T)$.
%for  $i \in [\ell]$..
\end{remark}

\begin{theorem}
\label{t kat and inv crystal}
%Let $\mu$ be a partition and $\mathbf{w}= (w_1,\dots, w_p) \in (\zH_\ell)^p$.
%The bijection  $\inv : \Tabloids_\ell \to \B^{\mu}$ which takes shape to content restricts to a
% bijection
For $\mu = (\mu_1 \ge \cdots \ge \mu_p \ge 0)$ and $\mathbf{w}= (w_1,\dots, w_p) \in (\zH_\ell)^p$,
$\inv$ gives a bijection
\begin{align}
\big\{T \in \Tabloids_\ell(\mu) \mid T \text{ is $\mathbf{w}$-katabolizable}\big\}
\xrightarrow{\inv} \B^{\mu;\mathbf{w}}
\end{align}
which takes shape to content.
\end{theorem}
%\begin{proposition}
%\label{p kat and inv crystal}
%For any tabloid $T \in \Tabloids_\ell(\mu)$ and  $\mathbf{w}= (w_1,\dots, w_p) \in (\zH_\ell)^p$,
%$\inv(T) \in \B^{\mu; \mathbf{w}}$ if and only if $T$ is $\mathbf{w}$-katabolizable.
%\end{proposition}
\begin{proof}
We must show that for any $T \in \Tabloids_\ell(\mu)$,
$T$ is $\mathbf{w}$-katabolizable if and only if $\inv(T) \in \B^{\mu; \mathbf{w}}$.
We prove this by induction on  $p+\sum_i \length(w_i)$.
The base case~$p=1,$ $w_1 = \idelm$ is clear.
%Since the only katabolizable tabloid is the single row consisting of all 1's,
Now suppose $w_1 \ne \idelm$.
We can write $\B^{\mu; \mathbf{w}} = \F_{w_1}(\B^{\mu;(\idelm, w_2, \dots, w_p)})$ and then
\begin{align*}
& \inv(T) \in \B^{\mu; \mathbf{w}}  \\
\iff & \ce_{w_1^{-1}}^{\spa \max}( \inv(T) )\in \B^{\mu; (\idelm, w_2, \dots, w_p)}   \\
\iff &\text{$\inv \big( \ce_{w_1^{-1}}^{\spa \max} \big(\inv(T) \big)\big) = P_{w_1^{-1}}(T)$ is $(\idelm, w_2, \dots, w_p)$-katabolizable} \\
\iff & \text{$T$ is $(w_1, w_2, \dots, w_p)$-katabolizable},
\end{align*}
where the second equivalence uses Proposition \ref{p Ei max equal partial P} and the inductive hypothesis.

Next suppose  $p>1$ and $w_1 = \idelm$.
%??this tensor bad for tensor prod order issue.  all switched to nonkashiwara
Note that $\B^{\mu; \mathbf{w}} = (\dtau \B^{(\mu_2, \dots, \mu_p); (w_2, \dots, w_p)}) \tsr \sfb_{\mu_1} $.
Then  $\inv(T) \in \B^{\mu; \mathbf{w}}$ if and only if
(1) the rightmost block  of $\inv(T)$ is $\sfb_{\mu_1} = 1^{\mu_1}$ and (2)  $b' \in \dtau \B^{(\mu_2, \dots, \mu_p); (w_2, \dots, w_p)}$, where $b'$
is  the biword obtained from  $\inv(T)$ by removing this block of 1's and subtracting 1 from its top letters.
%Recalling the explicit combinatorial realization of  $\dtau^{-1}$ in ? and the definition of  $\kat'$,
By \eqref{d katprime def}, condition (2) is equivalent to
%$U \in \dtau \B^{(\mu_2, \dots, \mu_p); (w_2, \dots, w_p)}  \iff  $
$\kat'(\inv(T)) = \dtaui(b') \in \B^{(\mu_2, \dots, \mu_p); (w_2, \dots, w_p)}$;
by  the inductive hypothesis and Proposition~\ref{p kat and kat'},
this is equivalent to  $\kat(T)= \inv (\kat'(\inv(T))$ being $(w_2, \dots, w_p)$-katabolizable.
Hence, noting that (1) %the last block  of $\inv(T)$ is all 1's %or say $1^{\mu_1}$ hmm.
is equivalent to  $T$ having no $1$'s outside its first row, we conclude that (1) and (2) are
equivalent to $T$ being $\mathbf{w}$-katabolizable.
\end{proof}

%\big\{T \in \Tabloids_\ell(\mu) \mid P(T) \text{ is $(\ell, \nr(\Psi))$-katabolizable}\big\}
%\xrightarrow{\inv} \B^{\mu;(\mathsf{w}_0, \ns(\Psi))}
\begin{samepage}
\begin{theorem}
\label{t crystal comps w0}
%note: seems like not worth stating the key version until later
%The $U_q(\sl_\ell)$-restriction of  $\B^{\mu; \mathbf{w}}$
%is a disjoint union of $U_q(\sl_\ell)$-Demazure crystals.
For $w_1 = \mathsf{w}_0$, the DARK crystal
$\B^{\mu; \mathbf{w}}$ (regarded as a subset of the $U_q(\gl_\ell)$-crystal  $\B^\mu$)
is a disjoint union of highest weight $U_q(\gl_\ell)$-crystals, with decomposition given  by
%decomposes into a disjoint union of $U_q(\sl_\ell)$-highest weight crystals as
\begin{align}
\label{et crystal comps w0}
\B^{\mu; \mathbf{w}} = \!\! \bigsqcup_{\substack{U \in \SSYT_\ell(\mu) \\ \text{U is $(\idelm, w_2, \dots, w_p)$-katabolizable}}} \!\! \!\! \cC_U,
\qquad \text{ where
$\cC_U = \{b \in \B^{\mu} \mid Q(b) = U\}$.}
\end{align}
\end{theorem}
\end{samepage}
\begin{proof}
This follows from Theorems \ref{t Blambda comps etc} and \ref{t kat and inv crystal}, using that,
when $w_1 = \mathsf{w}_0$,
$T \in \Tabloids_\ell(\mu)$ is  $\mathbf{w}$-katabolizable if and only if
$P(T) = P_{\mathsf{w}_0^{-1}}(T)$ is $(\idelm, w_2, \dots, w_p)$-katabolizable.
%(this equality of tableaux holds by Proposition \ref{p Pi braid} (iv)).
\end{proof}

%We can also prove Theorem \ref{t intro katable inv match crystal}, the variant of Theorem  from \S\ref{}:
Let us now also prove Theorem \ref{t intro katable inv match crystal}:
apply Theorem \ref{t kat and inv crystal}
with $\mathbf{w} = (\mathsf{w}_0, \ns(\Psi))$, then the ``if and only if'' statement in the proof of Theorem \ref{t crystal comps w0}, and then Proposition \ref{p easy kat vs gen kat}.

\section{Schur and key positivity}
\label{s Schur and key positivity}
%Here we apply our knowledge of DARK and AGD crystals to compute their characters.
%The Schur/key positivity of these characters comes from
%They are related to katabolism by the previous section and to nonsymmetric Catalan functions by Theorem \ref{t character rotation formula}.

We connect charge to $\hatsl_\ell$-weights and
then combine the results of Section \ref{s DARK and katabolism} with
%the results on crystals,
Corollary \ref{c monomial times Demazure} and Theorem \ref{t intro KR to affine iso subsets}
%assemble results of previous sections (most importantly Corollary \ref{c monomial times Demazure} and Theorem \ref{t intro KR to affine iso subsets})
to give several character formulas for DARK and AGD crystals;
this yields our katabolism formula
%???give this thm a name?
Theorem~\ref{t kat conjecture resolution} upon combining with the rotation
theorem.
%~\ref{t character rotation formula}.
Stronger key positivity results are then obtained via the restriction Theorem \ref{t restrict Demazure}.

\subsection{Characters}

Let $\ZZ[P]$ denote the group ring of $P$ with  $\ZZ$-basis $\{ e^\lambda \}_{\lambda \in P}$.
The  $\hatsl_\ell$-\emph{Demazure operators} are linear operators $D_i$ on
 $\ZZ[P]$ defined for each $i \in I$ by
\[D_i(f) = \frac{f-e^{-\alpha_i}\cdot s_i(f)}{1-e^{-\alpha_i}},\]
where $s_i$ acts on $\ZZ[P]$ by $s_i(e^\lambda) = e^{s_i(\lambda)}$
(see \S\ref{ss the affine symmetric group and 0Hecke monoid}).
The action of  $\Sigma = \{\tau^i \mid i \in [\ell]\}$ on  $P$ yields an action on $\ZZ[P]$ given by $\sigma(e^\lambda) = e^{\sigma(\lambda)}$ for  $\sigma \in \Sigma$.
Then $\tau$ and the  $D_i$ ($i \in I$)
satisfy the 0-Hecke relations
\eqref{e extended affine Weyl far commutation}--\eqref{e 0Hecke pi i squared pi i} of  $\zaH_\ell$
(it is well known that they satisfy \eqref{e extended affine Weyl far commutation},
\eqref{e extended affine Weyl braid} \cite[Corollary 8.2.10]{Kumarbook}
 and the others are easily checked).
Thus, just as we discussed for $\F_w$ in \S\ref{ss hatsl demazure crystals},
 $D_w$ makes sense for any  $w \in \zaH_\ell$ and  $D_w D_{w'} = D_{w w'}$ for all  $w, w' \in \zaH_\ell$.
%thus for any  $v \in \widetilde{\SS}_\ell$, we can write  $v = w \tau^i$ and define $D_v$ by $D_{v} = D_w \circ \tau^i$.
%$for $w \in W$ and $\tau \in Aut(\Gamma)$, we define an operator $D_{w\tau}$ on $\ZZ[P]$ by $D_{w\tau} = D_w \circ \tau$, where $\tau$ acts on $\ZZ[P]$ by $\tau(e^\mu) = e^{\tau(\mu)}$.

The \emph{character} of a subset  $S$ of a $U_q(\hatsl_\ell)$-crystal is
$\chr(S) := \sum_{b \in S} e^{\wt(b)} \in \ZZ[P]$.
Kashiwara \cite{KashiwaraDemazure} gave a Demazure operator formula for the
character of any Demazure crystal,
and Naoi extended this to encompass the action of  $\Sigma$, as follows:

\begin{corollary}[{\cite[Corollary 4.6]{NaoiVariousLevels}}]
\label{c Naoi Corollary 4.6}
%Let $G$ be a disjoint union of Demazure crystals.
%For every $w \in W$, there holds  $\chr(\F_w(G)) =  D_w(\chr(G))$.
For any $w \in \zaH_\ell$ and $S \in \mathcal{D}(\hatsl_\ell)$ (see Definition \ref{d mathcal D}),
\[ \chr(\F_w(S)) = D_w(\chr(S)). \]
%\text{} \sum_{b \in \F_w(S)} e^{\wt(b)} = D_w\Big( \sum_{b \in S} e^{\wt(b)} \Big).\]
\end{corollary}

Set  $\AA = \ZZ[q^{1/2\ell},q^{-1/2\ell}]$.
Define the ring homomorphism  $\zeta$ by
\begin{align}
\label{e P0 iso}
\zeta \colon  \ZZqqix \to \ZZ[P], \ \ \ x_i \mapsto e^{\varpi_i - \varpi_{i-1}} \, , \ q^{1/2\ell} \mapsto e^{-\delta/2\ell}.
%\cong \ZZqqix / (x_1x_2 \cdots x_\ell-1),
%e^{\varpi_i} \mapsto x_1 x_2 \cdots x_i, \text{ and } e^{-\delta/2\ell} \mapsto q^{1/2\ell}\\
\end{align}
It is  $\SS_\ell$-equivariant ($s_i$ acts on  $\ZZqqix$ by permuting the variables) and
has kernel $(x_1\cdots x_\ell-1)$. It is an extension of the map $\zeta$ from \eqref{e intro P0 iso} to a larger domain.

We wish to recover an element of $\ZZqqix$ given its image under $\zeta$,
and this is possible if we know it to be homogenous of a given degree.
Accordingly, let $\mathbf{X}_m \subset \ZZqqix$ denote the homogeneous component of $\mathbf{x}$-degree $m$.
The restricted map $\zeta \colon  \mathbf{X}_m \to \ZZ[P]$ is injective; let $\zeta(\mathbf{X}_m)$ denote the image and
$Z_m \colon  \zeta(\mathbf{X}_m) \xrightarrow{\cong} \mathbf{X}_m$ the inverse of this restriction of $\zeta$ (which is only a $\ZZ$-linear map).

Let  $\mu$ be  a partition and set $m = |\mu|$.
Suppose $G$ is a $U_q(\hatsl_\ell)$-crystal such that $e^{-\mu_1\Lambda_0} \chr(G) \in \zeta(\mathbf{X}_m)$
(by the proof of Theorem \ref{t character AGD} below, this holds for  $G = \AGD(\mu;\mathbf{w})$, our main case of interest).
We define the \emph{$\mathbf{x}$-character} of  $G$ by
\begin{align}
\label{e def x character}
%\chr(G) = \sum_{g \in G} e^{\wt(g)} \in \ZZ[P], \\
\chr_{\mathbf{x}; \mu}(G) = \sum_{g \in G}  Z_m(e^{\wt(g) - \mu_1\Lambda_0})\in \mathbf{X}_m \, .
\end{align}
%Note that $\chr_{\mathbf{x}; \mu}(G) $ only depends on $\mu_1$ and $|\mu|$.
In other words, if we find $f \in \mathbf{X}_m$ such that $\zeta(f) = \chr(G)e^{-\mu_1\Lambda_0}$, then $f = \chr_{\mathbf{x};\mu}(G)$.
 %in particular, given we know  $f \in \mathbf{X}_m$, $\zeta(f)$ determines the character $f$ uniquely.
%, allowing us to recover  $f$ from  $G$.
%Note that we can go back and forth between these two versions:
%\begin{align}
%\label{e two versions of chr}
%\zeta(\chr_{\mathbf{x};\mu}(G)) = \chr(G)e^{-\mu_1\Lambda_0}, \quad
%Z_m(\chr(G)e^{-\mu_1\Lambda_0}) = \chr_{\mathbf{x};\mu}(G).
%\end{align}

We will need two facts which relate $\pi_i$ and  $D_i$, and $\Phi$ and  $\tau$ via $\zeta$.
%Also straightforward from the definitions of  $\pi_i, D_i$
%and the action of $\SS_\ell$ on  $P \subset \h^*$,
%Finally, we need the following fact which relates the  $\gl$ and  $\hatsl$-Demazure operators;
First, it is straightforward to show from  the  $\SS_\ell$-equivariance of   $\zeta$ that
%in and the realization of $\SS_\ell$ as a subgroup of $GL(\h^*)$,
\begin{align}
\label{e Di and pii 1}
\zeta(\pi_i(f)) &= D_i(\zeta(f)) \quad \text{ for $i \in [\ell-1]$}.
%\label{e Di and pii 2}
%\pi_i(Z_m(g)) &= Z_m (D_i(g)) \text{ for any  $g \in \ZZ[Q+\zeta(x_1^m)]$ and $i \in [\ell-1]$}.
\end{align}
Second, we claim that for any  $f \in \mathbf{X}_m$,
\begin{align}
\label{e Phi and tau}
\zeta(\Phi(f)) = e^{-m\delta/\ell} \tau (\zeta(f)).
%\qquad
%Z_m(\tau \lambda) = q^{-m/\ell} \Phi(Z_m(\lambda)).
\end{align}
%\begin{align}
%\tau \zeta(\mathbf{x}^\mu\mathbf{x}^{-\mu}f) =
%\tau (\zeta(\mathbf{x}^\mu))\tau (\zeta(\mathbf{x}^{-\mu}f)) =
%\tau (\zeta(\mathbf{x}^\mu))\zeta(\Phi(\mathbf{x}^{-\mu}f)) =
%e^{\sum_{i} \mu_i(\bar{\Lambda}_{i+1}-\bar{\Lambda}_i)} \zeta(\Phi(\mathbf{x}^{-\mu}f)) =
%\zeta(\Phi(q^{-\mu_\ell}f)).
%\end{align}
%This can be shown direct computation in the monomial basis.  Key to this computation is
Since  $\zeta, \tau,$ and  $\Phi$ are ring homomorphisms, it is enough to prove  $\zeta(\Phi(x_i)) = e^{-\delta/\ell} \tau (\zeta(x_i))$.
This is readily verified from the computation
%$It is enough to verify this for is reduces to checking it for  $$
\begin{align*}
\tau (\zeta(x_i)) = \tau(e^{\varpi_i-\varpi_{i-1}}) = e^{\varpi_{i+1}-\varpi_{i} + \delta(m_{i+1}-m_i-(m_i-m_{i-1}))} =
\begin{cases}
e^{\varpi_{i+1}-\varpi_{i} + \delta/\ell}   & \text{ if  $i \in [\ell-1]$} \\
e^{\varpi_{i+1}-\varpi_{i} + \delta/\ell - \delta} & \text{ if  $i = \ell$},
\end{cases}
\end{align*}
where $m_i := \langle d, \Lambda_i \rangle$ and the last equality is by \eqref{e tau computation}.

\subsection{Charge and $\hatsl_\ell$-weights}
\label{ss charge vs d}

The pairing  $\langle d, \cdot \rangle$ on  $\hatsl_\ell$-weights gives a
%??$\frac{1}{2\}\ZZ$-valued
statistic on $U_q(\hatsl_\ell)$-crystal elements, which
is not available for $U'_q(\hatsl_\ell)$-seminormal crystals.
Naoi \cite{NaoiVariousLevels} showed that the strict embedding $\Theta_\mu$ matches this statistic to energy, thereby effectively allowing the full information of
$\hatsl_\ell$-weights to be seen on the DARK side.
Since energy on  $\B^\mu$ matches charge on $\Tabloids_\ell(\mu) = \inv(\B^\mu)$ \cite{NYenergy},
the charge and $\langle d, \cdot \rangle$ statistics agree.
% with
%energy on  $\B^\mu$ the pairing  $\langle d, \cdot \rangle$ on  $\hatsl_\ell$-weights under the embedding  $\Theta_\mu$.
%These results connect both statistics to energy.

\begin{remark}
It is actually more natural to connect charge and $\langle d, \cdot \rangle$ directly
%(but we omit details in the interest of space)
as they both essentially measure the number
of  $\cf_0$-edges required to construct the crystal element,
whereas energy is a more complicated statistic.
 %involving combinatorial $R$-matrices.
In the interest of space, we just give the idea:
for  $b  = \cf_{i_1}^{a_1}\cdots \cf_{i_k}^{a_k} u_\Lambda \in  \F_{i_1}\cdots \F_{i_k} \{u_{\Lambda} \}$
with  $i_j \in I$,  $\langle -d, \wt(b)- \wt(u_\Lambda)\rangle$ is the number of  $\cf_0$'s appearing in  $\cf_{i_1}^{a_1}\cdots \cf_{i_k}^{a_k}$.
A similar statement can be made for AGD crystals.
Charge also has a similar flavor since
%can be given a similar interpretation as
$\cf_0$-edges are related to property  (C3) below by the  $\inv$ map---see \cite[\S4.2]{Shimozonoaffine}.
%can be an interpreted as a statistic of this flavor as well.
%We omit details in the interest of space.
\end{remark}

%\begin{example}
%We compute the charge of a word $u$.
%The extracted permutations are indicated as $u_i$.
%For each permutation $u_i$, the quantities ci are indicated as subscripts.
%u = 2 1 3 5 1 4 1 2 4 3 2 3
%u1 = 21 53 10 42 31
%u2 = 10 42 31 20
%u3 = 10 31 20
%c(u1) = 1 + 3 + 0 + 2 + 1 = 7, c(u2) = 0 + 2 + 1 + 0 = 3, c(u3) = 0 + 1 + 0 = 1, c(u) = 7 + 3 + 1 = 11.
%\end{example}

%\begin{theorem}[{\cite[]{LS?}}, see e.g. {\cite[Theorem 5.7]{Shimozono2005CrystalsFD}}]
%\label{t charge characterization}
%Charge is the unique function on words of partition content which satisfies
%\begin{itemize}
%\item[(C1)] $\charge(u)=\sum_{i = 1}^p(i-1)\mu_i$ for the weakly increasing word $u$ of content $\mu$.
%\item[(C2)] $\charge$ is constant on Knuth equivalence classes.
%\item[(C3)] If $u = vx$ with $x \neq 1$ a letter, then $\charge(vx) = \charge(xv) + 1$.
%\end{itemize}
%\end{theorem}

Charge is a statistic on words of partition content which is commonly defined by
a circular-reading procedure (see, e.g., \cite[\S3.6]{SW}).
We prefer to take the following theorem of
Lascoux and Sch\"utzenberger as its definition.

\begin{theorem}[{\cite{LS}, see \cite[Theorem 24]{SW}}]
\label{t charge characterization}
Charge is the unique function from words of partition content to  $\ZZ_{\ge 0}$ satisfying
%which satisfies the following properties:
%let  $u$ be a word of partition content  $\lambda$.
\begin{itemize}
\item[\emph{(C1)}] Charge of the empty word is 0.
\item[\emph{(C2)}] For a word of partition content  $\lambda$ and of the form $u = v 1^{\lambda_1}$, $\charge(u) = \charge(v^-)$,
 where  $v^-$ is obtained from  $v$ by subtracting 1 from all its letters.
%, where  $m = \content_1(u)$,   has all its , where  $v$ is a word
\item[\emph{(C3)}] For a word  of partition content and of the form $u = vx$ with $x \neq 1$ a letter, $\charge(vx) = \charge(xv) + 1$.
\item[\emph{(C4)}] Charge is constant on Knuth equivalence classes.
\end{itemize}
\end{theorem}

We will view $\charge$ as a statistic on tabloids by setting
$\charge(T) = \charge(T^\ell \cdots T^2 T^1)$ for any $T \in \Tabloids_\ell$, where
the concatenation %$T^i$ is the  $i$-th row of  $T$ and
$T^\ell \cdots T^2 T^1$ is the row reading word of  $T$.

\begin{corollary}
\label{c delta charge}
Let $\mu$ be a partition and $\Theta_\mu \colon
\B^\mu \tsr B(\mu_1\Lambda_0)  \hookrightarrow B(\mu^p\Lambda_p) \tsr \cdots \tsr B(\mu^1 \Lambda_1)$ the strict embedding of $U'_q(\hatsl_\ell)$-seminormal crystals
from Theorem \ref{t intro KR to affine iso subsets}.
For any $b \in \B^\mu$,
\begin{align}
\label{ep delta charge}
& \wt \big(\Theta_\mu(b \tsr u_{\mu_1 \Lambda_0} )\big) =
 \mu_1 \Lambda_0 +  \aff( \wt(b)) - \delta \big(\charge(\inv(b)) + n_\ell(\mu) \big), \\[.4mm]
\label{e content and Zlambda 1}
& \zeta \big( q^{\charge(\inv(b))+n_\ell(\mu)} \mathbf{x}^{\content(b)} \big)
= e^{-\mu_1\Lambda_0 + \wt(\Theta_\mu(b \tsr u_{\mu_1 \Lambda_0} ))},
%\label{e content and Zlambda 2}
%q^{\charge(\inv(b))+n_\ell(\mu)} \mathbf{x}^{\content(b)}
%=Z_m \big(  e^{-\mu_1\Lambda_0}\wt(\Theta_\mu(u_{\mu_1 \Lambda_0} \tsr b)) \big).
\end{align}
where
%$n_\ell(\mu) = \frac{1}{2\ell} \sum_{i=1}^p \mu_i (\ell+1-2i)$.
$n_\ell(\mu) :=  \textstyle  \frac{|\mu|(\ell-1)-2n(\mu)}{2\ell}$, a variant of the well-known statistic $n(\mu):= \sum_{i=1}^p (i-1)\mu_i$.
%C_\mu = \langle \wt u_{hw}, d \rangle = \langle\sum_{i=1}^p \mu^i \Lambda_i , d \rangle$.
\end{corollary}
\begin{proof}
As $\wt(b) \in P_{\cl}$ is given by \eqref{e KR weight},
$\zeta(\mathbf{x}^{\content(b)}) = e^{\aff \wt(b)}$; hence \eqref{ep delta charge} implies~\eqref{e content and Zlambda 1}.

%and thus $\mathbf{x}^{\content(b)} = Z_m(\aff \wt(b))$.
We now prove \eqref{ep delta charge}.
Set $m_i = \langle d, \Lambda_i \rangle$ for  $i \in I$.
Since $\Theta_\mu$ commutes with the $P_\cl$-valued weight functions,
%$\wt: B \to P_{\cl}$ for $U'_q(\hatsl_\ell)$-seminormal crystals,
%and the  $B(s\Lambda_i)$ are regarded as $U'_q(\hatsl_\ell)$-seminormal crystals by restriction,
$\cl (\wt \Theta_\mu( b \tsr u_{\mu_1 \Lambda_0}  )) =
 \cl(\mu_1 \Lambda_0) + \wt(b)$.
%Recall that $\langle d, \aff (\lambda) \rangle = 0$ for all $\lambda \in P$ (see \S\ref{ss lie algebra hatsl});
Thus \eqref{ep delta charge} is equivalent to
\begin{align}
\label{e remains to show}
\mu_1 m_0 + \langle -d, \wt \Theta_\mu( b \tsr u_{\mu_1 \Lambda_0}) \rangle - \charge(\inv(b)) =  n_\ell(\mu).
%\langle -d, \wt \Theta_\mu(u_{\mu_1 \Lambda_0} \tsr b) \rangle = -\mu_1 m_0 + \charge(\inv(b)) + n_\ell(\mu).
\end{align}
By \cite[Theorem 7.1]{NaoiVariousLevels},
$\langle -d, \wt \Theta_\mu(b \tsr u_{\mu_1 \Lambda_0} ) \rangle = D(b) + C$,
where $D(b)$ is the energy of $b$ and $C$ is a constant that depends on $\mu$ and  $\ell$ but not $b$.
Further, $D(b) = \charge(\inv(b))$ by \cite{NYenergy} (see also \cite[Proposition 4.25]{Shimozono2005CrystalsFD}).
Hence, to pin down the constant, we need only verify \eqref{e remains to show}
for a single $b\in \B^\mu$.
We choose
$b_{hw} := \twist{\tau} \big( \cdots \twist{\tau} (\twist{\tau} \sfb_{\mu_p} \tsr \sfb_{\mu_{p-1}}) \cdots \tsr \sfb_{\mu_2}\big) \tsr \sfb_{\mu_1}$,
the element satisfying
$\Theta_\mu(u_{hw}) = b_{hw} \tsr u_{\mu_1\Lambda_0}$
for $u_{hw} := u_{\mu^p\Lambda_p} \tsr \cdots \tsr u_{\mu^1\Lambda_1}$,
as can be seen from the proof of \cite[Theorem 3.7]{generalizedNaoi}.  We compute
%The following computation then pins down the constant  $n_\ell(\mu)$ and completes the proof.
%
%It remains to pin down the constant  $n_\ell(\mu)$, which is done as follows:
%It remains to compute $\mu_1 m_0 + \langle  \wt (u_{hw}) , -d \rangle - \charge(\inv(b_{hw}))= n_\ell(\mu)$: compute the constant down the cons
%The constant  $n_\ell(\mu)$ is obtained by and then the result follows from
\begin{align*}
& \mu_1 m_0+ \langle -d, \wt (u_{hw}) \rangle - \charge(\inv(b_{hw})) =
 \mu_1 m_0-\sum_{i=1}^p \mu^i m_{i}  - \charge(\inv(b_{hw}))  \\
& = \  -\sum_{i=1}^p \mu_i (m_{i}-m_{i-1}) - \charge(\inv(b_{hw}))
=  -\sum_{i=1}^p \mu_i \textstyle \frac{2\modd_1^{\ell}(i)-1-\ell}{2\ell} - \sum_{i=1}^p \lfloor \frac{i-1}{\ell} \rfloor  \mu_i \\
& = \ \textstyle  \frac{|\mu|\spa (\ell-1)}{2\ell}-\frac{1}{\ell}\sum_{i=1}^p (i-1)\mu_i
 %= \ \textstyle  \frac{|\mu|\spa (\ell-1)-2n(\mu)}{2\ell}
%& = \  \frac{1}{\ell} \big(\frac{|\mu|(\ell-1)}{2}-\sum_{i=1}^p (i-1)\mu_i \big)
%=  \frac{1}{2\ell} \sum_{i=1}^p \mu_i (\ell+1-2i)
= n_\ell(\mu),
\end{align*}
where  $\{\modd_1^{\ell}(i)\} = (i+\ell\ZZ) \cap [\ell]$.
The third equality is by \eqref{e tau computation}
and a direct computation of the charge of $\inv(b_{hw})$, the tabloid of content  $\mu$ with all letters $i$ in row $\modd_1^\ell(i)$.
\end{proof}

%\subsection{Character identities}
\subsection{A Schur positive formula for Catalan functions: proof of Theorem \ref{t kat conjecture resolution}}
% and the Shimozono-Weyman conjecture}
\label{ss Schur positive formula for Catalan functions}

\begin{theorem}
\label{t character AGD}
Let $\mathbf{w} = (w_1, w_2, \dots, w_p) \in (\zH_\ell)^p$ and $\mu = (\mu_1 \ge \cdots \ge  \mu_p \ge 0)$ be a partition; set $\mu^i = \mu_i-\mu_{i+1}$, where  $\mu_{p+1} := 0$.
The  $\mathbf{x}$-character of the crystal
$\AGD(\mu;\mathbf{w})$
%\F_{w_1} \big( u_{\mu^1\Lambda_{1}} \tsr \tau \F_{w_{2}} \big( u_{\mu^2\Lambda_{1}} \tsr \cdots
%\tsr \tau \F_{w_p} (u_{\mu^p\Lambda_1 }) \big) \big)$
agrees with the charge weighted character of the DARK crystal  $\B^{\mu;\mathbf{w}}$ and these have an explicit description in terms of  $\pi_i$ and  $\Phi$:
\begin{align}
\label{ec character AGD 1}
&  \pi_{w_1}  \spa x_1^{\mu_1}\Phi \spa \pi_{w_2} \spa x_1^{\mu_2}\Phi \spa \pi_{w_3} \spa x_1^{\mu_3} \cdots \Phi \spa \pi_{w_p} \spa x_1^{\mu_p} =
q^{-n_\ell(\mu)}\chr_{\mathbf{x};\mu}(\AGD(\mu;\mathbf{w})) \\[1mm]
& \! =  \sum_{b \in \B^{\mu;\mathbf{w}}} q^{\charge(\inv(b))} \mathbf{x}^{\content(b)}
\label{ec character AGD 2}
= \sum_{\substack{\text{$T \in \Tabloids_\ell(\mu)$}\\ \text{$T$ is $\mathbf{w}$-katabolizable} }} \!\!\! q^{\charge(T)}\mathbf{x}^{\sh(T)}\spa ,
\end{align}
where $n_\ell(\mu) =  \frac{|\mu|\spa (\ell-1)}{2\ell}- \frac{1}{\ell} \sum_{i=1}^p (i-1)\mu_i$ as in Corollary \ref{c delta charge}.
\end{theorem}

\begin{proof}
The first equality of \eqref{ec character AGD 2}
follows from Theorem \ref{t intro KR to affine iso subsets} and \eqref{e content and Zlambda 1},
and the second holds by Theorem \ref{t kat and inv crystal}.
%We complete the proof by showing
%$e^{-\mu_1\Lambda_0}\chr(\AGD(\mu;\mathbf{w}))$
%$= e^{-\delta \spa n_\ell(\mu)} \zeta(\pi_{w_1}  \spa x_1^{\mu_1}\Phi \spa \pi_{w_2} \spa x_1^{\mu_2}  \cdots \Phi \spa \pi_{w_p} \spa x_1^{\mu_p})$,
%which will prove that the first and last expressions in \eqref{ec character AGD 1}--\eqref{ec character AGD 2} are
%equal.
We will establish \eqref{ec character AGD 1} by proving
$e^{-\mu_1\Lambda_0}\chr(\AGD(\mu;\mathbf{w}))$
$ = e^{-\delta \spa n_\ell(\mu)} \zeta(\pi_{w_1}  \spa x_1^{\mu_1}\Phi \spa \pi_{w_2} \spa x_1^{\mu_2}  \cdots \Phi \spa \pi_{w_p} \spa x_1^{\mu_p})$.
By Corollaries \ref{c monomial times Demazure} and \ref{c Naoi Corollary 4.6},
\begin{align}
%\label{e character AGD Naoi}
\chr(\AGD(\mu;\mathbf{w})) =
D_{w_1}\big( e^{\mu^1\Lambda_1} \cdot \tau D_{w_2}\big(e^{\mu^2\Lambda_1} \cdot \tau D_{w_3} \cdots \tau D_{w_p}(e^{\mu^p\Lambda_1})\big)\big).
\end{align}
(A similar character formula is proved in \cite[\S7]{NaoiVariousLevels} with this argument.)
Using the operator identities $e^{\Lambda_1}\tau = \tau e^{\Lambda_0}$ and $e^{\Lambda_0}D_i = D_i e^{\Lambda_0}$ for  $i \in [\ell-1]$, we compute
\begin{align*}
e^{-\mu_1\Lambda_0}\chr(\AGD(\mu;\mathbf{w}))
&= e^{-\mu_1\Lambda_0} D_{w_1}\big( e^{\mu^1\Lambda_1} \cdot \tau D_{w_2}\big(e^{\mu^2\Lambda_1}  \cdots   \tau D_{w_p}(e^{\mu^p\Lambda_1})\big)\big)\\
&=D_{w_1}\big( e^{\mu_1(\Lambda_1-\Lambda_0)-\mu_2\Lambda_1}\cdot  \tau D_{w_2}\big(e^{\mu^2\Lambda_1}  \cdots   \tau D_{w_p}(e^{\mu^p\Lambda_1})\big)\big)\\
&=D_{w_1}\big( e^{\mu_1(\Lambda_1-\Lambda_0)} \cdot \tau D_{w_2}\big(e^{-\mu_2\Lambda_0}e^{\mu^2\Lambda_1}  \cdots   \tau D_{w_p}(e^{\mu^p\Lambda_1})\big)\big)\\
&=D_{w_1}\big( e^{\mu_1(\Lambda_1-\Lambda_0)} \cdot \tau D_{w_2}\big(e^{\mu_2(\Lambda_1-\Lambda_0)-\mu_3\Lambda_1} \cdots   \tau D_{w_p}(e^{\mu^p\Lambda_1})\big)\big)\\
&=\cdots\\
&=D_{w_1}\big( e^{\mu_1(\Lambda_1-\Lambda_0)} \cdot \tau D_{w_2}\big(e^{\mu_2(\Lambda_1-\Lambda_0)}  \cdots   \tau D_{w_p}(e^{\mu_p(\Lambda_1-\Lambda_0)})\big)\big)\\
&=e^{-\delta \spa n_\ell(\mu)} \zeta\Big(\pi_{w_1}  \spa x_1^{\mu_1}\Phi \spa \pi_{w_2} \spa x_1^{\mu_2} \cdots \Phi \spa \pi_{w_p} \spa x_1^{\mu_p}\Big).
\end{align*}
The last equality follows from \eqref{e Di and pii 1} and \eqref{e Phi and tau};
in particular, the constant $n_\ell(\mu)$ appears since, as we pull  $\zeta$ to the right through the operators,
we pick up a factor $e^{-\frac{\delta}{\ell}\sum_{i=1}^p (i-1)\mu_i }$
for converting $\Phi$'s to $\tau$'s and a factor $e^{\delta  |\mu|\frac{\ell-1}{2\ell}}$ for converting multiplication by  $x_1$ to multiplication by $e^{\Lambda_1-\Lambda_0}$ since $\zeta(x_1) = e^{\varpi_1} = e^{\delta \frac{\ell-1}{2\ell}}e^{\Lambda_1-\Lambda_0}$ by \eqref{e tau computation}.
\end{proof}

\begin{corollary}
\label{c xchar kat schur pos}
In the case  $w_1 = \mathsf{w}_0$ (the longest element in  $\zH_\ell$),
the characters in Theorem \ref{t character AGD}
have the following Schur positive expansion:
\begin{align*}
&  \pi_{w_1}  \spa x_1^{\mu_1}\Phi \spa \pi_{w_2} \spa x_1^{\mu_2}\Phi \spa \pi_{w_3} \spa x_1^{\mu_3} \cdots \Phi \spa \pi_{w_p} \spa x_1^{\mu_p} =
q^{-n_\ell(\mu)}\chr_{\mathbf{x};\mu}(\AGD(\mu;\mathbf{w}))  \\[.8mm]
& = \sum_{b \in \B^{\mu;\mathbf{w}}} q^{\charge(\inv(b))} \mathbf{x}^{\content(b)}
= \!\! \sum_{\substack{U \in \SSYT_\ell(\mu) \\ \text{$U$ is $(\idelm, w_2, \dots, w_p)$-katabolizable} }} \!\!\! q^{\charge(U)}s_{\sh(U)}(\mathbf{x}).
\end{align*}
\end{corollary}
\begin{proof}
Combine Theorems \ref{t crystal comps w0} and \ref{t character AGD},
noting that each component  $\cC_U$
% = \{b \in \B^{\mu} \mid Q(b) = U\}$
of  the  $U_q(\gl_\ell)$-crystal  $\B^{\mu; \mathbf{w}}$
contributes  $q^{\charge(U)}$ times
$\sum_{b \in \cC_U} \mathbf{x}^{\content(b)}
= \sum_{b \in B^\gl(\sh(U))}  \mathbf{x}^{\wt(b)} =
s_{\sh(U)}(\mathbf{x})$
to the left side of \eqref{ec character AGD 2};
this last (well-known) equality follows from Proposition \ref{p key character}.
\end{proof}

Combining Corollary \ref{c xchar kat schur pos}, Theorem \ref{t character rotation formula}, and Proposition \ref{p easy kat vs gen kat}
yields Theorem \ref{t kat conjecture resolution}.
This proves the katabolism conjecture of Shimozono-Weyman \cite[Conjecture 27]{SW} upon verifying
that our katabolism Definition \ref{d intro kat}
agrees with that of \cite{SW} in the parabolic case:
% which we do in the next proposition.
%Corollary \label{c kat conjecture resolution} and Proposition \ref{p recover SW def}.

\begin{proposition}
\label{p recover SW def}
When $\Psi$ is the parabolic root ideal  $\Delta(\eta)$ for some composition $\eta$ of $\ell$ (see \eqref{e d parabolic root ideal}),
a tableau $T$ of partition content  $\mu$ is $\nr(\Psi)$-katabolizable
%(see Definition \ref{d intro kat} and Proposition \ref{p easy kat vs gen kat})
if and only if it is  $R(\eta, \mu)$-katabolizable in the sense of  \cite[\S3.7]{SW}.
%where  $R = ((\mu_1, \dots, \mu_{\eta_1}), \dots, (\mu_{}, \dots, \mu_{\eta_1}))$.
\end{proposition}
\begin{proof}
%Let  $\mathbf{n}=\nr(\Psi)$.
%The tableau $T$ consists of a northern part  $T^{[\eta_1]}$ and southern part  $T^{[\eta_1+1, \ell]}$.
Checking whether  $T$ is $\nr(\Psi)$-katabolizable begins
with the computation $U = P_{1,\ell} \circ \kat \cdots P_{\eta_1-1,\ell} \circ \kat \circ P_{\eta_1,\ell} \circ \kat(T)$.
The key observation is that each row $T^1, T^2, \dots, T^{\eta_1}$ of  $T$ is never touched by the column insertions
until it is rotated to become the new  $\ell$-th row.
Hence the computation of $U$ amounts to the following: check whether  $T^1$ contains  $\mu_1$ 1's, remove these 1's, then column insert the result into
$T^{[\eta_1+1, \ell]}$ to obtain a new tableau $V$, then check whether  $T^2$ contains  $\mu_2$ 2's, remove these 2's, column insert the result into
$V$, and so on.
%$\widehat{T^1}$ into  $T^{[\eta_1+1, \ell]}$, then $\widehat{T^2}$, $\dots,$ up to $\widehat{T^{\eta_1}}$, where  $\widehat{T^i}$ is the  $i$-th row of
% $T$ with letters  $i$ removed;
%%since the insertion of  $T^i$ does not affect the lower into  one row at a time, starting with the bottom row;
%taking into account that  $\kat$ involves rotating rows and subtracting 1 from all entries, the first application of  $\kat$ checks whether
%all the 1's of $T$ lie in its first row, the next checks whether all the $2$'s of  $T$ lie in its second row, and so on.
These checks are equivalent to checking whether  $T$ contains the superstandard tableau  $Z$ of shape $(\mu_1, \dots, \mu_{\eta_1})$.
%Hence the algorithm proceeds if and only if $T^{[\eta_1]}$ contains a superstandard tableau  $Z$,
Thus,  $T$ is not rejected in this computation %these first  $\eta_1$ steps of the algorithm do not reject $T$
if and only if  $T$ contains $Z$,
and if so, $U$ is obtained by column inserting  $T^{[\eta_1]}\setminus Z$ into $T^{[\eta_1+1, \ell]}$ one row at a time, which is the same
as the rectification of the skew tableau formed by placing $T^{[\eta_1]}\setminus Z$ and $T^{[\eta_1+1, \ell]}$ catty-corner.
This is exactly the first step in the katabolism algorithm of \cite{SW}.
%computing $R(\eta, \mu)$-katabolizability.
Continuing in this way with  $\eta_2, \eta_3, \dots$ gives the result.
\end{proof}

\subsection{Key positivity}
\label{ss key positivity}

We generalize the results above to key positive formulas for
characters of AGD and DARK crystals and tame nonsymmetric Catalan functions.
%We require some additional facts about $U_q(\gl_\ell)$-Demazure crystal in addition to
%as well as some additional facts
To do this, we address the algorithmic
problem of obtaining explicit key expansions for characters
%of a set we know problem
%in mind: given
of subsets which we know  to be disjoint unions of $U_q(\gl_\ell)$-Demazure crystals;
%, identify the element  of weight  $\gamma$$$ of i,
%adding to our discussion in \S\ref{ss gl Demazure crystal},
%which we need to identify tabloids whose shapes will yield key polynomial expansion;
some of this material, in particular Proposition \ref{p find which key}, is similar in spirit to
\cite[\S4]{AssafGonzalez}.

Let  $B$ be a $U_q(\gl_\ell)$-crystal.
%Write  $\wt(b) = (\wt_1(b), \dots, \wt_\ell(b)) \in \ZZ^\ell$ for the weight of  $b\in B$.
The weight function takes values in  $\ZZ^\ell$ and we write
$\wt(b) = (\wt_1(b), \dots, \wt_\ell(b))$ for the entries of  $\wt(b)$.
%$\wt : B \to \ZZ^\ell$
The \emph{crystal reflection operators}  $S_i: B \to B$,  $i \in [\ell-1]$, are given by
\begin{align*}
S_i(b)  =
\begin{cases}
\cf_i^{\wt_i(b)-\wt_{i+1}(b)} (b)  & \text{if }  \wt_i(b) \ge \wt_{i+1}(b), \\
\ce_i^{-\wt_i(b)+\wt_{i+1}(b)} (b)  & \text{if }  \wt_i(b) \le \wt_{i+1}(b).
\end{cases}
\end{align*}
Note that  $s_i (\wt(b)) = \wt(S_i (b))$.
The operators  $S_i$ were first studied by Lascoux and Sch\"utzenberger \cite{LS}, and later generalized
%which are treated in the setting of abstract crystals of any Kac Moody algebra
by Kashiwara \cite{Kashiwaramodifiedquantized}.
They satisfy the braid relations and therefore generate an action of  $\SS_\ell$ on  $B$.
%and in particular the  $U_q(\gl_\ell)$-restriction of $\B^\mu$.
For  $1 \le i < j \le \ell$, let  $s_{ij} =$ $s_i s_{i+1}\cdots s_{j-2} s_{j-1} s_{j-2} \cdots s_i \in \SS_\ell$ denote the
transposition swapping  $i$ and  $j$,
and $S_{ij} = $ $S_i S_{i+1} \cdots S_{j-2} S_{j-1} S_{j-2} \cdots S_i$
the corresponding reflection operator.

We define \emph{Bruhat order on  $\ZZ^\ell$} by  $\alpha < \beta$ if and only if  $\alpha^+ = \beta^+$ and $p(\alpha) < p(\beta)$ in Bruhat order on  $\SS_\ell$, where
$\alpha^+$ denotes the weakly decreasing rearrangement of  $\alpha$ and
$p(\alpha) \in \SS_\ell$ the shortest element
%(minimal coset representative)
such that $p(\alpha) \alpha^+ = \alpha$.

\begin{proposition}
\label{p bruhat covering}
%For $\beta, \alpha \in \ZZ^\ell$,
The relation
$\beta > \alpha$ is a covering relation in Bruhat order on  $\ZZ^\ell$ if and only if
there exist $1 \le i<k \le \ell$ such that $\alpha = s_{i \spa k}\spa \beta$ with $\alpha_i > \alpha_k$,
and
$\alpha_j \notin [\alpha_i, \alpha_k]$ for all $j \in [i+1, k-1]$.
% $\alpha = s_{i k}\beta$ for some $1 \le i<k \le \ell$ such that  $\alpha_i < \alpha_k$
%and for all  $j = i+1, \dots, k-1$, it is not true that $\alpha_i < \alpha_j < \alpha_k$.
\end{proposition}
\begin{proof}
For permutations $\alpha$ and  $\beta$, this is a well-known combinatorial description of the Bruhat order covering relations of  $\SS_\ell$ (see, e.g., \cite[Lemma 2.1.4]{BjornerBrentiBook}).
%In general, this is a description of Bruhat order on
The general case can be deduced from this one by a standardization argument and the fact that any covering relation in the Bruhat order poset restricted to minimal coset representatives %of a parabolic subgroup of  $\SS_\ell$ ,
is actually a covering relation in the full Bruhat poset (by, e.g., \cite[Theorem 2.5.5]{BjornerBrentiBook}).
%is easily deduced from this one.
\end{proof}
For example, $\beta = 32812852 > 52812832 = \alpha$ is a covering relation and  $\alpha = s_{1 \spa 7} \beta$.

%Suppose we are given
%a $U_q(\gl_\ell)$-Demazure crystal contained in $B(\lambda)$ but we do not know which one.
The next proposition is motivated by the following algorithmic problem: suppose we have access to the elements of
a $U_q(\gl_\ell)$-Demazure crystal  $G$ and want to determine the  $\gamma\in \ZZ^\ell$ for which  $G = BD(\gamma)$
(see \S\ref{ss gl Demazure crystal} for the definition of  $BD(\gamma)$).

\begin{proposition}
\label{p find which key}
Let  $G$ be a $U_q(\gl_\ell)$-Demazure crystal.
There is a unique element  $u_{lw} \in G$ such that, setting  $\gamma = \wt(u_{lw})$,  (1) $\gamma^+$ is the highest weight of  $G$, and
%, i.e. $\wt(u_{lw})\SS_\ell \lambda$,  and
(2) $S_{ij}(u_{lw}) \notin G$ for all covering relations $\gamma < s_{ij}\gamma$ in Bruhat order on  $\ZZ^\ell$.
Moreover, $G = BD(\gamma)$.
\end{proposition}
\begin{proof}
Consider a highest weight $U_q(\gl_\ell)$-crystal $B^\gl(\nu)$.
For each weight  $\alpha$ in the orbit  $\SS_\ell \cdot \nu$, there is a unique element  $u_{\alpha} \in B^\gl(\nu)$ of weight $\alpha$;
it belongs to $BD(\beta)$, $\beta \in \SS_\ell \cdot \nu$, if and only if $\alpha \le \beta$ in Bruhat order on  $\ZZ^\ell$.
It follows that if $G = BD(\tilde{\gamma})$,  then $u_{\tilde{\gamma}} \in B^{\gl}(\tilde{\gamma}^+)$ is the unique element  $u_{lw} \in G$ satisfying (1) and (2), and  $G = BD(\wt(u_{lw}))$.
%The result then follows by applying the following easy general observation:
%if  $I$ is a lower order ideal of a poset  $P$, then  $b \in I$ is a maximal element of  $I$
%if and only if every element in  $P$ covering  $b$ does not belong to  $I$.
%to the poset of minimal coset reps
\end{proof}

For a tabloid  $T \in \Tabloids_\ell(\mu)$ and  $i\in [\ell-1]$, define $S'_i := \inv  \circ S_i  \circ \inv(T)$
and $S'_{ij} = \inv \circ  S_{ij} \circ  \inv(T)$.
%These generate an action  $\SS_\ell$ on  $\Tabloids_\ell(\mu)$.
In fact, we only need this action on the set of \emph{row frank tabloids}:
\begin{align*}
\RowFrank_\ell(\mu) := \{T \in \Tabloids_\ell(\mu)\mid \text{ $\sh(T)$ is a rearrangement of $\sh(P(T))$} \},
\end{align*}
which is also the set of inverses of the extremal weight elements of the crystal $\B^\mu$.
Since  $\sh(S'_i(T)) = s_i(\sh(T))$ for any  $T \in \RowFrank_\ell(\mu)$,
 %By weight considerations alone,
the $S'_i$ preserve the set  $\RowFrank_\ell(\mu)$.
This also gives a simple description of  $S'_{ij}(T)$ for $T \in \RowFrank_\ell(\mu)$:
$S'_{ij}(T)$ is the unique row frank tabloid Knuth equivalent to  $T$ with shape  obtained from $\sh(T)$ by exchanging
the  $i$-th and  $j$-th parts.
\begin{definition}
\label{d extreme kat}
A tabloid  $T \in \RowFrank_\ell(\mu)$ is \emph{extreme $\mathbf{w}$-katabolizable}
if  $T$ is  $\mathbf{w}$-katabolizable and $S'_{ij}(T)$ is not  $\mathbf{w}$-katabolizable for all
$i < j$ such that $\sh(T) < s_{ij}(\sh(T))$ is a covering relation in Bruhat order on  $\ZZ^\ell$.
\end{definition}

\begin{theorem}
\label{t crystal comps any w}
The DARK crystal $\B^{\mu; \mathbf{w}}$
is isomorphic to a disjoint union of $U_q(\gl_\ell)$-Demazure crystals,
%??should we say see part where we define its gl structure? no, okay
with decomposition given by
\begin{align*}
\B^{\mu; \mathbf{w}} = \!\! \bigsqcup_{\substack{T \in \RowFrank_\ell(\mu)\\ \text{T is extreme $\mathbf{w}$-katabolizable}}} \!\! \!\! \tilde{\cC}_T, \quad
\text{ where
$\tilde{\cC}_T = \{b \in \B^{\mu;\mathbf{w}} \mid Q(b) = P(T)\} \cong BD(\sh(T))$.}
\end{align*}
\end{theorem}
\begin{proof}
By Corollary \ref{c monomial times Demazure} and  Theorem \ref{t restrict Demazure}, the  $U_q(\sl_\ell)$-restriction of
$\AGD(\mu;\mathbf{w})$ is isomorphic to a disjoint union of $U_q(\sl_\ell)$-Demazure crystals.
%which are crystals of the form $B_{w}(\lambda)$ for $w \in \SS_\ell$ and  $\lambda \in P^+_0$;
%is finite case a special case of finite precisely? or take some quotient on P to P_0?
So the same is true of $\B^{\mu; \mathbf{w}} \tsr u_{\mu_1 \Lambda_0}$ (by Theorem~\ref{t intro KR to affine iso subsets})
and therefore $\B^{\mu; \mathbf{w}}$ as well.
%??what happens to weights \Lambda_0 under restriction? they go to 0 according to def of restriction
Hence by Remark \ref{r sl to gl crystal}, $\B^{\mu; \mathbf{w}}$ is isomorphic to a disjoint union of $U_q(\gl_\ell)$-Demazure crystals; this decomposition can be written as
$\B^{\mu; \mathbf{w}}= \bigsqcup \cC_U \cap \B^{\mu;\mathbf{w}}$, where
$\cC_U$
ranges over the $U_q(\gl_\ell)$-components of $\F_{\mathsf{w}_0}\B^{\mu;\mathbf{w}}$
(see Theorem \ref{t crystal comps w0}).
%$\B^{\mu; \mathbf{w}}= \bigsqcup_U \cC_U \cap \B^{\mu;\mathbf{w}}$, where  $U$
%ranges over the highest weight elements of the $U_q(\gl_\ell)$-components of $\F_{\mathsf{w}_0}\B^{\mu;\mathbf{w}}$ and  $\cC_U = \{b \in \B^{\mu} \mid Q(b) = U\}$ (see Theorem \ref{t crystal comps w0}).
Then by Proposition \ref{p find which key} and Theorem \ref{t kat and inv crystal}, each set
$\inv(\cC_U \cap \B^{\mu;\mathbf{w}})$ contains a unique
$T \in \RowFrank_\ell(\mu)$ which is extreme $\mathbf{w}$-katabolizable,
and  $\cC_U \cap \B^{\mu;\mathbf{w}} = \{b \in \B^{\mu;\mathbf{w}} \mid Q(b) = U\} = \tilde{\cC}_T \cong BD(\sh(T))$.
\end{proof}

%\begin{proposition}
%The  components of the $U_q(\gl_\ell)$-restriction of  $\B^{\mu;\mathbf{w}}$ are in bijection with
%the set of $T \in \RowFrank_\ell(\mu)$ such that  $T$ is \emph{extreme $\mathbf{w}$-katabolizable},
%and the component containing  $\inv(T)$ is isomorphic to $BD(\sh(T))$.
%\end{proposition}

\begin{corollary}
\label{c AGD key pos}
%Maintain the notation of Theorem \ref{t character AGD}.
%Then the characters therein
The characters in Theorem \ref{t character AGD}
%\eqref{ec character AGD}
%have the following key positive expansion:
are key positive with key expansion
\begin{align*}
& \pi_{w_1}  \spa x_1^{\mu_1}\Phi \spa \pi_{w_2} \spa x_1^{\mu_2}\Phi \spa \pi_{w_3} \spa x_1^{\mu_3} \cdots \Phi \spa \pi_{w_p} \spa x_1^{\mu_p} = q^{-n_\ell(\mu)}\chr_{\mathbf{x};\mu}(\AGD(\mu;\mathbf{w}))\\[.6mm]
& = \sum_{b \in \B^{\mu;\mathbf{w}}} q^{\charge(\inv(b))} \mathbf{x}^{\content(b)} = \sum_{\substack{\text{$T \in \RowFrank_\ell(\mu)$} \\ \text{$T$ is extreme $\mathbf{w}$-katabolizable} }} q^{\charge(T)}\kappa_{\sh(T)}(\mathbf{x}).
\end{align*}
\end{corollary}
\begin{proof}
Combine Theorems \ref{t crystal comps any w} and \ref{t character AGD};
each $U_q(\gl_\ell)$-Demazure crystal $\tilde{\cC}_T$
%of the $U_q(\gl_\ell)$-restriction of  $\B^{\mu; \mathbf{w}}$
contributes
$\sum_{b \in \tilde{\cC}_T} q^{\charge(T)} \mathbf{x}^{\content(b)}$
$= \sum_{b \in BD(\sh(T))} q^{\charge(T)} \mathbf{x}^{\wt(b)}
=q^{\charge(T)} \kappa_{\sh(T)}(\mathbf{x})$
to the left side of \eqref{ec character AGD 2},
where we have used Proposition \ref{p key character} for the second equality.
\end{proof}

Combining Corollary \ref{c AGD key pos} and Theorem \ref{t character rotation formula}
yields a positive combinatorial formula for the key expansions of tame nonsymmetric Catalan functions, generalizing Theorem~\ref{t kat conjecture resolution}:

\begin{corollary}
\label{c kat conjecture resolution 0}
Let $(\Psi, \mu, w)$ be a tame labeled root ideal of length $\ell$ with partition $\mu$.
Set $\mathbf{w} = (w, \ns(\Psi)) \in (\zH_\ell)^\ell$ with  $\ns(\Psi)$ as in \eqref{e ns def}.
The associated nonsymmetric Catalan function
%\F_{w} \big( u_{\mu^1\Lambda_{1}} \tsr \tau \F_{\sfc(n_1)} \big( u_{\mu^2\Lambda_{1}} \tsr \cdots
%\tsr \tau \F_{\sfc(n_{\ell-1})} (u_{\mu^\ell\Lambda_1 }) \big) \big)$
has the key positive expansion
% = \F_{w} \sfb_{\mu_1} \tsr \tau \F_{\sfc(n_1)} \sfb_{\mu_2} \tsr \dots \tsr
%\tau \F_{\sfc(n_{p-2})}  \sfb_{\mu_{p-1}} \tsr \tau \F_{\sfc(n_{\ell-1})} (\sfb_{\mu_\ell}) $.  Hence
\begin{align}
\label{ec kat conjecture resolution 0 v0}
H(\Psi; \mu; w)(\mathbf{x};q) =
\!\! \sum_{\substack{\text{$T \in \RowFrank_\ell(\mu)$} \\ \text{$T$ is extreme $\mathbf{w}$-katabolizable} }} \!\! q^{\charge(T)}\kappa_{\sh(T)}(\mathbf{x})
\end{align}
and is the character of a AGD crystal
and DARK crystal:
%and is the $\mathbf{x}$-character of a AGD crystal and the charge weighted character of a DARK crystal:
%\F_{w_1} \big( u_{\mu^1\Lambda_{1}} \tsr \tau \F_{w_2} \big( u_{\mu^2\Lambda_{1}} \tsr \cdots
%\tsr \tau \F_{w_{\ell}} (u_{\mu^\ell\Lambda_1 }) \big) \big)$$$
\begin{align}
\label{ec kat conjecture resolution 0}
H(\Psi; \mu; w)(\mathbf{x};q) =  q^{-n_\ell(\mu)}\chr_{\mathbf{x};\mu}(\AGD(\mu;\mathbf{w})) =
\sum_{b \in \B^{\mu;\mathbf{w}}} q^{\charge(\inv(b))} \mathbf{x}^{\content(b)}.
\end{align}
\end{corollary}
%\begin{proof}
%The right hand sides of \eqref{et character affine Demazure KR} and \eqref{ec character rotation formula} are different ways of saying the same thing.
%\end{proof}

See the last three lines of Figure \ref{ex DARK crystals} (\S\ref{ss intro examples}).
The bold tabloids in Figure \ref{ex DARK crystals2} are the extreme $\mathbf{v}$-katabolizable tabloids,
for  $\mathbf{v}= (\zs_1 \zs_2\zs_1, \zs_2\zs_1, \zs_2\zs_1)$ (left),  $\mathbf{v}= (\zs_2,\zs_2\zs_1, \zs_2\zs_1)$ (right); reading off their shapes and charges yields the rightmost two key expansions in Figure \ref{ex DARK crystals}.

\begin{example}
\label{ex key positive}
Let $\ell = 5$, $\mu = 22211$, and $\Psi$ be the root ideal defined by $\nr(\Psi) = (2,2,2,2)$.
Let $w = \zs_3\zs_4\zs_3$.
Then $\mathbf{w} = (w, \ns(\Psi)) = (\zs_3\zs_4\zs_3, \zs_4\zs_3\zs_2, \zs_4\zs_3\zs_2, \zs_4\zs_3\zs_2, \zs_4\zs_3\zs_2)$.
Figure \ref{f key positive} (right) depicts the set of $T$ in $\RowFrank_\ell(\mu)$  such that $T$ is extreme $\mathbf{w}$-katabolizable.
By \eqref{ec kat conjecture resolution 0 v0}, reading off their shapes and charges yields the key positive expansion
\begin{align*}H(\Psi; \mu; w) = \
 &\kappa_{22112} +
q\kappa_{32111} +
q\kappa_{22013} +
q^2\kappa_{33011} +
q^2\kappa_{32012} +
q^2\kappa_{23003} +
q^2\kappa_{42011} +\\
q^3&\kappa_{42011} +
q^3\kappa_{43001} +
q^3\kappa_{42002} +
q^3\kappa_{33002} +
q^4\kappa_{43001} +
q^4\kappa_{52001} +
q^5\kappa_{53000}.
\end{align*}
%By Theorem \ref{t crystal comps any w}, this is also $\inv(b)$ as $b$ ranges over lowest weight elements in the components of $\B^{\mu;\mathbf{w}}$.
On the left of Figure \ref{f key positive} are the inverses of the $U_q(\gl_\ell)$-highest weight elements of  $\B^{\mu;\mathbf{w}}$
obtained by computing  $P(T)$ of the tabloids on the right.
This is also the set of $U \in \SSYT_\ell(\mu)$ which are $(\idelm, \ns(\Psi))$-katabolizable
($= \nr(\Psi)$-katabolizable),
providing an example of Theorem \ref{t kat conjecture resolution} and Corollary
\ref{c xchar kat schur pos}
 as well:
 reading off their shapes and charges yields
the following Schur positive expression for $H(\Psi; \mu; \mathsf{w}_0) = \sum_{b \in \B^{\mu;(\mathsf{w}_0, \ns(\Psi))}} q^{\charge(\inv(b))} \mathbf{x}^{\content(b)}$.
\begin{align*} H(\Psi; \mu; \mathsf{w}_0) = \
 &s_{22211} +
qs_{32111} +
qs_{3122} +
q^2s_{3311} +
q^2s_{3221} +
q^2s_{332} +
q^2s_{4211} + \hphantom{aaaaaaaa}\\
q^3&s_{4211} +
q^3s_{431} +
q^3s_{422} +
q^3s_{332} +
q^4s_{431} +
q^4s_{521} +
q^5s_{53}.
\end{align*}

\begin{figure}[h]
\centerfloat
\vspace{-3.4mm}
\[\begin{array}{cc@{\hspace{12mm}}c}
\text{charge} \\
0 &{\fontsize{5.5pt}{4pt}\selectfont \tableau{} \tableau{
1&1\\2&2\\3&3\\4\\5\\} } &
{\fontsize{5.5pt}{4pt}\selectfont \tableau{
1&1\\2&2\\3\\4\\3&5\\} }\\
\vspace{0mm} \\
1 &{\fontsize{5.5pt}{4pt}\selectfont \tableau{
1&1&3\\2&2\\3\\4\\5\\}}  \quad {\fontsize{5.5pt}{4pt}\selectfont \tableau{
1&1&5\\2&2\\3&3\\4\\ \bl }}
& {\fontsize{5.5pt}{4pt}\selectfont \tableau{
1&1&3\\2&2\\3\\4\\5\\} }\quad{\fontsize{5.5pt}{4pt}\selectfont \tableau{
1&1\\2&2\\ \bl \fr[l] \\3\\ 3&4&5} }\\
\vspace{0mm} \\
2&{\fontsize{5.5pt}{4pt}\selectfont \tableau{
1&1&3\\2&2&5\\3\\4\\ \bl } }  \quad {\fontsize{5.5pt}{4pt}\selectfont \tableau{
1&1&3\\2&2\\3&5\\4\\ \bl } } \quad {\fontsize{5.5pt}{4pt}\selectfont \tableau{
1&1&4\\2&2&5\\3&3\\ \bl \\ \bl} } \quad {\fontsize{5.5pt}{4pt}\selectfont \tableau{
1&1&3&5\\2&2\\3\\4\\ \bl }} &
{\fontsize{5.5pt}{4pt}\selectfont \tableau{
1&1&3\\2&2&5\\ \bl \fr[l] \\ 3\\4\\} }\quad{\fontsize{5.5pt}{4pt}\selectfont \tableau{
1&1&3\\2&2\\ \bl \fr[l] \\ 3\\4&5\\} }\quad{\fontsize{5.5pt}{4pt}\selectfont \tableau{
1&1\\2&2&4\\ \bl \fr[l] \\  \bl \fr[l] \\ 3&3&5\\} }\quad{\fontsize{5.5pt}{4pt}\selectfont \tableau{
1&1&3&5\\2&2\\ \bl \fr[l] \\ 3\\4\\} }\\
\vspace{0mm} \\
3&{\fontsize{5.5pt}{4pt}\selectfont \tableau{
1&1&3&4\\2&2\\3\\5\\ \bl} }\quad{\fontsize{5.5pt}{4pt}\selectfont \tableau{
1&1&3&4\\2&2&5\\3\\ \bl \\ \bl} }\quad{\fontsize{5.5pt}{4pt}\selectfont \tableau{
1&1&3&4\\2&2\\3&5\\ \bl \\ \bl } }\quad{\fontsize{5.5pt}{4pt}\selectfont \tableau{
1&1&3\\2&2&4\\3&5\\ \bl \\ \bl } } &
{\fontsize{5.5pt}{4pt}\selectfont \tableau{
1&1&3&4\\2&2\\\bl \fr[l] \\3\\5\\} }\quad{\fontsize{5.5pt}{4pt}\selectfont \tableau{
1&1&3&4\\2&2&5\\ \bl \fr[l] \\ \bl \fr[l] \\ 3\\} }\quad{\fontsize{5.5pt}{4pt}\selectfont \tableau{
1&1&3&4\\2&2\\ \bl \fr[l] \\ \bl \fr[l] \\ 3&5\\} }\quad{\fontsize{5.5pt}{4pt}\selectfont \tableau{
1&1&3\\2&2&4\\ \bl \fr[l] \\ \bl \fr[l] \\ 3&5\\} }\\
\vspace{0mm} \\
4&{\fontsize{5.5pt}{4pt}\selectfont \tableau{
1&1&3&3\\2&2&4\\5\\ \bl \\ \bl } }\quad{\fontsize{5.5pt}{4pt}\selectfont \tableau{
1&1&3&4&5\\2&2\\3\\ \bl \\ \bl } } &
{\fontsize{5.5pt}{4pt}\selectfont \tableau{
1&1&3&3\\2&2&4\\ \bl \fr[l] \\ \bl \fr[l] \\ 5\\} }\quad{\fontsize{5.5pt}{4pt}\selectfont \tableau{
1&1&3&4&5\\2&2\\ \bl \fr[l] \\ \bl \fr[l] \\ 3\\} }\\
\vspace{0mm} \\
5&{\fontsize{5.5pt}{4pt}\selectfont \tableau{
1&1&3&3&5\\2&2&4\\ \bl \\ \bl \\ \bl }}
& {\fontsize{5.5pt}{4pt}\selectfont \tableau{
1&1&3&3&5\\2&2&4\\ \bl \fr[l] \\ \bl \fr[l] \\ \bl \fr[l]}}
\end{array}\]
\vspace{-3mm}
\captionsetup{width=.87\linewidth}
\caption{\label{f key positive}
%{\small
%Tabloids giving the key expansion of $H(\Psi;\mu;w)$ (right), and their insertion tableaux giving the Schur expansion of $H(\Psi;\mu;\mathsf{w}_0)$ (left).  See Example~\ref{ex key positive}.}
}
\end{figure}

Let us check that the tabloid $T = {\fontsize{5.5pt}{4pt}\selectfont \tableau{
1&1&3&4\\2&2\\ \bl \fr[l] \\ 3\\5\\} }$ is extreme $\mathbf{w}$-katabolizable.
First, the following computation shows it is $\mathbf{w}$-katabolizable:
\begin{align*}
{\fontsize{5.5pt}{4pt}\selectfont \tableau{
1&1&3&4\\2&2\\ \bl \fr[l] \\ 3\\5\\} }
\xrightarrow{P_{\zs_3\zs_4\zs_3}}
{\fontsize{5.5pt}{4pt}\selectfont \tableau{
1&1&3&4\\2&2 \\ 3\\5\\ \bl \fr[l] } }
\xrightarrow{\kat}
{\fontsize{5.5pt}{4pt}\selectfont \tableau{
1&1 \\ 2\\4\\ \bl \fr[l] \\ 2&3} }
\xrightarrow{P_{\zs_2\zs_3\zs_4}}
{\fontsize{5.5pt}{4pt}\selectfont \tableau{
1&1\\2&2&4\\3 \\ \bl \fr[l] \\ \bl \fr[l]  } }
\xrightarrow{\kat}
{\fontsize{5.5pt}{4pt}\selectfont \tableau{
1&1&3\\2 \\ \bl \fr[l] \\ \bl \fr[l] \\ \bl \fr[l] } } \\[1.8mm]
\xrightarrow{P_{\zs_2\zs_3\zs_4}}
{\fontsize{5.5pt}{4pt}\selectfont \tableau{
1&1&3 \\2 \\ \bl \fr[l]  \\\bl \fr[l] \\ \bl \fr[l] } }
\xrightarrow{\kat}
{\fontsize{5.5pt}{4pt}\selectfont \tableau{
1 \\ \bl \fr[l]  \\\bl \fr[l] \\ \bl \fr[l] \\2 } }
\xrightarrow{P_{\zs_2\zs_3\zs_4}}
{\fontsize{5.5pt}{4pt}\selectfont \tableau{
1 \\ 2 \\ \bl \fr[l] \\ \bl \fr[l] \\ \bl \fr[l] } }
\xrightarrow{\kat}
{\fontsize{5.5pt}{4pt}\selectfont \tableau{
1 \\ \bl \fr[l] \\ \bl \fr[l] \\ \bl \fr[l]\\ \bl \fr[l] } }
\xrightarrow{P_{\zs_2\zs_3\zs_4}}
{\fontsize{5.5pt}{4pt}\selectfont \tableau{
1 \\ \bl \fr[l] \\ \bl \fr[l] \\ \bl \fr[l] \\ \bl \fr[l] } }
\xrightarrow{\kat}
\varnothing
\end{align*}
%\begin{align*}
%{\fontsize{5.5pt}{4pt}\selectfont \tableau{
%1&1&3&4\\2&2\\ \bl \fr[l] \\ 3\\5\\} }
%\xrightarrow{P_{\zs_3\zs_4\zs_3}}
%{\fontsize{5.5pt}{4pt}\selectfont \tableau{
%1&1&3&4\\2&2 \\ 3\\5\\ \bl \fr[l] } }
%\xrightarrow{\kat}  \ \xrightarrow{P_{\zs_2\zs_3\zs_4}}
%{\fontsize{5.5pt}{4pt}\selectfont \tableau{
%1&1\\2&2&4\\3 \\ \bl \fr[l] \\ \bl \fr[l]  } }
%\xrightarrow{\kat} \xrightarrow{P_{\zs_2\zs_3\zs_4}}
%{\fontsize{5.5pt}{4pt}\selectfont \tableau{
%1&1&3 \\2 \\ \bl \fr[l]  \\\bl \fr[l] \\ \bl \fr[l] } }
%\xrightarrow{\kat} \xrightarrow{P_{\zs_2\zs_3\zs_4}}
%{\fontsize{5.5pt}{4pt}\selectfont \tableau{
%1 \\ 2 \\ \bl \fr[l] \\ \bl \fr[l] \\ \bl \fr[l] } }
%\xrightarrow{\kat} \xrightarrow{P_{\zs_2\zs_3\zs_4}}
%{\fontsize{5.5pt}{4pt}\selectfont \tableau{
%1 \\ \bl \fr[l] \\ \bl \fr[l] \\ \bl \fr[l] \\ \bl \fr[l] } }
%\xrightarrow{\kat} \xrightarrow{P_{\zs_2\zs_3\zs_4}}
%\varnothing
%\end{align*}

We must also show that $S'_{ij}(T)$ is not $\mathbf{w}$-katabolizable for all covering relations $\sh(T) < s_{ij} \spa (\sh(T))$.
We have $\sh(T) = 42011$, and there are three covering relations corresponding to $(i,j) = (1,2)$, $(2,3)$, and $(2,4)$.
\begin{align*}
\!\!\! S'_{12}(T) = \spa
{\fontsize{5.5pt}{4pt}\selectfont \tableau{
1&1\\2&2&3&4\\ \bl \fr[l] \\ 3\\5\\} }
&
\xrightarrow{P_{\zs_3\zs_4\zs_3}}
{\fontsize{5.5pt}{4pt}\selectfont \tableau{
1&1\\2&2&3&4 \\ 3\\5\\ \bl \fr[l] } }
\xrightarrow{\kat}
{\fontsize{5.5pt}{4pt}\selectfont \tableau{
1&1&2&3 \\ 2\\4\\ \bl \fr[l] \\\bl \fr[l] } }
\! \xrightarrow{P_{\zs_2\zs_3\zs_4}}
{\fontsize{5.5pt}{4pt}\selectfont \tableau{
1&1&2&3\\2\\4 \\ \bl \fr[l] \\ \bl \fr[l]  } }
\xrightarrow{\kat}
{\fontsize{5.5pt}{4pt}\selectfont \tableau{
1 \\ 3\\ \bl \fr[l] \\ \bl \fr[l] \\ 1&2 } }
\! \xrightarrow{P_{\zs_2\zs_3\zs_4}}
{\fontsize{5.5pt}{4pt}\selectfont \tableau{
1 \\1&2&3 \\ \bl \fr[l]  \\\bl \fr[l] \\ \bl \fr[l] } }
\, \text{ not katabolizable}
\\[3.2mm]
 S'_{23}(&T)  = \spa
{\fontsize{5.5pt}{4pt}\selectfont \tableau{
1&1&3&4\\ \bl \fr[l] \\2&2\\ 3\\5} }
\xrightarrow{P_{\zs_3\zs_4\zs_3}}
{\fontsize{5.5pt}{4pt}\selectfont \tableau{
1&1&3&4\\ \bl \fr[l] \\2&2\\ 3\\5} }
\xrightarrow{\kat}
{\fontsize{5.5pt}{4pt}\selectfont \tableau{
\bl \fr[l] \\1&1\\ 2\\4 \\ 2&3} }
\xrightarrow{P_{\zs_2\zs_3\zs_4}}
{\fontsize{5.5pt}{4pt}\selectfont \tableau{
\bl \fr[l]\\1&1&4\\2&2 \\ 3 \\  \bl \fr[l] } }
\ \text{ not katabolizable}
\\[3.2mm]
 S'_{24}(&T)  =  \spa
{\fontsize{5.5pt}{4pt}\selectfont \tableau{
1&1&3&4\\ 2 \\ \bl \fr[l] \\2&3\\ 5} }
\xrightarrow{P_{\zs_3\zs_4\zs_3}}
{\fontsize{5.5pt}{4pt}\selectfont \tableau{
1&1&3&4\\ 2  \\2&3\\ 5 \\ \bl \fr[l]} }
\xrightarrow{\kat}
{\fontsize{5.5pt}{4pt}\selectfont \tableau{
1  \\1&2\\ 4 \\ \bl \fr[l] \\ 2&3} }
\xrightarrow{P_{\zs_2\zs_3\zs_4}}
{\fontsize{5.5pt}{4pt}\selectfont \tableau{
1 \\1&2&4 \\ 2&3 \\ \bl \fr[l]\\ \bl \fr[l] } }
 \ \text{ not katabolizable}
\end{align*}
\vspace{-2mm}
\end{example}
%[  \[{\fontsize{5.5pt}{4pt}\selectfont \tableau{
%1&1\\2&2\\3&3\\4\\5\\} }\quad\],  \[{\fontsize{5.5pt}{4pt}\selectfont \tableau{
%1&1&3\\2&2\\3\\4\\5\\} }\quad{\fontsize{5.5pt}{4pt}\selectfont \tableau{
%1&1\\2&2\\3&3&5\\4\\} }\quad\],  \[{\fontsize{5.5pt}{4pt}\selectfont \tableau{
%1&1&3\\2&2&5\\3\\4\\} }\quad{\fontsize{5.5pt}{4pt}\selectfont \tableau{
%1&1&3\\2&2\\3&5\\4\\} }\quad{\fontsize{5.5pt}{4pt}\selectfont \tableau{
%1&1\\2&2&4\\3&3&5\\} }\quad{\fontsize{5.5pt}{4pt}\selectfont \tableau{
%1&1&3&5\\2&2\\3\\4\\} }\quad\],  \[{\fontsize{5.5pt}{4pt}\selectfont \tableau{
%1&1&3&4\\2&2\\3\\5\\} }\quad{\fontsize{5.5pt}{4pt}\selectfont \tableau{
%1&1&3&4\\2&2&5\\3\\} }\quad{\fontsize{5.5pt}{4pt}\selectfont \tableau{
%1&1&3&4\\2&2\\3&5\\} }\quad{\fontsize{5.5pt}{4pt}\selectfont \tableau{
%1&1&3\\2&2&4\\3&5\\} }\quad\],  \[{\fontsize{5.5pt}{4pt}\selectfont \tableau{
%1&1&3&3\\2&2&4\\5\\} }\quad{\fontsize{5.5pt}{4pt}\selectfont \tableau{
%1&1&3&4&5\\2&2\\3\\} }\quad\],  \[{\fontsize{5.5pt}{4pt}\selectfont \tableau{
%1&1&3&3&5\\2&2&4\\} }\quad\]

\section{Consequences for $t=0$ nonsymmetric Macdonald polynomials}
\label{s consequences for nonsymmetric macdonald polynomials}

We show that the  $t=0$ specialized nonsymmetric Macdonald polynomials
are characters of AGD crystals and equal to certain nonsymmetric Catalan functions.
We thus obtain a key positive formula for these polynomials as a special case
of Corollary \ref{c AGD key pos}.

The Knop-Sahi recurrence \cite{Knopkostka, Sahiinterpolation}
determines the nonsymmetric Macdonald polynomials
 $E_\alpha(\mathbf{x}; q, t) = E_\alpha(x_1, \dots, x_\ell; q, t)$
for all weak compositions $\alpha \in \ZZpl$.
At $t=0$, the recurrence becomes
\begin{align}
E_{(0, \dots, 0)}(\mathbf{x};q,0) &= 1, \label{e E recursion 1} \\
E_{\zs_i \alpha}(\mathbf{x}; q,0) &= \pi_i (E_\alpha(\mathbf{x};q,0) ),
      \label{e E recursion 3} \\
\hspace{-.1cm} E_{(\alpha_\ell+1,\alpha_1, \dots, \alpha_{\ell-1})}(\mathbf{x};q,0) &= q^{\alpha_\ell} x_1 E_\alpha(x_2, \dots, x_\ell, x_1/q; q,0),  \hphantom{aaaaaaaa}\label{e E recursion 2}
\end{align}
which determines the specializations $E_\alpha(\mathbf{x}; q, 0)$.
We have adopted the notation of \cite[Equations (40)--(42)]{H}, except that
in \eqref{e E recursion 3}  we have used the action of  $\zH_\ell$ on  $\ZZ^\ell$ from \eqref{e Hl action}
to put what are often two equations into one.
%a nonstandard but useful convention.
% so that no condition on  $\alpha$ is required.
% but makes the formulas slightly cleaner.
%For connecting to nonsymmetric Catalan functions,
For this paper,
it is more convenient to work with a renormalization of the $E_\alpha(\mathbf{x}; q, 0)$, denoted $\tE_\alpha = \tE_\alpha(\mathbf{x}; q)$ and defined by
\begin{align}
\tE_{(0, \dots, 0)} &= 1, \label{e tE recursion 1}\\
\tE_{\zs_i \alpha} &= \pi_i(\tE_\alpha),  \   \label{e tE recursion 3} \\
\tE_{(\alpha_\ell+1,\alpha_1, \dots, \alpha_{\ell-1})} &= x_1 \tE_\alpha(x_2, \dots, x_\ell, qx_1) = x_1 \Phi(\tE_\alpha). \label{e tE recursion 2}
\end{align}
The two versions are related by  $E_\alpha(\mathbf{x};q,0) =q^{\sum_i {\alpha_i \choose 2}} \tE_\alpha(\mathbf{x};q^{-1})$;
note that the exponent of  $q$ here is also $n(\eta) = \sum_{i} (i-1)\eta_i$ for $\eta = (\alpha^+)'$ the conjugate partition of  $\alpha^+$
(recall that $\alpha^+$ denotes the weakly decreasing rearrangement of $\alpha$).
In addition, our notation  $E_\alpha(\mathbf{x};q,t)$ agrees with that of \cite{AssafGonzalez, HHLnonsymmetric, H}, while
the version used by Sanderson
 %\cite{IonNSMacdonald} by
\cite{Sanderson}, call it  $E^S_{\alpha}$, is related by $E^{S}_\alpha(x_1,\dots, x_\ell;q,t)
= E_{(\alpha_\ell, \dots, \alpha_1)}(x_\ell,\dots, x_1; q,t)$.

We suggest that on a first reading of this section,  the reader focus on the case $|\alpha| = \ell$,
as it captures the main ideas but with fewer technical details.

\subsection{Sanderson's theorem and key positivity}
\label{ss Sanderson theorem}
%this paragraph good finally!

%We will need some (presumably well-known) facts about the 0-Hecke monoid of $\widetilde{S}_\ell$ (which already hold in the extended affine braid group).
Recall from \S\ref{ss the affine symmetric group and 0Hecke monoid} that $\widetilde{\SS}_\ell \subset GL(\h^*)$.
Set $y_1 = \tau s_{\ell-1} \cdots s_1$ and $y_i = s_{i-1}\cdots s_1 y_1 s_1\cdots s_{i-1}$ for  $i =2,\dots, \ell$.
These elements commute pairwise and satisfy only one additional relation $y_1\cdots y_\ell = \idelm$;
hence they
generate a subgroup of
translations $T$, with $T \xrightarrow{\cong} \ZZ^{\ell}/{\ZZ(1,\dots, 1)} \xrightarrow{\cong}
\bigoplus_{i \in I} \ZZ \varpi_i$
via $y_i \mapsto \epsilon_i \mapsto \varpi_i - \varpi_{i-1}$.
Write  $\mathbf{y}^\lambda\in T$ for the element mapping to  $\lambda \in \bigoplus_{i \in I} \ZZ \varpi_i$,
so that  $\mathbf{y}^{\varpi_i} = y_1\cdots y_i$.
%One can check  $y_1\cdots y_\ell = \idelm$ and thus
%We identify
%$\ZZ^{\ell}/{\ZZ(1,\dots, 1)} \cong
%\bigoplus_{i \in I} \ZZ \varpi_i$,
%$\epsilon_1 + \cdots + \epsilon_i \mapsfrom \varpi_i$ %, an  $\SS_\ell$ equivariant isomorphism;
%and accordingly write  $T = \{\mathbf{y}^\lambda \mid \lambda \in \bigoplus_{i \in I} \ZZ \varpi_i\}$ with
%%$\mathbf{y}^{\varpi_i} = y_1 \cdots y_i$
%$\mathbf{y}^{\sum_i c_i\varpi_i} := y_1^{c_1+\cdots +c_\ell} y_2^{c_2+\cdots + c_\ell} \cdots y_\ell^{c_\ell}$.
These satisfy $\mathbf{y}^\lambda \mathbf{y}^\mu = \mathbf{y}^{\lambda+\mu}$
and $w \mathbf{y}^\lambda w^{-1} = \mathbf{y}^{w(\lambda)}$ for
$\lambda, \mu \in \bigoplus_{i \in I} \ZZ \varpi_i$ and
 $w \in \SS_\ell$.
Hence $\widetilde{\SS}_\ell = \SS_\ell \ltimes T$.
One can check that  $\mathbf{y}^\lambda \in GL(\h^*)$
is the same
%$T = \{t_\lambda \mid \lambda \in  \widetilde{M}\}$ with $\widetilde{M} := \bigoplus_{i \in I} \ZZ \varpi_i$ and
as the $t_\lambda$ defined in \cite[Equation 6.5.2]{KacBook},
which is one way to verify the above facts about the  $y_i$ and $T$ and matches our notation with \cite{KacBook,NaoiVariousLevels}.

\begin{lemma}
\label{l si commutes}
View a translation  $\mathbf{y}^\lambda \in \widetilde{S}_\ell$ as an element of the 0-Hecke monoid $\zaH_\ell$ by taking any reduced
expression for it.
Let  $d, d' \in [\ell]$. The following
hold in $\zaH_\ell$:
\begin{list}{\emph{(\roman{ctr})}} {\usecounter{ctr} \setlength{\itemsep}{1pt} \setlength{\topsep}{2pt}}
\item $\zs_i$ commutes with $\tau \sfc(d)$ for $d < i \le \ell-1$.
\item $(\tau \sfc(d))^d$ is a reduced expression for $\mathbf{y}^{\varpi_d}$ in $\widetilde{\SS}_\ell$, and thus $\mathbf{y}^{\varpi_d} = (\tau \sfc(d))^d$ in  $\zaH_\ell$.
\item For weights $\lambda, \mu \in \sum_{i \in I} \ZZ_{\ge 0} \varpi_i$,  $\mathbf{y}^\lambda \mathbf{y}^\mu = \mathbf{y}^{\lambda+\mu}= \mathbf{y}^{\mu} \mathbf{y}^\lambda$.
\item $(\tau \sfc(d))^d$ commutes with $(\tau \sfc(d'))^{d'}$.
\end{list}
\end{lemma}
\begin{proof}
%In fact all these facts already hold in the extended affine braid group of type $A_{\ell-1}^{(1)}$.
For (i), we compute using the relations \eqref{e extended affine Weyl far commutation}--\eqref{e extended affine Weyl tau}:
\begin{align*}
\zs_i (\tau \zs_{\ell-1}\cdots \zs_{d})
&= \tau \zs_{i-1}\zs_{\ell-1} \cdots \zs_d
= \tau \zs_{\ell-1} \cdots \zs_{i+1}\zs_{i-1}\zs_{i}\zs_{i-1}\zs_{i-2} \cdots \zs_d \\
&= \tau \zs_{\ell-1} \cdots \zs_{i+1}\zs_{i}\zs_{i-1}\zs_i \zs_{i-2} \cdots \zs_d
=  (\tau \zs_{\ell-1} \cdots \zs_d) \zs_i.
\end{align*}
Statement (ii) can be proved using the description of $\widetilde{S}_\ell$ as certain permutations of $\ZZ$; see, for instance, (15), (18), (19), and
Proposition 4.1 of \cite{B3}.
%??check these two papers and this one have compatible conventions. in paritcular why is \mathbf{y}^\varpi_d = y_1...y_d?  yes, checked
For $\lambda, \mu \in \sum_{i \in I} \ZZ_{\ge 0} \varpi_i$,
$\length(\mathbf{y}^\lambda) + \length(\mathbf{y}^\mu) = \length(\mathbf{y}^{\lambda+\mu})$ (see, e.g., \cite[\S4.1]{H}),
which gives (iii).  Statement (iv) is immediate from (ii) and (iii).
%$(\tau \sfc(d))^d (\tau \sfc(d'))^{d'} = \mathbf{y}^{\varpi_d+\varpi_{d'}} = (\tau \sfc(d'))^{d'} (\tau \sfc(d))^{d}$ holds in the extended affine braid group and
%therefore in the 0-Hecke monoid as well.
\end{proof}

We will need the observation that affine generalized Demazure crystals $\AGD(\mu;\mathbf{w})$ for constant  $\mu$ are just affine Demazure crystals.
\begin{proposition}
\label{p lambda constant crystal}
Suppose $\mu = (a^m, 0^{p-m})$ for $a \in \ZZ_{> 0}$
%(\mu_1, \dots, \mu_p)$ with $\mu_1 = \mu_2 = \cdots = \mu_m$ and $\mu_{m+1} = \cdots = \mu_p = 0$.
and $\mathbf{w} = (w_1, w_2, \dots, w_p) \in (\zH_\ell)^p$.
%Set $\mu^i = \mu_i-\mu_{i+1}$, with  $\mu_{m+1}:=0$.
Then $\AGD(\mu;\mathbf{w})= \F_{w_1 \tau w_{2} \cdots \tau w_m} \{ u_{a \Lambda_1 } \}
= $ $ B_{w_1\tau w_2 \cdots \tau w_m}(a \Lambda_1)
\subset B(a \Lambda_m)$.
Further, $B_{w_1\tau w_2 \cdots \tau w_m}(a \Lambda_1) = B_{w_1\tau w_2 \cdots \tau w_m\tau w}(a \Lambda_0)$ for any  $w \in \zH_\ell$.
\end{proposition}
\begin{proof}
As $\mu^i = 0$ for  $i\ne m$,
 %and  $\mu^m = a$,
the first statement is immediate from the definition of  $\AGD(\mu;\mathbf{w})$ in \eqref{e AGD def}.
The second follows from the fact that  $\cf_i \spa u_{a \Lambda_0} = 0$ for  $i \in [\ell-1]$.
\end{proof}

Recall from \eqref{e def x character} the definition of the $\mathbf{x}$-character of a crystal.
%Equivalently,
%\begin{align}
%\chr (B_{v}(p\Lambda_1))  = \zeta(\tE_\alpha) e^{\Lambda_0} e^{\delta \frac{p(p-\ell)}{2\ell}}
%\end{align}
The next result is partially
a restatement of Sanderson's theorem \cite{Sanderson}
(specifically, $\tE_\alpha = q^{\frac{p(p-\ell)}{2\ell}}\chr_{\mathbf{x};\mu}( B_{v}(\Lambda_0) )$).
However, we now have the advantage of seeing it as part of the more general Theorem~\ref{t character AGD}
and  can make it combinatorially explicit in a way which encompasses earlier work of Lascoux \cite{La} and
Shimozono-Weyman \cite{SW}
on cocharge Kostka-Foulkes polynomials.
%Lascoux and SW combinatorics (really talking about this and the later combinatorial corollary).

%there are some minor changes in conventions
%Also of note notably, our description of  $v$ is more explicit.

\begin{theorem}
%[\cite{Sanderson}]
\label{t Sanderson}
The  $t=0$ nonsymmetric Macdonald polynomials are $\mathbf{x}$-characters of
affine Demazure crystals:
let $\alpha \in \ZZpl$ and $\eta = (\eta_1,\dots, \eta_k) =  (\alpha^+)'$ be the conjugate of  $\alpha^+$.
Let $\mathsf{z} \in \zH_\ell$ be any element satisfying $\mathsf{z} \spa \alpha^+ = \alpha$.
Set  $p = |\alpha|$ and $\mu = 1^p$.
Then
\begin{align}
\label{et Sanderson}
\tE_\alpha(\mathbf{x}; q)
&= \pi_{\mathsf{z}}  \spa (x_1 \Phi \spa \pi_{\sfc(\eta_k)})^{\eta_k} \cdots (x_1\Phi \spa \pi_{\sfc(\eta_1)})^{\eta_1} \cdot  1 \\
%= q^{-n_\ell(\mu)}\chr_{\mathbf{x};\mu}(\AGD(1^p; \mathbf{w}) )
%= q^{\frac{p(p-\ell)}{2\ell}}\chr_{\mathbf{x};\mu}(\AGD(\mu; \mathbf{w}) )
\label{et Sanderson 2}
&= q^{\frac{p(p-\ell)}{2\ell}}\chr_{\mathbf{x};\mu}( B_{v}(\Lambda_0) )
= \sum_{b \in \B^{\mu;\mathbf{w}}} q^{\charge(\inv(b))} \mathbf{x}^{\content(b)},
%\char (\F_{v}u_{\ell\Lambda_0}),
\end{align}
where
$v = \mathsf{z} \spa (\tau \sfc(\eta_k))^{\eta_k} (\tau \sfc(\eta_{k-1}))^{\eta_{k-1}} \cdots  (\tau \sfc(\eta_1))^{\eta_1} =
\mathsf{z} \spa \mathbf{y}^{\varpi_{\eta_1} + \spa \cdots \spa + \varpi_{\eta_k}}  \in  \spa \zaH_\ell$, and
 \[\mathbf{w} = \big( \mathsf{z}, \underbrace{\sfc(\eta_k), \ldots, \sfc(\eta_k)}_{\eta_k\text{ times}},
\dots,  \underbrace{\sfc(\eta_2), \ldots, \sfc(\eta_2)}_{\eta_2\text{ times}}, \underbrace{\sfc(\eta_1),\ldots, \sfc(\eta_1)}_{\eta_1-1\text{ times}} \big).\]
\end{theorem}

\begin{proof}
First,
%we can write  $\tE_\alpha$ in terms of  $\tE_{\alpha^+}$
using \eqref{e tE recursion 3}, we obtain
$\tE_\alpha = \pi_{\mathsf{z}} \tE_{\alpha^+}$.
Let  $\beta = \alpha^+- \epsilon_{\eta_k}$ be the result of subtracting 1 from the rightmost occurrence of the largest part of
% $\eta_1$-st smallest nonzero part of
$\alpha^+$;
by \eqref{e tE recursion 3}--\eqref{e tE recursion 2},  $\tE_{\alpha^+} = x_1 \Phi \pi_{\sfc(d)} (\tE_{\beta})$ for any  $d$ such that  $\beta_d = (\alpha^+)_1-1$, which is equivalent to
 $\eta_k \le d \le \eta_{k-1}$.
The same argument %(with the ``new''  $\eta$ being  $\beta'$)
shows that $\tE_{\beta} = x_1 \Phi \pi_{\sfc(d')} (\tE_{\beta-\epsilon_{\eta_k-1}})$ for  any
$\eta_k -1\le d' \le \eta_{k-1}$.
Repeating this $\eta_k - 2$ more times, we obtain
%it follows that
$\tE_{\alpha^+} = (x_1 \Phi \pi_{\sfc(\eta_k)})^{\eta_k} (\tE_{\alpha^+ - (\epsilon_1 + \cdots + \epsilon_{\eta_k})})$.
%??useful: this last subscipt is also (\eta_1,\dots, \eta_{k-1})' padded with zeros to make length \ell
%we can repeat this to obtain  this  $\eta_{k}-1$ more times,
%Note that $\beta$ is weakly decreasing.
Continuing in this way we obtain \eqref{et Sanderson}.
%until we get  $\tE_\mathbf{0} = 1$ yields

By Theorem \ref{t character AGD}, the right side of \eqref{et Sanderson} is equal to
$\sum_{b \in \B^{\mu;\mathbf{w}}} q^{\charge(\inv(b))} \mathbf{x}^{\content(b)} = $
$q^{\frac{p(p-\ell)}{2\ell}}\chr_{\mathbf{x};\mu}(\AGD(\mu; \mathbf{w}) )$,
and this is also equal to
$q^{\frac{p(p-\ell)}{2\ell}}\chr_{\mathbf{x};\mu}( B_{v}(\Lambda_0) )$
by Proposition~\ref{p lambda constant crystal}.
Finally, the two descriptions of  $v$ are equal by Lemma \ref{l si commutes} (ii)--(iii).
\end{proof}

Combining \eqref{et Sanderson} with Corollary \ref{c AGD key pos} we obtain
\begin{corollary}
\label{c t0 nsMac key pos}
Maintain the notation of Theorem \ref{t Sanderson}.
The  $t=0$ nonsymmetric Macdonald polynomials
are key positive with key expansion given by
\begin{align}
\label{ec t0 nsMac key pos}
q^{\sum_i {\alpha_i \choose 2}} E_\alpha(\mathbf{x};q^{-1},0)
= \tE_\alpha(\mathbf{x};q) =
 \!\!   \sum_{\substack{\text{$T \in \RowFrank_\ell(\mu)$} \\ \text{$T$ is extreme $\mathbf{w}$-katabolizable} }} \!\!\! q^{\charge(T)}\kappa_{\sh(T)}(\mathbf{x}).
\end{align}
\end{corollary}

\begin{example}
\label{ex ns mac 0302}
We illustrate Corollary \ref{c t0 nsMac key pos} for
$\ell=4, \alpha = 0302$.
We have $\alpha^+ = 3200$,  $\eta = (\alpha^+)'= 221$, $\mu = 1^5$, and $\mathbf{w} = (\zs_1\zs_3\zs_2, \zs_3\zs_2\zs_1, \zs_3\zs_2, \zs_3\zs_2, \zs_3\zs_2)$.
%(we have chosen  $\mathsf{z}= $ of minimal length).
%
%all extremal elements (I think)
%[
%[ 135 empty empty 24 , 135 empty 24 empty , 135 24 empty empty , empty 135 empty
%24 , empty 135 24 empty , 13 245 empty empty , 145 2 3 empty , 1 245 empty 3 ,
%145 2 empty 3 , 145 empty 2 3 , 1 245 3 empty , empty 145 2 3 , 3 14 empty 25 ,
%3 14 25 empty , 13 24 empty 5 , 13 4 empty 25 , 13 empty 24 5 , empty 13 4 25 ,
%13 24 5 empty , 13 4 25 empty , 13 empty 4 25 , empty 13 24 5 , 1 24 empty 35 ,
%1 24 35 empty , 14 25 empty 3 , 1 empty 24 35 , 14 2 empty 35 , 14 empty 25 3 ,
%empty 14 2 35 , 14 25 3 empty , 14 2 35 empty , 14 empty 2 35 , empty 14 25 3 ,
%empty 1 24 35 , 1 24 3 5  , 14 2 3 5  , 1 25 3 4  , 15 2 3 4  , 1 2 35 4  , 1 2
%3 45   ]
%],
%lwtabs,
%[
%[ 1 2 3 45  , empty 1 24 35 , 1 24 3 5  , empty 13 4 25 , empty 145 2 3 , empty
%135 empty 24  ]
%],
\[\begin{array}{cccccccccccc}
\substack{\text{$T \in \RowFrank_\ell(\mu)$} \\[.4mm] \text{$T$ is extreme $\mathbf{w}$-katabolizable} }  & {\fontsize{6pt}{5pt}\selectfont\tableau{1\\2\\3\\4&5  } }&&
{\fontsize{6pt}{5pt}\selectfont \tableau{\bl \fr[l] \\1\\2&4\\3&5 } }&&
{\fontsize{6pt}{5pt}\selectfont \tableau{1\\2&4\\3\\5 } }&&
{\fontsize{6pt}{5pt}\selectfont \tableau{\bl \fr[l] \\1&3\\4\\2&5 } }&&
{\fontsize{6pt}{5pt}\selectfont \tableau{\bl \fr[l] \\1&4&5\\2\\3} }&&
{\fontsize{6pt}{5pt}\selectfont \tableau{\bl \fr[l] \\1&3&5\\ \bl \fr[l] \\ 2&4 } }
\vspace{1.9mm}
\\
\hspace{15mm}\tE_{0302} \ \ = & q \kappa_{1112} &\hspace{-2mm}+\hspace{-3mm}& q^2\kappa_{0122} &\hspace{-2mm}+\hspace{-3mm}& q^2\kappa_{1211} &\hspace{-2mm}+\hspace{-3mm}& q^3\kappa_{0212} &\hspace{-2mm}+\hspace{-3mm}& q^3\kappa_{0311} &\hspace{-2mm}+\hspace{-3mm}& q^4\kappa_{0302}
\end{array}\]
Here are the corresponding inverses of highest weight elements obtained by computing $P(T)$ of these tabloids.
This is also the subset of tableaux in the Lascoux/Shimozono-Weyman formula for $\omega\tilde{H}_{\eta}$ with at most  $\ell$ rows
(see Theorems \ref{t SW and Lascoux Koskta} and \ref{t t0 nsMac symmetric schur pos B}).
Taking $\sum_U q^{\charge(U)} s_{\sh(U)}$ over these
tableaux gives the symmetrization of  $\tE_{0302}$.
\[\hspace{6mm} \begin{array}{ccccccccccccccccccc}
\substack{U \in \SYT_\ell^{|\alpha|} \\[.3mm] \text{$U$ is $\nr(\Delta(\eta))$-katabolizable} }
& {\fontsize{6pt}{5pt}\selectfont\tableau{1&5\\2\\3\\4  } }&&
{\fontsize{6pt}{5pt}\selectfont \tableau{1&4\\2&5\\3\\ \bl} }&&
{\fontsize{6pt}{5pt}\selectfont \tableau{1&4\\2\\3\\5} }&&
{\fontsize{6pt}{5pt}\selectfont \tableau{1&3\\2&4\\5\\ \bl} }&&
{\fontsize{6pt}{5pt}\selectfont \tableau{1&4&5\\2\\3\\ \bl} }&&
{\fontsize{6pt}{5pt}\selectfont \tableau{1&3&5\\2&4\\ \bl \\ \bl } }
\vspace{1.2mm}
\\
\pi_{\mathsf{w}_0}\tE_{0302} = \tE_{0023} =  &q s_{2111} &\hspace{-1.5mm}+\hspace{-2.5mm}& q^2s_{221} &\hspace{-1.5mm}+\hspace{-2.5mm}& q^2s_{2111} &\hspace{-1.5mm}+\hspace{-2.5mm}& q^3s_{221} &\hspace{-1.5mm}+\hspace{-2.5mm}& q^3s_{311} &\hspace{-1.5mm}+\hspace{-2.5mm}& q^4s_{32} & \hphantom{aaaaaaaaaaaaaa}
\end{array}\]
%The set of $\nr(\Delta(\eta))-$katabolizable tableaux consist of these tableaux together with
%${\fontsize{6pt}{5pt}\selectfont\tableau{1\\2\\3\\4\\5} }$.
\end{example}

\begin{remark}
\label{r compare AssafGonzalez}
Another key positive formula for the $t=0$ nonsymmetric Macdonald polynomials was given by Assaf and Gonzalez in
\cite{AssafGonzalez, AssafGonzalez2}.
Their approach also uses crystals
but their indexing combinatorial objects are rather different---compare Example \ref{ex ns mac 0302} with \cite[Figure 31]{AssafGonzalez}.
An interesting problem is to find an explicit bijection between the two objects.
\end{remark}

Let us now explain how Corollary \ref{c t0 nsMac key pos} is a nonsymmetric generalization of Lascoux's
formula for the cocharge Kostka-Foulkes polynomials  $\tilde{K}_{\lambda \mu}(q)$.
For partition  $\mu$,
let $\tilde{H}_\mu(x_1, x_2,\dots;q)$ $= \sum_\lambda \tilde{K}_{\lambda \mu}(q)s_\lambda$
be the \emph{cocharge variant modified Hall-Littlewood polynomial};
it equals $q^{n(\mu)}Q'_\mu(x_1,x_2,\dots ;q^{-1})$ in the notation of \cite[p. 234]{Macbook},
%let $H_{\mu}(x_1,x_2, \dots;q)$ be the modified Hall-Littlewood polynomial
%(denoted $Q'_\mu$ in \cite[p. 234]{Macbook});
%%it is equal to the Catalan function $H(\Delta^+; \mu; \mathsf{w}_0)$ and
specializes to the homogeneous symmetric function $h_\mu$ at $q = 1$, and the coefficient of  $s_\mu$ is $q^{n(\mu)}$.
Lascoux \cite{La} gave a formula
for $\tilde{H}_\mu$ in terms of a function $\kattype$ from standard tableaux to partitions (see \cite[\S4.1]{SW});
this version of katabolism
was shown \cite[\S4]{SW} to agree
with a special case of the Shimozono-Weyman version
despite its somewhat different looking definition.
%but  to agree with the a special case of their version.
We assemble these results as follows:
 %^the notion of therein \cite[\S4]{SW}.
  %\in \SSYT(1^\ell)$$\mu= 1^\ell$.

% a certain case of \cite{SW}produce
%the same set of tableaux
%to
%%However, these combinatorial notions and accompanying symmetric functions are related as follows.
%However, they are related as follows:
%%However, it is shown in \cite[\S4.3--4.4]{SW}they are related as follows:

\begin{theorem}[{Shimozono-Weyman \cite[\S4]{SW}, Lascoux \cite{La}}]
\label{t SW and Lascoux Koskta}
Let  $\eta$ be a partition of~$\ell$. Recall that  $\Delta(\eta)$ is the parabolic root ideal with blocks given by  $\eta$ (see \eqref{e d parabolic root ideal}). Then
\begin{align}
\label{et SW and Lascoux Koskta}
 \omega\tilde{H}_{\eta} =   \sum_{U} q^{\charge(U)}s_{\sh(U)}(\mathbf{x}) = H(\Delta(\eta); 1^\ell ; \mathsf{w}_0)(\mathbf{x}; q),
\end{align}
where
the sum is over the set of standard tableaux  $U$ which are $R(\eta, 1^\ell)$-katabolizable in the sense of \cite[\S3.7]{SW}, and, moreover,
this set equals $\{U  \mid \kattype(U^t) \gd \eta\}$, where $U^t$ denotes the transpose of $U$, and $\gd$ denotes
dominance order on partitions.
Here, $\mathbf{x} = (x_1,\dots, x_\ell)$ and $\omega$ denotes the $\ZZ[q]$-algebra
homomorphism from symmetric functions in $x_1, x_2, \dots$
to $\ZZ[q][\mathbf{x}]^{\SS_\ell}$ which satisfies $\omega(s_{\lambda}(x_1,x_2,\dots)) = s_{\lambda'}(\mathbf{x})$.
\end{theorem}

We have not defined either notion of katabolism appearing here, but by Proposition~\ref{p recover SW def},
$R(\eta, 1^\ell)$-katabolizability agrees with
%the  $U$ appearing in are the set of $U \in \SSYT_\ell(1^\ell)$
$\mathbf{n}(\Delta(\eta))$-katabolizability of standard tableaux.
%a standard tableau  $U$ is $R(\eta, 1^\ell)$-katabolizable if and only if it
%%the  $U$ appearing in are the set of $U \in \SSYT_\ell(1^\ell)$
% is $\mathbf{n}(\Delta(\eta))$-katabolizable.
%%$U \in \SSYT_\ell(1^\ell)$.

For $\ell \ge |\eta|$, define $\Delta_\ell(\eta) \subset \Delta^+_\ell$ to be the root ideal
$\Delta(\eta) \sqcup \{(i, j) \in \Delta^+_\ell \mid j > |\eta|\}$
(this is just a convenient way to extend $\Delta(\eta)$ to length  $\ell$---see Proposition \ref{p poly part} (v)).
Our key positive formula \eqref{ec t0 nsMac key pos} ``symmetrizes'' to the Lascoux/Shimozono-Weyman formula \eqref{et SW and Lascoux Koskta} in the following sense:

\begin{theorem}
\label{t t0 nsMac symmetric schur pos}
Maintain the notation of Theorem \ref{t Sanderson};
also set $m = p = |\alpha|$  and   assume  $m \le \ell$.
Let  $\SYT^m = \SSYT_\ell(\mu)$, the standard Young tableaux with  $m$ boxes.
Then
% symmetrization of the  $t=0$ nonsymmetric Macdonald polynomial  $\tE_{\alpha}$ is given by
\begin{align}
\label{et t0 nsMac symmetric schur pos}
\pi_{\mathsf{w}_0}\tE_{\alpha}(\mathbf{x};q) = H(\Delta_\ell(\eta); \mu; \mathsf{w}_0)(\mathbf{x};q) =
%& q^{-n_\ell(\mu)}\chr_{\mathbf{x};\mu}(\AGD(\mu;\mathbf{w}))
 \!\! \sum_{\substack{U \in \SYT^m\\ \text{$U$ is $\nr(\Delta_\ell(\eta))$-katabolizable} }} \!\!\! q^{\charge(U)}s_{\sh(U)}(\mathbf{x}).
\end{align}
%Let $\tilde{\mathbf{w}} \in (\zH_\ell)^p$ be obtained from  $\mathbf{w}$ by replacing its first entry with  $\mathsf{w}_0$.
Moreover, $\F_{\mathsf{w}_0} \B^{\mu; \mathbf{w}} = \B^{\mu; (\mathsf{w}_0, \ns(\Delta_\ell(\eta)))}$ as $U_q(\gl_\ell)$-crystals, and so
\begin{align}
\label{et t0 nsMac symmetric schur pos 2}
&\{P(T) \mid \text{ $T \in \RowFrank_\ell(\mu)$ is extreme $\mathbf{w}$-katabolizable} \} =\\
&\{U \in \SYT^m \mid  \text{  $U$ is $\nr(\Delta_\ell(\eta))$-katabolizable} \}
=\{U \in \SYT^m \mid \kattype(U^t) \gd \eta\}.
\notag
\end{align}
%which is also equal to the set of standard tableaux  $U$ with  $p$ boxes such that  $\kattype(U^t) \gd \eta$, where
% $U^t$ is the transpose of $U$,  $\gd$ denotes dominance order  has $\kattype$.
\end{theorem}

The proof is given in \S\ref{ss Symmetrization to the Lascoux/Shimozono-Weyman formula}, along with a similar result for  $m > \ell$.
%which is more technical.

\subsection{Connection to nonsymmetric Catalan functions}
\label{ss nsmac and nonsymmetric Catalan functions}
For $\alpha \in \ZZpl$, define $\tsfp(\alpha) \in \zH_\ell$ to be the longest element
such that  $\tsfp(\alpha) \spa \alpha^+ = \alpha$.  This choice of  $\mathsf{z}$ in
Theorem \ref{t Sanderson} will be important below.
%This is the most convenient choice of  $z$

We now show that the  $\tE_\alpha$ can be realized as certain tame nonsymmetric Catalan functions.
Note that this is not immediate from Theorem \ref{t Sanderson} as
typically  $\eta_k+1$  does not lie in the right descent set of $\tsfp(\alpha)$
%(the longest choice for $z$)
and so \eqref{et Sanderson} does not match
a tame nonsymmetric Catalan function via Theorem \ref{t character rotation formula}.
However, there is a way to rewrite \eqref{et Sanderson} which does the job.
%in fact our computer explorations suggest this is the expression
%different choice of $\mathbf{w}$ which gives the same character and does correspond to a nonsymmetric Catalan function.
%Though we already have a nice key positive formula for the  $\tE_\alpha$ from
%Corollary \ref{c t0 nsMac key pos},
%)insight offered by this %However, it does not offer any new combinatorial insight into the $t=0$ nonsymmetric Macdonalds over
%
%This improves our understanding of nonsymmetric Catalan functions and nonsymmetric Macdonald polynomials in several ways as we discussed in \S\ref{ss intro nsmac} and \S\ref{ss intro kschur}, including
This adds to the list of interesting functions which are encompassed in the
nonsymmetric Catalan functions,
proves that the $\tE_\alpha$ are Euler characteristics of vector bundles on Schubert varieties
(see Theorem~\ref{t intro t0 nsMac key pos}),
and is important in the proofs of
Theorems \ref{t t0 nsMac symmetric schur pos} and \ref{t t0 nsMac symmetric schur pos B}.

\begin{definition}
For a partition $\eta = (\eta_1, \dots, \eta_k)$ of $m$, let $\Delta'(\eta) \subset \Delta^+_m$ be the root ideal determined by
\begin{align*}
\nr(\Delta'(\eta)) = & \big( (\eta_1)^{\eta_2},\eta_1, \eta_1-1, \eta_1-2, \dots, \eta_2+1,
(\eta_2)^{\eta_3},\eta_2, \eta_2-1, \eta_2-2, \dots, \eta_3+1,  \\
& \dots (\eta_{k-1})^{\eta_k},\eta_{k-1}, \eta_{k-1}-1, \dots, \eta_k+1, {\eta_k}, \eta_k-1, \dots, 2
\big),
\end{align*}
where $(\eta_1)^{\eta_2}$ indicates that $\eta_1$ appears in the list $\eta_2$ times, and similarly for $(\eta_2)^{\eta_3}$, etc.
Informally,  $\Delta'(\eta)$ is obtained from the parabolic root ideal
$\Delta(\eta)$ by removing trapezoids between consecutive blocks.
%the trapezoids of roots with height $\eta_i$ and base  $\eta_{i+1}$ between consecutive blocks.
%
%the This can be thought of as the result of which is contained in the parabolic root ideal $\Delta(\eta)$ with triangles filled in between each consecutive blocks.
For  $\ell \ge m = |\eta|$, we also define $\Delta'_\ell(\eta) \subset \Delta^+_\ell$ to be the root ideal
$\Delta'(\eta) \sqcup \{(i, j) \in \Delta^+_\ell \mid j > |\eta|\}$.
%by adding  $\ell-|\eta|$ full columns of roots.
\end{definition}

%\begin{remark}
%The root ideal $\Delta'_\ell(\eta)$ is just a convenient choice which extends  $\Delta'(\eta)$ to length  $\ell$.
%It follows from Proposition \ref{p poly part} that
%$H(\Delta'_\ell(\eta); 1^m0^{\ell-m}; w) = H(\Psi; 1^m0^{\ell-m} ; w)$ for any  $\Psi$ such that  $\Psi \cap \Delta^+_\ell = \Delta'(\eta)$.
%%Note that the roots in  $\Delta'_\ell(\eta)$ in columns  $> |\eta|$ are irrelevant by .
%\end{remark}

\begin{example}
For $\eta = 732$, we have depicted $\Delta'(\eta)$ in red,  $\Delta^+ \setminus \Delta(\eta)$ in light blue, and $\Delta(\eta) \setminus \Delta'(\eta)$ in blue
(the two trapezoidal regions).
%the latter consists of the two trapezoidal regions between consecutive blocks.
%the latter consists of the two trapezoidal regions between blocks.
\ytableausetup{mathmode, boxsize=.74em,centertableaux}
\[\Delta'(\eta) =
\Tiny
\begin{ytableau}
      \  &*(blue!10)&*(blue!10) &*(blue!10)&*(blue!10)&*(blue!10)&*(blue!10)&*(red!100)&*(red!100)&*(red!100)&*(red!100)&*(red!100)\\
        &       &*( blue!10)&*( blue!10)&*( blue!10)&*(blue!10)&*(blue!10)&*(blue!36)&*(red!100)&*(red!100)&*(red!100)&*(red!100)\\
        &        &        &*( blue!10)&*( blue!10)&*(blue!10)&*(blue!10)&*(blue!36)&*(blue!36)&*(red!100)&*(red!100)&*(red!100)\\
        &        &         &       &*( blue!10)&*(blue!10)&*(blue!10)&*(blue!36)&*(blue!36)&*(blue!36)&*(red!100)&*(red!100)\\
        &        &         &        &       &*(blue!10)&*(blue!10)&*(blue!36)&*(blue!36)&*(blue!36)&*(red!100)&*(red!100)\\
        &        &         &        &          &    &*( blue!10)&*( blue!36)&*(blue!36)&*(blue!36)&*(red!100)&*(red!100)\\
        &        &         &        &          &     &  &*( blue!36)&*(blue!36)&*(blue!36)&*(red!100)&*(red!100)\\
        &        &         &        &          &     &   &  &*(blue!10)&*(blue!10)&*(red!100)&*(red!100)\\
        &        &         &        &          &     &   &   &        &*(blue!10)&*(blue!36)&*(red!100)\\
        &        &         &        &          &     &   &   &   & &*( blue!36)&*(blue!36)\\
        &        &         &        &          &     &   &   &   &   & &*(blue!10)\\
        &        &         &        &          &     &   &   &   &   &   &
\end{ytableau}
\]
\end{example}

%the inverse of  $\tsfp(\alpha)$ in one-line notation is obtained
%by relabeling the largest letter of $\alpha$ with $1, 2, 3, \dots, k$ from right to left, the next largest letter with $k+1$, $k+2$, $\dots$, and so on.
%For example, for $\alpha = 331100000003$, $\tsfp(\alpha) = (3254CBA98761)^{-1} = C2143BA98765$.
%?useful the inverse here is an artifact of the bad convention that si acts on the left on words

\begin{theorem}
\label{t nonsymmetric Macdonald}
Let $\alpha \in \ZZpl$ and set $\eta = (\alpha^+)'$.
Set $m = |\alpha|$ and put
$\mu = 1^m 0^{\ell-m}$ if  $m \le \ell$ and  $\mu = 1^m$ otherwise.
 %be the conjugate of its partition rearrangement.
The $t=0$ nonsymmetric Macdonald polynomials agree with certain nonsymmetric Catalan functions:
\begin{align}
\label{et nonsymmetric Macdonald}
\tE_\alpha(x_1, \dots, x_\ell; q) =
\begin{cases}
H(\Delta'_\ell(\eta); \mu; \tsfp(\alpha))(x_1,\dots, x_\ell;q)  & \text{ if  $m \le \ell$} \\
H(\Delta'(\eta); \mu; \tsfp(\alpha, 0^{m-\ell}))(x_1,\dots, x_m;q)|_{x_{\ell+1} = \spa \cdots \spa = \spa x_m \spa = \spa 0}  & \text{ if  $m > \ell$}.
\end{cases}
\end{align}
\end{theorem}

\begin{proof}
First assume  $m \le \ell$.
The labeled root ideal $(\Delta'_\ell(\eta), \mu, \tsfp(\alpha))$ is tame because the  $\ell-\eta_1$ 0's in  $\alpha^+$ ensure that
$\eta_1+1, \dots, \ell-1$ is contained in the right descent set of  $\tsfp(\alpha)$.
Hence by \eqref{ec kat conjecture resolution 0} and then Proposition \ref{p lambda constant crystal},
\[
H(\Delta'_\ell(\eta); \mu;  \tsfp(\alpha)) = q^{-n_\ell(\mu)}\chr_{\mathbf{x};\mu}(\AGD(\mu; \mathbf{w})) = q^{-n_\ell(\mu)}\chr_{\mathbf{x};\mu}(B_w(\Lambda_0)),\]
where  $\mathbf{w} = (\tsfp(\alpha), \ns(\Delta'_\ell(\eta)))$,
$w = \tsfp(\alpha) \tau \sfc(n_1) \tau \sfc(n_2) \cdots \tau \sfc(n_{m-1}) \tau \sfc(n_m) \in \zaH_\ell$,
$(n_1, \dots, n_{\ell-1})$ $= \nr(\Delta'_\ell(\eta))$, and  $n_m := 1$
 (if $m<\ell$ this was already true otherwise we define it);
the $\sfc(n_m)$ here
%???this takes long time to explain ...tricky issue about whether worth defining \nr to have \ell parts with last one 1
(allowed
%could have been anything in $\zH_\ell$
by Proposition \ref{p lambda constant crystal}) makes $w$ end in $\tau \sfc(\eta_k)  \tau \sfc(\eta_k-1) \cdots  \tau \sfc(1)$,
enabling the parts of  $\eta$ to be handled uniformly below.
%It is more convenient to view $\AGD(\mu; \mathbf{w}) = \F_{w'}u_{\Lambda_1}= \F_{w}u_{\Lambda_0}$ with  $w = w'\tau$.
Thus by Theorem \ref{t Sanderson}, to prove the top case of \eqref{et nonsymmetric Macdonald}
it suffices to show  $B_w(\Lambda_0) = B_v(\Lambda_0)$, where
$v = \tsfp(\alpha)(\tau \sfc(\eta_k))^{\eta_k} (\tau \sfc(\eta_{k-1}))^{\eta_{k-1}} \cdots  (\tau \sfc(\eta_1))^{\eta_1}$
as in Theorem \ref{t Sanderson} with  $\mathsf{z} = \tsfp(\alpha)$.
Since $\F_{\mathsf{w}_0} \{ u_{\Lambda_0}\}  = \{u_{\Lambda_0}\}$, it is enough to show
$w \spa \mathsf{w}_0 = v \spa \mathsf{w}_0$
in the 0-Hecke monoid $\zaH_\ell$.

%very good shape
For an interval  $[i,j] \subset [m]$, set $w^{[i,j]} = \tau \sfc(n_i) \tau \sfc(n_{i+1}) \cdots \tau \sfc(n_{j})$,
so that  $w = \tsfp(\alpha) w^{[m]}$.
%($w_L$ is the contribution of the first parabolic block to $w$).
By definition of $\Delta'_\ell(\eta)$, $w^{[\eta_1]} = (\tau \sfc(\eta_1))^{\eta_2} \tau \sfc(\eta_1) \tau \sfc(\eta_1-1)\tau \sfc(\eta_1-2)\cdots \tau \sfc(\eta_2+1)$.
By Lemma \ref{l si commutes} (i),
 for  $j \in [m]$ and $n_j < i \le \ell-1$,
%$\zs_{i}$ commutes with $w^{[j,m]}$,; hence also
$\zs_{i} w^{[j,m]} \mathsf{w}_0 = w^{[j,m]} \zs_i \mathsf{w}_0 = w^{[j,m]} \mathsf{w}_0$. Hence
\begin{align*}
w^{[m]}\mathsf{w}_0 & = (\tau \sfc(\eta_1))^{\eta_2+1} \tau \sfc(\eta_1-1)\tau \sfc(\eta_1-2)\cdots \tau \sfc(\eta_2+1) w^{[\eta_1+1,m]} \mathsf{w}_0 \\
%& = (\tau \sfc(\eta_1))^{\eta_2+2}\tau \sfc(\eta_1-2)\cdots \tau \sfc(\eta_2+1) w^{[\eta_1+1,m]} \mathsf{w}_0 \\
& = (\tau \sfc(\eta_1))^{\eta_2+2} \tau \sfc(\eta_1-2) \cdots \tau \sfc(\eta_2+1) w^{[\eta_1+1,m]} \mathsf{w}_0
= \cdots =
(\tau \sfc(\eta_1))^{\eta_1} w^{[\eta_1+1, m]} \spa \mathsf{w}_0.
%
%& = (\tau \sfc(\eta_1))^{\eta_2+3} \tau \sfc(\eta_1-3) \cdots \tau \sfc(\eta_2+1) w^{[\eta_1+1,m]} \mathsf{w}_0
%= \cdots =
%(\tau \sfc(\eta_1))^{\eta_1} w^{[\eta_1+1, m]} \spa \mathsf{w}_0.
\end{align*}
%Visualizing using
In Example \ref{ex nsmac}, this amounts to removing the triangle
\ytableausetup{mathmode, boxsize=1.04em,centertableaux}
$\TTiny
\begin{ytableau}
*(red!80)5 \\
*(red!80) 4 &*(red!0)5 \\
*(red!80) 3& *(red!0) 4  &*(red!0)5
\end{ytableau}$
as the first step in going from the left to middle diagram.
%(We have effectively added a small triangle to the trapezoidal region between the first and second blocks
%to make a full triangle of nonroots in rows  $1,\dots, \eta_1$---see Example.)
%into a full larger triangle can be added with only affecting the right word by something in $\SS_\ell$.
Repeating this for the other parts of  $\eta$
% (the last part of  $\eta$ we need extra guy added by Proposition \ref{p lambda constant crystal},
%parabolic blocks (i.e., %parts of  $\eta$),
we obtain
\begin{align}
\label{e proof nonsymmetric}
w \spa \mathsf{w}_0 \spa = \spa \tsfp(\alpha) (\tau \sfc(\eta_1))^{\eta_1}  \cdots  (\tau \sfc(\eta_k))^{\eta_k} \mathsf{w}_0
\spa = \spa \tsfp(\alpha) (\tau \sfc(\eta_k))^{\eta_k} \cdots  (\tau \sfc(\eta_1))^{\eta_1} \mathsf{w}_0 = v \spa \mathsf{w}_0,
\end{align}
as desired. Here, we have used Lemma \ref{l si commutes} (iv) for the second equality.

Now to handle the case  $m > \ell$, we use the
following stability property of the $t=0$ nonsymmetric Macdonald polynomials
which is straightforward to verify from the Haglund-Haiman-Loehr formula \cite[Theorem 3.5.1]{HHLnonsymmetric}:
for any  $\beta\in \ZZpl$,
$E_{(\beta,0)}(x_1,\dots, x_{\ell+1};q,0)|_{x_{\ell+1}=0} $
$= E_\beta(x_1,\dots, x_{\ell}; q,0)$.
Thus we also have
$\tE_{(\beta,0)}(x_1,\dots, x_{\ell+1};q)|_{x_{\ell+1}=0}
= \tE_\beta(x_1,\dots, x_{\ell}; q)$.
Applying this to $\tE_{(\alpha, 0^{m-\ell})}(x_1, \dots, x_m; q) =
H(\Delta'(\eta); 1^m ; \tsfp(\alpha, 0^{m-\ell}))(x_1, \dots, x_m;q)$,
which holds by the top case of \eqref{et nonsymmetric Macdonald},
yields the bottom case of \eqref{et nonsymmetric Macdonald}.
\end{proof}

\begin{example}
\label{ex nsmac}
We assemble several expressions for $\tE_\alpha$ for $\alpha = 3221110000$ ($\ell = 10$).
We have $\eta = (\alpha^+)' = 631$ and $\Delta'(\eta)$ is indicated by the red squares in the left diagram below.
Let $\mu = 1^\ell$ and note that $\tsfp(\alpha) = \zs_2 \spa \zs_4 \zs_5 \zs_4 \spa \zs_7 \zs_8 \zs_9 \zs_8 \zs_7 \zs_8$.
%\eqref{e si root diagrams}
%(and $\Delta^+ \setminus \Delta'(\eta)$ is shown in blue).
% = \mathsf{w}_{2,3} \mathsf{w}_{4,6} \mathsf{w}_{7,10}.$
% appearing in the results above:
\begin{align}
\label{e ex nsmac 1}
\tE_{\alpha} &=
 x_1 \Phi \spa \pi_{\sfc(1)} (x_1 \Phi \spa \pi_{\sfc(3)})^3 (x_1\Phi \spa \pi_{\sfc(6)})^{6} \cdot  1 \\
\label{e ex nsmac 2}
&= \chr_{\mathbf{x};\mu} \! \big( B_{\tau \sfc(1) (\tau \sfc(3))^3 (\tau \sfc(6))^{6} }(\Lambda_0) \big) \\
%?Get rid of usage of Z_ notation. yes
% Z_\ell \big( e^{-\Lambda_0} \tau D_{\sfc(1)} (\tau D_{\sfc(3)})^3 (\tau D_{\sfc(6)})^{6} (e^{\Lambda_0}) \big) \\
\label{e ex nsmac 3}
&= \chr_{\mathbf{x};\mu} \! \big( B_{(\tau \sfc(6))^6 (\tau \sfc(3))^3 (\tau \sfc(1))^{1} }(\Lambda_0) \big)\\
\label{e ex nsmac 4}
&= \chr_{\mathbf{x};\mu} \! \big( B_{(\tau \sfc(6))^4 \tau \sfc(5) \tau \sfc(4) \tau \sfc(3) \tau \sfc(3) \tau \sfc(2) \tau \sfc(1) }(\Lambda_0) \big)\\
\label{e ex nsmac 5}
&= H(\Delta'(\eta); \mu;  \tsfp(\alpha))  \\
\label{e ex nsmac 6}
&= \pi_{\tsfp(\alpha)} (x_1 \Phi \spa \pi_{\sfc(6)})^4 x_1 \Phi \spa \pi_{\sfc(5)} x_1 \Phi \spa \pi_{\sfc(4)} x_1 \Phi \spa \pi_{\sfc(3)} x_1 \Phi \spa \pi_{\sfc(3)} x_1 \Phi \spa \pi_{\sfc(2)} x_1 \Phi \spa \pi_{\sfc(1)} \cdot  1.
\end{align}
The formulas \eqref{e ex nsmac 1}--\eqref{e ex nsmac 5} come from \eqref{et Sanderson}, \eqref{et Sanderson 2}, \eqref{e proof nonsymmetric}, \eqref{e proof nonsymmetric}, and  \eqref{et nonsymmetric Macdonald}, respectively; the last equality holds by
%$\tE_{\alpha} = H(\Delta'(\eta); 1^\ell;  ) $
Theorem \ref{t character rotation formula}.
The left diagram below
gives a way of visualizing
\eqref{e ex nsmac 4}--\eqref{e ex nsmac 6} (in the style of Example \ref{ex big kat}),
the middle diagram corresponds to \eqref{e ex nsmac 3}, and the right to
\eqref{e ex nsmac 1}--\eqref{e ex nsmac 2}.
\ytableausetup{mathmode, boxsize=1.04em,centertableaux}
\begin{align*}
%\label{e si root diagrams}
\Tiny
\begin{ytableau}
*(black!12)&*(blue!48)&*(blue!48) &*(blue!48) & *(blue!48) & *(blue!48) &*(red!80)6&*(red!80)7&*(red!80)8&*(red!80)9\\
*(black!12)& *(black!12)      &*(blue!48)&*(blue!48) &*(blue!48) & *(blue!48)  &*(blue!48) &*(red!80)6&*(red!80)7&*(red!80)8&*(red!0)9\\
*(black!12)&*(black!12)&*(black!12)&*(blue!48)& *(blue!48)  &*(blue!48)  & *(blue!48) &*(blue!48) &*(red!80)6&*(red!80)7&*(red!0)8&*(red!0)9\\
*(black!12)&*(black!12)&*(black!12)&*(black!12)&*(blue!48)&*(blue!48)& *(blue!48) & *(blue!48) &*(blue!48) &*(red!80)6&*(red!0)7&*(red!0)8&*(red!0)9\\
*(black!12)&*(black!12)&*(black!12)&*(black!12)&*(black!12)&*(blue!48) &*(blue!48) &*(blue!48) & *(blue!48) &*(red!80)5 &*(red!0)6&*(red!0)7&*(red!0)8&*(red!0)9\\
*(black!12)&*(black!12)&*(black!12)&*(black!12)&*(black!12)&*(black!12)&*(blue!48) & *(blue!48) &*(blue!48)  & *(red!80) 4  &*(red!0)5 &*(red!0)6&*(red!0)7&*(red!0)8&*(red!0)9\\
*(black!12)&*(black!12)&*(black!12)&*(black!12)&*(black!12)&*(black!12)&*(black!12)&*(blue!48) & *(blue!48) &*(red!80)3& *(red!0) 4  &*(red!0)5 &*(red!0)6&*(red!0)7&*(red!0)8&*(red!0)9\\
*(black!12) &*(black!12)&*(black!12)&*(black!12)&*(black!12)&*(black!12)&*(black!12)& *(black!12) & *(blue!48) &*(blue!48) &*(red!0)3& *(red!0) 4  &*(red!0)5 &*(red!0)6&*(red!0)7&*(red!0)8&*(red!0)9\\
*(black!12)&*(black!12)&*(black!12)&*(black!12)&*(black!12)&*(black!12)&*(black!12)&*(black!12)&*(black!12)&*(blue!48)&*(red!0)2&*(red!0) 3& *(red!0) 4  &*(red!0)5 &*(red!0)6&*(red!0)7&*(red!0)8&*(red!0)9\\
*(black!12) &*(black!12)&*(black!12)&*(black!12)&*(black!12)&*(black!12)&*(black!12)&*(black!12)&*(black!12)&*(black!12)&*(red!0)1&*(red!0)2&*(red!0) 3& *(red!0) 4  &*(red!0)5 &*(red!0)6&*(red!0)7&*(red!0)8&*(red!0)9
\end{ytableau}
\qquad
\begin{ytableau}
*(black!12)&*(blue!0)&*(blue!0) &*(blue!0) & *(blue!0) & *(blue!0) &*(red!0)6&*(red!0)7&*(red!0)8&*(red!0)9\\
*(black!12)&*(black!12)&*(blue!0)&*(blue!0) &*(blue!0) & *(blue!0)  &*(blue!0) &*(red!0)6&*(red!0)7&*(red!0)8&*(red!0)9\\
*(black!12)&*(black!12)&*(black!12)&*(blue!0)& *(blue!0)  &*(blue!0)  & *(blue!0) &*(blue!0) &*(red!0)6&*(red!0)7&*(red!0)8&*(red!0)9\\
*(black!12)&*(black!12)&*(black!12)&*(black!12)&*(blue!0)&*(blue!0)& *(blue!0) & *(blue!0) &*(blue!0) &*(red!0)6&*(red!0)7&*(red!0)8&*(red!0)9\\
*(black!12)&*(black!12)&*(black!12)&*(black!12)&*(black!12)&*(blue!0) &*(blue!0) &*(blue!0) & *(blue!0) &*(blue!0) &*(red!0)6&*(red!0)7&*(red!0)8&*(red!0)9\\
*(black!12)&*(black!12)&*(black!12)&*(black!12)&*(black!12)&*(black!12)&*(blue!0) & *(blue!0) &*(blue!0)  & *(blue!0)   &*(blue!0)&*(red!0)6&*(red!0)7&*(red!0)8&*(red!0)9\\
*(black!12)&*(black!12)&*(black!12)&*(black!12)&*(black!12)&*(black!12)&*(black!12)&*(blue!0) & *(blue!0) &*(red!0)3& *(red!0) 4  &*(red!0)5 &*(red!0)6&*(red!0)7&*(red!0)8&*(red!0)9\\
*(black!12)&*(black!12)&*(black!12)&*(black!12)&*(black!12)&*(black!12)&*(black!12)&*(black!12) & *(blue!0) &*(blue!0) &*(red!0)3& *(red!0) 4  &*(red!0)5 &*(red!0)6&*(red!0)7&*(red!0)8&*(red!0)9\\
*(black!12)&*(black!12)&*(black!12)&*(black!12)&*(black!12)&*(black!12)&*(black!12)&*(black!12)&*(black!12)&*(blue!0)&*(blue!0) &*(red!0) 3& *(red!0) 4  &*(red!0)5 &*(red!0)6&*(red!0)7&*(red!0)8&*(red!0)9\\
*(black!12) &*(black!12)&*(black!12)&*(black!12)&*(black!12)&*(black!12)&*(black!12)&*(black!12)&*(black!12)&*(black!12)&*(red!0)1&*(red!0)2&*(red!0) 3& *(red!0) 4  &*(red!0)5 &*(red!0)6&*(red!0)7&*(red!0)8&*(red!0)9
\end{ytableau}
\qquad
\begin{ytableau}
*(black!12) &*(red!0) 1&*(red!0) 2 &*(red!0)  3 & *(red!0)  4  &*(red!0) 5 &*(red!0) 6&*(red!0) 7&*(red!0) 8&*(red!0) 9\\
*(black!12)&*(black!12)&*(blue!0)&*(blue!0)&*(red!0)  3 & *(red!0)  4  &*(red!0) 5 &*(red!0) 6&*(red!0) 7&*(red!0) 8&*(red!0) 9\\
*(black!12)&*(black!12)&*(black!12)&*(blue!0)& *(blue!0)  &*(red!0)  3 & *(red!0)  4  &*(red!0) 5 &*(red!0) 6&*(red!0) 7&*(red!0) 8&*(red!0) 9\\
*(black!12)&*(black!12)&*(black!12)&*(black!12)&*(blue!0)&*(blue!0)& *(red!0) 3 & *(red!0) 4  &*(red!0) 5 &*(red!0) 6&*(red!0) 7&*(red!0) 8&*(red!0) 9\\
*(black!12)&*(black!12)&*(black!12)&*(black!12)&*(black!12)&*(blue!0) &*(blue!0) &*(blue!0)  & *(blue!0) &*(blue!0) &*(red!0) 6&*(red!0) 7&*(red!0) 8&*(red!0) 9\\
*(black!12)&*(black!12)&*(black!12)&*(black!12)&*(black!12)&*(black!12)&*(blue!0) & *(blue!0) &*(blue!0)  & *(blue!0) &*(blue!0) &*(red!0) 6&*(red!0) 7&*(red!0) 8&*(red!0) 9\\
*(black!12)&*(black!12)&*(black!12)&*(black!12)&*(black!12)&*(black!12)&*(black!12)&*(blue!0) & *(blue!0) &*(blue!0)&*(blue!0) &*(blue!0) &*(red!0) 6&*(red!0) 7&*(red!0) 8&*(red!0) 9\\
*(black!12)&*(black!12)&*(black!12)&*(black!12)&*(black!12)&*(black!12)&*(black!12)&*(black!12) & *(blue!0) &*(blue!0)&*(blue!0)&*(blue!0)&*(blue!0) &*(red!0) 6&*(red!0) 7&*(red!0) 8&*(red!0) 9\\
*(black!12)&*(black!12)&*(black!12)&*(black!12)&*(black!12)&*(black!12)&*(black!12)&*(black!12)&*(black!12)&*(blue!0) &*(blue!0)&*(blue!0)&*(blue!0)&*(blue!0)&*(red!0) 6&*(red!0) 7&*(red!0) 8&*(red!0) 9\\
*(black!12)&*(black!12)&*(black!12)&*(black!12)&*(black!12)&*(black!12)&*(black!12)&*(black!12)&*(black!12)& *(black!12) &*(blue!0)&*(blue!0)&*(blue!0)&*(blue!0)&*(blue!0)&*(red!0) 6&*(red!0) 7&*(red!0) 8&*(red!0) 9
\end{ytableau}
\end{align*}
\end{example}

%\subsection{Connection to combinatorics of Lascoux and Shimozono-Weyman}
%\subsection{Symmetrization to classical tableau combinatorics}
\subsection{Symmetrization to the Lascoux/Shimozono-Weyman formula}
\label{ss Symmetrization to the Lascoux/Shimozono-Weyman formula}

%done, checked carefully.  only a few notations/remark that might be updated later
\begin{proof}[Proof of Theorem \ref{t t0 nsMac symmetric schur pos}]
By Theorem \ref{t nonsymmetric Macdonald} and Definition \ref{d HH gamma Psi},
$\pi_{\mathsf{w}_0}\tE_{\alpha} = \pi_{\mathsf{w}_0} H(\Delta_\ell'(\eta); \mu; \tsfp(\alpha))$
$= H(\Delta_\ell'(\eta); \mu; \mathsf{w}_0)$.
Hence
%and 0-Hecke relations
the first equality
of \eqref{et t0 nsMac symmetric schur pos}
will follow from
\begin{align}
\label{e proof Delta eq Deltaprime}
H(\Delta_\ell'(\eta); \mu; \mathsf{w}_0) = H(\Delta_\ell(\eta); \mu; \mathsf{w}_0).
\end{align}
This identity can be seen
by starting with  $\Delta_\ell'(\eta)$ and filling
%adding one root at a time to  $\Delta_\ell'(\eta)$,
in the trapezoidal regions of $\Delta_\ell(\eta) \setminus \Delta_\ell'(\eta)$ one root at a time,
using \cite[Lemma 8.9]{BMPS2} to show that the corresponding
Catalan functions remain the same
(this is essentially the same argument used to prove \cite[Lemma 10.1]{BMPS2}).
The second equality
of \eqref{et t0 nsMac symmetric schur pos}
holds by Theorem \ref{t kat conjecture resolution}.

Now to prove $\F_{\mathsf{w}_0} \B^{\mu; \mathbf{w}}=\B^{\mu; (\mathsf{w}_0, \ns(\Delta_\ell(\eta)))}$,
we first apply Proposition \ref{p lambda constant crystal} to obtain \break
 $\AGD\big(\mu; (\mathsf{w}_0, \ns(\Delta_\ell'(\eta))) \big) = B_{w}(\Lambda_0)$
 and
  $\AGD\big(\mu; (\mathsf{w}_0, \ns(\Delta_\ell(\eta))) \big) = B_{y}(\Lambda_0)$,
where
\[ w  := \, \mathsf{w}_0 \tau \ns(\Delta_\ell'(\eta))_1   \cdots \tau \ns(\Delta_\ell'(\eta))_{m-1} \tau,
\quad y := \, \mathsf{w}_0 \tau \ns(\Delta_\ell(\eta))_1 \tau \ns(\Delta_\ell(\eta))_2 \cdots \tau \ns(\Delta_\ell(\eta))_{m-1} \tau.
\]
%?useful Since $w \le y$ in Bruhat order. true but sort of round about
Considering the given expressions for  $w$ and  $y$ as words in the $\zs_i$'s and $\tau$,
we see the expression for $w$ is a subword of that of $y$, and hence
%on  $\widetilde{S}_\ell$,
$B_{w}(\Lambda_0)= \F_w \{ u_{\Lambda_0} \}  \subset  \F_y \{ u_{\Lambda_0} \} = B_{y}(\Lambda_0)$.
But we also have $\chr_{\mathbf{x};\mu}(B_{w}(\Lambda_0)) = \chr_{\mathbf{x};\mu}(B_{y}(\Lambda_0))$
by Corollary \ref{c kat conjecture resolution 0} and \eqref{e proof Delta eq Deltaprime}.
So equality $B_{w}(\Lambda_0) = B_{y}(\Lambda_0)$ holds.

Let  $v$ and $\mathbf{w}$ be as in Theorem \ref{t Sanderson} with  $\mathsf{z} = \tsfp(\alpha)$.
Then $\mathsf{w}_0 \spa v \spa  \mathsf{w}_0 \spa  = \spa w \spa \mathsf{w}_0$ by \eqref{e proof nonsymmetric}, and hence
 %??useful w s differ here and there by omega0 in front
$\F_{\mathsf{w}_0} B_v(\Lambda_0) = B_{w}(\Lambda_0) = B_{y}(\Lambda_0)$ (by Proposition \ref{p lemma 4.3 Naoi}).
%??useful using that \F_i "satisfy" 0-Hecke relations and probably several other places nearby...now cited here
By Theorem \ref{t intro KR to affine iso subsets}, the equality of AGD crystals $\F_{\mathsf{w}_0} B_v(\Lambda_0) = B_{y}(\Lambda_0)$
implies that the corresponding DARK crystals are equal (as subsets of  $\B^\mu$): $\F_{\mathsf{w}_0}\B^{\mu; \mathbf{w}} = \B^{\mu; (\mathsf{w}_0, \ns(\Delta_\ell(\eta)))}$.
Equating the  $U_q(\gl_\ell)$-highest weights of these crystals then gives the first equality of \eqref{et t0 nsMac symmetric schur pos 2} (with the help of Proposition \ref{p easy kat vs gen kat});
the connection to  $\kattype$ follows from
Proposition \ref{p recover SW def} and
Theorem~\ref{t SW and Lascoux Koskta}.
\end{proof}

The companion  result to Theorem \ref{t t0 nsMac symmetric schur pos} for  $m > \ell$ is
more technical and requires a crystal version of setting $x_{\ell+1} = \cdots = x_m=0$, which we now describe.

Let  $B$ be a $U_q(\gl_m)$-crystal which is a isomorphic to a disjoint union of highest weight crystals  $B^\gl(\nu)$ for  $\nu = (\nu_1 \ge \cdots \ge \nu_m \ge 0)$.
The weight function takes values in  $\ZZ^m_{\ge 0}$ and we write
$\wt(b) = (\wt_1(b), \dots, \wt_m(b))$ for the entries of  $\wt(b)$.
Let  $S \subset B$ be isomorphic to a disjoint union of $U_q(\gl_m)$-Demazure crystals.
Let  $\Res_J B$ denote the $U_q(\gl_\ell)$-restriction of  $B$ corresponding to Dynkin node subset  $J =[\ell-1] \subset [m-1]$ (see \S\ref{ss restrict Demazure}).
%and let $\Res_J S$ denote the set  $S$ regarded as a subset of $\Res_J B$. ?said already in res thm
By Theorem~\ref{t restrict Demazure}  $\Res_J S$ is isomorphic to  a disjoint union of  $U_q(\gl_\ell)$-Demazure crystals.
Define
\begin{align}
\Rez^m_\ell S = \big\{b \in S \mid \wt_i(b) = 0 \text{ for all } i \in [\ell+1, {m}]\big\} \, \subset \Res_J S.
\end{align}
Since $\cf_j$,  $j \in [\ell-1]$, fixes $\wt_i$ for  $i > \ell$,
$\Rez^m_\ell S$ is also a disjoint union of $U_q(\gl_\ell)$-Demazure crystals and its character is obtained from that of  $S$ by setting $x_{\ell+1} = \cdots = x_m=0$.

%As we work with both $\zH_\ell$ and $\zH_m$
Below we work with  $\zH_m$ and its submonoid  $\zH_\ell$ generated by  $\zs_1,\dots, \zs_{\ell-1}$ ($\ell \le m$);
denote by $\iota \colon \zH_\ell \hookrightarrow \zH_m$ the inclusion
and $\wnotb{1}{m}$ and $\wnotb{1}{\ell}$ their longest elements.

%We denote by $\wnotb{1}{\ell}$ (resp. $\wnotb{1}{m}$) the longest element of $\zH_\ell$ (resp. $\zH_m$).
%Recall that  $\wnotb{1}{m}$ denotes the longest element of  $\zH_m$.

\begin{lemma}
\label{l res crystal}
Let  $m \ge \ell$ and  $S \subset B$ as above. Then
%Let  $S$ be a subset of a $U_q(\gl_m)$-crystal which is a disjoint union of Demazure crystals.  Then
\begin{align}
\Rez^m_\ell \F_{\wnotb{1}{m}} S = \F_{\wnotb{1}{\ell}} \Rez^m_\ell S.
\end{align}
\end{lemma}
\begin{proof}
Let  $\zs_{i_1} \cdots \zs_{i_k}$ be a reduced
word for  $\wnotb{1}{m}$.
For any  $b \in B$,  $\sum_{i = \ell+1}^{m} \wt_i(\cf_\ell(b)) > \sum_{i = \ell+1}^{m} \wt_i(b)$ and
 for  $j \ne \ell$, $\sum_{i = \ell+1}^{m} \wt_i(\cf_j(b)) = \sum_{i = \ell+1}^{m} \wt_i(b)$;
also,  $\sum_{i = \ell+1}^{m} \wt_i(b) = 0$ implies $\cf_j(b) = 0$ for  $j > \ell$.
It follows that an arbitrary element  $\cf_{i_1}^{a_{i_1}} \cdots \cf_{i_k}^{a_{i_k}} (b)$ of $\F_{\wnotb{1}{m}} S$ lies in
$\Rez^m_\ell \F_{\wnotb{1}{m}} S$ if and only if  $b\in \Rez^m_\ell S$ and $a_{i_j} = 0$ whenever  $i_j \ge \ell$.
Hence  $\Rez^m_\ell \F_{\wnotb{1}{m}} S = \F_{v} \Rez^m_\ell S$ where  $v$ is the product of the  $\zs_{i_j}$ with  $i_j < \ell$.
Since $\zs_{i_1} \cdots \zs_{i_k}$ contains a reduced word for  $\wnotb{1}{\ell}$, it follows from
Remark \ref{r gl demazure crystals} that  $\F_v \Rez^m_\ell S= \F_{\wnotb{1}{\ell}} \Rez^m_\ell S$.
 \end{proof}

\begin{lemma}
\label{l restrict katab}
Given $\mathbf{n} = (n_1,\dots, n_{p-1}) \in [\ell]^{p-1}$  and $\mathsf{z} \in \zH_\ell$,
% and $\iota(\mathsf{z})$ be the corresponding element of  $\zH_m$, where
%$\iota \colon \zH_\ell \hookrightarrow \zH_m$,  $\zs_i \mapsto \zs_i$ for  $i\in [\ell-1]$.
%% is the inclusion of the submonoid generated by $\zs_1,\dots,\zs_{\ell-1}$.
%%but regarded as an element of $\zH_m$ (viewing $\zH_\ell$ as the submonoid of  $\zH_m$ generated by $\zs_1,\dots,\zs_{\ell-1}$).
%%, which we can also regard as an element of  $\zH_m$.
define the tuples $\mathbf{w} =$ $(\mathsf{z}, \zs_{\ell-1} \cdots \zs_{n_1},\dots, \zs_{\ell-1} \cdots \zs_{n_{p-1}}) \in (\zH_\ell)^p$
and  $\widetilde{\mathbf{w}} = (\iota(\mathsf{z}), \zs_{m-1} \cdots \zs_{n_1},\dots, \zs_{m-1} \cdots \zs_{n_{p-1}}) \in (\zH_m)^p$.
%Let  $\mathbf{w} \in (\zH_\ell)^p$ be as in Theorem \ref{t Sanderson} with  $\mathsf{z} =\tsfp(\alpha)$
%and  $\widetilde{\mathbf{w}} \in (\zH_m)^p$ be the same but for  $m$ in place of  $\ell$ and  $\mathsf{z} = \tsfp(\alpha, 0^{m-\ell})$.
Then for any partition  $\mu = (\mu_1,\dots, \mu_p)$,
$\Rez^m_\ell \B^{\mu; \widetilde{\mathbf{w}}}
= \B^{\mu; \mathbf{w}}.$
\end{lemma}
\begin{proof}
By Theorem \ref{t kat and inv crystal}, this is equivalent to showing that $T \in \Tabloids_\ell$ is  $\mathbf{w}$-katabolizable if and only if  $\widetilde{T}$ is  $\widetilde{\mathbf{w}}$-katabolizable, where  $\widetilde{T}$ is the same as  $T$ but regarded as an element of  $\Tabloids_m$.
One checks easily by induction that these two katabolism computations are
essentially identical, the only difference being that
%An easy induction shows that after each application of
whenever $\kat$ is applied in the $\mathbf{w}$-katabolism algorithm,
%is applied in the $\widetilde{\mathbf{w}}$-katabolism algorithm,
it matches the application of $P_{\zs_{\ell} \cdots \zs_{m-1}} \circ \kat$ in the
$\widetilde{\mathbf{w}}$-katabolism algorithm;
%since the latter cycles the 1-st row to the  $m$-th, and then moves the  $m$-th row $m$-th row to become the  $\ell$-th row;
this holds because at every step (in either algorithm) just before  $\kat$ is applied, the input tabloid is empty in rows  $\ell+1, \dots, m$.
\end{proof}

\begin{theorem}
\label{t t0 nsMac symmetric schur pos B}
%Let $\alpha \in \ZZpl$ and $\eta = (\eta_1,\dots, \eta_k) =  (\alpha^+)'$. Assume $m = |\alpha| > \ell$. Then
Maintain the notation of Theorem \ref{t Sanderson};
also set $m = p = |\alpha|$  and  assume  $m > \ell$.
Let  $\SYT^m_\ell = \SSYT_\ell(\mu)$, the  $\SYT$ with  $m$ boxes and at most  $\ell$ rows.
Then
\begin{align}
\label{et t0 nsMac symmetric schur pos B}
\pi_{\wnotb{1}{\ell}}\tE_{\alpha} = H(\Delta(\eta); \mu; \wnotb{1}{m})|_{x_{\ell+1}\spa = \spa \cdots\spa = \spa x_m =\spa  0} \, =
 \!\!\! \sum_{\substack{U \in \SYT^m_\ell \\ \text{$U$ is $\nr(\Delta(\eta))$-katabolizable} }} \!\!\!\!\!\!\! q^{\charge(U)}s_{\sh(U)}(x_1, \dots, x_\ell).
\end{align}
Moreover, $\F_{\wnotb{1}{\ell}} \B^{\mu; \mathbf{w}}
= \Rez^m_\ell \B^{\mu; (\wnotb{1}{m}, \ns(\Delta(\eta)))}$ as $U_q(\gl_\ell)$-crystals, and so
\begin{align}
\label{et t0 nsMac symmetric schur pos 2 B}
&\{P(T) \mid \text{ $T \in \RowFrank_\ell(\mu)$ is extreme $\mathbf{w}$-katabolizable} \} =\\
%&\{U \in \SSYT_\ell(\mu) \mid  \text{  $U$ is $(\wnotb{1}{m}, \ns(\Delta(\eta)))$-katabolizable} \}\\
&\{U \in \SYT^m_\ell \mid  \text{  $U$ is $\nr(\Delta(\eta))$-katabolizable} \}
=\{U \in \SYT^m_\ell \mid \kattype(U^t) \gd \eta\}.
\notag
\end{align}
\end{theorem}
\begin{proof}
%Let  $\mathbf{w} \in (\zH_\ell)^p$ be as in Theorem \ref{t Sanderson}
By Theorem \ref{t t0 nsMac symmetric schur pos}, applied with $m$ in place of  $\ell$ and $(\alpha,0^{m-\ell})$ in place of  $\alpha$,
$\F_{ \wnotb{1}{m} } \B^{ \mu; \widetilde{\mathbf{w}} }$ $=\B^{ \mu; (\wnotb{1}{m}, \ns(\Delta(\eta))) }$
as $U_q(\gl_m)$-crystals, where $\widetilde{\mathbf{w}} =
(\iota({\mathsf{z}}), \zs_{m-1} \cdots \zs_{n_1},\dots, \zs_{m-1} \cdots \zs_{n_{p-1}}) \in (\zH_m)^p$
%is as in Lemma \ref{l restrict katab}
with $\mathbf{n} = \eta_k^{\eta_k}\cdots \eta_2^{\eta_2}\eta_1^{\eta_1-1}$.
%and  $\tilde{\mathsf{z}}$ equal to $\mathsf{z}$, the first part of  $\mathbf{w}$
%(we are using that  $z (\alpha^+) = \alpha$ implies  that $z((\alpha,0^{m-\ell})^+) = (\alpha,0^{m-\ell})$, regarding  $z$  as an element of  $\zH_m$).
(Here $\mathsf{z}$ is one of the inputs to Theorem \ref{t t0 nsMac symmetric schur pos B}, which can be any element of  $\zH_\ell$ satisfying $\mathsf{z} \spa \alpha^+ = \alpha$. In this application of
Theorem \ref{t t0 nsMac symmetric schur pos} we must choose a $\tilde{\mathsf{z}} \in \zH_m$ such that  $\tilde{\mathsf{z}} \spa (\alpha,0^{m-\ell})^+ = (\alpha,0^{m-\ell})$; we choose  $\tilde{\mathsf{z}} = \iota(\mathsf{z})$.)
Applying $\Rez^m_\ell$ to both sides and then using Lemmas \ref{l res crystal} and \ref{l restrict katab} yields
the ``moreover'' statement:
\begin{align*}
\Rez^m_\ell \B^{\mu; (\wnotb{1}{m}, \ns(\Delta(\eta)))} = \Rez^m_\ell \F_{\wnotb{1}{m}} \B^{\mu; \widetilde{\mathbf{w}}}
= \F_{\wnotb{1}{\ell}} \Rez^m_\ell \B^{\mu; \widetilde{\mathbf{w}}}
= \F_{\wnotb{1}{\ell}} \B^{\mu; \mathbf{w}},
\end{align*}
and the consequence
\eqref{et t0 nsMac symmetric schur pos 2 B} follows much like
the proof of \eqref{et t0 nsMac symmetric schur pos 2}.
Next, it follows from Theorem \ref{t Sanderson} that the charge weighted character of  $\F_{\wnotb{1}{\ell}} \B^{\mu; \mathbf{w}}$ is
$\pi_{\wnotb{1}{\ell}} \tE_{\alpha}$
%a couple steps hidden here in this application of Theorem sanderso, but they're easy so maybe okay \tE_{\wnotb{1}{\ell}(\alpha)} =
and that of  $\Rez^m_\ell \F_{\wnotb{1}{m}} \B^{\mu; \widetilde{\mathbf{w}}}$ is
$(\pi_{\wnotb{1}{m}} \tE_{(\alpha, 0^{m-\ell})}) |_{x_{\ell+1}\spa = \spa \cdots\spa = \spa x_m =\spa  0}$.
Hence
$\big(\pi_{\wnotb{1}{m}} \tE_{(\alpha,0^{m-\ell})}\big)|_{x_{\ell+1}\spa = \spa \cdots\spa = \spa x_m =\spa  0}$  $=$ $\pi_{\wnotb{1}{\ell}}\tE_{\alpha}$.
This fact given, \eqref{et t0 nsMac symmetric schur pos B} is obtained by applying Theorem \ref{t t0 nsMac symmetric schur pos} (specifically \eqref{et t0 nsMac symmetric schur pos}),
with $(\alpha,0^{m-\ell})$ in place of  $\alpha$ and then setting $x_{\ell+1} = \dots = x_m = 0$.
\end{proof}

\vspace{2mm}
\noindent
\textbf{Acknowledgments.}
We are grateful to  Peter Littelmann and Wilberd van der Kallen for the proof of Theorem \ref{t restrict Demazure}.
We thank Mark Haiman, George Seelinger, Mark Shimozono, and Weiqiang Wang for helpful discussions
and Elaine~So for help typing and typesetting figures.

%\setlength{\bibsep}{0pt plus .3ex}
%\bibliographystyle{plain}
%\bibliography{mycitations}
%\def\cprime{$'$} \def\cprime{$'$} \def\cprime{$'$}
\def\cprime{$'$} \def\cprime{$'$} \def\cprime{$'$}

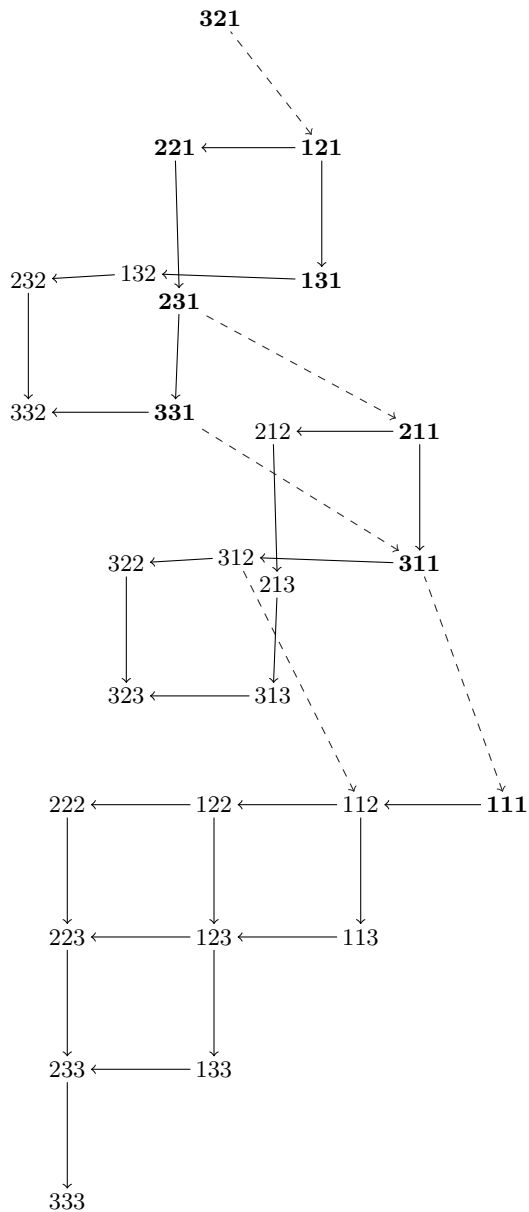
\begin{figure}
        \centerfloat
        \vspace{.2cm}
\begin{tikzpicture}[xscale = 1.95, yscale = 1.76]
\tikzstyle{vertex}=[inner sep=0pt, outer sep=2pt]
\tikzstyle{framedvertex}=[inner sep=3pt, outer sep=4pt, draw=gray]
\tikzstyle{aedge} = [draw, <-, thin, black]
\tikzstyle{dotaedge} = [draw, thin, <-,black!86, dashed]
\tikzstyle{edge} = [draw, thick, -,black]
\tikzstyle{doubleedge} = [draw, thick, double distance=1pt, -,black]
\tikzstyle{hiddenedge} = [draw=none, thick, double distance=1pt, -,black]
\tikzstyle{dashededge} = [draw, very thick, dashed, black]
\tikzstyle{LabelStyleH} = [text=black, anchor=south]
\tikzstyle{LabelStyleHn} = [text=black, anchor=north]
\tikzstyle{LabelStyleV} = [text=black, anchor=east]

\begin{scope}[xshift= 24, yshift =-24.000000]
\node[vertex] (v27) at (0.200000,1.35000){\scriptsize$\mathbf{321}$};
\end{scope}

\begin{scope}[xshift=-3, yshift =-95.0000]
\node[vertex] (v11) at (1.84000,2.87000) {\scriptsize$\mathbf{121}$};
\node[vertex] (v12) at (1.84000,1.87000) {\scriptsize$\mathbf{131}$};
\node[vertex] (v13) at (0.840000,2.87000){\scriptsize$\mathbf{221}$};
\node[vertex] (v14) at (0.5900,1.92000)   {\scriptsize$132$};
\node[vertex] (v15) at (0.87000,1.7100)  {\scriptsize$\mathbf{231}$};
\node[vertex] (v16) at (0.840000,0.87000){\scriptsize$\mathbf{331}$};
\node[vertex] (v17) at (-0.160000,1.8700){\scriptsize$232$};
\node[vertex] (v18) at (-0.160000,0.8700){\scriptsize$332$};
\end{scope}

\begin{scope}[xshift = 16,yshift =-156.000]

\node[vertex] (v19) at (1.84000,2.87000)  {\scriptsize$\mathbf{211}$};
\node[vertex] (v20) at (1.84000,1.87000)  {\scriptsize$\mathbf{311}$};
\node[vertex] (v21) at (0.840000,2.87000) {\scriptsize$212$};
\node[vertex] (v22) at (0.5900,1.92000)    {\scriptsize$312$};
\node[vertex] (v23) at (0.87000,1.7100)   {\scriptsize$213$};
\node[vertex] (v24) at (0.840000,0.87000) {\scriptsize$313$};
\node[vertex] (v25) at (-0.160000,1.8700) {\scriptsize$322$};
\node[vertex] (v26) at (-0.160000,0.8700) {\scriptsize$323$};
\end{scope}

\begin{scope}[yshift =-240.000]

\node[vertex] (v1) at (3.00000,3.00000){\scriptsize$\mathbf{111}$};
\node[vertex] (v2) at (2.00000,3.00000){\scriptsize$112$};
\node[vertex] (v3) at (2.00000,2.00000){\scriptsize$113$};
\node[vertex] (v4) at (1.00000,3.00000){\scriptsize$122$};
\node[vertex] (v5) at (1.00000,2.00000){\scriptsize$123$};
\node[vertex] (v6) at (1.00000,1.00000){\scriptsize$133$};
\node[vertex] (v7) at (0.000000,3.00000){\scriptsize$222$};
\node[vertex] (v8) at (0.000000,2.00000){\scriptsize$223$};
\node[vertex] (v9) at (0.000000,1.00000){\scriptsize$233$};
\node[vertex] (v10) at (0.000000,0.000000){\scriptsize$333$};
\end{scope}

%0 edges
\draw[dotaedge] (v1) to node[LabelStyleH]{\Tiny$  $} (v20);
\draw[dotaedge] (v2) to node[LabelStyleH]{\Tiny$  $} (v22);
%\draw[dotaedge] (v3) to node[LabelStyleH]{\Tiny$  $} (v24);
%\draw[dotaedge] (v4) to node[LabelStyleH]{\Tiny$  $} (v25);
\draw[dotaedge] (v11) to node[LabelStyleH]{\Tiny$  $} (v27);
\draw[dotaedge] (v19) to node[LabelStyleH]{\Tiny$  $} (v15);
\draw[dotaedge] (v20) to node[LabelStyleV]{\Tiny$  $} (v16);
%\draw[dotaedge] (v21) to node[LabelStyleV]{\Tiny$  $} (v17);
%\draw[dotaedge] (v22) to node[LabelStyleV]{\Tiny$  $} (v18);

\draw[aedge] (v21) to node[LabelStyleH]{\Tiny$  $} (v19);
\draw[aedge] (v22) to node[LabelStyleH]{\Tiny$  $} (v20);
\draw[aedge] (v20) to node[LabelStyleV]{\Tiny$  $} (v19);
\draw[aedge] (v15) to node[LabelStyleV]{\Tiny$  $} (v13);
\draw[aedge] (v2) to node[LabelStyleH]{\Tiny$  $} (v1);
\draw[aedge] (v3) to node[LabelStyleV]{\Tiny$  $} (v2);
\draw[aedge] (v4) to node[LabelStyleH]{\Tiny$  $} (v2);
\draw[aedge] (v5) to node[LabelStyleV]{\Tiny$  $} (v4);
\draw[aedge] (v5) to node[LabelStyleH]{\Tiny$  $} (v3);
\draw[aedge] (v6) to node[LabelStyleV]{\Tiny$  $} (v5);
\draw[aedge] (v7) to node[LabelStyleH]{\Tiny$  $} (v4);
\draw[aedge] (v8) to node[LabelStyleH]{\Tiny$  $} (v5);
\draw[aedge] (v8) to node[LabelStyleV]{\Tiny$  $} (v7);
\draw[aedge] (v9) to node[LabelStyleV]{\Tiny$  $} (v8);
\draw[aedge] (v9) to node[LabelStyleH]{\Tiny$  $} (v6);
\draw[aedge] (v10) to node[LabelStyleV]{\Tiny$  $} (v9);
\draw[aedge] (v12) to node[LabelStyleV]{\Tiny$  $} (v11);
\draw[aedge] (v13) to node[LabelStyleH]{\Tiny$  $} (v11);
\draw[aedge] (v14) to node[LabelStyleH]{\Tiny$  $} (v12);
\draw[aedge] (v16) to node[LabelStyleV]{\Tiny$  $} (v15);
\draw[aedge] (v17) to node[LabelStyleH]{\Tiny$  $} (v14);
\draw[aedge] (v18) to node[LabelStyleH]{\Tiny$  $} (v16);
\draw[aedge] (v18) to node[LabelStyleV]{\Tiny$  $} (v17);
\draw[aedge] (v23) to node[LabelStyleV]{\Tiny$  $} (v21);
\draw[aedge] (v24) to node[LabelStyleV]{\Tiny$  $} (v23);
\draw[aedge] (v25) to node[LabelStyleH]{\Tiny$  $} (v22);
\draw[aedge] (v26) to node[LabelStyleH]{\Tiny$  $} (v24);
\draw[aedge] (v26) to node[LabelStyleV]{\Tiny$  $} (v25);
\end{tikzpicture}
\vspace{.4mm}
\captionsetup{width=\linewidth}
\caption{\label{f last fig}
{\small The tensor product of KR crystals  $\B^{\mu} = B^{1,1} \tsr B^{1,1} \tsr B^{1,1}$, which is also the DARK crystal  $\B^{\mu;\mathbf{w}}$ for  $\mu = (1,1,1)$ and $\mathbf{w} = (\zs_2\zs_1,\zs_2\zs_1, \zs_2\zs_1)$.
Its charge weighted character is the modified Hall-Littlewood polynomial
$H_{111}(\mathbf{x};q) = s_{111} + (q + q^2) s_{21} + \spa q^3 s_{3} =
\kappa_{111} + \spa (q + q^2) \kappa_{012} + \spa q^3\kappa_{003}$.
Horizontal and vertical arrows give the  $\cf_1, \cf_2$ edges, respectively.
The dashed arrows are the $\cf_0$-edges of $ \B^{\mu} \tsr B(\Lambda_0)$
which have both ends in 
$\B^{\mu;\mathbf{w}} \tsr u_{\Lambda_0} \subset \B^{\mu} \tsr B(\Lambda_0)$.
Since  $\mu$ is constant, the corresponding
generalized Demazure crystal
$\AGD(\mu;\mathbf{w}) = \Theta_\mu (B^{\mu;\mathbf{w}} \tsr u_{\Lambda_0})$
(via Theorem \ref{t intro KR to affine iso subsets})
is an actual affine Demazure crystal, namely  $\F_{\zs_2\zs_1 \tau \zs_2\zs_1 \tau \zs_2\zs_1} \{u_{\Lambda_1}\} \subset B(\Lambda_0)$.
\newline \vspace{-3.3mm} \newline 
%\hphantom{as} WARNING arxiv doesnt compile this...weird
\mbox{\ \ \, Shown} bold is the DARK crystal
$\B^{\mu; \mathbf{v}} =
\big(\twist{\tau} \F_{\zs_2 \zs_1} (\twist{\tau} \F_{\zs_2\zs_1}(\sfb_1)  \tsr  \sfb_1) \big) \tsr   \sfb_1$, for  $\mu = (1,1,1)$,  $\mathbf{v}= (\idelm,\zs_2\zs_1, \zs_2\zs_1)$,
which has charge weighted character
$\kappa_{111} + \spa q \spa \kappa_{102} + \spa q^2\kappa_{201} + \spa q^3\kappa_{300}$,
also equal to the  $t=0$ nonsymmetric Macdonald polynomial  $\tE_{300}(\mathbf{x};q)$ and the
nonsymmetric Catalan function  $H(\Delta^+;\mu;\zs_2)$.
Again, the corresponding
generalized Demazure crystal $\AGD(\mu;\mathbf{v})$ $= \Theta_\mu (B^{\mu;\mathbf{v}} \tsr u_{\Lambda_0}) =
\F_{\tau \zs_2\zs_1 \tau \zs_2\zs_1} \{ u_{\Lambda_1}\} \subset B(\Lambda_0)$
is an affine Demazure crystal.}
}
\vspace{-0mm}
\end{figure}

\end{document}